\documentclass{amsart}

\usepackage{amssymb}
\usepackage{amscd}
\usepackage{epsfig}
\usepackage{color}

\usepackage{mathtools}

\usepackage[setpagesize=false,hypertexnames=false]{hyperref}

\usepackage{ascmac}

\usepackage[all]{xy}
\usepackage{time}

\usepackage[OT2,T1]{fontenc}
\DeclareSymbolFont{cyrletters}{OT2}{wncyr}{m}{n}
\DeclareMathSymbol{\Sha}{\mathalpha}{cyrletters}{"58}

\usepackage{stackengine}
\stackMath
\newcommand\undertilde[2][1]{%
 \def\useanchorwidth{T}%
  \ifnum#1>1%
    \stackunder[0pt]{\tenq[\numexpr#1-1\relax]{#2}}{\scriptscriptstyle\sim}%
  \else%
    \stackunder[1pt]{#2}{\scriptscriptstyle\sim}%
  \fi%
}

\usepackage{yhmath}
\usepackage {wasysym}

\DeclareMathAlphabet{\mathpzc}{OT1}{pzc}{m}{it}

\usepackage{ulem}

\title[Associators in mould theory]
{
Associators 
in mould theory
}

\keywords{Moulds, associators}
\subjclass[2020]{Primary~16T05, Secondary~11M32}

\author{Hidekazu Furusho}
\author{Minoru Hirose}
\author{Nao Komiyama}

\address{Graduate School of Mathematics, Nagoya University,
Furo-cho, Chikusa-ku, Nagoya, 464-8602, Japan}
\email{furusho@math.nagoya-u.ac.jp}

\address{Graduate School of Science and Engineering, Kagoshima University, 1-21-35 Korimoto, Kagoshima, Kagoshima 890-0065, Japan}
\email{hirose@sci.kagoshima-u.ac.jp}

\address{Department of Mathematics, Josai University, 1-1 Keyakidai, Sakado, Saitama 350-0295, Japan}
\email{nkomiyama@josai.ac.jp}

\date{June 23, 2026}

\newtheorem{thm}{Theorem}
\newtheorem{lem}[thm]{Lemma}
\newtheorem{cor}[thm]{Corollary}
\newtheorem{prop}[thm]{Proposition}

\theoremstyle{definition}
\newtheorem{rem}[thm]{Remark}
\newtheorem{defn}[thm]{Definition}
\newtheorem{thm-defn}[thm]{Theorem-Definition}
\newtheorem{eg}[thm]{Example}
\newtheorem{nota}[thm]{Notation}

{\theoremstyle{remark} }

\numberwithin{equation}{section}
\newcommand{\Q}{\mathbb{Q}}
\newcommand{\C}{\mathbb{C}}
\newcommand{\R}{\mathbb{R}}
\newcommand{\Z}{\mathbb{Z}}
\newcommand{\N}{\mathbb{N}}

\newcommand{\shuffle}{\scalebox{.8}{$\Sha$}}
\newcommand{\Sh}[3]{{\rm Sh}\binom{#1;#2}{#3}}
\newcommand{\Shstar}[3]{{\rm Sh}_*\binom{#1;#2}{#3}}

\newcommand{\vecx}{{\bf x}}
\newcommand{\varia}[2]{\left({}^{#1}_{#2}\right)}

\newcommand{\anti}{\mathsf{anti}}

\newcommand{\ev}{\mathsf{ev}}
\newcommand{\Ev}{\mathsf{Ev}}

\newcommand{\ma}{\mathsf{ma}}
\newcommand{\vimo}{\mathsf{vimo}}
\newcommand{\mi}{\mathsf{mi}}

\newcommand{\minus}{\mathsf{minus}}

\newcommand{\swap}{\mathsf{swap}}
\newcommand{\pari}{\mathsf{pari}}
\newcommand{\dima}{\mathsf{dima}}
\newcommand{\divimo}{\mathsf{divimo}}
\newcommand{\dimi}{\mathsf{dimi}}

\newcommand{\madec}{\mathsf{madec}}

\newcommand{\ARI}{\mathsf{ARI}}

\newcommand{\ulflex}[2]{{#1}\rceil_{\scalebox{.7}{$#2$}}}
\newcommand{\urflex}[2]{{}_{\scalebox{.7}{$#1$}}\lceil{#2}}
\newcommand{\llflex}[2]{{#1}\rfloor_{\scalebox{.7}{$#2$}}}
\newcommand{\lrflex}[2]{{}_{\scalebox{.7}{$#1$}}\lfloor{#2}}
\newcommand{\arit}{\mathsf{arit}}
\newcommand{\ari}{\mathsf{ari}}
\newcommand{\preari}{\mathsf{preari}}
\newcommand{\garit}{\mathsf{garit}}
\newcommand{\gari}{\mathsf{gari}}
\newcommand{\ganit}{{\sf ganit}}
\newcommand{\invgari}{{\sf invgari}}
\newcommand{\pal}{{\sf pal}}
\newcommand{\pic}{{\sf pic}}
\newcommand{\adari}{{\sf adari}}
\newcommand{\shmap}{\mathpzc{Sh}}
\newcommand{\unitmould}{1_{{\mathcal M}(\mathcal{F}; \Gamma)}}
\newcommand{\zeromould}{0_{{\mathcal M}(\mathcal{F}; \Gamma)}}
\newcommand{\unitdimould}{1_{\mathcal{M}_2(\mathcal F;\Gamma_1,\Gamma_2)}}
\newcommand{\unitpolymould}{1_{\mathcal M_n(\mathcal F;\Gamma_1,\dots, \Gamma_n)}}

\newcommand{\pol}{\mathsf{pol}}
\newcommand{\al}{\mathsf{al}}
\newcommand{\il}{\mathsf{il}}

\newcommand{\id}{\mathsf{id}}

\newcommand{\mulp}{\diamond}

\newcommand{\codec}{\mathsf{dec}^{-1}}

\newcommand{\flip}{\mathsf{flip}}
\newcommand{\rev}{\mathsf{rev}}

\newcommand{\bal}{\mathsf{bal}}

\newcommand{\abal}{\mathsf{bal}}

\newcommand{\GARI}{\mathsf{GARI}}
\newcommand{\as}{\mathsf{as}}
\newcommand{\is}{\mathsf{is}}

\newcommand{\pent}{\mathsf{pent}}
\newcommand{\0}{\mathbf{0}}
\newcommand{\dec}{\mathsf{dec}}

\newcommand{\Zag}{\mathsf{Zag}}
\newcommand{\Zig}{\mathsf{Zig}}
\newcommand{\Mini}{\mathsf{Mini}}
\newcommand{\Mono}{\mathsf{Mono}}

\newcommand{\Li}{\mathsf{Li}}

\newcommand{\GRT}{\mathsf{GRT}}
\newcommand{\ASTR}{\mathsf{ASTR}}

\newcommand{\grt}{\mathfrak{grt}}
\newcommand{\DMR}{\mathsf{DMR}}

\newcommand{\paj}{\mathsf{paj}}

\newcommand{\zerocut}{\mathfrak{c}_0}

\newcommand{\Lau}{\mathsf{Lau}}
\newcommand{\ser}{\mathsf{ser}}

\newcommand{\expari}{\mathsf{expari}}

\newcommand{\coeff}[3]{\left\langle #1\, \middle|\, {}^{#2}_{#3} \right\rangle}

\begin{document}
\bibliographystyle{amsalpha+}
\maketitle

\begin{abstract}
By developing various techniques of mould theory
and establishing a quasi-involutive reformulation of Drinfeld's associator set,
we introduce
$\GARI(\mathcal{F})_{\as+\abal}$,
a mould theoretic formulation of
Drinfeld's  associator set.
We give a mould-theoretical generalization of the result that
associator relations imply double shuffle relations,
namely,
we explain that $\GARI(\mathcal{F})_{\as+\abal}$
is  embedded into  \'{E}calle's set
$\GARI(\mathcal{F})_{\as\ast\is}$
which is a mould theoretic version of
Racinet's double shuffle set.
\end{abstract}

{\small \tableofcontents}

\setcounter{section}{-1}
\section{Introduction}\label{introduction}
It is known that multiple zeta values (MZVs) satisfy a vast number of algebraic/linear relations.
Associator relations and double shuffle relations are
the most important ones among them.
The associator relations are defining equations of
the set $\ASTR$ of associators introduced by Drinfeld \cite{Dr}.
Geometrically they arise from symmetries of  configuration
 spaces of the affine line.
On the other hand the double shuffle relations are rather combinatorial.
They arise from the iterated integral presentation and the
power series expansion of MZVs.
They are explored in \cite{E-flex, IKZ, G, R}.
Particularly their mould-theoretic aspect is revealed in
\cite{E-flex, SS, S-ARIGARI}
where Racinet's work on the set $\DMR$ 
of solutions of the relations are 
regarded as a polynomial part of the set
$\GARI(\mathcal{F})_{\as\ast\is}$ (cf. Definition \ref{def:GARIas*is})
of certain related moulds.

The objective of this  paper  is to present a mould-theoretic formulation
$\GARI(\mathcal{F})_{\as+\abal}$ of
$\ASTR$ 
and show its properties.
Our results are exhibited as follows:
\begin{itemize}
\item Quasi-involutive reformulation of $\ASTR$ in Theorem \ref{thm: flip characterization}.
\item Mould theoretic variant
       $\GARI(\mathcal{F})_{\as+\abal}$ of $\ASTR$
       in Definition \ref{defn: GARI as+bal}.
\item {The isomorphism (Theorem \ref{thm: recovery theorem of ASTR and GRT})
      $$
     {\ASTR}\cdot \exp{\Q f_{1}}
      \simeq\GARI(\mathcal{F})_{{\as+\abal}}$$
under the map $\ma$ (cf. \eqref{eq: ma Gamma})}
when  ${\mathcal F}={\mathcal F}_\ser$.
\item
      {The inclusion  (Theorem \ref{thm: inclusion GARIas+pent to GARIas*is})
      $$
      \GARI(\mathcal{F})_{\as+\abal}\hookrightarrow\GARI(\mathcal{F})_{\as\ast\is}
      $$
      under the map $\minus$ (cf. \eqref{eq:neg})
      which extends the inclusion $\ASTR\hookrightarrow \DMR$
      shown in \cite{EF2, F11} for a general $\mathcal F$.}
      \item The specific element $\paj$ gives an element in $\GARI(\mathcal{F})_{\as+\abal}\setminus \ma(\ASTR)$ under the map $\minus$
         (Theorem \ref{thm: paj in GARI as+pent})
         when  ${\mathcal F}={\mathcal F}_\Lau$.
   \end{itemize}

\subsubsection*{Acknowledgements}
H.F. has been supported by grants JSPS KAKENHI  JP18H01110 
and JP24K00520.
M.H. has been supported by grants JSPS KAKENHI JP18J00982 and JP18K13392.
N.K. has been supported by grants JSPS KAKENHI JP23KJ1420.
They are grateful to the referee for valuable comments.

\section{Preparation on mould theory}\label{sec:preparation}
We start by preparing the notion of family of functions
in \S \ref{sec: function family}.
The notion of moulds and dimoulds is reviewed in \S \ref{sec: moulds and dimoulds}.
We extend these notions to  polymoulds in \S \ref{sec: polymoulds}.
In \S \ref{sec: Mould-proper map and the unique prolongation theorem},
we prove  the unique prolongation theorem (Theorem \ref{thm:UniqueProlongation})
by introducing the notion of mould-proper maps (Definition \ref{defn:mould-proper map}),
which are significant tools to prove several  properties of moulds in the later sections.
Several properties of moulds  including alternal(il)ity, etc. are recalled in
\S \ref{sec: alternality, alternility, symmetrality, symmetrilty}
and the notions  of  flexions and  of gari are recalled in
\S \ref{subsec:flexions and gari}.
The map $\ma$ connecting moulds with non-commutative formal power series is reviewed
in \S \ref{subsec:map ma}.
We extend the map to $\dima$ for dimoulds in \S \ref{sec: dima and ma}.
In \S \ref{sec:Polymoulds of pure braid type}, we consider polymoulds
related with  braids and encode them with
several fundamental operations there by exploiting the unique prolongation theorem.

\subsection{Family of functions}\label{sec: function family}

The notion of moulds which will be explained in the next subsection
is  associated with each  family of functions. 
In this subsection, we provide an abstract framework of families of functions and
introduce the notion of divisible family
which is rather a working hypothesis in the  mould theory.

We denote by $\Q$-${\mathbf{Alg} }$ the category of commutative unital $\Q$-algebras, and by $\Q$-$\mathbf{Vec}\mathsf{Inj}$ the category whose objects are finite-dimensional $\Q$-vector spaces and Hom sets are injective homomorphisms.
We consider the functor $F_{\rm pol}:\Q$-$\mathbf{Vec}\mathsf{Inj} \to \Q$-${\mathbf{Alg} }$ which sends  $V$ to its symmetric algebra $S(V)$.

\begin{defn}
A {\it family of functions} means a pair
$\mathcal{F} = (F_{\mathcal F},i_{\mathcal F})$ with a functor
$F_{\mathcal F}:\Q$-$\mathbf{Vec}\mathsf{Inj}\to \Q$-${\mathbf{Alg} }$ and
a natural transformation $i_{\mathcal F}:F_{\rm pol} \Rightarrow F_{\mathcal F}$.
\end{defn}

Let $V_n := \Q x_1 \oplus \cdots \oplus \Q x_n$ where $x_1,\dots,x_n$ are formal symbols.
For a family of functions ${\mathcal F}$,
we put
$$\mathcal{F}_m := F_{\mathcal F}(V_m).$$
Note that $F_{\rm pol}(V_m)$ is canonically identified with $\Q[x_1,\dots,x_m]$, and $\mathcal{F}_m$ is regarded as a $\Q[x_1,\dots,x_m]$-algebra.
For $M\in \mathcal{F}_m$, a vector space $V$, and linearly independent vectors $u_1,\dots,u_m$, we denote by $M(u_1,\dots,u_m)$ the element of $F_{\mathcal F}(V)$ defined by $F_{\mathcal F}(f)(M)$ where $f$ is an injective homomorphism from $V_m$ to $V$ defined by $x_i \mapsto u_i$.
\begin{defn}\label{def:divisibility}
We say a family $\mathcal F$ is {\it divisible} when
for any vector space $V$ and independent vectors $u_1,\dots,u_n$, and $M\in F_\mathcal{F}(V)$, there exists a unique $N\in F_\mathcal{F}(V)$ satisfying
\[
M(u_{1},\dots u_{i-1},u_{i},u_{i+2},\dots,u_{n})-M(u_{1},\dots u_{i-1},u_{i+1},u_{i+2},\dots,u_{n})=(u_{i}-u_{i+1})\cdot N
\]
for all $i$.
\end{defn}

We denote the unique element $N$ in the above definition by
\[
\frac{1}{u_{i}-u_{i+1}}\left\{ M(u_{1},\dots u_{i-1},u_{i},u_{i+2},\dots,u_{n})-M(u_{1},\dots u_{i-1},u_{i+1},u_{i+2},\dots,u_{n}) \right\}.
\]
The following are some examples of divisible families of functions:
\begin{eg}
\begin{itemize}
\item The case where $F_{\mathcal{F}}=F_{{\rm pol}}$ and $i_{\mathcal{F}}$
is the identity. Then
\[
\mathcal{F}_{m}=\mathbb{Q}[x_{1},\dots,x_{m}].
\]
We write this case as $\mathcal{F}_{{\rm pol}}=\mathcal{F}$.
\item The case where $F_{\mathcal{F}}(V)$ is the quotient field of $S(V)$ and $i_{\mathcal{F}}$
is the obvious embedding. Then
\[
\mathcal{F}_{m}=\mathbb{Q}(x_{1},\dots,x_{m}).
\]
We write this case as $\mathcal{F}_{{\rm rat}}=\mathcal{F}$.
\item The case where $F_{\mathcal{F}}(V)$ is the completion of the symmetric
algebra $S(V)$ by degree and $i_{\mathcal{F}}$ is an obvious embedding.
Then
\[
\mathcal{F}_{m}=\mathbb{Q}[[x_{1},\dots,x_{m}]].
\]
We write this case as $\mathcal{F}_{{\ser}}=\mathcal{F}$.
\item The case where $F_{\mathcal{F}}(V)$ is the quotient field of the
completion of $S(V)$ by degree and $i_{\mathcal{F}}$ is an obvious
embedding. In this case,
\[
\mathcal{F}_{m}=\mathbb{Q}((x_{1},\dots,x_{m})).
\]
We write this case as $\mathcal{F}_{{\rm Lau}}=\mathcal{F}$.
\item The case where $F_{\mathcal{F}}(V)$ is the ring of holomorphic functions
on $\hat{V}\otimes\mathbb{C}$ where $\hat{V}$ is the dual vector
space of $V$ and $i_{\mathcal{F}}$ is an obvious embedding.
\item The case where $F_{\mathcal{F}}(V)$ is the field of meromorphic functions
on $\hat{V}\otimes\mathbb{C}$ where $\hat{V}$ is the dual vector
space of $V$ and $i_{\mathcal{F}}$ is an obvious embedding.
\item The case where $F_{\mathcal{F}}(V)$ is a free ${\rm Frac}(S(V))$-algebra
generated by formal symbols $(v_{1},\dots,v_{m})$ where $v_{1},\dots,v_{m}$
are independent vectors in $V$ and $i_{\mathcal{F}}$ is an obvious
embedding.
\end{itemize}
\end{eg}

\subsection{Review on moulds and dimoulds}\label{sec: moulds and dimoulds}
We recall the notion of moulds indexed by a set,
which generalizes  moulds labelled by a group  discussed in \cite{FK}
and also the notion of dimoulds  and  their tensor products
which can be found in \cite{Sau}.
Hereafter we take  $\Gamma$ to be any set
and put $\N_0=\{0,1,2,\dots\}$.

\begin{defn}[cf. {\cite[Definition 1.1]{FK}}]\label{def:mould}
Let $\mathcal{F}$ be a family of functions.
A {\it mould} 
indexed by $\Gamma$ in a lower layer
and  valued in  $\mathcal F$
is a collection 
\begin{equation*}
	M=\left( M\varia{x_1, \dots, x_m}{\sigma_1, \dots, \sigma_m} \right)_{m\in\N_0, \sigma_i\in\Gamma}
\end{equation*}
with 
$M\varia{x_1, \dots, x_m}{\sigma_1, \dots, \sigma_m}\in \mathcal{F}_m$ for $m\geq 0$,
\footnote{
In this case, we have $M(\emptyset)\in\mathcal{F}_0$ for $m=0$.
}
which we call the {\it length} $m$ component of $M$.
We denote the set of all moulds by $\mathcal{M}(\mathcal{F};\Gamma)$.
The set $\mathcal M(\mathcal{F}; \Gamma)$ forms a $\Q$-linear space by
\begin{align*}
	A+ B
	&:= \left(
	A\varia{x_1, \dots, x_m}{\sigma_1, \dots, \sigma_m}
	+ B\varia{x_1, \dots, x_m}{\sigma_1, \dots, \sigma_m}
	\right)_{m\in\N_0}, \\
	c A
	&:= \left(
	c A\varia{x_1, \dots, x_m}{\sigma_1, \dots, \sigma_m}
	\right)_{m\in\N_0},
\end{align*}
for $A, B\in\mathcal M(\mathcal{F}; \Gamma)$ and $c\in \Q$,
namely the addition and the scalar are taken componentwise.
The zero vector ${\zeromould}\in\mathcal M(\mathcal F;\Gamma)$ is given by ${\zeromould}\varia{x_1, \dots, x_m}{\sigma_1, \dots, \sigma_m}:=0$ for $m\geq0$ and for $(\sigma_1,\dots,\sigma_m)\in\Gamma^m$.
We define a {\it product} on $\mathcal M(\mathcal{F}; \Gamma)$ by
\begin{equation*}
	 (A\times B)\varia{x_1, \dots, x_m}{\sigma_1, \dots, \sigma_m}
	:=\sum_{i=0}^mA\varia{x_1, \dots, x_i}{\sigma_1, \dots, \sigma_i}
	B\varia{x_{i+1}, \dots, x_m}{\sigma_{i+1}, \dots, \sigma_m},
\end{equation*}
for $A,B\in{\mathcal M(\mathcal{F}; \Gamma)}$ and for $m\geq0$ and for $(\sigma_1,\dots,\sigma_m)\in\Gamma^m$.
Then the pair $(\mathcal M(\mathcal{F}; \Gamma),\times)$ is a non-commutative, associative, unital $\Q$-algebra.
Here, the unit ${\unitmould}\in\mathcal M(\mathcal{F}; \Gamma)$ is given by
$$
{\unitmould}\varia{x_1, \dots, x_m}{\sigma_1, \dots, \sigma_m}
:=\left\{\begin{array}{ll}
1 & (m=0), \\
0 & (\mbox{otherwise}),
\end{array}\right.
$$
for $m\geq0$ and for $(\sigma_1,\dots,\sigma_m)\in\Gamma^m$.
A {\it constant-mould} $C\in\mathcal M(\mathcal{F}; \Gamma)$ is a mould with all components in $\Q$.
So {$\zeromould$} and {$\unitmould$} are constant.
{By abuse of notation we simply denote them by $0$ and $1$ if there is no fear of confusion.}
\end{defn}


The above notions are nothing but a simple extension of moulds indexed by a group $\Gamma$ as discussed in \cite{FK}
to the ones labelled by a set.


\begin{rem}
Assume that $u_{1},\dots,u_{m}$ in $\mathbb{Q}x_{1}+\cdots+\mathbb{Q}x_{r}$ are linearly independent over $\Q$.
For $M \in \mathcal M(\mathcal{F};\Gamma)$ we denote $M\varia{u_1, \dots, u_m}{\sigma_1, \dots, \sigma_m}$ to be the image of $M\varia{x_1, \dots, x_m}{\sigma_1, \dots, \sigma_m}$ under the field embedding $\mathcal{F}_{m} \to \mathcal{F}$ sending $x_i \mapsto u_i$.
\end{rem}

\begin{rem}
For simplicity we occasionally denote $M\in\mathcal M(\mathcal{F}; \Gamma)$ by
\begin{equation*}
	 M=( M(\vecx_m))_{m\in\N_0},
\end{equation*}
where $\vecx_0:=\emptyset$ and $\vecx_m:=\varia{x_1,\ \dots,\ x_m}{\sigma_1,\ \dots,\ \sigma_m}$ for $m\geq1$.
\end{rem}

We define two subsets of $\mathcal M(\mathcal F; \Gamma)$ by
\begin{align*}
& \ARI(\mathcal F;\Gamma):=\{ M\in\mathcal M(\mathcal F;\Gamma)\ |\ M(\emptyset)=0 \}, \\
& \GARI(\mathcal F;\Gamma):=\{ M\in\mathcal M(\mathcal F;\Gamma)\ |\ M(\emptyset)=1 \}.
\end{align*}
We note that the pair $(\ARI(\mathcal F;\Gamma),\times)$ forms a subalgebra of $(\mathcal M(\mathcal F; \Gamma), \times)$ and the pair $(\GARI(\mathcal F;\Gamma), \times)$ forms a group.

We denote  $\mathcal M(\mathcal F;\Gamma)$,
$\ARI(\mathcal F;\Gamma)$ and $\GARI(\mathcal F;\Gamma)$ simply by $\mathcal M(\mathcal F)$,
$\ARI(\mathcal F)$ and $\GARI(\mathcal F)$ respectively
when $\Gamma=\{1\}$.
In such a case, we simply denote $M\varia{x_1, \dots, x_m}{\ 1, \dots,\ 1}$
by $M(x_1, \dots, x_m)$.
For a set $\Gamma$ and an element  $M$ in  $\mathcal M(\mathcal F)$,
we denote $M_\Gamma$ to be the element  in $\mathcal M(\mathcal F; \Gamma)$
which is defined by
\begin{equation}\label{eq: mould extension by Gamma}
M_\Gamma\varia{x_1, \dots, x_m}{\sigma_1, \dots, \sigma_m}=M(x_1, \dots, x_m)
\end{equation}
for $\sigma_1,\dots,\sigma_m\in\Gamma$.

\begin{rem}\label{rem:copy-of-set-of-moulds}
For our convenience, we prepare $\overline{\mathcal M}(\mathcal F; \Gamma)$,
a copy of $\mathcal M(\mathcal F;\Gamma)$.
We denote the components of the element $M\in\overline{\mathcal M}(\mathcal F;\Gamma)$ by
$$
M\varia{\sigma_1, \dots, \sigma_m}{x_1, \dots, x_m}
$$
instead of $M\varia{x_1, \dots, x_m}{\sigma_1, \dots, \sigma_m}$ for $m\in\mathbb N$ and $(\sigma_1,\dots,\sigma_m)\in\Gamma^m$.
Again when $\Gamma=\{1\}$, we simply denote
$M\varia{\sigma_1, \dots, \sigma_m}{x_1, \dots, x_m}$
by $M(x_1, \dots, x_m)$.
We define two subsets of $\overline{\mathcal M}(\mathcal F;\Gamma)$ by
\begin{align*}
& \overline{\ARI}(\mathcal{F};\Gamma):=\{ M\in\overline{\mathcal M}(\mathcal F;\Gamma)\ |\ M(\emptyset)=0 \}, \\
& \overline{\GARI}(\mathcal{F};\Gamma):=\{ M\in\overline{\mathcal M}(\mathcal F;\Gamma)\ |\ M(\emptyset)=1 \}.
\end{align*}
\end{rem}

\begin{defn}[cf. {\cite[Definition 5.2]{Sau}}]\label{def:2.2.1}
Let $\mathcal{F}$ be a family of functions and $\Gamma_1,\Gamma_2$ sets.
A {\it dimould}
\footnote{
We note that the notion of {\bf di}moulds is different from that of  {\bf bi}moulds in \cite{E-ARIGARI} and \cite{E-flex}
which are \lq moulds  with double-layered parameters'.
}
with values in  $\mathcal F$ and
$M$ indexed by $\Gamma_1$ and $\Gamma_2$ in a lower layer is a collection
\begin{equation*}
	M:=\left(
	M\varia{x_1, \dots, x_r;\,x_{r+1}, \dots, x_{r+s}}{\sigma_1, \dots, \sigma_r;\,\sigma_{r+1}, \dots, \sigma_{r+s}}
	\right)_{r,s\in\N_0,\sigma_1, \dots, \sigma_r\in\Gamma_1,\sigma_{r+1}, \dots, \sigma_{r+s}\in\Gamma_2},
\end{equation*}
with 
$M\varia{x_1, \dots, x_r;\,x_{r+1}, \dots, x_{r+s}}{\sigma_1, \dots, \sigma_r;\,\sigma_{r+1}, \dots, \sigma_{r+s}} \in \mathcal{F}_{r+s}$ for $r, s\in\N_0$.
We denote the set of all dimoulds with values in $\mathcal F$ by $\mathcal{M}_2(\mathcal F;\Gamma_1,\Gamma_2)$.
By the component-wise summation and the component-wise scalar multiple, the set $\mathcal{M}_2(\mathcal F;\Gamma_1,\Gamma_2)$ forms a $\Q$-linear space.
The {\it product} of $\mathcal{M}_2(\mathcal F;\Gamma_1,\Gamma_2)$ which is denoted by the same symbol $\times$ as the product of $\mathcal{M}_2(\mathcal F;\Gamma_1,\Gamma_2)$ is defined by
\begin{align}\label{eqn:2.2.1}
	&(A\times B)\varia{x_1, \dots, x_r;\,x_{r+1}, \dots, x_{r+s}}{\sigma_1, \dots, \sigma_r;\,\sigma_{r+1}, \dots, \sigma_{r+s}} \\
	&:=\sum_{i=0}^r\sum_{j=0}^s
	A\varia{x_1, \dots, x_i;\,x_{r+1}, \dots, x_{r+j}}{\sigma_1, \dots, \sigma_i;\,\sigma_{r+1}, \dots, \sigma_{r+j}}
	B\varia{x_{i+1}, \dots, x_r;\,x_{r+j+1}, \dots, x_{r+s}}{\sigma_{i+1}, \dots, \sigma_r;\,\sigma_{r+j+1}, \dots, \sigma_{r+s}}, \nonumber
\end{align}
for $A,B\in\mathcal{M}_2(\mathcal F;\Gamma_1,\Gamma_2)$ and for $r,s\geq0$ and for $(\sigma_1,\dots,\sigma_{r}) \in \Gamma_1^r$ and for $(\sigma_{r+1},\dots,\sigma_{r+s}) \in \Gamma_2^s$.
Then $(\mathcal{M}_2(\mathcal F;\Gamma_1,\Gamma_2), \times)$ is a non-commutative, associative $\Q$-algebra.
The unit $\unitdimould$ of $(\mathcal{M}_2(\mathcal F;\Gamma_1,\Gamma_2), \times)$ is given by
\begin{equation*}
\unitdimould\varia{x_1, \dots, x_r;\,x_{r+1}, \dots, x_{r+s}}{\sigma_1, \dots, \sigma_r;\,\sigma_{r+1}, \dots, \sigma_{r+s}}
:=\left\{
\begin{array}{ll}
	1 & (r=s=0), \\
	0 & (\mbox{otherwise}).
\end{array}\right.
\end{equation*}
\end{defn}

We denote  $\mathcal M_2(\mathcal F;\Gamma_1,\Gamma_2)$,
simply by $\mathcal M_2(\mathcal F)$ when $\Gamma_1=\Gamma_2=\{1\}$,
which is the case treated by Sauzin in \cite{Sau}.
In such a case, we simply denote $M\varia{x_1, \dots, x_r;x_{r+1}, \dots, x_{r+s}}{\ 1, \dots,\ 1; \quad 1, \ \dots, \ 1}$
by $M(x_1, \dots, x_r; x_{r+1},\dots, x_{r+s})$.

\begin{rem}
Similarly to Remark \ref{rem:copy-of-set-of-moulds}, we prepare $\overline{\mathcal{M}_2}(\mathcal F;\Gamma_1,\Gamma_2)$, a copy of $\mathcal{M}_2(\mathcal F;\Gamma_1,\Gamma_2)$. To avoid confusion, we often denote each  component of the element $M\in\overline{\mathcal{M}_2}(\mathcal F;\Gamma_1,\Gamma_2)$ by
$$
M\varia{\sigma_1, \dots, \sigma_r;\,\sigma_{r+1}, \dots, \sigma_{r+s}}{x_1, \dots, x_r;\,x_{r+1}, \dots, x_{r+s}}
$$
instead of $M\varia{x_1, \dots, x_r;\,x_{r+1}, \dots, x_{r+s}}{\sigma_1, \dots, \sigma_r;\,\sigma_{r+1}, \dots, \sigma_{r+s}}$ for $r,s\in\mathbb{N}_0$, $\sigma_1, \dots, \sigma_r\in\Gamma_1$ and $\sigma_{r+1}, \dots, \sigma_{r+s}\in\Gamma_2$.
\end{rem}

We consider morphisms between sets of dimoulds and sets of moulds.

\begin{defn}\label{def:tensor product}
Let $\mathcal{F}$ be a family of functions and $\Gamma_1,\Gamma_2$ sets.
Let
$$i_\otimes:\mathcal{M}(\mathcal F;\Gamma_1) \otimes_\Q \mathcal{M}(\mathcal F;\Gamma_2) \rightarrow \mathcal{M}_2(\mathcal F;\Gamma_1,\Gamma_2)$$
be the $\Q$-linear map defined by
$$
i_\otimes(M\otimes N)\varia{x_1, \dots, x_r;\,x_{r+1}, \dots, x_{r+s}}{\sigma_1, \dots, \sigma_r;\,\sigma_{r+1}, \dots, \sigma_{r+s}}
:=M\varia{x_1, \dots, x_r}{\sigma_1, \dots, \sigma_r}\cdot
N\varia{x_{r+1}, \dots, x_{r+s}}{\sigma_{r+1}, \dots, \sigma_{r+s}}
$$
for $r,s\geq0$ and $(\sigma_1,\dots,\sigma_{r}) \in \Gamma_1^r$ and for $(\sigma_{r+1},\dots,\sigma_{r+s}) \in \Gamma_2^s$.
By abuse of notation, we denote $i_\otimes(M\otimes N)$ simply
by $M\otimes N$ and call it as the {\it tensor product} of $M$ and $N$.
\end{defn}


We see that the above  $i_\otimes$  yields
an algebra homomorphism  by the equality
\begin{equation}\label{eq: Sauzin (5.7)}
(M_1 \otimes M_2) \times (N_1 \otimes N_2) = (M_1 \times N_1) \otimes (M_2 \times N_2).
\end{equation}
for $M_1, N_1\in \mathcal{M}(\mathcal F;\Gamma_1)$ and
$M_2, N_2\in \mathcal{M}(\mathcal F;\Gamma_2)$
which is shown in \cite{Sau} (5.7).
Actually when ${\mathcal F}={\mathcal F}_\ser$,
the above $i_\otimes$ induces an algebra isomorphism:

\begin{lem}\label{lem: tensor isom between moulds and dimoulds}
When ${\mathcal F}={\mathcal F}_\ser$,
the above $i_\otimes$ induces a $\Q$-algebra isomorphism
$$
i_{\widehat\otimes}:
\mathcal{M}(\mathcal F;\Gamma_1) \widehat\otimes_\Q \mathcal{M}(\mathcal F;\Gamma_2) \simeq \mathcal{M}_2(\mathcal F;\Gamma_1,\Gamma_2).
$$
Here $\widehat\otimes_\Q $ means the completed tensor product with respect to the
summation of the length and the total degree of formal power series.
\end{lem}

\begin{proof}
We define the mould $I_{\mathbf k;\alpha} \in \mathcal{M}(\mathcal F;\Gamma_1)$ and the dimould $I_{\mathbf k,\mathbf l;\alpha,\beta} \in \mathcal{M}_2(\mathcal F;\Gamma_1,\Gamma_2)$ by
\begin{align*}
&I_{\mathbf k;\alpha}\varia{x_1, \dots, x_m}{\sigma_1, \dots, \sigma_m}
:=\left\{\begin{array}{ll}
x_1^{k_1}\cdots x_p^{k_p} & (m=p,\ \alpha_i=\sigma_i), \\
0 & (\mbox{otherwise}),
\end{array}
\right. \\
&I_{\mathbf k,\mathbf l;\alpha,\beta}\varia{x_1, \dots, x_r;\,x_{r+1}, \dots, x_{r+s}}{\sigma_1, \dots, \sigma_r;\,\sigma_{r+1}, \dots, \sigma_{r+s}}  \\
&\qquad\qquad\qquad
:=\left\{\begin{array}{ll}
x_1^{k_1}\cdots x_p^{k_p}x_{p+1}^{l_{p+1}}\cdots x_{p+q}^{l_{p+q}} & (r=p,s=q,\ \alpha_i=\sigma_i, \beta_i=\sigma_{p+i}), \\
0 & (\mbox{otherwise}),
\end{array}
\right.
\end{align*}
for $p,q\in\N_0$ and $\mathbf k=(k_1,\dots,k_p)\in\N_0^p$, $\mathbf l=(l_1,\dots,l_q)\in\N_0^q$ and $\alpha=(\alpha_1,\dots,\alpha_p) \in \Gamma_1^{p}$, $\beta=(\beta_1,\dots,\beta_q) \in \Gamma_2^{q}$.
Note that by the definition of the tensor product $\otimes$
in Definition \ref{def:tensor product},
we have
$I_{\mathbf k;\alpha}\otimes I_{\mathbf l;\beta}=I_{\mathbf k,\mathbf l;\alpha,\beta}$.
Here, any mould $M \in \mathcal{M}(\mathcal F;\Gamma_1)$ is represented by
$$
M
=\sum_{p\in\N_0}\sum_{\mathbf k\in\N_0^p}\sum_{\alpha\in \Gamma_1^p}
c_p(M;\mathbf k;\alpha)I_{\mathbf k;\alpha}
$$
with $c_p(M;\mathbf k;\alpha)\in\Q$, and any dimould $M \in  \mathcal{M}_2(\mathcal F;\Gamma_1,\Gamma_2)$ is represented by
$$
M
=\sum_{p,q\in\N_0}\sum_{\substack{\mathbf k\in\N_0^p \\ \mathbf l\in\N_0^q}}
\sum_{\substack{\alpha\in \Gamma_1^p \\ \beta\in \Gamma_2^q}}
c_{p,q}(M;\mathbf k,\mathbf l;\alpha,\beta)I_{\mathbf k,\mathbf l;\alpha,\beta}
$$
with $c_{p,q}(M;\mathbf k,\mathbf l;\alpha,\beta)\in\Q$.
Hence, by the correspondence $I_{\mathbf k;\alpha}\otimes I_{\mathbf l;\beta}=I_{\mathbf k,\mathbf l;\alpha,\beta}$ and by the algebra homomorphism by \eqref{eq: Sauzin (5.7)}, we get the $\Q$-algebra isomorphism between $\mathcal{M}(\mathcal F;\Gamma_1) \widehat\otimes \mathcal{M}(\mathcal F;\Gamma_2)$ and $\mathcal{M}_2(\mathcal F;\Gamma_1,\Gamma_2)$.
\end{proof}

\begin{rem}\label{rem:counterexample-of-isom}
We do not expect such an isomorphism  
in general.
Consider a dimould $M\in \mathcal{M}_2({\mathcal F}_\Lau)$
with (1,1)-component
$\frac{1}{x_1+x_2}\in{\mathcal F_{\Lau,2}}$. 
It does not look to be an element of  any reasonable completed double tensor product of $\mathcal{M}(\mathcal F)$
because we have
$\frac{1}{x_1+x_2}\not\in
\mathcal{F}_{\Lau,1}\hat\otimes\mathcal{F}_{\Lau,1}
=\Q((x_1))\hat\otimes \Q((x_2))$.
\end{rem}


\subsection{Poly-moulds}\label{sec: polymoulds}
We introduce a notion of poly-moulds which extends that of moulds and dimoulds
in our previous subsection.

\begin{defn}
Let $\mathcal{F}$ be a family of functions and $\Gamma_1,\dots,\Gamma_n$ sets with $n\in \N$.
A \textit{poly-mould}
\footnote{
A poly-mould with $n=1$ (resp. $n=2$) is
a mould (resp. a dimould).}
$M$ with values in  $\mathcal F$ and indexed by $\Gamma_1,\dots,\Gamma_n$ in a lower layer is a sequence
$$
M:=
\left(
M
\varia{x_1^{(1)}, \dots, x_{r_1}^{(1)};\,x_{1}^{(2)}, \dots, x_{r_2}^{(2)};\,\cdots\cdots;\,x_{1}^{(n)}, \dots, x_{r_n}^{(n)}}
{\sigma_1^{(1)}, \dots, \sigma_{r_1}^{(1)};\,\sigma_{1}^{(2)}, \dots, \sigma_{r_2}^{(2)};\,\cdots\cdots;\,\sigma_{1}^{(n)}, \dots, \sigma_{r_n}^{(n)}}
\right)_{r_i\in\N_0,\,\sigma_j^{(i)}\in\Gamma_i}
$$
with 
$$
M
\varia{x_1^{(1)}, \dots, x_{r_1}^{(1)};\,x_{1}^{(2)}, \dots, x_{r_2}^{(2)};\,\cdots\cdots;\,x_{1}^{(n)}, \dots, x_{r_n}^{(n)}}
{\sigma_1^{(1)}, \dots, \sigma_{r_1}^{(1)};\,\sigma_{1}^{(2)}, \dots, \sigma_{r_2}^{(2)};\,\cdots\cdots;\,\sigma_{1}^{(n)}, \dots, \sigma_{r_n}^{(n)}}
\in
{\mathcal F}_{r_1+\cdots+r_n}.
$$
Here we put $x_j^{(i)}:=x_{r_1+ \cdots + r_{i-1} + j}$ for $i=1,\dots,n$ and for $j\in\N$.
We denote the set of all poly-moulds indexed by $\Gamma_i$ ($i=1,\dots,n$) in a lower layer by $\mathcal M_n(\mathcal F;\Gamma_1,\dots, \Gamma_n)$.
Similarly to the cases of $n=1$  and $2$, the set {$\mathcal M_n(\mathcal F;\Gamma_1,\dots, \Gamma_n)$} forms a $\Q$-linear space by the component-wise summation and the component-wise scalar multiple.
The {\it product} of $\mathcal M_n(\mathcal F;\Gamma_1,\dots, \Gamma_n)$ is given by
\begin{align}\label{eq:product-of-polymoulds}
&(A\times B)
\varia{x_1^{(1)}, \dots, x_{r_1}^{(1)};\cdots\cdots;\,x_{1}^{(n)}, \dots, x_{r_n}^{(n)}}
{\sigma_1^{(1)}, \dots, \sigma_{r_1}^{(1)};\,\cdots\cdots;\,\sigma_{1}^{(n)}, \dots, \sigma_{r_n}^{(n)}} \\
&:=\sum_{i_1=0}^{r_1} \cdots \sum_{i_n=0}^{r_n}
A
\varia{x_1^{(1)}, \dots, x_{i_1}^{(1)};\cdots\cdots;\,x_{1}^{(n)}, \dots, x_{i_n}^{(n)}}
{\sigma_1^{(1)}, \dots, \sigma_{i_1}^{(1)};\,\cdots\cdots;\,\sigma_{1}^{(n)}, \dots, \sigma_{i_n}^{(n)}}
B
\varia{x_{i_1+1}^{(1)}, \dots, x_{r_1}^{(1)};\cdots\cdots;\,x_{i_n+1}^{(n)}, \dots, x_{r_n}^{(n)}}
{\sigma_{i_1+1}^{(1)}, \dots, \sigma_{r_1}^{(1)};\,\cdots\cdots;\,\sigma_{i_n+1}^{(n)}, \dots, \sigma_{r_n}^{(n)}} \nonumber
\end{align}
for $A,B\in\mathcal M_n(\mathcal F;\Gamma_1,\dots, \Gamma_n)$ and for $r_j\geq0$ and for $(\sigma_1,\dots,\sigma_{r_j}) \in \Gamma_j^{r_j}$ ($j=1,\dots,n$).
Then $(\mathcal M_n(\mathcal F;\Gamma_1,\dots, \Gamma_n), \times)$ forms a non-commutative, associative $\Q$-algebra with the unit $\unitpolymould\in\mathcal M_n(\mathcal F;\Gamma_1,\dots, \Gamma_n)$ given by
\begin{equation*}
\unitpolymould
\varia{x_1^{(1)}, \dots, x_{r_1}^{(1)};\cdots\cdots;\,x_{1}^{(n)}, \dots, x_{r_n}^{(n)}}
{\sigma_1^{(1)}, \dots, \sigma_{r_1}^{(1)};\,\cdots\cdots;\,\sigma_{1}^{(n)}, \dots, \sigma_{r_n}^{(n)}}
:=\left\{
\begin{array}{ll}
1 & (r_1=\cdots=r_n=0), \\
0 & (\mbox{otherwise}).
\end{array}\right.
\end{equation*}
\end{defn}

\begin{defn}\label{defn: tensor product of two polymoulds}
Let
$$
i_\otimes: \mathcal M_m(\mathcal F;\Gamma_1,\dots, \Gamma_m)\otimes \mathcal M_n(\mathcal F;\Gamma_{m+1},\dots, \Gamma_{m+n})
\to \mathcal M_{m+n}(\mathcal F;\Gamma_{1},\dots, \Gamma_{m+n})
$$
be the $\Q$-linear map defined  by
\begin{align*}
(M\otimes N)
&\varia{x_1^{(1)}, \dots, x_{r_1}^{(1)};\cdots\cdots;\,x_{1}^{(m+n)}, \dots, x_{r_{m+n}}^{(m+n)}}
{\sigma_1^{(1)}, \dots, \sigma_{r_1}^{(1)};\,\cdots\cdots;\,\sigma_{1}^{(m+n)}, \dots, \sigma_{r_{m+n}}^{(m+n)}} \\
&=M
\varia{x_1^{(1)}, \dots, x_{r_1}^{(1)};\cdots\cdots;\,x_{1}^{(m)}, \dots, x_{r_m}^{(m)}}
{\sigma_1^{(1)}, \dots, \sigma_{r_1}^{(1)};\,\cdots\cdots;\,\sigma_{1}^{(m)}, \dots, \sigma_{r_m}^{(m)}}
N
\varia{x_1^{(m+1)}, \dots, x_{r_{m+1}}^{(m+1)};\cdots\cdots;\,x_{1}^{(m+n)}, \dots, x_{r_{m+n}}^{(m+n)}}
{\sigma_1^{(m+1)}, \dots, \sigma_{r_{m+1}}^{(m+1)};\,\cdots\cdots;\,\sigma_{1}^{(m+n)}, \dots, \sigma_{r_{m+n}}^{(m+n)}}
\end{align*}
for $r_j\geq0$ and $(\sigma_1^{(j)},\dots,\sigma_{r_j}^{(j)}) \in \Gamma_j^{r_j}$ ($j=1,\dots,n$).
By abuse of notation, we denote $i_\otimes(M\otimes N)$ simply
by $M\otimes N$ and call it as the {\it tensor product} of $M$ and $N$.
\end{defn}

Again we note  that the above $i_\otimes$ forms
an algebra homomorphism  by the equality
$$
(M_1 \otimes M_2) \times (N_1 \otimes N_2) = (M_1 \times N_1) \otimes (M_2 \times N_2)$$
which can be proved by the same arguments as \eqref{eq: Sauzin (5.7)}.
Actually when ${\mathcal F}={\mathcal F}_\ser$,
the above $i_\otimes$ induces an algebra isomorphism:

\begin{lem}\label{lem: tensor isom between moulds and polymoulds}
Let $m,n\in\N$ and $\Gamma_1,\dots,\Gamma_{m+n}$ be sets.
When ${\mathcal F}={\mathcal F}_\ser$,
the above $i_\otimes$ induces a $\Q$-algebra isomorphism
$$
\mathcal M_m(\mathcal F;\Gamma_1,\dots, \Gamma_m) \widehat\otimes \mathcal M_n(\mathcal F;\Gamma_{m+1},\dots, \Gamma_{m+n}) \simeq \mathcal M_{m+n}(\mathcal F;\Gamma_{1},\dots, \Gamma_{m+n}).
$$
\end{lem}

\begin{proof}
Similarly to the proof of Lemma \ref{lem: tensor isom between moulds and dimoulds}, we obtain this claim.
\end{proof}

By the above lemma, we get the $\Q$-algebra isomorphism
\begin{equation}\label{eq: isom mould tensors and polymoulds}
i_{\widehat\otimes}:
\mathcal M(\mathcal F;\Gamma_1) \widehat\otimes
\cdots \widehat\otimes \mathcal M(\mathcal F;\Gamma_m)
\simeq \mathcal M_{m}(\mathcal F;\Gamma_{1},\dots, \Gamma_{m})
\end{equation}
for ${\mathcal F}={\mathcal F}_\ser$.

\subsection{Mould-proper map and the unique prolongation theorem}
\label{sec: Mould-proper map and the unique prolongation theorem}

In this subsection, we introduce the notion of mould-proper map (Definition \ref{defn:mould-proper map}) and  show the unique prolongation theorem
(Theorem \ref{thm:UniqueProlongation})
which will be repeatedly employed in our later arguments.

For this purpose we prepare
the  notion of $S_\bullet$-moulds which
extends that of poly-moulds in order to formulate mould-proper maps.
\begin{defn}
Let $\mathcal{F}$ be a family of functions and $S_{\bullet}=(S_{0},S_{1},S_{2},\dots)$ a sequence of sets.
An {\it $S_\bullet$-mould $M$
with values in $\mathcal{F}$}
means a collection
\[
M=(M_{s}(x_{1},\dots,x_{m}))_{m\geq 0,s\in S_{m}}
\]
with $M_{s}(x_{1},\dots,x_{m})\in{\mathcal{F}_{m}}$.
\end{defn}

We denote by $\mathcal{M}(\mathcal{F},S_\bullet)$
\footnote{
It should not be confused with  the symbol $\mathcal{M}(\mathcal{F};\Gamma)$ which stands for the set of moulds with values in the set $\Gamma$.
Actually our notion of $S_\bullet $ will not appear in our later section. }
the
set of $S_\bullet$-moulds with values in $\mathcal{F}$.
Note that a poly-mould $M\in\mathcal{M}_{n}(\mathcal{F};\Gamma_{1},\dots,\Gamma_{n})$
can be viewed as a special case of an $S_\bullet$-mould
by letting $S_{\bullet}=(S_{0},S_{1},S_{2},\dots)$ with
$S_m=\underset{i_1+\cdots+i_n=m}{\amalg}
\Gamma_1^{i_1}\times\cdots\times \Gamma_n^{i_n}$,
that is,
\[
S_{m}=\{(\sigma_{1}^{(1)},\dots,\sigma_{i_{1}}^{(1)},\dots,\sigma_{1}^{(n)},\dots,\sigma_{i_{n}}^{(n)})\mid i_{1}+\cdots+i_{n}=m,\sigma_{j}^{(k)}\in\Gamma_{k}\}.
\]
Furthermore, for any sequences of sets $S_{\bullet}^{(i)}=(S_{0}^{(i)},S_{1}^{(i)},S_{2}^{(i)}\dots)$
for $i=1,\dots,n$, we can identify $\mathcal{M}(\mathcal{F},S_{\bullet}^{(1)})\times\cdots\times\mathcal{M}(\mathcal{F},S_{\bullet}^{(n)})$
with $\mathcal{M}(\mathcal{F},S_{\bullet})$ where
$$S_{\bullet}=(\sqcup_{i=1}^{n}S_{0}^{(i)},\sqcup_{i=1}^{n}S_{1}^{(i)},\sqcup_{i=1}^{n}S_{2}^{(i)},\dots).
$$
Hence, the tuple of poly-moulds can also be viewed as a special case
of $S_{\bullet}$-moulds.

\begin{defn}\label{defn:mould-proper map}
Let $m\geq 0$ and $S_{\bullet}=(S_{0},S_{1},S_{2},\dots)$ a sequence of sets. Put
$$
\Lambda:=\left\{(n,s,v_1,\dots,v_n) \
 \middle| 
\begin{array}{l}
n\geq0,s\in S_{n}, 
\quad v_{1},\dots,v_{n}\text{  are linearly}\\
\text{  independent vectors in }\mathbb{Q}x_{1}+\cdots+\mathbb{Q}x_{m}
\end{array}
\right\}.
$$

(1) We say that a polynomial $H((w_\lambda)_{\lambda\in \Lambda})\in
    \mathbb{Q}(x_{1},\dots,x_{m})[w_\lambda\mid \lambda\in\Lambda]$ is {\it proper} if
$H((z^M_\lambda)_{\lambda\in \Lambda})$ with $z^M_\lambda=M_{s}(v_{1},\dots,v_{n})$
    is {well-defined}
    \footnote{That  means
    there exists $c\in \Q[x_1,\dots,x_m]$ such that $cH\in \mathbb{Q}[x_{1},\dots,x_{m}][w_\lambda\mid \lambda\in\Lambda]$ and there exists a unique $u\in \mathcal{F}_{m}$ satisfying $cu=(cH)((z^M_\lambda)_{\lambda\in \Lambda})$ .}
    in $\mathcal F$
     for any divisible family of functions $\mathcal{F}$ and $M\in \mathcal{M}(\mathcal{F},{S_\bullet})$.

(2) For a divisible family of functions $\mathcal{F}$,
a (set-theoretical) map $f:\mathcal{M}(\mathcal{F},{S_\bullet})\to\mathcal{F}_{m}$
   is {\it mould-proper} if and only if there exists a proper-polynomial $H((w_\lambda)_{\lambda\in \Lambda})\in
   \mathbb{Q}(x_{1},\dots,x_{m})[w_\lambda\mid \lambda\in\Lambda]$
   such that $$f(M)(x_{1},\dots,x_{m})=H((z^M_\lambda)_{\lambda\in \Lambda})$$
   for all $M\in \mathcal{M}(\mathcal{F},{S_\bullet})$
   with $z^M_\lambda=M_{s}(v_{1},\dots,v_{n})$ 
   for  $\lambda=(n,s,v_1,\dots,v_n)\in \Lambda$.
Furthermore, we
say that
{a (set-theoretical) map}
$g:\mathcal{M}(\mathcal{F},{S_\bullet})\to
\mathcal{M}(\mathcal{F},{T_\bullet})$
is {\it mould-proper} if the map
\[
g_{t}:\mathcal{M}(\mathcal{F},S_{\bullet})\to\mathcal{F}_{m}\ ;\ M\mapsto g(M)_{t}(x_1,\dots,x_m)
\]
is mould-proper for all $m\in\mathbb{Z}_{\geq0}$ and $t\in T_{m}$.
\end{defn}

\begin{eg}
Let $m=2$ and $s\in S_{1}$ be any element. Then a polynomial
\[
H((w_{\lambda})_{\lambda\in\Lambda})=\frac{1}{x_{2}-x_{1}}\left(w_{(1,s,x_{2})}-w_{(1,s,x_{1})}\right)
\]
is proper since
\[
H((z_{\lambda}^{M})_{\lambda\in\Lambda})=\frac{1}{x_{2}-x_{1}}\left(M_{s}(x_{2})-M_{s}(x_{1})\right)
\]
is well-defined by the definition of divisible families of functions.
\end{eg}

\begin{eg}
Let $m=2$ and $s\in S_{1}$ be any element. Then a polynomial
\[
H((w_{\lambda})_{\lambda\in\Lambda})=\frac{1}{x_{2}-x_{1}}w_{(1,s,x_{2})}
\]
is not proper since
\[
H((z_{\lambda}^{M})_{\lambda\in\Lambda})=\frac{1}{x_{2}-x_{1}}M_{s}(x_{2})
\]
is not well-defined when $\mathcal{F}=\mathcal{F}_{\pol}$ and $M_{s}(x_{2})\neq0$.
\end{eg}

More generally, we have the following.
\begin{lem}\label{lem: divisible}
Let $\mathcal{F}=(F_{\mathcal{F}},i_{\mathcal{F}})$ be a divisible
family of functions, $n,r_{1},\dots,r_{n}$ be non-negative integers,
and $V$ the vector space generated by formal symbols
$$\{u_{i,j}\}_{1\leq i\leq n,1\leq j\leq r_{i}}.$$
Then for $M\in\mathcal{F}_{n}$, the following term is well-defined
as an element of $\mathcal{F}_{V}$.
\[
\sum_{\substack{1\leq j_{1}\leq r_{1}\\
\vdots\\
1\leq j_{n}\leq r_{n}
}
}M(u_{1,j_{1}},\dots,u_{n,j_{n}})\prod_{i=1}^{n}\prod_{j'\in\{1,\dots,r_{i}\}\setminus\{j_{i}\}}\frac{1}{u_{1,j_{1}}-u_{1,j'}}.
\]
\end{lem}

\begin{proof}
It follows from the repeated use of the definition of divisibility of
family of functions in Definition \ref{def:divisibility}. For example, when $n=1$ and $r_{1}=3$, we have
\begin{align*}
 & \frac{M(x_{1})}{(x_{1}-x_{2})(x_{1}-x_{3})}+\frac{M(x_{2})}{(x_{2}-x_{1})(x_{2}-x_{3})}+\frac{M(x_{3})}{(x_{3}-x_{1})(x_{3}-x_{2})}\\
 & =\frac{1}{x_{1}-x_{2}}\left(\frac{M(x_{1})-M(x_{3})}{x_{1}-x_{3}}-\frac{M(x_{2})-M(x_{3})}{x_{2}-x_{3}}\right)
\in\mathcal{F}_{3}.
\end{align*}
\end{proof}

We note that $\mathcal{M}(\mathcal{F},S_{\bullet})\times\mathcal{M}(\mathcal{F},T_{\bullet})$
is naturally identified with $\mathcal{M}(\mathcal{F},(S_{n}\sqcup T_{n})_{n=0}^{\infty})$ as sets.
Thus the notion of ``mould-proper'' is  naturally extended to the product 
of sets of moulds.

\begin{rem}
We note that actually all the maps treated in this paper ($+$, $\times$, $\otimes$, $\shmap$, $\garit$, $\gari$, $\pari$ and $\anti$ in Definition \ref{defn: pari anti}, $\mulp$, $\bal$, $\Ev$, etc.) are mould-proper maps.
Other examples of mould-proper maps are also given in \cite[(6)]{E-ARIGARI} and \cite[(4.2)]{Sau}.
\end{rem}

We denote
$$\mathcal{F}_{\pol} \subset\mathcal{F}$$
when
$\Q[x_1,..,x_m]\to \mathcal F_m$ is injective for all $m$.
In this case,
we often regard
$\mathcal{M}(\mathcal{F}_\pol,{S_{\bullet}})$
to be a subset of
$\mathcal{M}(\mathcal{F},{S_{\bullet}})$


\begin{thm}[Unique prolongation theorem]\label{thm:UniqueProlongation}
Let $\mathcal{F}$ be a divisible family of functions such that $\mathcal{F}_{\pol}\subset \mathcal{F}$, and
$f:\mathcal{M}(\mathcal{F},{S_{\bullet}})\to\mathcal{F}_{m}$
be a mould-proper map given by
\[
f(M)(x_{1},\dots,x_{m})=H((z_{\lambda}^{M})_{\lambda\in\Lambda})
\]
where $H$ is a polynomial given in Definition \ref{defn:mould-proper map}. If $f(M)=0$ for
all $M\in\mathcal{M}(\mathcal{F}_{\pol},{S_{\bullet}})$ then
$H=0$.
\end{thm}

\begin{proof}
Note that it is enough to only prove the case where $S_{0}=\emptyset$ and
$H\in\Q[x_{1},\dots,x_{m}][w_{\lambda}\mid\lambda\in\Lambda]$
since the general case follows from this case. Note that there are
inclusions
\[
\Q[x_{1},\dots,x_{m}][w_{\lambda}\mid\lambda\in\Lambda_{1}]\subset\Q[x_{1},\dots,x_{m}][w_{\lambda}\mid\lambda\in\Lambda_{2}]
\]
for all $\Lambda_{1}\subset\Lambda_{2}\subset\Lambda$. Let $\Lambda'\subset\Lambda$
be the minimum subset such that $H\in\Q[x_{1},\dots,x_{m}][w_{\lambda}\mid\lambda\in\Lambda']$.
By definition, $\Lambda'$ is a finite set. Let $J$ be the image of the map
\[
\Lambda'\to\{(n,s)\mid n>0,s\in S_{n}\}\ ;\;(n,s,v_{1},\dots,v_{n})\mapsto(n,s).
\]
Since $\Lambda'$ is a finite set, $J$ is also a finite set.
Since $\mathcal{F}_{\pol,m}$ is dense in $\mathcal{F}_{\ser,m}$
in the adic topology, the map
\[
\mathcal{M}(\mathcal{F}_{\pol},S_{\bullet})\xrightarrow{f}
\mathcal{F}_{\pol,m}
\]
is continuously extended to
\[
\mathcal{M}(\mathcal{F}_{\mathsf{ser}},S_{\bullet})\xrightarrow{f}
\mathcal{F}_{\ser,m}
\]
and factors through
\[
\mathcal{M}(\mathcal{F}_{\mathsf{ser}},S_{\bullet})\xrightarrow{M\mapsto(M_{s})_{(n,s)\in J}}\prod_{(n,s)\in J}\Q[[z_{1},\dots,z_{n}]]\xrightarrow{F}\Q[[x_{1},\dots,x_{m}]]
=
\mathcal{F}_{\ser,m}
\]
where
\[
F((M_{n,s})_{(n,s)\in J})=H((M_{n,s}(v_{1},\dots,v_{n}))_{(n,s,v_{1},\dots,v_{n})\in\Lambda'}).
\]
Then, the map $F$ 
naturally extends to
\[
F_{{\C}}:\prod_{(n,s)\in J}{\C}[[z_{1},\dots,z_{n}]]\to{\C}[[x_{1},\dots,x_{m}]]
\]
by
\[
F_{{\C}}((M_{n,s})_{(n,s)\in J})=H((M_{n,s}(v_{1},\dots,v_{n}))_{(n,s,v_{1},\dots,v_{n})\in\Lambda'}).
\]
By assumption, we have $F_{{\C}}(P)=0$ for $P\in\prod_{(n,s)\in J}\Q[[z_{1},\dots,z_{n}]]$. Therefore, $F_{\C}(P)=0$ for $P\in\prod_{(n,s)\in J}\C[[z_{1},\dots,z_{n}]]$
by the identity theorem since for any monomial $x_{1}^{l_{1}}\cdots x_{m}^{l_{m}}$,
the coefficient of $x_{1}^{l_{1}}\cdots x_{m}^{l_{m}}$ in $F_{\C}(P)$
is expressed as a polynomial function of coefficients of components of $P$. Fix
$\Q$-linearly independent complex numbers $\{\alpha_{n,s,j}\}_{(n,s)\in J,1\leq j\leq n}$.
Let $\beta\in\C\setminus\Q$ and put
\[
Q=\left((1-\sum_{j=1}^{n}\alpha_{n,s,j}z_{j})^{\beta}\right)_{(n,s)\in J}\in\prod_{(n,s)\in J}{\C}[[z_{1},\dots,z_{n}]].
\]
For $\lambda=(n,s,v_{1},\dots,v_{n})\in\Lambda'$, put
\[
L_{\lambda}=\sum_{j=1}^{n}\alpha_{n,s,j}v_{j}\in{\C}x_{1}+\cdots+{\C}x_{m}.
\]
Note that $L_{\lambda}\neq L_{\lambda'}$ if $\lambda\neq\lambda'$
by the $\Q$-linear independency of $\{\alpha_{n,s,j}\}$
and the assumption $n>0$. Then
\[
F_{{\C}}(Q)=H\left(\left(\left(1-L_{\lambda}\right)^{\beta}\right)_{\lambda\in\Lambda'}\right).
\]
Now, for $\Lambda''\subset\Lambda'$ and $G\in\Q[x_{1},\dots,x_{m}][w_{\lambda}\mid\lambda\in\Lambda'']$,
we show that if
\[
G\left(\left(\left(1-L_{\lambda}\right)^{\beta}\right)_{\lambda\in\Lambda''}\right)=0
\]
then $G=0$ by the induction on $\#\Lambda''$. Assume that $\Lambda''\neq\emptyset$
since the case $\Lambda''=\emptyset$ is obvious. Choose any $\lambda^{(0)}\in\Lambda''$
and write $G$ as
\[
G=\sum_{k=0}^{K}w_{\lambda^{(0)}}^{k}G_{k}\qquad(G_{k}\in\Q[x_{1},\dots,x_{m}][w_{\lambda}\mid\lambda\in\Lambda''\setminus\{\lambda^{(0)}\}]).
\]
Then
\[
0=G\left(\left(\left(1-L_{\lambda}\right)^{\beta}\right)_{\lambda\in\Lambda''}\right)=\sum_{k=0}^{K}f_{k}\cdot (1-L_{\lambda^{(0)}})^{\beta k}
\]
where $f_{k}=G_{k}\left(\left((1-L_{\lambda})^{\beta}\right)_{\lambda\in\Lambda''\setminus\{\lambda^{(0)}\}}\right)$.
Note that $f_{k}$ and $(1-L_{\lambda^{(0)}})^{\beta k}$ can be regarded
as the analytic functions on the universal covering space of ${\C}^{m}\setminus\bigcup_{\lambda\in\Lambda''}\{1-L_{\lambda}=0\}$.
Then by considering the monodromies around $\{1-L_{\lambda^{(0)}}=0\}$,
we see that
$$\sum_{k=0}^Kf_k \cdot \exp(2\pi\sqrt{-1}\beta k)^n\cdot (1-L_{\lambda^{(0)}})^{\beta k}=0$$
holds for all $n\in\Z$ since $L_{\lambda^{(0)}}\neq L_{\lambda}$
for $\lambda\in\Lambda'\setminus\{\lambda^{(0)}\}$.\\
Here, $1,\exp(2\pi\sqrt{-1}\beta),\dots,\exp(2\pi\sqrt{-1}\beta K)$
are distinct since $\beta\notin\Q$.
Thus, by the regularity of the Vandermonde matrix, all
$f_k$ must be $0$, i.e.,
$$G_{k}\left(\left((1-L_{\lambda})^{\beta}\right)_{\lambda\in\Lambda''\setminus\{\lambda^{(0)}\}}\right)=0$$
for $0\leq k\leq K$. Therefore, by the induction hypothesis, $G_{k}=0$
for $0\leq k\leq K$. Thus $G=0$ is proved. Finally, by putting $\Lambda''=\Lambda$
and $G=H$, we obtain $H=0$, which completes the proof.
\end{proof}

\begin{rem}
We say that a mould-proper map $f:\mathcal{M}(\mathcal{F},S_{\bullet})\to\mathcal{F}_{m}$
is a {\it prolongation} of mould-proper map $g:\mathcal{M}(\mathcal{F}_{\ser},S_{\bullet})\to\mathcal{F}_{\ser,m}$
if both can be expressed by a common polynomial. Such $f$ is uniquely determined by $g$ by Theorem \ref{thm:UniqueProlongation}.
Particularly we have $f=0$ if $g=0$.
We call such uniqueness as the unique prolongation theorem.
The unique prolongation theorem is also very practical to prove several identities.
For example, let $g:\mathcal{M}(\mathcal{F}_{\ser};\Gamma)\times\mathcal{M}(\mathcal{F}_{\ser};\Gamma)\to\mathcal{M}(\mathcal{F}_{\ser};\Gamma)$
be a mould-proper map and $g':\mathcal{M}(\mathcal{F};\Gamma)\times\mathcal{M}(\mathcal{F};\Gamma)\to\mathcal{M}(\mathcal{F};\Gamma)$
its prolongation. Define $h:\mathcal{M}(\mathcal{F}_{\ser};\Gamma)^{3}\to\mathcal{M}(\mathcal{F}_{\ser};\Gamma)$
and $h':\mathcal{M}(\mathcal{F};\Gamma)^{3}\to\mathcal{M}(\mathcal{F};\Gamma)^{3}$
by $h(x,y,z)=g(g(x,y),z)-g(x,g(y,z))$ and $h'(x,y,z)=g'(g'(x,y),z)-g'(x,g'(y,z))$.
Then, by definition, $h'$ is a prolongation of $h$. Thus, by the
unique prolongation theorem, if $h=0$, i.e., $g$ is associative
then $h'=0$, i.e., $g'$ is associative.
\end{rem}

\subsection{Review on alternal(il)ity and symmetral(il)ity} 
\label{sec: alternality, alternility, symmetrality, symmetrilty}
To define the alternality, alternility, symmetrality and symmetrilty for moulds, we prepare the following algebraic formulation.
Put $X:=\{\varia{x_i}{\,\sigma}\}_{i\in\mathbb N,\sigma\in\Gamma}$.
Let $X_\Z$ be the set defined by
$$
X_\Z:=X_\Z^\Gamma
:=\{ \varia{u}{\sigma} \ |\ u=a_1x_1+\cdots+a_kx_k,\ k\in\mathbb N,\ a_j\in\Z,\ \sigma\in\Gamma \},
$$
and let $X_\Z^\bullet$ be the non-commutative free monoid generated by all elements of $X_\Z$ with the empty word $\emptyset$ as the unit.
Whence it is regarded to be the set of sequences of $X_\Z$.
We put
$\mathcal A_X:=\Q \langle X_\Z \rangle$
to be the non-commutative polynomial $\Q$-algebra generated by
$X_\Z$
(i.e. $\mathcal A_X$ is the $\Q$-linear space generated by $X_\Z^\bullet$).
We equip $\mathcal A_X$ a product $\shuffle:\mathcal A_X^{\otimes2} \rightarrow \mathcal A_X$
which is linearly defined by $\emptyset\, \shuffle\, \omega:=\omega\, \shuffle\, \emptyset:=w$ and
\begin{equation}\label{eqn:shuffle product}
	a\omega\ \shuffle\ b\eta
	:=a(\omega\, \shuffle\, b\eta)+b(a\omega\, \shuffle\, \eta),
\end{equation}
for $a,b\in X_\Z$ and $\omega,\eta\in X_\Z^\bullet$.
Then the pair $(\mathcal A_X,\shuffle)$ forms a commutative, associative, unital $\Q$-algebra.
Let $\left\{ \Sh{\omega}{\eta}{\alpha} \right\}_{\omega,\eta,\alpha\in X_\Z^\bullet}$ to be the family in $\Z$ defined by
$$
\omega\ \shuffle\ \eta
=\sum_{\alpha\in X_\Z^\bullet} \Sh{\omega}{\eta}{\alpha}\alpha.
$$
\begin{defn}[{cf. \cite[Definition 1.4]{FK}}]\label{def:al-il-as-is}
Let $\mathcal{F}$ be a family of functions.
A mould $M\in \ARI(\mathcal F;\Gamma)$ (resp. $\in\GARI(\mathcal F;\Gamma)$) is called {\it alternal} (resp. {\it symmetral}) if we have
\begin{align}\label{eqn:def of al,as}
	&\sum_{\alpha\in X_\Z^\bullet}
	\Sh{\varia{x_1,\ \dots,\ x_p}{\sigma_1,\ \dots,\ \sigma_p}}{\varia{x_{p+1},\ \dots,\ x_{p+q}}	{\sigma_{p+1},\ \dots,\ \sigma_{p+q}}}{\alpha} M(\alpha)=0 \\
	&\hspace{5cm}
	(\mbox{resp. } =M\varia{x_1,\ \dots,\ x_p}{\sigma_1,\ \dots,\ \sigma_p}M\varia{x_{p+1},\ \dots,\ x_{p+q}}{\sigma_{p+1},\ \dots,\ \sigma_{p+q}}) \nonumber
\end{align}
for $p,q\geq1$.
The $\Q$-linear space {$\ARI(\mathcal F;\Gamma)_\al$ (resp. $\GARI(\mathcal F;\Gamma)_\as$)} is defined to be the subset of moulds $M$ in $\ARI(\mathcal F;\Gamma)$ (resp. $\GARI(\mathcal F;\Gamma)$) which are alternal (resp. symmetral).
Particularly when $\Gamma$ is a set with one element, we denote $\ARI(\mathcal{F};\Gamma)_\al$, (resp. $\GARI(\mathcal{F};\Gamma)_\as$) simply by $\ARI(\mathcal{F})_\al$ (resp. $\GARI(\mathcal{F})_\as$).
\end{defn}

To reformulate the notion of the alternality and the symmetrality,
we consider the following map.

\begin{defn}[{\cite[Definition 2.3]{Komi}}]\label{def:shmap}
Let $\mathcal{F}$ be a family of functions and $\Gamma$ a set.
The $\Q$-linear map
$$\shmap=\shmap_\Gamma: \mathcal{M}(\mathcal{F};\Gamma) \rightarrow\mathcal{M}_2(\mathcal{F};\Gamma,\Gamma)$$
is defined by
\begin{align*}
	\shmap(M)
	&:=\left(\sum_{\alpha\in X_\Z^\bullet}\Sh{\varia{x_1,\ \dots,\ x_p}{\sigma_1,\ \dots,\ \sigma_p}}{\varia{x_{p+1},\ \dots,\ x_{p+q}}	{\sigma_{p+1},\ \dots,\ \sigma_{p+q}}}{\alpha}
	M(\alpha)\right)_{p,q\in\N_0, \sigma_i\in\Gamma}
\end{align*}
for $M\in\mathcal{M}(\mathcal{F};\Gamma)$.
\end{defn}

This map is algebraic:

\begin{lem}[{\cite[Lemma 2.6]{Komi}, \cite[Lemma 5.1]{Sau}}]\label{lem:shmap is alg. hom.}
The map $\shmap$ is a $\Q$-algebra homomorphism.
\end{lem}

By using this map $\shmap$ and the tensor product (Definition \ref{def:tensor product}), we reformulate symmetral moulds and alternal moulds in terms of dimoulds:

\begin{prop}[{\cite[Proposition 2.4]{Komi}}]\label{prop:gp-like, Lie-like}
Let $\mathcal{F}$ be a family of functions and $\Gamma$ a set.
For a mould $M\in\mathcal{M}(\mathcal{F};\Gamma)$, we have the following equivalence:

{\rm (i).} $M\in\ARI(\mathcal F;\Gamma)_\al$
$\Longleftrightarrow
\shmap(M)=M\otimes \unitmould+\unitmould\otimes M$,

{\rm (ii).} $M\in\GARI(\mathcal F;\Gamma)_\as\Longleftrightarrow
\shmap(M)=M\otimes M$
and $M(\emptyset)=1$.
\end{prop}

In the following, we assume that $\Gamma$ is a group.\footnote{
It is because we need a product structure there
in the right hand side of the equation \eqref{eqn:shufflestar product}.
}
Put $Y:=\{\varia{\,\sigma}{x_i}\}_{i\in\mathbb N,\sigma\in\Gamma}$.
Similarly to $X_\Z$ (resp. $X_\Z^\bullet$), we define $Y_\Z$ (resp. $Y_\Z^\bullet$)
to be the set
$$
Y_\Z:=\{ \varia{\sigma}{u} \ |\ u=a_1x_1+\cdots+a_kx_k,\ k\in\mathbb N,\ a_j\in\Z,\ \sigma\in\Gamma \}
$$
(resp. the non-commutative free monoid generated by all elements of $Y_\Z$ with the empty word $\emptyset$ as the unit).
We put
\begin{equation}\label{eq: K}
\mathcal K:=\Q(x_i\ |\ i\in\N),
\end{equation}
that is, the commutative field generated by all $x_i$ over $\Q$.
We define
$$\mathcal A_Y:=\mathcal K \langle Y_\Z \rangle$$
to be the non-commutative polynomial $\mathcal K$-algebra generated by
$Y_\Z$
(i.e. $\mathcal A_Y$ is the $\mathcal K$-linear space generated by $Y_\Z^\bullet$).
The algebra $\mathcal A_Y$ is equipped with the product $\shuffle$ and the pair $(\mathcal A_Y,\shuffle)$ forms a commutative, associative, unital $\mathcal K$-algebra.
While, we also equip $\mathcal A_Y$ a product $\shuffle_*:\mathcal A_Y^{\otimes2} \rightarrow \mathcal A_Y$
which is linearly defined by $\emptyset\, \shuffle_*\, \omega:=\omega\, \shuffle_*\, \emptyset:=w$ and $\binom{\sigma}{v}\omega\, \shuffle_*\, \binom{\sigma'}{v'}\eta:=0$ for $\binom{\sigma}{v},\binom{\sigma'}{v'}\in Y_\Z$ with $v= v'$
and $\omega,\eta\in Y_\Z^\bullet$, and
\begin{align}\label{eqn:shufflestar product}
	\varia{\sigma}{v}\omega\, \shuffle_*\, \varia{\sigma'}{v'}\eta
	&:=\varia{\sigma}{v}\Bigl( \omega\, \shuffle_*\, \varia{\sigma'}{v'}\eta \Bigr)
		+\varia{\sigma'}{v'}\Bigl( \varia{\sigma}{v}\omega\, \shuffle_*\, \eta \Bigr) \\
	&\hspace{1.5cm}+\frac{1}{v-v'}\left\{\varia{\sigma\sigma'}{\ v}(\omega\, \shuffle_*\, \eta)
	-\varia{\sigma\sigma'}{\ v'}(\omega\, \shuffle_*\, \eta)\right\} \nonumber
\end{align}
for $\binom{\sigma}{v},\binom{\sigma'}{v'}\in Y_\Z$ with $v\neq v'$
and $\omega,\eta\in Y_\Z^\bullet$.
Then the pair $(\mathcal A_Y,\shuffle_*)$ forms a commutative, non-associative, unital $\mathcal K$-algebra.
Let $\left\{ \Shstar{\omega}{\eta}{\alpha} \right\}_{\omega,\eta,\alpha\in Y_\Z^\bullet}$ to be the family in $\mathcal K$ defined by
$$
\omega\ \shuffle_*\ \eta
=\sum_{\alpha\in Y_\Z^\bullet} \Shstar{\omega}{\eta}{\alpha}\alpha.
$$

\begin{defn}
Let $\mathcal{F}$ be a divisible family of functions and $\Gamma$ be a group.
A mould $M\in \overline{\ARI}(\mathcal{F};\Gamma)$ (resp. $\in\overline{\GARI}(\mathcal{F};\Gamma)$) is called {\it alternil} (resp. {\it symmetril}) if we have
\begin{align}\label{eqn:def of il,is}
	&\sum_{\alpha\in Y_\Z^\bullet}
	\Shstar{\varia{\sigma_1,\ \dots,\ \sigma_p}{x_1,\ \dots,\ x_p}}{\varia{\sigma_{p+1},\ \dots,\ \sigma_{p+q}}{x_{p+1},\ \dots,\ x_{p+q}}}{\alpha} M(\alpha)=0 \\
	&\hspace{5cm}
	(\mbox{resp. } =M\varia{\sigma_1,\ \dots,\ \sigma_p}{x_1,\ \dots,\ x_p}M\varia{\sigma_{p+1},\ \dots,\ \sigma_{p+q}}{x_{p+1},\ \dots,\ x_{p+q}}) \nonumber
\end{align}
for $p,q\geq1$.
\end{defn}

\begin{defn}[{\cite[Remark 2.7]{Komi}}]
The $\Q$-linear map
$$
\shmap_*=\shmap_{*,\Gamma}: \overline{\mathcal M}(\mathcal F; \Gamma) \rightarrow \overline{\mathcal{M}_2}(\mathcal{F};\Gamma,\Gamma)$$
is defined by
\begin{align*}
	\shmap_*(M)
	&:=\left(\sum_{\alpha\in Y_\Z^\bullet}\Shstar{\varia{\sigma_1,\ \dots,\ \sigma_p}{x_1,\ \dots,\ x_p}}{\varia{\sigma_{p+1},\ \dots,\ \sigma_{p+q}}{x_{p+1},\ \dots,\ x_{p+q}}}{\alpha}
	M(\alpha)\right)_{p,q\in\N_0, \sigma_i\in\Gamma}
\end{align*}
for $M\in\overline{\mathcal M}(\mathcal F; \Gamma)$.
\end{defn}

This map is also algebraic:

\begin{lem}[{cf. \cite[Remark 2.7]{Komi}}]
The map $\shmap_*$ is a $\Q$-algebra homomorphism.
\end{lem}

\begin{proof}
The proof can be done in the same way as that of
Lemma \ref{lem:shmap is alg. hom.}.
\end{proof}

\begin{rem}
In \cite[Remark 2.7]{Komi},
an analogue of Proposition \ref{prop:gp-like, Lie-like} is shown, that is, we have the following equivalences:

{\rm (i).} 
$M\in \overline{\mathcal M}(\mathcal F;\Gamma)$ is alternil if and only if
$\shmap_*(M)=M\otimes 1_{\overline{\mathcal M}(\mathcal F;\Gamma)}+1_{\overline{\mathcal M}(\mathcal F;\Gamma)}\otimes M$.

{\rm (ii).} 
$M\in\overline{\mathcal M}(\mathcal F;\Gamma)$ is symmetril if and only if
$\shmap_*(M)=M\otimes M$
and $M(\emptyset)=1$.
\end{rem}

To introduce $\GARI(\mathcal{F};\Gamma)_{\as\ast\is}$,
we prepare the map $\swap$:

\begin{defn}\label{defn:swap}
Let $\Gamma$ be a group.
We define the $\mathbb Q$-linear map $\swap:{\mathcal M}(\mathcal{F};\Gamma) \rightarrow \overline{\mathcal M}(\mathcal{F};\Gamma)$ by
\begin{equation}\label{eq:swap}
	\swap( M)
{\scriptsize\left(\begin{array}{rrr}
	\sigma_1,& \dots,& \sigma_m \\
	v_1,& \dots,& v_m
\end{array}\right)}
= M
{\scriptsize\left(\begin{array}{rrrrr}
	v_m,& v_{m-1}-v_m,& \dots,& v_2-v_3,& v_1-v_2 \\
	\sigma_1\cdots\sigma_m,& \sigma_1\cdots\sigma_{m-1},& \dots,& \sigma_1\sigma_2,& \sigma_1
\end{array}\right)}
\end{equation}
for $M\in\mathcal M(\mathcal{F};\Gamma)$.
\end{defn}


\begin{defn}\label{def:GARIas*is}
Let $\mathcal{F}$ be a divisible family of functions and
$\Gamma$ be a group.
The set {$\GARI(\mathcal{F};\Gamma)_{\as\ast\is}$} is defined to be the subset of moulds $M\in \GARI(\mathcal{F};\Gamma)_\as$ which satisfies the condition that there exists a constant mould $C$ such that $C\times\swap(M)$ is symmetril.
The subset $\GARI(\mathcal{F};\Gamma)_{\underline{\as\ast\is}}$ is defined to be the subset of moulds $M\in\GARI(\mathcal{F};\Gamma)_{\as\ast\is}$ whose length 1 components are even, that means
$M\varia{x_1}{\sigma_1}= M\varia{-x_1}{\sigma_1^{-1}}$.
\end{defn}

\begin{eg}\label{eg: Zig and Zag}
We consider the
multiple polylogarithm which is the several-variable complex function
defined by the following power series
\begin{equation}\label{eq:Li}
\Li_{n_1,\dots,n_r}(z_1,\dots,z_r):=\sum_{0<k_1<\cdots <k_r}\frac{z_1^{k_1}\cdots z_r^{k_r}}{k_1^{n_1}\cdots k_r^{n_r}}
\end{equation}
for $n_1,\dots,n_r, r\in\N$.
For $N\in\N$, we denote $\mu_N$  to be the group of $N$-th roots of unity in  $\C$.
The associated moulds
$$\Zag=\{\Zag\varia{u_1,\ \dots,\ u_m}{\epsilon_1,\ \dots,\ \epsilon_m}\}_m
\quad \text{ and }\quad
\Zig=\{\Zig\varia{\epsilon_1,\ \dots,\ \epsilon_m}{v_1,\ \dots,\ v_m}\}_m
$$
in $\mathcal M(\mathcal F;\Gamma)$ with $\Gamma=\mu_N$
defined by
\begin{align*}
&\Zag\varia{u_1,\ \dots,\ u_m}{\epsilon_1,\ \dots,\ \epsilon_m} \\
&\quad=\sum_{n_1,\dots, n_m>0}
\Li^{\scalebox{.5}{$\Sha$}}_{n_1,\dots,n_m}(\frac{\epsilon_1}{\epsilon_2},\dots,\frac{\epsilon_{m-1}}{\epsilon_m},\epsilon_m)
u_1^{n_1-1}(u_1+u_2)^{n_2-1}\cdots
(u_1+\cdots+u_m)^{n_m-1} \\
&\Zig\varia{\epsilon_1,\ \dots,\ \epsilon_m}{v_1,\ \dots,\ v_m}=
\sum_{n_1,\dots, n_m>0}
\Li^\ast_{n_m,\dots,n_1}(\epsilon_m, \dots, \epsilon_1)
v_1^{n_1-1}\cdots v_m^{n_m-1}
\end{align*}
are concerned in \cite{E-ARIGARI, E-flex}.
Here
 $\Li^{\scalebox{.5}{$\Sha$}}_{n_1,\dots,n_m}(\epsilon_1,\dots,\epsilon_m)$
 and
  $\Li^{\ast}_{n_1,\dots,n_m}(\epsilon_1,\dots,\epsilon_m)$
mean the shuffle regularization and the harmonic (stuffle) regularization of
$\Li_{n_1,\dots,n_m}(\epsilon_1,\dots,\epsilon_m)$
respectively (cf. \cite{AK}).
It is explained in \cite{E-flex} that the mould
$\Zag$ belongs to $\GARI(\mathcal{F}_\ser;\Gamma)_{\as\ast\is}$ when $\Gamma=\mu_N$
by the formula
$$\Zig=\Mini\times \swap(\Zag)$$
(see also \cite[Appendix B]{FK} and see Definition \ref{defn:Mini} for $\Mini$).
\end{eg}

Actually it is shown in Proposition \ref{prop: Hopf str on MF ser}
that $\mathcal{M}(\mathcal{F}_\ser;\Gamma)$  is equipped with a
structure of completed Hopf algebra whose coproduct is given by $\shmap$.

%
%
%
\subsection{Review on flexions and gari-product}\label{subsec:flexions and gari}
We introduce the flexions, which are binary maps on the monoid $X_\Z^\bullet$,
in Definition \ref{def:flexion}
in order to define the product $\gari$ on $\GARI(\Gamma)$ in Definition \ref{def:gari}.

\begin{defn}[{\cite[Definition 1.14]{FK}}]\label{def:flexion}
The {\it flexions} are the 
 four binary operators $\urflex{*}{*},\ \ulflex{*}{*},\ \lrflex{*}{*},\ \llflex{*}{*}:X_{\Z}^\bullet\times X_{\Z}^\bullet\rightarrow X_{\Z}^\bullet$
 which are defined by
\begin{align*}
	\urflex{\beta}{\alpha}
	&:= {\scriptsize\left(\begin{array}{rrrr}
		b_1+\cdots+b_n+a_1,& a_2,& \dots,& a_m \\
		\sigma_1,& \sigma_2,& \dots,& \sigma_m
	\end{array}\right)} , \\
	\ulflex{\alpha}{\beta}
	&:= {\scriptsize\left(\begin{array}{rrrl}
		a_1,& \dots,& a_{m-1},& a_m+b_1+\cdots+b_n \\
		\sigma_1,& \dots,& \sigma_{m-1},& \sigma_m
	\end{array}\right)} , \\
	\lrflex{\beta}{\alpha}
	&:={\scriptsize\left(\begin{array}{rrr}
		a_1,& \dots,& a_m \\
		\tau_n^{-1}\sigma_1,& \dots,& \tau_n^{-1}\sigma_m
	\end{array}\right)} , \\
	\llflex{\alpha}{\beta}
	&:={\scriptsize\left(\begin{array}{rrr}
		a_1,& \dots,& a_m \\
		\sigma_1\tau_1^{-1},& \dots,& \sigma_m\tau_1^{-1}
	\end{array}\right)}  , \\
	\urflex{\emptyset}{\gamma}&:=\ulflex{\gamma}{\emptyset}:=\lrflex{\emptyset}{\gamma}:=\llflex{\gamma}{\emptyset}:=\gamma , \\
	\urflex{\gamma}{\emptyset}&:=\ulflex{\emptyset}{\gamma}:=\lrflex{\gamma}{\emptyset}:=\llflex{\emptyset}{\gamma}:=\emptyset ,
\end{align*}
for $\alpha=\varia{a_1,\dots,a_m}{\sigma_1,\dots,\sigma_m}$, $\beta=\varia{b_1,\dots,b_n}{\tau_1,\dots,\tau_n}\in X_{\Z}^\bullet$ ($m,n\geq1$) and $\gamma\in X_{\Z}^\bullet$.
\end{defn}

To give a group structure on $\GARI(\mathcal{F};\Gamma)$, we introduce the map
$\garit$.

\begin{defn}[{\cite[(2.37)]{E-flex}}]\label{def:garit}
Let $\mathcal{F}$ be a family of functions and $\Gamma$ a group.
For $B\in\GARI(\mathcal{F};\Gamma)$, we define the map
$$
\garit(B):\mathcal M(\mathcal F;\Gamma) \rightarrow \mathcal M(\mathcal F;\Gamma)
$$
by $\garit(B)(A)(\vecx_0):=A(\emptyset)$ and
\begin{align*}
	&\garit(B)(A)(\vecx_m) \\
	&:=\sum_{s\geq1}\sum_{\substack{
			\vecx_m=\alpha_1\beta_1\gamma_1\cdots\alpha_s\beta_s\gamma_s \\
			\beta_i, \
			\gamma_j\alpha_{j+1}\neq\emptyset \\
			 1\leq i\leq s, \
			 1\leq j\leq s-1}}
		A(\ulflex{\urflex{\alpha_1}{\beta_1}}{\gamma_1}\cdots \ulflex{\urflex{\alpha_s}{\beta_s}}{\gamma_s}) \\
		&\hspace{5cm}\cdot B(\llflex{\alpha_1}{\beta_1})\cdots B(\llflex{\alpha_s}{\beta_s})B^{\times-1}(\lrflex{\beta_1}{\gamma_1})\cdots B^{\times-1}(\lrflex{\beta_s}{\gamma_s})
\end{align*}
for any $A\in\mathcal M(\mathcal F;\Gamma)$.
Here the symbol $B^{\times-1}$ means the inverse element of $B$ on the group $(\GARI(\mathcal{F};\Gamma),\times)$.
\end{defn}

\begin{eg}
For our simplicity, we denote $\omega_i:=\varia{x_i}{\sigma_i}$ for $i\in\N$, that is, we have $\vecx_m=(\omega_1,\dots,\omega_m)$.
\begin{align*}
\garit(B)(A)(\vecx_1)&=A(\omega_1), \\
\garit(B)(A)(\vecx_2)&=A(\omega_1,\omega_2) + A(\urflex{\omega_1}{\omega_2})B(\llflex{\omega_1}{\omega_2}) + A(\ulflex{\omega_1}{\omega_2})B^{\times-1}(\lrflex{\omega_1}{\omega_2}), \\
\garit(B)(A)(\vecx_3)&=A(\omega_1,\omega_2,\omega_3)
+ A(\urflex{\omega_1}{\omega_2,\omega_3})B(\llflex{\omega_1}{\omega_2,\omega_3})
+ A(\ulflex{\omega_1,\omega_2}{\omega_3})B^{\times-1}(\lrflex{\omega_1,\omega_2}{\omega_3}) \\
&\quad +A(\urflex{\omega_1,\omega_2}{\omega_3})B(\llflex{\omega_1,\omega_2}{\omega_3})
+ A(\ulflex{\omega_1}{\omega_2,\omega_3})B^{\times-1}(\lrflex{\omega_1}{\omega_2,\omega_3}) \\
&\quad + A(\urflex{\omega_1}{\ulflex{\omega_2}{\omega_3}}) B(\llflex{\omega_1}{\omega_2})B^{\times-1}(\lrflex{\omega_2}{\omega_3})
+ A(\ulflex{\omega_1}{\omega_2}\omega_3)B^{\times-1}(\lrflex{\omega_1}{\omega_2}) \\
&\quad + A(\omega_1\urflex{\omega_2}{\omega_3})B(\llflex{\omega_2}{\omega_3}).
\end{align*}
\end{eg}

\begin{defn}[{\cite[(2.43)]{E-flex}}]\label{def:gari}
Let $\mathcal{F}$ be a family of functions and $\Gamma$ a group.
We define the map
$$\gari:\GARI(\mathcal{F};\Gamma)\times \GARI(\mathcal{F};\Gamma)\rightarrow\GARI(\mathcal{F};\Gamma)$$
by
\begin{equation*}
	\gari(A,B):=\garit(B)(A)\times B
\end{equation*}
for $A,B\in\GARI(\mathcal{F};\Gamma)$.
\end{defn}

\begin{prop}\label{prop: GARI forms a group}
Under the above $\gari$-product,
$\GARI(\mathcal{F};\Gamma)$ forms a group.
\end{prop}

\begin{proof}
Theorem \ref{thm:ma(circledast)=gari(ma,ma)} in Appendix
\ref{app:correspondence between circledast and gari}
and Theorem \ref{thm:UniqueProlongation} verify the validity of the associativity.
By definition, $\garit(A)(\unitmould)=\unitmould$ and $\garit(\unitmould)(A)=A$.
So we have
$$
\gari(A,\unitmould)=\gari(\unitmould,A)=A,
$$
for any $A\in\GARI(\mathcal{F};\Gamma)$.
For any $B\in\GARI(\mathcal{F};\Gamma)$, the mould $\invgari(B)\in\GARI(\mathcal{F};\Gamma)$
which satisfies
\begin{align*}\label{eqn:def of invgari}
\bigl( \gari(\invgari(B),B) \bigr)(\vecx_m)=0
\end{align*}
for all $m\in\N$,
can be constructed recursively length by length.
We see  that it is the inverse element of $B$ under the product $\gari$.
Hence  $(\GARI(\mathcal{F};\Gamma),\gari)$ forms a group.
\end{proof}

To give a Lie algebra structure on $\ARI(\mathcal{F};\Gamma)$, we introduce the map
$\arit$.

\begin{defn}[{\cite[Definition 1.9]{FK}}]\label{def:aritu}
Let $\mathcal{F}$ be a family of functions and $\Gamma$ a group.
Let $ B\in \mathcal M(\mathcal F;\Gamma)$.
The linear map
$$\arit( B):\mathcal M(\mathcal F;\Gamma) \rightarrow \mathcal M(\mathcal F;\Gamma)$$
is defined by
\footnote{In \cite{FK}, the map $\arit(B)$ is defined only for $B\in\ARI(\Gamma)$, but this map is defined for any $B\in \mathcal M(\mathcal F;\Gamma)$ as a linear map on $\mathcal M(\mathcal F;\Gamma)$.}
\begin{align*}
\arit( B)( A)(\vecx_m)
:=\left\{\begin{array}{ll}
	\displaystyle \sum_{\substack{\vecx_m=\alpha\beta\gamma \\
	\beta,\gamma\neq\emptyset}}
	A(\alpha\urflex{\beta}{\gamma}) B(\llflex{\beta}{\gamma})
	-\sum_{\substack{\vecx_m=\alpha\beta\gamma \\
		\alpha,\beta\neq\emptyset}}
	A(\ulflex{\alpha}{\beta}\gamma) B(\lrflex{\alpha}{\beta}) & (m\geq2),\\
0 & (m=0,1),
\end{array}\right.
\end{align*}
for $A\in\mathcal M(\mathcal F;\Gamma)$.
\end{defn}

\begin{defn}[{\cite[Definition 1.11]{FK}}]\label{defn:ari-bracket}
Let $\mathcal{F}$ be a family of functions and $\Gamma$ a group.
The {\it $\ari$-bracket} means
the bilinear map
$$\ari:\ARI(\mathcal{F};\Gamma)\times \ARI(\mathcal{F};\Gamma)\rightarrow\ARI(\mathcal{F};\Gamma)$$
which is defined by
\begin{equation}\label{ari-bracket}
	\ari( A, B)
	:=\arit( B)( A)-\arit( A)( B)
	+[ A,  B]
\end{equation}
for $ A, B\in\ARI(\mathcal{F};\Gamma)$.
Here we have
$ [A, B]:= A\times  B- B\times  A$.
\footnote{In the papers \cite{E-flex}, \cite{S-ARIGARI} and \cite{RS}, the product $A\times B$ (resp. the bracket $[A,B]$) is denoted by $mu(A,B)$ (resp. $lu(A,B)$).}
\end{defn}

We note that  $\arit$ and $\ari$ are denoted by
$\arit_u$ and $\ari_u$ respectively in \cite{FK}.

\begin{prop}
Under the above $\ari$-bracket,
$\ARI(\mathcal{F};\Gamma)$ forms a Lie algebra.
\end{prop}

\begin{proof}
It is shown in \cite[Appendix A.1]{FK}.
\end{proof}

\begin{rem}
Actually we could obtain the above proposition from Theorem \ref{thm:ma(<,>)=ari(ma,ma)} in Appendix \ref{app:correspondence between exp and expari}
by combining  with
the unique prolongation theorem (Theorem \ref{thm:UniqueProlongation}).
\end{rem}

\subsection{Review on the map $\ma_\Gamma$}\label{subsec:map ma}
We review the notion of the map $\ma_\Gamma$ (cf. \cite{S})
which relates non-commutative formal power series with a mould.

Let $\Gamma$ be a set with $n-1$ elements.
Let $\frak f_{\Gamma}$
\footnote{
Particularly when $\Gamma$ is given by $\{1,2,\dots, n-1\}$,
we simply denote it by $\frak f_n$.}
be the free Lie algebra over $\Q$ with $n$ variables $f_0$ and $f_\sigma$ ($\sigma\in\Gamma$), and $U\frak f_{\Gamma}:=\Q\langle f_0,f_\sigma\ |\ \sigma\in\Gamma\rangle$ be its universal enveloping algebra.
We denote $\widehat{U\frak f_{\Gamma}}:=\Q\langle\langle f_0,f_\sigma\ |\ \sigma\in\Gamma\rangle\rangle$ to be their completions by degree.
Let
\begin{equation}\label{eq:e}
e_0:\widehat{U\frak f_{\Gamma}}\to\mathbb{Q}
\end{equation}
be the map taking the coefficient of $f_{0}$. Define
\begin{equation}\label{eq:UfGamma}
{\widehat{U\frak f_{\Gamma}}^\dagger}=\{w\in\widehat{U\frak f_{\Gamma}}\mid(e_0\otimes{\rm id})\circ\Delta(w)=0
\}.
\end{equation}
We note that $1,f_\sigma\in\widehat{U\frak f_{\Gamma}}^\dagger$ for $\sigma\in\Gamma$
but $f_0^m\not\in\widehat{U\frak f_{\Gamma}}^\dagger$ for $m\in\N$.


\begin{defn}\label{def:ma}
With $h\in\widehat{U\frak f_{\Gamma}}$ presented as
\begin{equation}\label{eq: word expansion}
h=  \sum_{r=0}^\infty\sum_{(\sigma_1,\dots,\sigma_r)\in\Gamma^r}\sum_{k_0,\dots,k_r\in\N_{0}}
\coeff{h}{k_0,\dots,k_r}{\sigma_1,\dots,\sigma_r}
f_0^{k_0}f_{\sigma_1}\cdots f_{\sigma_r}f_0^{k_r}
\in \widehat{U\frak f_{\Gamma}},
\end{equation}
we associate a mould
$$
\ma_{\Gamma,h}=\{\ma^r_{\Gamma,h}({}_{\sigma_1,\dots,\sigma_r}^{x_1,\dots,x_r})\}_{r\in\Z_{\geq 0},(\sigma_1,\dots,\sigma_r)\in\Gamma^r}
\in{\mathcal M(\mathcal{F}_{{\ser}};\Gamma)}
$$
with
\begin{align}\notag
&
\ma^r_{\Gamma,h}({}_{\sigma_1,\dots,\sigma_r}^{x_1,\dots,x_r})
=\vimo^r_{\Gamma,h}({}_{\qquad \sigma_1,\dots,\sigma_r}^{0,x_1,x_1+x_2,\dots,x_1+\cdots+x_r}), \\
\label{eq:vimo}
&
\vimo^r_{\Gamma,h}({}^{z_0,\dots,z_r}_{\sigma_1,\dots,\sigma_r})
=\sum_{k_0,\dots,k_r\in\N_0}
\coeff{h}{k_0,\dots,k_r}{\sigma_1,\dots,\sigma_r}
z_0^{k_0}z_1^{k_1}z_2^{k_2}\cdots z_r^{k_r}.
\end{align}
\end{defn}

{
We note that the above
$\ma^r_{\Gamma,h}({}_{\sigma_1,\dots,\sigma_r}^{x_1,\dots,x_r})$  and
$\vimo^r_{\Gamma,h}({}^{z_0,\dots,z_r}_{\sigma_1,\dots,\sigma_r})$
actually agree with
$\ma^r_{h}({}_{\sigma_1^{-1},\dots,\sigma_r^{-1}}^{x_1,\dots,x_r})$
and
$\vimo^r_h({}^{z_0,\dots,z_r}_{\sigma_1^{-1},\dots,\sigma_r^{-1}})$
in \cite[Definition 2.4]{FK}
which are defined only when $\Gamma$ forms a group.

We note that $\ma^r_{\Gamma,h}({}_{\sigma_1,\dots,\sigma_r}^{x_1,\dots,x_r})$ is the constant term of $h$
when $r=0$.

\begin{lem}\label{lem:UfGamma}
For $h\in \widehat{U\frak f_{\Gamma}}$,
the following conditions are equivalent.
\begin{enumerate}
\item $h\in \widehat{U\frak f_{\Gamma}}^\dagger$.
\item For all $r\in\Z_{\geq 0}$, $\sigma_1,\dots,\sigma_r\in\Gamma$ and $k_0,\dots,k_r\in\N_0$,
\[
\sum_{i=0}^{r}(k_i+1)\coeff{h}{k_0,\dots,k_i+1,\dots,k_r}{\sigma_1,\dots,\sigma_r}=0.
\]
\item For all $r\in\Z_{\geq 0}$, $\sigma_1,\dots,\sigma_r\in\Gamma$,
\[
\sum_{i=0}^{r}(\frac{\partial}{\partial z_i}) \vimo^r_{\Gamma,h}({}^{z_0,\dots,z_r}_{\sigma_1,\dots,\sigma_r})=0.
\]
\item For all $r\in\Z_{\geq 0}$, $\sigma_1,\dots,\sigma_r\in\Gamma$, $\vimo^r_{\Gamma,h}({}^{z_0,\dots,z_r}_{\sigma_1,\dots,\sigma_r})$ is invariant under the translation $(z_0,\dots,z_r)\mapsto (z_0+\alpha,\dots,z_r+\alpha)$ for any $\alpha\in\Q$.
\end{enumerate}
\end{lem}

\begin{proof}
The equivalence of (1) and (2) is immediate by
\begin{equation}\label{eq: en otimes id Delta}
(e_{0}\otimes{\rm id})\circ\Delta(h)=\sum_{r=0}^\infty
\sum_{(\sigma_1,\dots,\sigma_r)\in\Gamma^r}\sum_{k_0,\dots,k_r\in\N_{0}}
\sum_{i=0}^{r}(k_i+1)\coeff{h}{k_0,\dots,k_i+1,\dots,k_r}{\sigma_1,\dots,\sigma_r}
f_0^{k_0}f_{\sigma_1}\cdots f_{\sigma_r}f_0^{k_r}.
\end{equation}
The equivalence of (2) and (3) is direct.
The equivalence of (3) and (4) follows from the fact that, for a $P\in\Q[[z_0,\dots,z_r]]$,
$\sum_{i=0}^{r}(\frac{\partial}{\partial z_i})P=0$
if and only if $P$ is translation invariant.
\end{proof}

Denote $d_0$ to be the derivation of $\widehat{U\frak f_\Gamma}$
sending $f_0\mapsto 1$ and $f_\sigma\mapsto 0$ ($\sigma\in\Gamma$).
Then  by \eqref{eq: en otimes id Delta}, 
we see that
$$d_0=(e_{0}\otimes{\rm id})\circ\Delta.
$$
Whence we have a characterization
\begin{equation}\label{eq: kernel diff charzn}
\widehat{U\frak f_\Gamma}^\dag=\ker d_0.
\end{equation}

\begin{lem}\label{lem: Hopf structure of UF-Gamma-dag}
The $\Q$-linear space $\widehat{U\frak f_\Gamma}^\dag$ forms a complete Hopf subalgebra of
$\widehat{U\frak f_\Gamma}$
under the normal non-commutative  product $\cdot$ and the coproduct $\Delta$.
\end{lem}

\begin{proof}
By \eqref{eq: kernel diff charzn},  it is direct to see  that
$\widehat{U\frak f_\Gamma}^\dag$ forms a
$\Q$-subalgebra  under the  product $\cdot$.
We note that
\begin{align*}
(d_0\otimes\id)\circ \Delta&=
((e_0\otimes \id)\otimes\id)\circ(\Delta\otimes \id)\circ\Delta
=(e_0\otimes (\id\otimes\id))\circ(\id\otimes\Delta)\circ\Delta \\
&=(e_0\otimes\Delta)\circ\Delta
=(\id\otimes\Delta)\circ(e_0\otimes\id)\circ\Delta, \\
(\id\otimes d_0)\circ \Delta&=
(\id\otimes(\id\otimes e_0))\circ(\id\otimes\Delta)\circ\Delta
=( (\id\otimes\id)\otimes e_0)\circ(\Delta\otimes\id)\circ\Delta. \\
&=(\Delta\otimes e_0)\circ\Delta
=(\Delta\otimes\id)\circ(\id\otimes e_0)\circ\Delta.
\end{align*}
Assume that $h\in \widehat{U\frak f_\Gamma}$.
Then we have $(e_0\otimes \id)\circ \Delta(h)=0$
and also $(\id\otimes e_0)\circ \Delta(h)=0$
because the coproduct $\Delta$ is co-commutative.
Then by the above two equalities we have
$(d_0\otimes\id)\circ \Delta(h)=0$ and
$(\id\otimes d_0)\circ \Delta(h)=0$.
That means $\Delta(h) \in \widehat{U\frak f_\Gamma}^\dag\widehat{\otimes}\widehat{U\frak f_\Gamma}^\dag$.
By \eqref{eq: kernel diff charzn}, it is immediate to see that
it is closed under the antipode.
Whence $\widehat{U\frak f_\Gamma}^\dag$ forms  a Hopf subalgebra.
\end{proof}

Let $D^1{\frak f}_\Gamma$ denote the Lie subalgebra of ${\frak f}_\Gamma$
which is the depth $\geq$ 1-part of
$\frak f_\Gamma$, that is,
$${\frak f}_\Gamma=\Q f_0 \oplus D^1{\frak f}_\Gamma.$$

\begin{prop}\label{prop: UF=UD1F}
The complete Hopf algebra $\widehat{U\frak f_\Gamma}^\dag$ is identified with the completed universal enveloping algebra  of the Lie algebra $D^1{\frak f}_\Gamma$.
$$
\widehat{U\frak f_\Gamma}^\dag\simeq \widehat{U}(D^1{\frak f}_\Gamma).
$$
\end{prop}

\begin{proof}
It is because the Lie-like part of $\widehat{U\frak f_\Gamma}^\dag$,
the subspace of elements $h$ satisfying $\Delta(h)=h\otimes1 +1\otimes h$,
is equal to $D^1{\frak f}_\Gamma$.
\end{proof}

%

Actually  $\widehat{U\frak f_{\Gamma}}^\dagger$ is isomorphic to
 $\mathcal M(\mathcal{F}_{{\ser}};\Gamma)$: 

\begin{prop}\label{prop: isom ma}
The map
\begin{equation}\label{eq: ma Gamma}
\ma_\Gamma:\widehat{U\frak f_{\Gamma}}^\dagger \to \mathcal M(\mathcal{F}_{{\ser}};\Gamma)\ ;\ h \mapsto \ma_{\Gamma,h}
\end{equation}
is an isomorphism of algebras.
\footnote{Particularly  in the case of $\Gamma=\{1\}$,
we simply denote
$\ma$
instead of $\ma_\Gamma$.
}
\end{prop}

\begin{proof}
The bijectivity of the map $\ma_\Gamma$ follows from the equivalence of (1) and (4) in Lemma \ref{lem:UfGamma}.
We prove that it forms an algebra homomorphism, i.e.,
\begin{equation}\label{eqn:ma(hg)=ma(h)timesma(g)}
\ma_\Gamma(hg)=\ma_\Gamma(h) \times \ma_\Gamma(g)
\end{equation}
for $h,g\in \widehat{U\frak f_{\Gamma}}^\dagger$.
By using the expression \eqref{eq: word expansion}, we have
\begin{align*}
hg
=\sum_{m=0}^\infty \sum_{\sigma_i\in\Gamma} \sum_{k_i\in\N_{0}}
\left\{
\sum_{r=0}^m\sum_{\substack{k'_r+k''_r=k_r \\ k'_r,k''_r\geq0}}
\coeff{h}{k_0,\dots,k_{r-1},k'_r}{\sigma_1,\dots,\sigma_r}\coeff{g}{k''_r,k_{r+1},\dots,k_m}{\sigma_{r+1},\dots,\sigma_m}
\right\}
f_0^{k_0}f_{\sigma_1}\cdots f_{\sigma_m}f_0^{k_m}.
\end{align*}
So, for $m\geq0$, we calculate
\begin{align*}
&\vimo^m_{\Gamma,hg}({}^{z_0,\dots,z_m}_{\sigma_1,\dots,\sigma_m}) \\
&=\sum_{k_0,\dots,k_m\in\N_0}
\left\{
\sum_{r=0}^m\sum_{\substack{k'_r+k''_r=k_r \\ k'_r,k''_r\geq0}}
\coeff{h}{k_0,\dots,k_{r-1},k'_r}{\sigma_1,\dots,\sigma_r}\coeff{g}{k''_r,k_{r+1},\dots,k_m}{\sigma_{r+1},\dots,\sigma_m}
\right\}
z_0^{k_0}z_1^{k_1}\cdots z_m^{k_m} \\
&=\sum_{r=0}^m
\left\{
\sum_{k_0,\dots,k_{r-1},k'_r\in\N_0}
\coeff{h}{k_0,\dots,k_{r-1},k'_r}{\sigma_1,\dots,\sigma_r}
z_0^{k_0}z_1^{k_1}\cdots z_{r-1}^{k_{r-1}}z_r^{k'_r}
\right\} \\
&\hspace{3cm} \cdot \left\{
\sum_{k''_r,k_{r+1},\dots,k_{r+s}\in\N_0}
\coeff{g}{k''_r,k_{r+1},\dots,k_m}{\sigma_{r+1},\dots,\sigma_m}
z_r^{k''_r}z_{r+1}^{k_{r+1}}\cdots z_m^{k_m}
\right\} \\
&=\sum_{r=0}^m
\vimo^r_{\Gamma,h}({}^{z_0,\dots,z_r}_{\sigma_1,\dots,\sigma_r})
\cdot \vimo^{m-r}_{\Gamma,g}({}^{z_r,z_{r+1},\dots,z_m}_{\sigma_{r+1},\dots,\sigma_m}).
\end{align*}
By Lemma \ref{lem:UfGamma} (4), we get
$$
\vimo^m_{\Gamma,hg}({}^{z_0,\dots,z_m}_{\sigma_1,\dots,\sigma_m})
=\sum_{r=0}^m
\vimo^r_{\Gamma,h}({}^{z_0,\dots,z_r}_{\sigma_1,\dots,\sigma_r})
\cdot \vimo^{m-r}_{\Gamma,g}({}^{0, z_{r+1}-z_r,\dots,z_m-z_r}_{\sigma_{r+1},\dots,\sigma_m}).
$$
By putting $z_0=0$ and $z_i=x_1+\cdots + x_i$ ($1\leq i\leq m$), we obtain
$$
\ma^m_{\Gamma,hg}({}^{x_1,\dots,z_m}_{\sigma_1,\dots,\sigma_m})
=\sum_{r=0}^m
\ma^r_{\Gamma,h}({}^{x_1,\dots,x_r}_{\sigma_1,\dots,\sigma_r})
\cdot \ma^{m-r}_{\Gamma,g}({}^{x_{r+1},\dots,x_m}_{\sigma_{r+1},\dots,\sigma_m})
$$
for $m\geq0$, that is, we get \eqref{eqn:ma(hg)=ma(h)timesma(g)}.
Hence, the map $\ma_{\Gamma}$ is an algebra isomorphism.

\end{proof}

We note that  for an invertible $g$, we have
\begin{equation}\label{eq: ma and inverse}
 \ma_{\Gamma}(g^{-1})=\ma_{\Gamma}(g)^{\times-1}
\end{equation}
by \eqref{eqn:ma(hg)=ma(h)timesma(g)}.


%

We will see  in Proposition \ref{prop: Hopf str on MF ser}
that the algebra isomorphism \eqref{eq: ma Gamma} actually
realizes an isomorphism of completed
Hopf algebras by encoding a Hopf algebra structure in
the algebra $\mathcal{M}(\mathcal{F}_\ser;\Gamma)$.

We remark that
a \lq polynomial' version of the isomorphism \eqref{eq: ma Gamma}
is discussed in \cite[Theorem 3.2.3]{S-ARIGARI}.
\footnote{
It looks there are errors on the statements of  \cite[Theorem 3.2.3]{S-ARIGARI} :
$\Q\langle C\rangle$ must be $\Q\langle C\rangle_{>0}$
and $\ARI^{\pol}$ must be its finite part $\ARI^{\pol,\text{fin}}$
(cf. \cite[Definition 1.3]{FK}).
}

\begin{defn}\label{defn: pari anti}
Let $\mathcal{F}$ be a family of functions.
We define the following $\mathbb Q$-linear involutions $\pari$, $\anti$, $\minus:\mathcal M(\mathcal F;\Gamma) \rightarrow \mathcal M(\mathcal F;\Gamma)$ by
\begin{align}\label{eq:swap}
\pari( M)
{\scriptsize\left(\begin{array}{rrr}
	u_1,& \dots,& u_m \\
	\sigma_1,& \dots,& \sigma_m
\end{array}\right)}
&= (-1)^mM
{\scriptsize\left(\begin{array}{rrrrr}
	u_1,& \dots,& u_m \\
	\sigma_1,& \dots,& \sigma_m
\end{array}\right)}, \\
\anti( M)
{\scriptsize\left(\begin{array}{rrr}
	u_1,& \dots,& u_m \\
	\sigma_1,& \dots,& \sigma_m
\end{array}\right)}
&= M
{\scriptsize\left(\begin{array}{rrrrr}
	u_m,& \dots,& u_1 \\
	\sigma_m,& \dots,& \sigma_1
\end{array}\right)}, \\
\minus( M)
{\scriptsize\left(\begin{array}{rrr}
	u_1,& \dots,& u_m \\
	\sigma_1,& \dots,& \sigma_m
\end{array}\right)}
&	= M
	{\scriptsize\left(\begin{array}{rrr}
		-u_1,& \dots,& -u_m \\
		\sigma_1,& \dots,& \sigma_m
\end{array}\right)} \label{eq:neg}
\end{align}
for $M\in\mathcal M(\mathcal F;\Gamma)$.
\end{defn}
}

We could extend the operation $\minus$ for polymoulds
$\minus: \mathcal M_n(\mathcal F;\Gamma_1,\dots, \Gamma_n)\to
\mathcal M_n(\mathcal F;\Gamma_1,\dots, \Gamma_n)$
by
\begin{align*}
M &
\varia{x_1^{(1)}, \dots, x_{r_1}^{(1)};\,x_{1}^{(2)}, \dots, x_{r_2}^{(2)};\,\cdots\cdots;\,x_{1}^{(n)}, \dots, x_{r_n}^{(n)}}
{\sigma_1^{(1)}, \dots, \sigma_{r_1}^{(1)};\,\sigma_{1}^{(2)}, \dots, \sigma_{r_2}^{(2)};\,\cdots\cdots;\,\sigma_{1}^{(n)}, \dots, \sigma_{r_n}^{(n)}} \\
& \qquad\qquad\mapsto
M
\varia{-x_1^{(1)}, \dots, -x_{r_1}^{(1)};\,-x_{1}^{(2)}, \dots, -x_{r_2}^{(2)};\,\cdots\cdots;\,-x_{1}^{(n)}, \dots, -x_{r_n}^{(n)}}
{\sigma_1^{(1)}, \dots, \sigma_{r_1}^{(1)};\,\sigma_{1}^{(2)}, \dots, \sigma_{r_2}^{(2)};\,\cdots\cdots;\,\sigma_{1}^{(n)}, \dots, \sigma_{r_n}^{(n)}}.
\end{align*}





\subsection{The maps $\dima_{\Gamma,\Gamma}$ and $\ma_\Gamma$}
\label{sec: dima and ma}
We introduce  $\dima_{\Gamma_1,\Gamma_2}$ \lq dimould variants' of  $\ma$ in Proposition \ref{prop: isom dima}
and then we show  that  the map $\ma_\Gamma$ gives a Hopf algebra isomorphism  in Proposition \ref{prop: Hopf str on MF ser} and it realizes a correspondence between
$\exp\widehat{\frak f_\Gamma}^\dag$ and $\GARI(\mathcal F_\ser;\Gamma)_{\as}$
and the one between  $\widehat{\frak f_\Gamma}^\dag$ and $\ARI(\mathcal F_\ser;\Gamma)_{\as}$
in Proposition \ref{prop:bijection expf and GARIas}.


\begin{defn}\label{defn: dima and divimo}
Let $\Gamma_1$ and $\Gamma_2$ be finite sets.
For $\Phi\in \widehat{U\frak f_{\Gamma_1}} \widehat\otimes \widehat{U\frak f_{\Gamma_2}}$
which we describe as
\begin{align}
\Phi
=\sum_{p,q\geq0}
\sum_{\substack{\sigma'_i\in\Gamma_1 \\
\sigma''_j\in\Gamma_2}}
\sum_{d'_i,d''_j\in\N_{0}}
\left\langle
 \Phi \Bigm|\varia{d'_{0},\dots, d'_{p}}{\sigma'_{1},\dots, \sigma'_{p}};\varia{d''_{0},\dots, d''_{q}}{\sigma''_{1},\dots, \sigma''_{q}}
 \right\rangle
f_0^{d'_{0}}f_{\sigma'_1}\cdots f_{\sigma'_p}f_0^{d'_{p}}\otimes f_0^{d''_{0}}f_{\sigma''_1}\cdots f_{\sigma''_q}f_0^{d''_{q}},
\end{align}
we associate a dimould
\begin{align*}
\dima_{\Gamma_1,\Gamma_2,\Phi}
=\left\{
\dima^{p,q}_{\Gamma_1,\Gamma_2,\Phi}\varia{x_1, \dots, x_p;\,x_{p+1}, \dots, x_{p+q}}{\sigma'_{1},\dots, \sigma'_{p};\,\sigma''_{1},\dots, \sigma''_{q}}
\right\}_{p,q\in\N_0,\sigma'_i\in\Gamma_1,\sigma''_j\in\Gamma_2}
\in{\mathcal M(\mathcal{F}_{{\ser}};\Gamma_1,\Gamma_2)}
\end{align*}
with
\begin{align*}
&
\dima^{p,q}_{\Gamma_1,\Gamma_2,\Phi}\varia{x_1, \dots, x_p;\,x_{p+1}, \dots, x_{p+q}}{\sigma'_{1},\dots, \sigma'_{p};\,\sigma''_{1},\dots, \sigma''_{q}} \\
&\quad =\divimo^{p,q}_{\Gamma_1,\Gamma_2,\Phi}\varia{0,x_1,x_1+x_2,\dots,x_1+\cdots+x_p;\,0,x_{p+1},x_{p+1}+x_{p+2},\dots,x_{p+1}+\cdots+x_{p+q}}{\hspace{1.cm} \sigma'_{1},\dots, \sigma'_{p} \hspace{1.1cm};\,\hspace{1.5cm} \sigma''_{1},\dots, \sigma''_{q}}, \\
&
\divimo^{p,q}_{\Gamma_1,\Gamma_2,\Phi}\varia{z_0, \dots, z_p;\,z'_0, \dots, z'_q}{\sigma'_{1},\dots, \sigma'_{p};\,\sigma''_{1},\dots, \sigma''_{q}} \\
&\quad =\sum_{d'_i,d''_j\in\N_{0}}
\left\langle
 \Phi \Bigm|\varia{d'_{0},\dots, d'_{p}}{\sigma'_{1},\dots, \sigma'_{p}};\varia{d''_{0},\dots, d''_{q}}{\sigma''_{1},\dots, \sigma''_{q}}
 \right\rangle
z_0^{d'_{0}}\cdots z_p^{d'_{p}}(z'_0)^{d''_{0}}\cdots(z'_q)^{d''_{q}}.
\end{align*}
\end{defn}

Particularly for
$\Phi=h_1\otimes h_2$ with $h_1 \in \widehat{U\frak f_{\Gamma_1}} ,h_2 \in \widehat{U\frak f_{\Gamma_2}}$, we have
\begin{equation}\label{eq: dima=ma otimes ma}
\dima_{\Gamma_1,\Gamma_2, \Phi}=\ma_{\Gamma_1, h_1}\otimes\ma_{\Gamma_2, h_2}.
\end{equation}

The following is a dimould version of  Proposition \ref{prop: isom ma}.

\begin{prop}\label{prop: isom dima}
Let $\Gamma_1$ and $\Gamma_2$ be finite sets.
For
The map
$$
\dima_{\Gamma_1,\Gamma_2}:{\widehat{U\frak f_{\Gamma_1}}^\dagger \widehat\otimes \  \widehat{U\frak f_{\Gamma_2}}^\dagger} \to \mathcal M_2(\mathcal{F}_\ser;\Gamma_1,\Gamma_2)
$$
sending $\Phi \mapsto \dima_{\Gamma_1,\Gamma_2,\Phi}$
is a $\Q$-algebra isomorphism
which is actually obtained by  the composition of two $\Q$-linear isomorphisms
$\ma_{\Gamma_1}\otimes\ma_{\Gamma_2}:
\widehat{U\frak f_{\Gamma_1}}^\dagger \widehat\otimes \ \widehat{U\frak f_{\Gamma_2}}^\dagger\simeq \mathcal{M}(\mathcal F_\ser;\Gamma_1) \widehat\otimes \mathcal{M}(\mathcal F_\ser;\Gamma_2)
$
and
$\widehat\otimes: \mathcal{M}(\mathcal F_\ser;\Gamma_1) \widehat\otimes \mathcal{M}(\mathcal F_\ser;\Gamma_2) \simeq \mathcal{M}_2(\mathcal F_\ser;\Gamma_1,\Gamma_2)$.
\end{prop}

\begin{proof}
By using Lemma \ref{lem: tensor isom between moulds and dimoulds} and the equation \eqref{eq: dima=ma otimes ma}, we get this claim.
\end{proof}

The coproduct $\Delta$ of $\widehat{U\frak f_{2}}^\dag$
(cf. Lemma \ref{lem: Hopf structure of UF-Gamma-dag})
and the map $\shmap$ in Definition \ref{def:shmap}
are compatible under our maps:


\begin{lem}\label{lem: dima delta=sh ma}
Let $\Gamma$  be a finite set.
We have the following commutative diagram:
\begin{equation}\label{eq:CD dima delta=sh ma}
\xymatrix{
 \widehat{U\frak f_{\Gamma}}^\dag \ar@{->}[rr]^{\simeq}_{\ma_{\Gamma}}\ar@{^{(}->}[d]_{\Delta} && \mathcal{M}(\mathcal F_\ser;\Gamma)\ar@{^{(}->}[d]^{\shmap} \\
(\widehat{U\frak f_{\Gamma}}^\dag)^{\widehat{\otimes}2}  \ar@{->}[rr]^{\simeq}_{\dima_{\Gamma,\Gamma}}&&\mathcal{M}_2(\mathcal F_\ser;\Gamma,\Gamma).
}
\end{equation}
That is,  we have $\dima_{\Gamma,\Gamma} \circ \Delta = \shmap \circ \ma_{\Gamma}$.
\end{lem}

\begin{proof}
We give a proof only the case of $\Gamma=\{1\}$.
For any $\varphi \in \widehat{U\frak f_2}^\dag$ with
$$
\varphi
=  \sum_{r=0}^\infty
\sum_{k_0,\dots,k_r\in\N_{0}}
\left\langle
\varphi \, \middle| (k_0,\dots,k_r)
\right\rangle
f_0^{k_0}f_1\cdots f_1f_0^{k_r}
$$
we have
\begin{align*}
\Delta(\varphi)
&=\sum_{p,q\geq0}\sum_{d'_i,d''_j\in\N_{0}}
\left\langle
 \Delta(\varphi)\, \middle|(d'_{0},\dots, d'_{p}) ; (d''_{0},\dots, d''_{q})
 \right\rangle
f_0^{d'_{0}}f_{1}\cdots f_{1}f_0^{d'_{p}}\otimes f_0^{d''_{0}}f_{1}\cdots f_{1}f_0^{d''_{q}},
\end{align*}
where
\begin{align*}
&\left\langle
 \Delta(\varphi)\, \middle|(d'_{0},\dots, d'_{p}) ; (d''_{0},\dots, d''_{q})
 \right\rangle \\
&=\sum_{\{i_1,\dots,i_p\}\sqcup\{j_1,\dots,j_q\}=[p+q]}\quad
\sum_{\substack{d'_{k}=e'_{i_{k}}+ \cdots+e'_{i_{k+1}-1} \\
0\leq k\leq p}}\quad
\sum_{\substack{d''_{l}=e''_{j_{l}}+ \cdots+e''_{j_{l+1}-1} \\
0\leq l\leq q}} \\
&\hspace{3cm}
\left\{ \prod_{i=0}^{p+q}\binom{e'_{i}+e''_{i}}{e'_{i}} \right\}
\cdot
\left\langle
\varphi \, \middle| (e'_{0}+e''_{0},\dots, e'_{p+q}+e''_{p+q})
\right\rangle.
\end{align*}
Here, $i_0:=j_0:=0$ and $i_{p+1}:=j_{q+1}:=p+q+1$ and $[p+q]:=\{1,2,\dots,p+q\}$
and $\{i_1,\dots,i_p\}\sqcup\{j_1,\dots,j_q\}=[p+q]$ runs over $i_1<\cdots<i_p$ and $j_1<\cdots<j_q$.
By the definition of $\divimo$ and $\vimo$, we get
\footnote{For simplicity, 
we omit the lower indices $(1,\dots, 1)$ in $\divimo^{p,q}_{\{1\},\{1\},\Delta(\varphi)}$ and in $\dima^{p,q}_{\{1\},\{1\},\Delta(\varphi)}$.}
\begin{align}\label{eqn:divimo=sum vimo}
&\divimo^{p,q}_{\{1\},\{1\},\Delta(\varphi)}\left( z_0, \dots, z_p;\,z'_0, \dots, z'_q \right) \\
&\qquad\qquad
=\sum_{\{i_1,\dots,i_p\}\sqcup\{j_1,\dots,j_q\}=[p+q]}
\vimo_{\{1\},\varphi}^{p+q}\left( z_{a_0}+z'_{b_0},\dots,z_{a_{p+q}}+z'_{b_{p+q}}\right)
\nonumber
\end{align}
for all $p,q\in\N_0$.
Here, $a_i:=k-1$ for $i_{k-1}\leq i\leq i_k-1$ and for $1\leq k\leq p+1$, and $b_j:=l-1$ for $j_{l-1}\leq j\leq j_l-1$ and for $1\leq l\leq q+1$.
By substituting
\begin{align*}
& z_0=0,\qquad
z_i=x_1 + \cdots + x_i \quad
(1\leq i\leq p), \\
& z'_0=0,\qquad
z'_j=x_{p+1} + \cdots + x_{p+j} \quad
(1\leq j\leq p+q)
\end{align*}
in \eqref{eqn:divimo=sum vimo}, we get
\begin{align*}
&\dima^{p,q}_{\{1\},\{1\},\Delta(\varphi)}\left( x_1, \dots, x_p;\,x_{p+1}, \dots, x_{p+q} \right) \\
&\qquad =\sum_{\alpha\in X_\Z^\bullet}
\Sh{(x_1,\ \dots,\ x_p)}{(x_{p+1},\ \dots,\ x_{p+q})}{\alpha} \ma_{\{1\},\varphi}^{p+q}(\alpha).
\end{align*}
Hence, we obtain
\begin{align*}
\dima_{\{1\},\{1\}} \circ \Delta(\varphi)
&=\left\{
\dima^{p,q}_{\{1\},\{1\},\Delta(\varphi)}\left( x_1, \dots, x_p;\,x_{p+1}, \dots, x_{p+q} \right)
\right\}
_{p,q\in\N_0} \\
&=\left\{
\sum_{\alpha\in X_\Z^\bullet}
\Sh{(x_1,\ \dots,\ x_p)}{(x_{p+1},\ \dots,\ x_{p+q})}{\alpha} \ma_{\{1\},\varphi}^{p+q}(\alpha)
\right\}
_{p,q\in\N_0} \\
&=\shmap \left( \ma_{\{1\},\varphi} \right) \\
&=\shmap \circ \ma_{\{1\}}\left( \varphi \right).
\end{align*}
So we get the claim.
The proof for general $\Gamma$ can be done by the same arguments.
\end{proof}

We learn that the map $\shmap$ equips a structure of a completed Hopf algebra
on the algebra $\mathcal{M}(\mathcal{F}_\ser;\Gamma)$:

\begin{prop}\label{prop: Hopf str on MF ser}
Let $\Gamma$  be a finite set.
The algebra $\mathcal{M}(\mathcal{F}_\ser;\Gamma)$ is equipped with a structure of
a completed Hopf algebra whose  coproduct is given by $\shmap$  in
Definition \ref{def:shmap}
composed with the isomorphism in Lemma \ref{lem: tensor isom between moulds and dimoulds},
and the antipode is given by $\anti \circ \pari$.
Furthermore the algebra isomorphism $\ma_\Gamma$ in
\eqref{eq: ma Gamma} realizes an isomorphism of completed
Hopf algebras.
\end{prop}

\begin{proof}
By Lemmas
\ref{lem: tensor isom between moulds and dimoulds} and
\ref{lem: dima delta=sh ma}, the map $\shmap$ induces the coproduct of $\mathcal{M}(\mathcal{F}_\ser;\Gamma)$.
Note that the antipode $S$ of $\widehat{U\frak f_\Gamma}^\dagger$ is given by
$$
S(f_{\sigma_1}\cdots f_{\sigma_m})=(-1)^m f_{\sigma_m}\cdots f_{\sigma_1}
$$
for $\sigma_1,\dots,\sigma_m\in\Gamma\cup\{0\}$.
We prove $\ma_\Gamma \circ S=(\anti \circ \pari) \circ \ma_\Gamma$.
Assume
$$h=  \sum_{r=0}^\infty\sum_{(\sigma_1,\dots,\sigma_r)\in\Gamma^r}\sum_{k_0,\dots,k_r\in\N_{0}}
\coeff{h}{k_0,\dots,k_r}{\sigma_1,\dots,\sigma_r}f_0^{k_0}f_{\sigma_1}f_0^{k_1}\cdots f_{\sigma_r}f_0^{k_r}
\in \widehat{U\frak f_\Gamma}^\dagger.$$
We have
\begin{align*}
S(h)
&=\sum_{r=0}^\infty\sum_{(\sigma_1,\dots,\sigma_r)\in\Gamma^r}\sum_{k_0,\dots,k_r\in\N_{0}}
(-1)^{k_0+\cdots k_r+r}\coeff{h}{k_0,\dots,k_r}{\sigma_1,\dots,\sigma_r}f_0^{k_r}f_{\sigma_r}\cdots f_0^{k_1}f_{\sigma_1}f_0^{k_0} \\
&=\sum_{r=0}^\infty\sum_{(\sigma_1,\dots,\sigma_r)\in\Gamma^r}\sum_{k_0,\dots,k_r\in\N_{0}}
(-1)^{k_0+\cdots k_r+r}\coeff{h}{k_r,\dots,k_0}{\sigma_r,\dots,\sigma_1}f_0^{k_0}f_{\sigma_1}\cdots f_0^{k_{r-1}}f_{\sigma_r}f_0^{k_r},
\end{align*}
so we calculate
\begin{align*}
\vimo^r_{\Gamma,S(h)}\varia{z_0,\dots,z_r}{\sigma_1,\dots,\sigma_r}
&=\sum_{k_0,\dots,k_r\in\N_{0}}
(-1)^{k_0+\cdots k_r+r}\coeff{h}{k_r,\dots,k_0}{\sigma_r,\dots,\sigma_1}z_0^{k_0}\cdots z_{r-1}^{k_{r-1}}z_r^{k_r} \\
&=(-1)^{r}\sum_{k_0,\dots,k_r\in\N_{0}}
\coeff{h}{k_0,\dots,k_r}{\sigma_r,\dots,\sigma_1}(-z_r)^{k_0}(-z_{r-1})^{k_1}\cdots (-z_0)^{k_r} \\
&=(-1)^r \vimo^r_{\Gamma,h}\varia{-z_r,-z_{r-1}\dots,-z_0}{\sigma_r,\sigma_{r-1},\dots,\sigma_1}.
\end{align*}
By Lemma \ref{lem:UfGamma}. (4) we get
$$
\vimo^r_{\Gamma,S(h)}\varia{z_0,\dots,z_r}{\sigma_1,\dots,\sigma_r}
=(-1)^r \vimo^r_{\Gamma,h}\varia{0,z_r-z_{r-1}\dots,z_r-z_0}{\sigma_r,\sigma_{r-1},\dots,\sigma_1}.
$$
By putting $z_0=0$ and $z_i=x_1+\cdots + x_i$ ($1\leq i\leq m$), we obtain
$$
\ma^r_{\Gamma,S(h)}\varia{x_1,\dots,x_r}{\sigma_1,\dots,\sigma_r}
=(-1)^r \ma^r_{\Gamma,h}\varia{x_r,\dots,x_1}{\sigma_r,\dots,\sigma_1}
=\anti \circ \pari \left( \ma^r_{\Gamma,h}\varia{x_1,\dots,x_r}{\sigma_1,\dots,\sigma_r} \right).
$$
Hence, the antipode of $\mathcal{M}(\mathcal{F}_\ser;\Gamma)$ is $\anti \circ \pari$.
\end{proof}

\begin{rem}
The condition ${\mathcal F}={\mathcal F}_\ser$ in
Proposition \ref{prop: Hopf str on MF ser} is crucial since it depends on
the isomorphisms given in Lemmas \ref{lem: tensor isom between moulds and dimoulds}
and \ref{lem: dima delta=sh ma}. 
\end{rem}

For a series $h\in \widehat{U\frak f_\Gamma}$ without constant term,
we put
$\exp(h)=\sum_{k=0}^\infty\frac{h^k}{k!}\in\widehat{U\frak f_\Gamma}$.
We denote $\exp\widehat{\frak f_\Gamma}\cap \widehat{U\frak f_\Gamma}^\dag$ by
$\exp\widehat{\frak f_\Gamma}^\dag$ and
$\widehat{\frak f_\Gamma}\cap \widehat{U\frak f_\Gamma}^\dag$ by
$\widehat{\frak f_\Gamma}^\dag$.
\footnote{By definition, $\widehat{\frak f_\Gamma}^\dag$  (resp. $\exp\widehat{\frak f_\Gamma}^\dag$) is the set of Lie (resp. group-like) series whose  coefficient of $f_0$ is zero.
Whence we have $\exp (\widehat{\frak f_\Gamma}^\dag)=\exp\widehat{\frak f_\Gamma}^\dag$.
}


\begin{prop}\label{prop:bijection expf and GARIas}
Let $\Gamma$  be a finite set.
The map $\ma_\Gamma$ induces a group isomorphism from $\exp\widehat{\frak f_\Gamma}^\dag$ to $\GARI(\mathcal F_\ser;\Gamma)_{\as}$ and a Lie algebra isomorphism from $\widehat{\frak f_\Gamma}^\dag$ to $\ARI(\mathcal F_\ser;\Gamma)_{\al}$.
\end{prop}

\begin{proof}
Let $\varphi\in \widehat{U\frak f_\Gamma}^\dag$.
By 
\eqref{eq: dima=ma otimes ma} and \eqref{eq:CD dima delta=sh ma},
$\varphi\in\exp\widehat{\frak f_\Gamma}$ is equivalent to
$\shmap (\ma_\Gamma(\varphi))=\ma_\Gamma(\varphi) \otimes \ma_\Gamma(\varphi)$
and $\ma_\Gamma(\varphi)(\emptyset)=1$.
We also see that $\varphi\in\widehat{\frak f_\Gamma}$ is equivalent to
$\shmap (\ma_\Gamma(\varphi))=\ma_\Gamma(\varphi) \otimes \unitmould + \unitmould \otimes \ma_\Gamma(\varphi)$.
By Proposition \ref{prop:gp-like, Lie-like},
we obtain claim.
\end{proof}

%
%
%

We note that a similar claim for  polynomial-valued moulds is treated in
\cite[(3.5.13)]{S-ARIGARI}.

\subsection{Polymoulds related with braids}\label{sec:Polymoulds of pure braid type}
We investigate polymoulds which are  related with braids (Definition \ref{defn:polymoulds of pure braid type}).
By applying  the unique prolongation theorem (Theorem \ref{thm:UniqueProlongation}),
we transmit several operations for infinitesimal braid Lie algebra
to those for such polymoulds,  which will play an important role in the next section.

\begin{defn}
For $n\geq2$, the {\it infinitesimal braid Lie algebra}
$\mathfrak{t}_{n}$ means the Lie algebra generated
by $(t_{ij})_{0\leq i,j< n}$ with the relations
\[
t_{ii}=0, \quad
t_{ij}=t_{ji},\quad [t_{ij},t_{ik}+t_{jk}]=0,\quad [t_{ij},t_{kl}]=0
\]
where $i,j,k,l$ are distinct integers between $0$ and $n-1$, and
$U\mathfrak{t}_{n}$ its universal enveloping algebra. Let $\widehat{\mathfrak{t_{n}}}$
(resp. $\widehat{U\mathfrak{t}_{n}}$) be the completion of $\mathfrak{t}_{n}$
(resp. $U\mathfrak{t}_{n}$) by degree.
\end{defn}

We note that the center of $\mathfrak t_n$ is one-dimensional generated by
\begin{equation}\label{eq: center of braid Lie algebra}
c_n=\sum_{i<j}t_{ij}.
\end{equation}
We consider  the configuration space
$$
\mathsf{Conf}_n(\mathbb{A}^1)
:=\{(z_0,\dots,z_{n-1})\in\mathbb{A}^n\bigm|
z_i\neq z_j \quad (i\neq j)\}
$$
and define ${\mathcal H}_n:=H^0\bar{B}(\mathsf{Conf}_n(\mathbb{A}^1))$.
Then by Chen's theory (\cite{Ch})
${\mathcal H}_n$ is identified with the dual of the Hopf algebra $U\mathfrak{t}_{n}$
by $<t_{ij}|\omega_{kl}>=\delta_{\{i,j\},\{k,l\}}$
($0\leq i,j,k,l <n$ with $i<j$, $k<l$)
where $\omega_{ij}:=[d\log (z_i-z_j)]\in H^1(\mathsf{Conf}_n(\mathbb{A}^1)) \subset {\mathcal H}_n$.
We note that the product structure of Hopf algebra ${\mathcal H}_n$ is given by the shuffle product
and its coproduct is given by the deconcatenation denoted by $\Delta^\dec$.
We consider the action
$$\xi_n:\widehat{U\mathfrak{t}_{n}}\times {\mathcal H}_n\to {\mathcal H}_n$$
given by
$$\xi_n(w,f):=<\overleftarrow{w}\otimes 1|\Delta^\dec(f)>$$
where $\overleftarrow{w}\in\widehat{U\mathfrak{t}_{n}}$
means the element obtained by  \lq reading backwards'  of $w$, i.e.
the image of $w$ under the anti-isomorphism inducing the identity on
$H^1(\mathsf{Conf}_n(\mathbb{A}^1))$.
By definition, there are natural inclusions $\widehat{U\mathfrak{t}_{n}}\subset \widehat{U\mathfrak{t}_{m}}$ and ${\mathcal H}_n\subset {\mathcal H}_m$ for each $n\leq m$. Put $\widehat{U\mathfrak{t}_{\infty}} = \bigcup_{n}\widehat{U\mathfrak{t}_{n}}$ and ${\mathcal H}_\infty = \bigcup_{n}{\mathcal H}_n$.
Then $\xi_n$ are naturally extended to
$$\xi:\widehat{U\mathfrak{t}_{\infty}}\times {\mathcal H}_\infty\to {\mathcal H}_\infty.$$
By definition, we have
 \begin{equation}
      \xi(t_{ij},fg)=\xi(t_{ij},f)\cdot g+\xi(t_{ij},g)\cdot f,
 \end{equation}
 \begin{equation}
     \xi(1,f)=f,
 \end{equation}
  \begin{equation}\label{eq: xi w1w2}
         \xi(w_{1}w_{2},f)=\xi(w_{1},\xi(w_{2},f)).
  \end{equation}
Let $a_0,\dots, a_{m+1}\in \Z_{\geq 0}\cup \{\infty\} $ with $a_j\neq \infty$ for $1\leq j\leq m$.
By the differential equation (\cite{G-MPL} Theorem 2.1 and \cite{P} Lemma 3.3.30)
of the iterated integral
$I(z_{a_0};z_{a_1},\dots,z_{a_m};z_{a_{m+1}})$ with $z_\infty=\infty$
of \eqref{eq: symbol iterated integral},
we have an  element $\mathbb{I}(z_{a_0};z_{a_1},\dots,z_{a_m};z_{a_{m+1}})$ in $\mathcal H_{\infty}$ satisfying the same equation
\begin{align*}
\mathbb{I}&(z_{a_0};z_{a_1},\dots,z_{a_m};z_{a_{m+1}}) \\
&\qquad
\qquad=\sum_{i=1}^{m}[(\omega_{a_{i},a_{i+1}}-\omega_{a_i,a_{i-1}})|\mathbb{I}(z_{a_0};z_{a_1},\dots,\widehat{z_{a_i}}\dots,z_{a_m};z_{a_{m+1}})
]
+\delta_{m,0}
\end{align*}
where we put $\omega_{a,\infty}=0$ and $\omega_{a,a}=0$ ($a\in \Z_{\geq 0}$).
We denote 
$\mathbb{I}(z_{a_0};z_{a_{1}},\dots,z_{a_{m}};z_{a_{m+1}})\in\mathcal H_{\infty}$
simply by $\mathbb{I}(a_0;{a_{1}},\dots,{a_{m}};{a_{m+1}})\in\mathcal H_{\infty}$. Furthermore, we denote $\mathbb{I}(\infty;{a_{1}},\dots,{a_{m}};{a_{m+1}})$ $\in\mathcal H_{\infty}$ simply by $\mathbb{I}({a_{1}},\dots,{a_{m}};{a_{m+1}})$
$\in\mathcal H_{\infty}$.
Particularly by the differential equation, 
we see
 \begin{align}
   \label{eq:apply_tij}
\xi(t_{ij},& \mathbb{I}(a_{0};a_{1},\dots,a_{m};a_{m+1})) \\ \notag
&=\sum_{\epsilon\in\{\pm1\},1\leq p\leq m
}\epsilon\delta_{\{i,j\},\{a_{p},a_{p+\epsilon}\}}\mathbb{I}(a_{0}; a_{1},\dots,\widehat{a_{p}},\dots,a_{m};a_{m+1}).
 \end{align}

\begin{defn}
Let
$$
{\dec_{n+2}}:\widehat{U\mathfrak{t}_{n+2}}\to\widehat{U\mathfrak{f}_{1}}\hat{\otimes}\widehat{U\mathfrak{f}_{2}}\hat{\otimes}\cdots\hat{\otimes}\widehat{U\mathfrak{f}_{n+1}}.
$$
be the {$\Q$-linear} map given by
\begin{align*}
&\dec_{n+2}(w) \\
& \quad =\sum_{k_{0},\dots,k_{n}\geq0}\sum_{\substack{(a_{i,j})_{0\leq i\leq n,1\leq j\leq k_{i}}\\0\leq a_{i,j}\leq i}}
c\left(\xi\left(w,\mathbb{I}(a_{0,1},\dots,a_{0,k_{0},};1)\cdots\mathbb{I}(a_{n,1},\dots,a_{n,k_{n}};n+1)\right)\right)\\
& \qquad \qquad \ \times f_{a_{0,1}}\cdots f_{a_{0,k_{0}}}\otimes f_{a_{1,1}}\cdots f_{a_{1,k_{1}}}\otimes\cdots\otimes f_{a_{n,1}}\cdots f_{a_{n,k_{n}}}
\end{align*}
where
			$$c:\mathcal H_n\to\mathbb{Q}$$
is the counit of $\mathcal{H}_n$.
\end{defn}
We note that
\[
c(\mathbb{I}(a_{1},\dots,a_{k};a_{k+1}))=\delta_{k,0}.
\]
If there exists no risk of confusion, we simply write $\dec_{n+2}$ as $\dec$.

\begin{lem}\label{lem coalg isom of UF tensors and UT}
The map $\dec$ gives an isomorphism of coalgebras, and the inverse
{$\Q$-linear}
map
\begin{equation}\label{eq:codec}
{\codec}:\widehat{U\mathfrak{f}_{1}}\hat{\otimes}\widehat{U\mathfrak{f}_{2}}\hat{\otimes}\cdots\hat{\otimes}\widehat{U\mathfrak{f}_{n+1}}\to\widehat{U\mathfrak{t}_{n+2}}
\end{equation}
is given by
\begin{align*}
w_{1}(f_{0})\otimes & w_{2}(f_{0},f_{1})\otimes\cdots\otimes w_{n+1}(f_{0},f_{1},\dots,f_{n})\\
&\mapsto w_{1}(t_{01})w_{2}(t_{02},t_{12})\cdots w_{n+1}(t_{0,n+1},\dots,t_{n,n+1}).
\end{align*}
\end{lem}


\begin{proof}
To avoid confusion, we denote the map in \eqref{eq:codec}
by $\mathsf{mulp}$ only in this proof.
We can easily check the surjectivity of $\mathsf{mulp}$. Thus it is enough
to check $\dec\circ\mathsf{mulp}={\rm id}$. By definition, we have
\begin{align*}
 & \dec\circ\mathsf{mulp}(f_{a_{0,1}}\cdots f_{a_{0,k_{0}}}\otimes f_{a_{1,1}}\cdots f_{a_{1,k_{1}}}\otimes\cdots\otimes f_{a_{n},1}\cdots f_{a_{n,k_{n}}})\\
 & =\dec(t_{a_{0,1},1}\cdots t_{a_{0,k_{0}},1}\times\cdots\times t_{a_{n,1},n+1}\cdots t_{a_{n,k_{n}},n+1}),
\end{align*}
and
\begin{align*}
& c(\xi(t_{a_{0,1},1}\cdots t_{a_{0,k_{0}},1}\times\cdots\times t_{a_{n,1},n+1}\cdots t_{a_{n,k_{n}},n+1},\\ & \qquad \mathbb{I}(a_{0,1},\dots,a_{0,k_{0}};1)\cdots\mathbb{I}(a_{n,1},\dots,a_{n,k_{n}};n+1)))\\
&  =\begin{cases}
1 & (a_{i,j})=(a_{i,j}')\ \ \ \text{for all }i,j\\
0 & {\rm otherwise}.
  \end{cases}
\end{align*}
Thus
\begin{align*}
\dec(t_{a_{0,1},1}&\cdots t_{a_{0,k_{0}},1}\times\cdots\times t_{a_{n,1},n+1}\cdots t_{a_{n,k_{n}},n+1}) \\
&=f_{a_{0,1}}\cdots f_{a_{0,k_{0}}}\otimes f_{a_{1,1}}\cdots f_{a_{1,k_{1}}}\otimes\cdots\otimes f_{a_{n},1}\cdots f_{a_{n,k_{n}}}.
\end{align*}
The compatibility under coproduct is immediate by definition 
because we have $\Delta(f_i)=f_i\otimes 1 +1\otimes f_i$
and $\Delta(t_{ij})=t_{ij}\otimes 1 +1\otimes t_{ij}$.
Hence the lemma is proved.
\end{proof}



\begin{defn}
For $k$ with $1 \leq k\leq n-1$, we denote
\begin{equation}
e_{0k}:\widehat{U\mathfrak{t}_{n}}\to\mathbb{Q}
\end{equation}
to be the continuous linear map which takes the sum of coefficients of $t_{0k}$
and $t_{k0}$ in the degree $1$-part of $\widehat{U\mathfrak{t}_{n}}$
when we expand  in terms of $t_{ij}$ with $0\leq i\neq j < n$.
We put
\begin{equation}
\widehat{U\mathfrak{t}_{n}}^{\dag}=\{w\in\widehat{U\mathfrak{t}_{n}}\mid(e_{0j}\otimes{\rm id})\circ\Delta(w)=0\text{ for all }1\leq j< n\}.
\end{equation}
\end{defn}

\begin{lem}
There is an involution
\begin{equation}\label{eq: flip}
\flip: \widehat{U\mathfrak{t}_{n}}^{\dag}\to \widehat{U\mathfrak{t}_{n}}^{\dag}
\end{equation}
which sends $t_{ij}$ ($i<j$) to $t_{n-i,n-j}$ ($0< i,j<n$)
and $-\sum_{k=0}^{n-1} t_{k,n-j}$ ($i=0$).
\end{lem}

\begin{proof}
It can be checked by direct calculation that the defining equations of
$\mathfrak t_n$ are preserved under $\flip$.
Seeing that it forms an involution is immediate.
\end{proof}

We also define the continuous linear map
\begin{equation*}
e_{nk}:\widehat{U\mathfrak{t}_{n}}\to\mathbb{Q}
\end{equation*}
with $0\leq k<n$
which takes the sum of coefficients of $\flip(t_{0,n-k})$
and $\flip(t_{n-k,0})$ in the degree $1$-part of $\widehat{U\mathfrak{t}_{n}}$
when we expand it in terms of $\flip(t_{ij})$ with $0\leq i\neq j < n$.

\begin{lem}\label{lem: Hopf algebra with flip}
The subspace $\widehat{U\mathfrak{t}_{n}}^{\dag}$ forms a Hopf subalgebra of
$\widehat{U\mathfrak{t}_{n}}$.
It is stable under
$\flip$ in \eqref{eq: flip}.
\end{lem}

\begin{proof}
It can be proved in the same way as Lemma \ref{lem: Hopf structure of UF-Gamma-dag}.
By definition, we see
$$\ker e_{0j}=\ker e_{n,n-j},$$
from which
the second half of the claim follows.
\end{proof}

\begin{lem}\label{lem:theta_dagger_isom}
The map $\dec$ induces an isomorphism of coalgebras
\begin{equation}\label{eq:dec isom for dag}
\dec:
\widehat{U\mathfrak{t}_{n+2}}^{\dag} 
\simeq\widehat{U\mathfrak{f}_{1}}^{\dag}\hat{\otimes}\widehat{U\mathfrak{f}_{2}}^{\dag}\hat{\otimes}\cdots\hat{\otimes}\widehat{U\mathfrak{f}_{n+1}}^{\dag}
 (\simeq\widehat{U\mathfrak{f}_{2}}^{\dag}\hat{\otimes}\cdots\hat{\otimes}\widehat{U\mathfrak{f}_{n+1}}^{\dag}).
\end{equation}
\end{lem}


\begin{proof}
For $w_{i}\in\widehat{U\mathfrak{f}_{i}}^{\dag}$ ($1\leq i\leq n+1$),
we have
\begin{align*}
(e_{0j}\otimes{\rm id})&\circ\Delta\circ\dec^{-1}(w_{1}\otimes\cdots\otimes w_{n+1})  \\
 & =(e_{0j}\otimes{\rm id})\circ\Delta\left(
 w_1(t_{01})\cdots w_{n+1}(t_{0,n+1},\dots,t_{n,n+1})
 \right)\\
 & =(e_{0j}\otimes{\rm id})\Bigl(
  \Delta(w_1(t_{01}))\cdots \Delta(w_{n+1}\bigl(t_{0,n+1},\dots,t_{n,n+1})\bigr) \Bigr) \\
 & =\dec^{-1}(w_{1}\otimes\cdots\otimes w_{j-1}\otimes(e_{0}\otimes{\rm id})(\Delta w_{j})\otimes w_{j+1}\otimes \cdots \otimes w_{n+1}).
\end{align*}
Thus the kernel of $(e_{0j}\otimes{\rm id})\circ\Delta(w)$ is equal
to
\[
\widehat{U\mathfrak{f}_{1}}\hat{\otimes}\cdots\hat{\otimes}\widehat{U\mathfrak{f}_{j-1}}\hat{\otimes}\widehat{U\mathfrak{f}_{j}}^{\dag}\hat{\otimes}\widehat{U\mathfrak{f}_{j+1}}\hat{\otimes}\cdots\hat{\otimes}\widehat{U\mathfrak{f}_{n+1}}.
\]
Thus the lemma holds since the intersection for all $1\leq j\leq n+1$
is equal to the right hand side.
The compatibility of the coproduct is
by Lemma \ref{lem coalg isom of UF tensors and UT}.
\end{proof}

We note that the above is not  a Hopf algebra isomorphism,
though both
$\widehat{U\mathfrak{t}_{n+2}}^{\dag}$
and
$\widehat{U\mathfrak{f}_{1}}^{\dag}\hat{\otimes}\widehat{U\mathfrak{f}_{2}}^{\dag}\hat{\otimes}\cdots\hat{\otimes}\widehat{U\mathfrak{f}_{n+1}}^{\dag}$
are equipped with Hopf algebra structures.
Now we define a mould-theoretic analogue of  $\widehat{U\mathfrak{t}_{n}}^{\dag}$
below:

\begin{defn}\label{defn:polymoulds of pure braid type}
(1).
Let $\mathcal{F}$ be a family of functions.
The  $\Q$-linear space ${\mathcal P}_{n+2}(\mathcal{F})$
is defined to be
\begin{equation}\label{eq: P n+2 F}
{\mathcal P}_{n+2}(\mathcal{F})\coloneqq\mathcal{M}_{n}(\mathcal{F};[1],[2],\dots, [n]),
\end{equation}
where
for any $n\in\N$, we denote the set
$$[n]:=\{1,\dots,n\}.$$


(2).
There exists a $\Q$-linear isomorphism
\[
{\madec_n}:
\widehat{U\mathfrak{t}_{n}}^{\dag}\simeq {\mathcal P}_{n}(\mathcal{F}_{{\ser}})
\]
given by
\begin{align*}
\madec_n:
\widehat{U\mathfrak{t}_{n}}^{\dag}
\overset{\dec}{\simeq}
\widehat{U\mathfrak{f}_{2}}^{\dag}\hat{\otimes}\cdots\hat{\otimes}\widehat{U\mathfrak{f}_{n-1}}^{\dag}
\overset{\otimes_k\ma_{[k]}}{\simeq}
{\mathcal M}(\mathcal{F}_{{\ser}};[1])\hat{\otimes}
\cdots\hat{\otimes}{\mathcal M}(\mathcal{F}_{{\ser}};[n-2])
\overset{i_{\widehat{\otimes}}}{\simeq}
{\mathcal P}_{n}(\mathcal{F}_{{\ser}}),
\end{align*}
where
$\dec$ is the coalgebra isomorphism \eqref{eq:dec isom for dag},
$\ma_{[k]}$ means $\ma_{\Gamma}$ with $\Gamma=[k]$ in \eqref{eq: ma Gamma}
and 
$i_{\widehat{\otimes}}$ is
the algebra isomorphism
in \eqref{eq: isom mould tensors and polymoulds}.
\end{defn}

We encode ${\mathcal P}_{n}(\mathcal{F}_{{\ser}})$ with a structure of Hopf algebra
by transmitting from that of $\widehat{U\mathfrak{t}_{n}}^{\dag}$
(cf. Lemma \ref{lem: Hopf algebra with flip})
under the map $\madec_n$.
Particularly we denote its product by $\diamond$, which is
the  binary operator on ${\mathcal P}_{n}(\mathcal{F}_{{\ser}})$
given by
\begin{equation}\label{eq: diamond product}
M \mulp N \coloneqq \madec_{n}\left(\madec_{n}^{-1}(M)\cdot\madec_{n}^{-1}(N)\right).
\end{equation}
It will be shown later (in  Lemma \ref{lem:mould-proper-mulB}) that it is
a mould-proper binary operator.

\begin{lem}\label{lem:mould-proper-bydef}
For any sets $\Gamma_1,\dots,\Gamma_s$ and
a given map
\[
g:\mathcal{M}_{s}(\mathcal{F}_{{\ser}};\Gamma_{1},\dots,\Gamma_{s})\to\widehat{U\mathfrak{t}_{n+2}}^{\dag},
\]
the composite map
\[
\madec_{n+2}\circ g:\mathcal{M}_{s}(\mathcal{F}_{{\ser}};\Gamma_{1},\dots,\Gamma_{s})\to\mathcal{P}_{n+2}(\mathcal{F}_{{\ser}})
\]
is mould-proper if and only if the map
\[
M\mapsto c\circ\xi\bigg(g(M),\prod_{i=1}^{n}\mathbb{I}\bigg(\langle0\rangle,\sigma_{i,1},\langle u_{i,1}\rangle,\dots,\sigma_{i,l_{i}},\langle u_{i,1}+\cdots+u_{i,l_{i}}\rangle;i+1\bigg)\bigg)
\in\mathbb{Q}[[ u_{1,1},\dots,u_{n,l_{n}}]]
\]
where we put
\begin{multline*}
\mathbb{I}(\langle v_{0}\rangle,\sigma_{1},\langle v_{1}\rangle,\dots,\sigma_{d},\langle v_{d}\rangle;\sigma_{d+1}) \\
\coloneqq\sum_{k_0,\dots,k_d=0}^{\infty}v_{0}^{k_{0}}\cdots v_{d}^{k_{d}} \mathbb{I}(\{0\}^{k_{0}},\sigma_{1},\{0\}^{k_{1}},\dots,\sigma_{d},\{0\}^{k_{d}};\sigma_{d+1}),
\end{multline*}
is mould-proper for any $l_1,\dots,l_n\geq 0$ and $\sigma_{i,j}$ with $0\leq\sigma_{i,j}\leq i$.
\end{lem}

\begin{proof}
It follows from definition.
\end{proof}

Below we prepare a technical criterion to be mould-proper.

\begin{lem}\label{lem:mould-proper-B-pre}
Let $r\geq0$, $m_{1},\dots,m_{r}\geq0$.
Let $\{s_{k,l}\}_{1\leq k\leq r,0\leq l\leq m_{k}}$
be a collection of elements in
$\bigoplus_{0\leq i<j}\mathbb{Q}t_{ij}$.
Assume that for any $1\leq k\leq r$,
\begin{enumerate}
\item \label{assump1} for $l>0$, we have $s_{k,l}\in \bigoplus_{1\leq i<j}\mathbb{Q}t_{ij}$;
\item \label{assump2}
for any $1\leq i,j$, at least one of
the following conditions holds:
\begin{description}
\item[$H_1(i,j;k)$] for any $l$, the coefficients of $t_{ij}$, $t_{0i}$ and $t_{0j}$ in $s_{k,l}$ are equal;
\item[$H_2(i,j;k)$] the coefficient of $t_{0i}$ and that of $t_{0j}$ in $s_{k,0}$ are
not equal.
\end{description}
\end{enumerate}
Furthermore, let $N\geq0$, $l_1,\dots,l_N\geq0$,
$\{\rho_{i,j}\}_{1\leq i\leq N,1\leq j\leq l_{i}+1}$
be a collection of elements of
$\Z_{\geq1}$.
Then the map
\[
\mathcal{M}(\mathcal{F}_{{\ser}};\{1,\dots,m_{1}\})\times\cdots\times\mathcal{M}(\mathcal{F}_{{\ser}};\{1,\dots,m_{r}\})
\to
\Q[[ u_{i,j} \mid 1\leq i\leq N, 1\leq j\leq l_j ]]
\]
given by
\begin{multline*}
(M_{1},\dots,M_{r})\mapsto\\
c\circ\xi\left(\prod_{p=1}^{r}\ma_{[m_p]}^{-1}(M_{p})(s_{p,0},\dots,s_{p,m_{p}}),\prod_{i=1}^{N}\mathbb{I}(\langle0\rangle,\rho_{i,1},\langle u_{i,1}\rangle,\dots,\rho_{i,l_{i}},\langle u_{i,0}+\cdots+u_{i,l_{i}}\rangle;\rho_{i,l_{i}+1})\right).
\end{multline*}
is mould-proper.
\footnote{
In this paper, a product $\prod_{i=1}^n r_i$ in a noncommutative algebra $R$  is taken from left to right, that is,  $\prod_{i=1}^n r_i:=r_1\cdots r_n\in R$.}
\end{lem}

\begin{proof}
Let $\mathcal{V}$ be the $\Q$-linear space generated by
$u_{1,1},\dots,u_{N,l_N}$.
We consider  the $\Q$-linear space $Q$ generated by
all non-commutative products
\[
\prod_{i=1}^{N}\tilde{\mathbb{I}}(\langle v_{i,0}\rangle,\sigma_{i,1},\langle v_{i,1}\rangle,\dots,\sigma_{i,d_i},\langle v_{i,d_i}\rangle;\sigma_{i,d_i+1}),
\]
of formal symbols
$\tilde{\mathbb{I}}(\langle v_{i,0}\rangle,\sigma_{i,1},\langle v_{i,1}\rangle,\dots,\sigma_{i,d_i},\langle v_{i,d_i}\rangle;\sigma_{i,d_i+1})$
with $\sigma_{1,1},\dots,\sigma_{N,d_{N}+1}\in \Z_{\geq 0}$
and $v_{1,0},\dots,v_{N,d_N}\in \mathcal{V}$
such that $\{v_{i,j}-v_{i,0}\mid1\leq i\leq N,1\leq j\leq d_{i}\}$ are linearly independent in $\mathcal V$.
We define the $\Q$-linear map
$$\ev:Q\to\mathcal H_{\infty}[[u_{1,1},\dots,u_{N,l_N}]]$$
by
sending 
\begin{align*}
\prod_{i=1}^{N}\tilde{\mathbb{I}}&
(\langle v_{i,0}\rangle,\sigma_{i,1},\langle v_{i,1}\rangle,\dots,\sigma_{i,d_i},\langle v_{i,d_i}\rangle;\sigma_{i,d_i+1}) \\
&\qquad\qquad \mapsto\qquad
\prod_{i=1}^{N}{\mathbb{I}}(\langle v_{i,0}\rangle,\sigma_{i,1},\langle v_{i,1}\rangle,\dots,\sigma_{i,d_i},\langle v_{i,d_i}\rangle;\sigma_{i,d_i+1}).
\end{align*}

For a monomial $a\in Q$ given by
\begin{equation}\label{a product presentation for a}
a  =\prod_{i=1}^{N}\tilde{\mathbb{I}}(\langle v_{i,0}\rangle,\sigma_{i,1},\langle v_{i,1}\rangle,\dots,\sigma_{i,d_{i}},\langle v_{i,d_{i}}\rangle;\sigma_{i,d_{i}+1}),
\end{equation}
we introduce elements $a'(j,p)$ and $a''(j,p)$ in $Q$ by
\begin{align*}
a'(j,p) & =\prod_{i=1}^N \begin{cases}
\tilde{\mathbb{I}}(\langle v_{j,0}\rangle,\sigma_{j,1},\langle v_{j,1}\rangle,\dots,\widehat{\sigma_{j,p},\langle v_{j,p}\rangle},\dots,\sigma_{j,d_{j}},\langle v_{j,d_{j}}\rangle;\sigma_{j,d_{j}+1}) & i = j\\
\tilde{\mathbb{I}}(\langle v_{i,0}\rangle,\sigma_{i,1},\langle v_{i,1}\rangle,\dots,\sigma_{i,d_{i}},\langle v_{i,d_{i}}\rangle;\sigma_{i,d_{i}+1}) & i\neq j,
\end{cases}\\
a''(j,p) & =\prod_{i=1}^N \begin{cases}
\tilde{\mathbb{I}}(\langle v_{j,0}\rangle,\sigma_{j,1},\dots,\widehat{\langle v_{j,p-1}\rangle,\sigma_{j,p}},\dots,\langle v_{j,d_{j}}\rangle;\sigma_{j,d_{j}+1}) & i = j\\
\tilde{\mathbb{I}}(\langle v_{i,0}\rangle,\sigma_{i,1},\langle v_{i,1}\rangle,\dots,\sigma_{i,d_{i}},\langle v_{i,d_{i}}\rangle;\sigma_{i,d_{i}+1}) & i\neq j.
\end{cases}
\end{align*}

For a such $a\in Q$ and
\begin{equation}\label{eq:s=sum ct}
s=\sum_{x,y}c_{x,y}t_{x,y}\in\bigoplus_{0\leq i<j}\mathbb{Q}t_{ij},
\end{equation}
 we have
\[
\xi(s,\ev(a))=\ev( \xi'_{s}(a))+V(s,a)\cdot \ev(a )
\]
where
\begin{align*}
\xi'_{s}(a)
 \coloneqq &\sum_{j=1}^{N}\sum_{\substack{1\leq p\leq d_{j}\\\sigma_{j,p}\neq\sigma_{j,p+1}}}
 \sum_{x,y}c_{x,y}\left(\delta_{\{x,y\},\{\sigma_{j,p},\sigma_{j,p+1}\}}-\delta_{\{x,y\},\{\sigma_{j,p},0\}}\right) a'(j,p)\\
 & \ -\sum_{j=1}^{N}\sum_{\substack{1\leq p\leq d_{j}\\\sigma_{j,p}\neq\sigma_{j,p-1}}}
 \sum_{x,y}c_{xy}\left(\delta_{\{x,y\},\{\sigma_{j,p},\sigma_{j,p-1}\}}-\delta_{\{x,y\},\{\sigma_{j,p},0\}}\right) a''(j,p)\\
 & \text{(\ensuremath{\sigma_{j,p-1}} is understood to be \ensuremath{\infty} if \ensuremath{p=1})}
\end{align*}
and $V(s,a)\in \mathcal{V}
\subset \mathcal H_\infty[[u_{1,1},\dots,u_{N,l_N}]]$ is given by
$$
V(s,a)
\coloneqq \sum_{j=1}^{N}
\sum_{x,y}c_{x,y}\left(\sum_{p=0}^{d_j}\delta_{\{x,y\},\{\sigma_{j,p+1},0\}}v_{j,p}-\sum_{p=1}^{d_j}\delta_{\{x,y\},\{\sigma_{j,p},0\}}v_{j,p}\right).
$$
For a monomial $a
\in Q$ given as in \eqref{a product presentation for a}
and
$s\in \bigoplus_{0\leq i<j}\mathbb{Q}t_{ij}$ given as in
\eqref{eq:s=sum ct},
we denote by $\mathcal W(a)$ the $\Q$-linear
subspace of $\mathcal{V}$ spanned by
$\{v_{i,j}-v_{i,0}\mid1\leq i\leq N,1\leq j\leq d_{i}\}$.
Let us write the above definitions of $\xi'_s$ as
\[
\xi_{s}'(a)=\sum_{{\bf p}\in D(s,a)}c_{{\bf p}}A({\bf p})\bf
\]
where $c_{{\bf p}}\in \Z$ and $A({\bf p})\in Q$ is a monomial and
$D(s,a)$ is chosen to be minimal, that is,
$A({\bf p})\neq\ A({\bf q})$
when ${\bf p} \neq {\bf q}\in D(s,a)$ and
$c_{\bf p}\neq 0$ for all ${\bf p}\in D(s,a)$.
Note that $\mathcal W(A({\bf p}))\subset \mathcal W(a)$.
By the assumption (\ref{assump1}), $V(s_{k,l},a)=0$ if $l>0$.

Note that, for any ${\bf p}\in D(s,a)$, all $A({\bf p})$ is expressed as, for some $j$ and $p$,
$a'(j,p)$ such that  $a'(j,p)\in \{A({\bf p})\mid{\bf p}\in D(s_{k,l},a)\}$
or $a''(j,p)$ such that $a''(j,p)\in \{A({\bf p})\mid{\bf p}\in D(s_{k,l},a)\}$.
If the condition ${H_{1}(\sigma_{j,p},\sigma_{j,p+1};k)}$ holds then
\begin{equation}\label{eq:H1a1}
a'(j,p)\notin\{A({\bf p})\mid{\bf p}\in D(s_{k,l},a)\}
\end{equation}
for any $l$ since the coefficient of $a'(j,p)$ in $\xi_{s_{k,l}}'(a)$ is given by
\begin{equation*}
\sum_{x,y}c_{xy}\left(\delta_{\{x,y\},\{\sigma_{j,p},\sigma_{j,p+1}\}}-\delta_{\{x,y\},\{\sigma_{j,p},0\}}\right)=0
\end{equation*}
where $c_{x,y}$ are coefficients of $t_{x,y}$ in $s_{k,l}$.
Similarly if the condition $H_1(\sigma_{j,p},\sigma_{j,p-1};k)$ holds then
\begin{equation}\label{eq:H1a2}
a''(j,p)\notin\{A({\bf p})\mid{\bf p}\in D(s_{k,l},a)\}
\end{equation}
for any $l$.
Therefore, if $A({\bf p})=a'(j,p)$ (resp. $a''(j,p)$) for some $j$ and $p$, then the condition ${H_{1}(\sigma_{j,p},\sigma_{j,p+1};k)}$ (resp. $H_1(\sigma_{j,p},\sigma_{j,p-1};k)$) does not hold.
Whence the condition ${H_{2}(\sigma_{j,p},\sigma_{j,p+1};k)}$ (resp. $H_2(\sigma_{j,p},\sigma_{j,p-1};k)$) must hold.

Let us calculate $V(s,a)-V(s,A({\bf p}))$
with $s=\sum_{x,y}c_{x,y}t_{x,y}$ and  $a\in Q$ given as \eqref{a product presentation for a} under the condition ${H_{2}(\sigma_{j,p},\sigma_{j,p+1};k)}$ or $H_2(\sigma_{j,p},\sigma_{j,p-1};k)$
for the later purpose:
We have
\begin{align*}
V(s,a) - & V(s,a'(j,p))  =\sum_{y}c_{0,y}\left(V(t_{0,y}',a)-V(t_{0,y}',a'(j,p))\right)\\
 & =\sum_{y}c_{0,y}\left(\delta_{y,\sigma_{j,p+1}}v_{j,p}-\delta_{y,\sigma_{j,p}}v_{j,p}+\delta_{y,\sigma_{j,p}}v_{j,p-1}-\delta_{y,\sigma_{j,p+1}}v_{j,p-1}\right)\\
 & =(v_{j,p}-v_{j,p-1})\sum_{y}c_{0,y}(\delta_{y,\sigma_{j,p+1}}-\delta_{y,\sigma_{j,p}})\\
 & =(v_{j,p}-v_{j,p-1})(c_{0,\sigma_{j,p+1}}-c_{0,\sigma_{j,p}}).
\end{align*}
Furthermore, $v_{j,p}-v_{j,p-1}\notin \mathcal W(a'(j,p))$ by the assumption that $\{v_{i,j}-v_{i,0}\mid1\leq i\leq N,1\leq j\leq d_{i}\}$ are linearly independent in $\mathcal V$.
Thus $V(s,a) - V(s,a'(j,p)) \notin \mathcal W(a'(j,p))$ if (and only if) the coefficients of $t_{0,\sigma_{j,p+1}}$ and $t_{0,\sigma_{j,p}}$ in $s$ are not equal.
Therefore if the condition $H_2(\sigma_{j,p},\sigma_{j,p+1};k)$ holds then
\begin{equation}\label{eq:H2a1}
V(s_{k,0},a)-V(s_{k,0},a'(j,p))\notin \mathcal W(a'(j,p)).
\end{equation}
Similarly, if the condition $H_2(\sigma_{j,p},\sigma_{j,p-1};k)$ holds then
\begin{equation}\label{eq:H2a2}
V(s_{k,0},a)-V(s_{k,0},a''(j,p))\notin \mathcal W(a''(j,p)).
\end{equation}

Thus by  \eqref{eq:H2a1}, \eqref{eq:H2a2} and
the assumption (\ref{assump2}), for any ${\bf p}\in D(s_{k,l},a)$, we have
\begin{align}
V(s_{k,0},a)-V(s_{k,0},A({\bf p})) \notin \mathcal W(A({\bf p})) \label{eq:V_diff_notin_W}.
\end{align}


We have
\begin{align}\label{eq: products of ma in terms of a}
 & \prod_{p=1}^{r}\ma_{[m_p]}^{-1}(M_p)(s_{p,0},\dots,s_{p,m_p})\\
 \notag &  =\sum_{b_{1},\dots,b_{r}=0}^{\infty}\sum_{\substack{(k_{p,q})_{1\leq p\leq r,0\leq q\leq b_{p}}\\(\sigma_{p,q})_{1\leq p\leq r,1\leq q\leq b_{p}}}}
 \prod_{p=1}^{r}\langle\ma_{[m_p]}^{-1}(M_p)|
{}^{k_{p,0},\dots,k_{p,b_{p}}}_{\sigma_{p,1},\dots,\sigma_{p,b_{p}}}\rangle\\
\notag & \ \ \
 \times
 \prod_{p=1}^{r}s_{p,0}^{k_{p,0}}s_{p,\sigma_{p,1}}s_{p,0}^{k_{p,1}}\cdots s_{p,\sigma_{p,b_{p}}}s_{p,0}^{k_{p,b_{p}}}
 \in \widehat{U\mathfrak{t}_{\infty}},
\end{align}
where we employ the notation in \eqref{eq: word expansion}.
Furthermore, by using
\[
\xi(s,\ev(a))=\sum_{{\bf p}\in D(s,a)}c_{{\bf p}}\ev(A({\bf p}))+V(s,a)\cdot\ev(a)
\]
and \eqref{eq: xi w1w2}
repeatedly, for $k\geq0$, we inductively have
\begin{align*}
&\xi(s_{p,0}^k,\ev(a)) \\
&=\sum_{0\leq N\leq k} \sum_{\substack{{\bf p}_h\in D(s_{p,0},a_h) \\
0\leq h\leq N-1}}
\left\{
\left( \prod_{h=0}^{N-1}c_{{\bf p}_h} \right)\left( \sum_{\substack{\sum_{h=0}^N\beta_h=k-N \\
\beta_h\geq0}}\prod_{h=0}^NV(s_{p,0},a_h)^{\beta_h} \right)
\ev(a_N)
\right\}.
\end{align*}
Here, $a_0:=a$ and $a_h:=A({\bf p}_{h-1})$ for $N\geq h\geq1$.
So by using this, we get
\begin{align*}
 & \xi(\prod_{p=1}^rs_{p,0}^{k_{p,0}}s_{p,\sigma_{p,1}}s_{p,0}^{k_{p,1}}\cdots s_{p,\sigma_{p,b_{p}}}s_{p,0}^{k_{p,b_{p}}},\ev(a))\\
 & =\sum_{\substack{(N_{p,q})_{1\leq p\leq r,0\leq q\leq b_{p}}\\
0\leq N_{p,q}\leq k_{p,q}
}
}\sum_{\substack{(k_{p,q,h})_{\substack{1\leq p\leq r\\
0\leq q\leq b_{p}\\
0\leq h\leq N_{p,q}
}
}\\
\sum_{h=0}^{N_{p,q}}k_{p,q,h}=k_{p,q}-N_{p,q}
}
}\sum_{\substack{{\bf p}_{p,q,h}\in D(s_{p,0},a_{p,q,h+1})\\
{\bf p}_{p,q}\in D(s_{p,\sigma_{p,q}},a_{p,q,0})
}
}\\
 & \ \ \
 \left[
 (\prod_{p,q,h}c_{{\bf p}_{p,q,h}})(\prod_{p,q}c_{{\bf p}_{p,q}})(\prod_{p,q,h}V(s_{p,0},a_{p,q,h})^{k_{p,q,h}})\ev(a_{1,0,0})
 \right]
\end{align*}
where  the above $a_{p,q,h}$ is recursively defined as
\[
a_{p,q,h}=\begin{cases}
a & h=N_{p,q},\ q=b_{p},\ p=r\\
a_{p+1,0,0}
 & h=N_{p,q},\ q=b_{p},\ p<r\\
A({\bf p}_{p,q+1}) & h=N_{p,q},\ q<b_{p}\\
A({\bf p}_{p,q,h+1}) & h<N_{p,q}.
\end{cases}
\]
Here for any $1\leq p\leq r$, the vectors $(V(s_{p,0}, a_{p,q,h}))_{q,h}$ are {\it affinely independent}\footnote{We say that vectors $w_0,\dots,w_m$ in a vector space are affinely independent if $w_1-w_0,\dots,w_m-w_0$ are linearly independent.} 
by (\ref{eq:V_diff_notin_W}).
Since we have
\begin{multline*}
\sum_{\substack{(k_{p,q,h})_{h}\\
k_{p,q,0}+\cdots+k_{p,q,N_{p,q}}=k_{p,q}-N_{p,q}
}
}\prod_{h=0}^{N_{p,q}}V(s_{p,0},a_{p,q,h})^{k_{p,q,h}}\\
=
\sum_{h=0}^{N_{p,q}}V(s_{p,0},a_{p,q,h})^{k_{p,q}}\prod_{h'\neq h}\frac{1}{V(s_{p,0},a_{p,q,h})-V(s_{p,0},a_{p,q,h'})}
\end{multline*}
by the following identity (cf. \cite{HMO} Lemma 2.1)
$$
\sum_{\substack{ e_1+\cdots+e_r=m \\ e_i\ge0\,(1\le i\le r) }}
a_1^{e_1}\cdots a_r^{e_r}
=\sum_{i=1}^r a_i^{m+r-1} \prod_{j\ne i}(a_i-a_j)^{-1},
$$
we obtain
\begin{align*}
 & \xi(\prod_{p}s_{p,0}^{k_{p,0}}s_{p,\sigma_{p,1}}s_{p,0}^{k_{p,1}}\cdots s_{p,\sigma_{p,b_{p}}}s_{p,0}^{k_{p,b_{p}}},\ev(a))\\
 & =\sum_{\substack{(N_{p,q})_{1\leq p\leq r,0\leq q\leq b_{p}}\\
0\leq N_{p,q}\leq k_{p,q}
}
}\sum_{\substack{{\bf p}_{p,q,h}\in D(s_{p,0},a_{p,q,h+1})\\
{\bf p}_{p,q}\in D(s_{p,q},a_{p,q,0})
}
}(\prod_{p,q,h}c_{{\bf p}_{p,q,h}})(\prod_{p,q}c_{{\bf p}_{p,q}})\\
 & \ \ \times\Bigg(\prod_{p,q}\sum_{h=0}^{N_{p,q}}V(s_{p,0},a_{p,q,h})^{k_{p,q}}\prod_{h'\neq h}\frac{1}{V(s_{p,0},a_{p,q,h})-V(s_{p,0},a_{p,q,h'})}\Bigg)\ev(a_{1,0,0}).
\end{align*}
Thus
{\small
\begin{align*}
 &c\left( \sum_{(k_{p,q})_{1\leq p\leq r,0\leq q\leq b_{p}}}\prod_{p}\langle\ma_{[m_p]}^{-1}(M_p)|
 {}^{k_{p,0},\dots,k_{p,b_{p}}}_{\sigma_{p,1},\dots,\sigma_{p,b_{p}}}\rangle
\xi(\prod_{p}s_{p,0}^{k_{p,0}}s_{p,\sigma_{p,1}}s_{p,0}^{k_{p,1}}\cdots s_{p,\sigma_{p,b_{p}}}s_{p,0}^{k_{p,b_{p}}},\ev(a) ) \right) \\
 & =\sum_{(N_{p,q})_{1\leq p\leq r,0\leq q\leq b_{p}}}\sum_{\substack{{\bf p}_{p,q,h}\in D(s_{p,0},a_{p,q,h+1})\\
{\bf p}_{p,q}\in D(s_{p,q},a_{p,q,0})}}c(\ev(a_{1,0,0}))(\prod_{p,q,h}c_{{\bf p}_{p,q,h}})(\prod_{p,q}c_{{\bf p}_{p,q}})\sum_{\substack{(h_{p,q})_{p,q}\\
0\leq h_{p,q}\leq N_{p,q}}}\\
 & \ \ \times\prod_{p,q}\prod_{h\neq h_{p,q}}\frac{1}{V(s_{p,0},a_{p,q,h_{p,q}})-V(s_{p,0},a_{p,q,h})}\\
 & \ \ \times\prod_{p}\Bigg(\sum_{(k_{p,q})_{0\leq q\leq b_{p}}}\langle\ma_{[m_p]}^{-1}(M_p)|
 {}^{k_{p,0},\dots,k_{p,b_{p}}}_{\sigma_{p,1},\dots,\sigma_{p,b_{p}}}\rangle
 \prod_{q}V(s_{p,0},a_{p,q,h_{p,q}})^{k_{p,q}}\Bigg)\\
 & =\sum_{(N_{p,q})_{1\leq p\leq r,0\leq q\leq b_{p}}}\sum_{\substack{{\bf p}_{p,q,h}\in D(s_{p,0},a_{p,q,h+1})\\{\bf p}_{p,q}\in D(s_{p,q},a_{p,q,0})}}
 c(\ev(a_{1,0,0}))(\prod_{p,q,h}c_{{\bf p}_{p,q,h}})(\prod_{p,q}c_{{\bf p}_{p,q}})\sum_{\substack{(h_{p,q})_{p,q}\\
0\leq h_{p,q}\leq N_{p,q}
}
}\\
 & \ \ \times\prod_{p,q}\prod_{h\neq h_{p,q}}\frac{1}{V(s_{p,0},a_{p,q,h_{p,q}})-V(s_{p,0},a_{p,q,h})}\\
 & \ \ \times\prod_{p}M_{p}\Bigg(\begin{array}{c}
V(s_{p,0},a_{p,1,h_{p,1}})-V(s_{p,0},a_{p,0,h_{p,0}}),\dots,V(s_{p,0},a_{p,b_{p},h_{p,b_{p}}})-V(s_{p,0},a_{p,b_{p-1},h_{p,b_{p-1}}})\\
\sigma_{p,1},\dots,\sigma_{p,b_{p}}
\end{array}\Bigg)\\
& \eqqcolon L_{a}(\sigma_{1,1},\dots,\sigma_{1,b_{1}};\dots;\sigma_{r,1},\dots,\sigma_{r,b_{r}})
\in\Q[[u_{1,1},\dots, u_{N,l_N}]]
.
\end{align*}
}
So $L_{a}(\sigma_{1,1},\dots,\sigma_{1,b_{1}};\dots;\sigma_{r,1},\dots,\sigma_{r,b_{r}})$ is mould-proper by Lemma \ref{lem: divisible} since the vectors $(V(s_{p,0}, a_{p,q,h}))_{q,h}$ are taken to be affinely independent.

Finally, by \eqref{eq: products of ma in terms of a}
we have
\begin{multline*}
c(\xi(\prod_{p}\ma_{[m_p]}^{-1}(M_{p})(s_{p,0},\dots,s_{p,m_{p}}),\ev(a) ))\\
=\sum_{b_{1},\dots,b_{p}=0}^{\infty}\sum_{(\sigma_{p,q})_{1\leq p\leq r,1\leq q\leq b_{p}}}L_{a}(\sigma_{1,1},\dots,\sigma_{1,b_{1}};\dots;\sigma_{r,1},\dots,\sigma_{r,b_{r}}).
\end{multline*}
We note that it is a finite sum 
by the definition of $\xi$.

Especially by considering the case
\[
a = \prod_{i=1}^{N}\tilde{\mathbb{I}}(\langle0\rangle,\sigma_{i,1},\langle u_{i,1}\rangle,\dots,\sigma_{i,l_{i}},\langle u_{i,1}+\cdots+u_{i,l_{i}}\rangle;i+1),
\]
we obtain the lemma  because $L_{a}(\sigma_{1,1},\dots,\sigma_{1,b_{1}};\dots;\sigma_{r,1},\dots,\sigma_{r,b_{r}})$ is shown to be mould-proper.
\end{proof}

\begin{lem}\label{lem:mould-proper-B}
Let $n,r\geq0$, $m_{1},\dots,m_{r}\geq0$.
Let $\{s_{k,l}\}_{1\leq k\leq r,0\leq l\leq m_{k}}$
be a collection of elements in
$\bigoplus_{1\leq i<j\leq n+2}\mathbb{Q}t_{ij}$ satisfying the same conditions as in Lemma \ref{lem:mould-proper-B-pre}.
Then the map
\[
\mathcal{M}(\mathcal{F}_{{\ser}};\{1,\dots,m_{1}\})\times\cdots\times\mathcal{M}(\mathcal{F}_{{\ser}};\{1,\dots,m_{r}\})\to {\mathcal P}_{n+2}(\mathcal{F}_\ser)
\]
given by
\begin{align*}
 & \mathcal{M}(\mathcal{F}_{{\ser}};\{1,\dots,m_{1}\})\times\cdots\times\mathcal{M}(\mathcal{F}_{{\ser}};\{1,\dots,m_{r}\})\simeq\widehat{U\mathfrak{f}_{1+m_{1}}}^{\dag}\times\cdots\times\widehat{U\mathfrak{f}_{1+m_{r}}}^{\dag}\\
 & \xrightarrow{(\varphi_{1},\dots,\varphi_{r})\mapsto\varphi_{1}(s_{1,0},\dots,s_{1,m_{1}})\cdots\varphi_{r}(s_{r,0},\dots,s_{r,m_{r}})}\widehat{U\mathfrak{t}_{n+2}}^{\dag}\xrightarrow[\simeq]{\madec_{n+2}} {\mathcal P}_{n+2}(\mathcal{F}_{{\ser}})
\end{align*}
is mould-proper.
\end{lem}

\begin{proof}
By taking the map $g$ in Lemma \ref{lem:mould-proper-bydef}
to be the composition of the first and the second maps above,
we obtain the claim by  Lemmas \ref{lem:mould-proper-bydef}
and \ref{lem:mould-proper-B-pre}.
\end{proof}

As an application of the above,
we see that  product $\mulp$ on ${\mathcal P}_{n}(\mathcal{F}_{{\ser}})$
given by \eqref{eq: diamond product}
is mould-proper.


\begin{lem}\label{lem:mould-proper-mulB}
The  binary operator $\mulp$ on ${\mathcal P}_{n}(\mathcal{F}_{{\ser}})$
is mould-proper.
\end{lem}

\begin{proof}
It follows by
\begin{align*}
 & (M_{1}\otimes\cdots\otimes M_{n-2})\mulp(N_{1}\otimes\cdots\otimes N_{n-2})\\
 & =\madec_{n}(\madec_{n}^{-1}(M_{1}\otimes\cdots\otimes M_{n-2})\cdot\madec_{n}^{-1}(N_{1}\otimes\cdots\otimes N_{n-2}))\\
 & =\madec_{n}\Bigl(\dec_{n}^{-1}(\ma_{[1]}^{-1}(M_{1})\otimes\cdots\otimes\ma_{[n-2]}^{-1}(M_{n-2})) \\
&\qquad\qquad\qquad\quad \cdot\dec_{n}^{-1}(\ma_{[1]}^{-1}(N_{1})\otimes\cdots\otimes\ma_{[n-2]}^{-1}(N_{n-2}))\Bigr)\\
 & =\madec_{n}\big(\ma_{[1]}^{-1}(M_{1})(t_{02},t_{12})\times\cdots\times\ma_{[n-2]}^{-1}(M_{n-2})(t_{0,n-1},\dots,t_{n-2,n-1})\\
 & \qquad \times\ma_{[1]}^{-1}(N_{1})(t_{02},t_{12})\times\cdots\times\ma_{[n-2]}^{-1}(N_{n-2})(t_{0,n-1},\dots,t_{n-2,n-1})\big)
\end{align*}
and Lemma \ref{lem:mould-proper-B}.
\end{proof}

For  a divisible family $\mathcal F$ of functions,
we can uniquely extend $\mulp$ to a mould-proper
 binary operator 
 \begin{equation}\label{eq: mulp:Pn}
 \mulp:{\mathcal P}_{n}(\mathcal{F})\times{\mathcal P}_{n}(\mathcal{F})\to{\mathcal P}_{n}(\mathcal{F})
 \end{equation}
by Theorem \ref{thm:UniqueProlongation} and Lemma \ref{lem:mould-proper-mulB},


\begin{defn}\label{def:the induced f*}
For a map $f:\{0,\dots,n-1\}\to\{0,\dots,m-1,\infty\}$, define the
continuous $\mathbb{Q}$-algebra homomorphism
\[
f^{*}:\widehat{U\mathfrak{t}_{m}}\to\widehat{U\mathfrak{t}_{n}}
\]
by $t_{ij}\mapsto\sum_{k\in f^{-1}(i),l\in f^{-1}(j)}t_{kl}$.
Note that $f^{*}(\widehat{U\mathfrak{t}_{m}}^\dag)\subset\widehat{U\mathfrak{t}_{n}}^\dag$ if $f(0)=0$. Furthermore, for such $f$,
define the map $f^{*,\mathcal{P}}:\mathcal{P}_{m}(\mathcal{F}_{{\ser}})\to\mathcal{P}_{n}(\mathcal{F}_{{\ser}})$
by $f^{*,\mathcal{P}}=\madec_{n}\circ f^{*}\circ\madec_{m}^{-1},$ and extend
it to
\[
f^{*,\mathcal{P}}:\mathcal{P}_{m}(\mathcal{F})\to\mathcal{P}_{n}(\mathcal{F})
\]
 by the following lemma and the unique prolongation theorem (Theorem
\ref{thm:UniqueProlongation}),
where $\mathcal{F}$ is a divisible family of functions.
\end{defn}

\begin{lem}
\label{lem:mould-proper-fpullback}
The map
$\madec_{n}\circ f^{*}\circ\madec_{m}^{-1}:\mathcal{P}_{m}(\mathcal{F}_{{\ser}})\to\mathcal{P}_{n}(\mathcal{F}_{{\ser}})$
is mould-proper if $f(0)=0$.
\end{lem}
\begin{proof}
By Lemma \ref{lem:mould-proper-bydef}, it is enough to show that
\begin{align*}
(M_{1},\dots,& M_{r})\mapsto\\
&c\circ\xi\Biggl(f^{*}(\prod_{p=1}^{r}\ma^{-1}(M_{p})(s_{p,0},\dots,s_{p,m_{p}})), \\
&\qquad\qquad\qquad\qquad \prod_{i=1}^{N}\mathbb{I}(\langle0\rangle,\rho_{i,1},\langle u_{i,1}\rangle,\dots,\rho_{i,l_{i}},\langle u_{i,0}+\cdots+u_{i,l_{i}}\rangle;\rho_{i,l_{i}+1})\Biggr)
\end{align*}
is mould-proper.

We consider the
$\mathbb{Q}$-algebra homomorphism $f_{*}:{\mathcal H}_n\to {\mathcal H}_m$
which is the dual of $f^{*}:\widehat{U\mathfrak{t}_{m}}\to\widehat{U\mathfrak{t}_{n}}$.
We note that
\[
f_{*}(\mathbb{I}(a_{1},\dots,a_{k};a_{k+1}))=\mathbb{I}(f(a_{1}),\dots,f(a_{k});f(a_{k+1}))
\]
if $f(a_{k+1})\neq \infty$,
where we regard the right hand side as $0$ if 
$f(a_j)=\infty$ for some $j$.
Then by definition,
\[
f_{*}\circ\xi(f^{*}(s),a)=\xi(s,f_{*}(a)),
\]
and since $c=c\circ f_{*}$, we have
\[
c\circ \xi(f^{*}(s),a)=c\circ\xi(s,f_{*}(a)).
\]
Thus
\begin{align*}
 & c\circ\xi\Bigg(f^{*}(\prod_{p=1}^{r}\ma^{-1}(M_{p})(s_{p,0},\dots,s_{p,m_{p}})),
 \\
 &\qquad\qquad \qquad
 \prod_{i=1}^{N}\mathbb{I}(\langle0\rangle,\rho_{i,1},\langle u_{i,1}\rangle,\dots,\rho_{i,l_{i}},\langle u_{i,0}+\cdots+u_{i,l_{i}}\rangle;\rho_{i,l_{i}+1})\Bigg)\\
 & =c\circ\xi\Bigg(\prod_{p=1}^{r}\ma^{-1}(M_{p})(s_{p,0},\dots,s_{p,m_{p}}),
 \\
 &\qquad \qquad\qquad
 \prod_{i=1}^{N}\mathbb{I}(\langle0\rangle,f(\rho_{i,1}),\langle u_{i,1}\rangle,\dots,f(\rho_{i,l_{i}}),\langle u_{i,0}+\cdots+u_{i,l_{i}}\rangle;f(\rho_{i,l_{i}+1}))\Bigg).
\end{align*}
Therefore the lemma follows by Lemma \ref{lem:mould-proper-B-pre}.
\end{proof}

For $\varphi\in\widehat{U\mathfrak{f}_{2}}$ and disjoint subsets
$S_{0},S_{1},S_{2}$ of $\{0,\dots,n-1\}$, we put
\begin{equation}\label{eq:varphi S0,S1,S2}
\varphi^{S_{0},S_{1},S_{2}}=\varphi(\sum_{(i,j)\in S_{0}\times S_{1}}t_{ij},\sum_{(j,k)\in S_{1}\times S_{2}}t_{jk})\in\widehat{U\mathfrak{t}_{n}}.
\end{equation}

\begin{lem}
In the above setting, if $0\in S_{0}$ and $\varphi\in\widehat{U\mathfrak{f}_{2}}^\dag$
then $\varphi^{S_{0},S_{1},S_{2}}\in\widehat{U\mathfrak{t}_{n}}^\dag$.
\end{lem}

\begin{proof}
By definition, $(e_{0j}\otimes{\rm id})\circ\Delta(\varphi^{S_{0},S_{1},S_{2}})=0$
if $j\notin S_{1}$. Assume $j\in S_{1}$. Let 
$\ev^{S_{0},S_{1},S_{2}}:\widehat{U\mathfrak{f}_{2}}\to\widehat{U\mathfrak{f}_{n}}$
be the map given by $\varphi\mapsto\varphi^{S_{0},S_{1},S_{2}}$.
Then
\[
\Delta(\varphi^{S_{0},S_{1},S_{2}})=(\ev^{S_{0},S_{1},S_{2}}\otimes \ev^{S_{0},S_{1},S_{2}})(\Delta\varphi).
\]
Thus
\begin{align*}
(e_{0j}\otimes{\rm id})\circ\Delta(\varphi^{S_{0},S_{1},S_{2}}) & =(e_{0j}\otimes{\rm id})\circ(\ev^{S_{0},S_{1},S_{2}}\otimes \ev^{S_{0},S_{1},S_{2}})(\Delta\varphi)\\
 & =((e_{0j}\circ \ev^{S_{0},S_{1},S_{2}})\otimes \ev^{S_{0},S_{1},S_{2}})(\Delta\varphi).\\
 \intertext{By using $e_0$ in \eqref{eq:e}, we have}
 & =(e_0\otimes \ev^{S_{0},S_{1},S_{2}})(\Delta\varphi) \\
 & =\ev^{S_{0},S_{1},S_{2}}\circ(e_0\otimes{\rm id})(\Delta\varphi).
\end{align*}
Thus if $\varphi\in\widehat{U\mathfrak{f}_{2}}^\dag$ i.e., $(e_0\otimes{\rm id})(\Delta\varphi)=0$
then $\varphi^{S_{0},S_{1},S_{2}}\in\widehat{U\mathfrak{t}_{n}}^\dag$.
\end{proof}

\begin{defn}\label{defn:Ev S0 S1 S2}
For disjoint subsets $S_{0},S_{1},S_{2}$ of $\{0,\dots,n-1\}$ with $1\in S_{0}$,
we define
the composition map
\[
\Ev^{S_{0},S_{1},S_{2}}:\mathcal{M}(\mathcal{F}_{{\ser}})
\xrightarrow[\simeq]{\ma_{[1]}^{-1}}
\widehat{U\mathfrak{f}_{2}}^\dag\xrightarrow{\varphi\mapsto\varphi^{S_{0},S_{1},S_{2}}}\widehat{U\mathfrak{t}_{n}}^\dag\xrightarrow[\simeq]{\madec_{n}}
{\mathcal P}_{n}(\mathcal{F}_{{\ser}}).
\]
For a divisible family $\mathcal{F}$ of functions,
it extends to  the map
$$
\Ev^{S_{0},S_{1},S_{2}}:\mathcal{M}(\mathcal{F})
\to
{\mathcal P}_{n}(\mathcal{F}),
$$
which is assured by the following lemma and the unique prolongation theorem (Theorem \ref{thm:UniqueProlongation}).
For $M\in\mathcal{M}(\mathcal{F})$, we write
{$M^{S_{0},S_{1},S_{2}}$
for $\Ev^{S_{0},S_{1},S_{2}}(M)\in {\mathcal P}_{n}(\mathcal{F})$}.
\end{defn}

\begin{lem}
The map $\Ev^{S_{0},S_{1},S_{2}}:
\mathcal{M}(\mathcal{F}_{{\ser}})\to
{\mathcal P}_{n}(\mathcal{F}_{{\ser}})$
is mould-proper.
\end{lem}

\begin{proof}
Define $f^{S_{0},S_{1},S_{2}}:\{0,\dots,n-1\}\to\{0,1,2,\infty\}$
by
\[
f(p)=\begin{cases}
i & p\in S_{i}\\
\infty & p\notin S_{1}\cup S_{2}\cup S_{3}.
\end{cases}
\]
Then $\varphi^{S_{1},S_{2},S_{3}}=(f^{S_{1},S_{2},S_{3}})^{*}\circ\varphi(t_{01},t_{12})$.
Thus the lemma follows by Lemmas \ref{lem:mould-proper-mulB} and \ref{lem:mould-proper-fpullback}.
\end{proof}

\section{Definitions of $\GARI(\mathcal{F})_{\as+\abal}$}
\label{sec:def of GARI as pent}
In \S\ref{subsec:review of M and GRT}
we recall  the associator set $\ASTR$,
the Grothendieck-Teichm\"{u}ller group $\GRT_1$ 
and its Lie algebra $\grt_1$
and give their quasi-involutive presentation 
in Theorem \ref{thm: flip characterization}.
In \S \ref{subsec: the map dec}
we prepare  the mould-theoretical involution $\bal$ and
the moulds $\Zag^w$ and $\Zag^{xy,w}$ ($w=x$ or $y$)
and introduce the set
$\GARI(\mathcal{F})_{\as+\abal}$ of balanced moulds, its even part
$\GARI(\mathcal{F})_{\underline{\as+\abal}}$
and a linear variant
$\ARI(\mathcal{F})_{\underline{\al+\abal}}$.
In \S\ref{subsec:def and prop of GARI as pent},
we prove that  $\ASTR$, $\GRT_1$  and $\grt_1$ agree with 
$\GARI(\mathcal{F})_{\as+\abal}$, 
$\GARI(\mathcal{F})_{\underline{\as+\abal}}$ and
$\ARI(\mathcal{F})_{\underline{\al+\abal}}$
respectively, under the map $\ma$,
in the case  ${\mathcal F}={\mathcal F}_\ser$ in Theorem \ref{thm: recovery theorem of ASTR and GRT}.
This is achieved by introducing an intermediate set of moulds
$\GARI(\mathcal{F})_{\as+\pent}$,
its even part 
$\GARI(\mathcal{F})_{\underline{\as+\pent}}$,
together with 
its linear variant $\ARI(\mathcal{F})_{\underline{\al+\pent}}$
in Definition \ref{defn: GARI as+pent}.
In \S \ref{subsec: pentagon equation for the associated constant moulds},
we prepare 
the notion of accompanied associators.
In \S \ref{sec: inclusion from GARI as pent to GARI as is}, we
introduce a series $\Zig$ and
show that our sets
$\GARI(\mathcal{F})_{\as+\pent}$ and $\GARI(\mathcal{F})_{\underline{\as+\pent}}$
are embedded into \'{E}calle's sets
$\GARI(\mathcal{F})_{\as\ast\is}$ and $\GARI(\mathcal{F})_{\underline{\as\ast\is}}$
respectively in Theorem \ref{thm: inclusion GARIas+pent to GARIas*is}.
In \S \ref{sec: specific element paj}, we show that \'{E}calle's specific mould $\paj$
gives an element in $\GARI(\mathcal{F})_{\as+\pent}$
in Theorem \ref{thm: paj in GARI as+pent}.


\subsection{Quasi-involutive reformulation of
$\ASTR$, $\GRT_1$ and $\grt_1$}
\label{subsec:review of M and GRT}
This subsection is devoted to recalling the associator set $\ASTR$, the Grothendieck-Teichm\"{u}ller group $\GRT_1$ and its Lie algebra $\grt_1$, introduced through the pentagon equation, and reformulate these objects using the involution $\flip$, culminating in Theorem \ref{thm: flip characterization}.

Here we consider the case of $\Gamma=\{1\}$.
We have  the free Lie algebra $\frak f_\Gamma=\frak f_2$
over $\Q$
with two variables $f_0$ and $f_1$
and its universal enveloping algebra $U\frak f_2=\Q\langle f_0,f_1\rangle$
as well as their completions
$\widehat{\frak f_2}$ and $\widehat{U\frak f_2}=\Q\langle\langle f_0,f_1\rangle\rangle$ 
with respect to degrees.

%
%

\begin{defn}[\cite{Dr, F10}]\label{defn:ASTR and GRT1}
(1).
We define  ${\ASTR}$ to be the set of series
$\varphi=\varphi(f_0,f_1)$ in
$\widehat{U\mathfrak f_2}$
which satisfies the following conditions:
\begin{itemize}
\item {the  group-like condition: $\varphi\in \exp \hat{\mathfrak f}_2$},
\item {the pentagon equation:}
\begin{equation}\label{eq: the pentagon equation}
\varphi^{01,2,3}\varphi^{0,1,23}=\varphi^{0,1,2}\varphi^{0,12,3}\varphi^{1,2,3}
\end{equation}
in  $\widehat{U\mathfrak t_4}$.
\footnote{
See \eqref{eq:varphi S0,S1,S2}, for notation $\varphi^{01,2,3}$, etc.
For technical reason, we reverse the order of the product in \cite{Dr}.
Note that we have $\overleftarrow{\Phi_{\mathrm{KZ}}}(f_0,f_1):=
\Phi_{\mathrm KZ}(-f_0,-f_1)^{-1}\in\ASTR$
where $\Phi_{\mathrm KZ}$ is the Drinfeld (KZ) associator (see. \cite{F10} Remark 6 (i) ).
}
\end{itemize}

(2).
For $\mu\in\Q$,
we define
$\ASTR_{\mu}$ 
to be the subset 
which consists of
$\varphi$ 
satisfying the above conditions
and $\langle\varphi|f_1f_0\rangle=\frac{\mu^2}{24}$.
Here $\langle\varphi|f_1f_0\rangle$ means the coefficient of $\varphi$ on $f_0f_1$.

(3).
The (graded) {\it Grothendieck-Teichm\"{u}ller group}  $\GRT_1$ is
set-theoretically
defined to be  $\ASTR_0$.

In \cite{Dr}, it is shown that
the set $\GRT_1$ forms a group
and $\ASTR_\mu$ forms a left $\GRT_1$-torsor for each $\mu\in\Q^\times$
under the operation
\begin{align}\label{eq:circled-ast-product}
\varphi_2\circledast\varphi_1
&=\varphi_2(f_0,f_1)\cdot \varphi_1\left(\varphi_2(f_0,f_1)^{-1} f_0\varphi_2(f_0,f_1),f_1\right) \\ \notag
&=\varphi_1\left(f_0,\varphi_2(f_0,f_1) f_1\varphi_2(f_0,f_1)^{-1}\right)
\cdot\varphi_2(f_0,f_1)
.
\end{align}

(4).
The {\it Grothendieck-Teichm\"{u}ller Lie algebra} $\grt_1$
is the associated  graded Lie algebra of $\GRT_1$ which is set-theoretically
the collection of series $\psi$ satisfying
\begin{itemize}
\item {the Lie-like condition: $\psi\in {\mathfrak f}_2$,}
\item the pentagon equation:
$\psi^{01,2,3}+\psi^{0,1,23}=\psi^{0,1,2}+\psi^{0,12,3}+\psi^{1,2,3}$
in $\mathfrak t_4$.
\end{itemize}
It forms a Lie algebra under the bracket
\begin{equation}\label{eq: Ihara bracket}
\{\psi_2,\psi_1\}
=[\psi_1,\psi_2]+{D}_{\psi_2}(\psi_1)-{D}_{\psi_1}(\psi_2)
\end{equation}
where
$D_\psi$ is the derivation of $\mathfrak f_2$,
defined by 
$$
D_\psi(f_0):=0, \qquad\qquad
D_\psi(f_1):=[\psi, f_1].
$$
\end{defn}

We prepare the following notions.

\begin{defn}
(1). A series $\varphi\in \exp \hat{\mathfrak f}_2$ 
is called
 {\it commutator group-like} if
$\varphi$ lies on the topological commutator $\exp [\hat{\mathfrak f}_2,\hat{\mathfrak f}_2]$.

(2). A series $\alpha\in \hat{\mathfrak f}_2$ 
is called
 {\it commutator Lie-like} if
$\alpha$ lies on the topological commutator $ [\hat{\mathfrak f}_2,\hat{\mathfrak f}_2]$.
\end{defn}

\begin{lem}\label{lem: commutator-like}
(1). Every element in $\ASTR$ is commutator group-like.

(2). Every element in $\grt_1$ is commutator Lie-like.
\end{lem}

\begin{proof}
(1).
By applying the map
\[
\widehat{U\mathfrak{t}_{4}}\to\widehat{U\mathfrak{t}_{3}}\ ;\ (t_{01},t_{02},t_{03},t_{12},t_{13},t_{23})\mapsto(t_{12},t_{13},0,t_{23},0,0),
\]
to the pentagon equation \eqref{eq: the pentagon equation}, 
we obtain
\[
\varphi(t_{13}+t_{23},0)\varphi(t_{12},t_{23})
=\varphi(t_{12},t_{23})\varphi(t_{12}+t_{13},0)\varphi(t_{23},0).
\]
Since $\varphi$ is group-like then $\varphi(f_{0},0)=\exp(\alpha f_{0})$
for some $\alpha$. Then we have
\begin{align*}
\varphi(t_{12},t_{23}) & =\exp(-\alpha(t_{13}+t_{23}))
\varphi(t_{12},t_{23})\exp(\alpha(t_{12}+t_{13}+t_{23}))\\
 & =\exp(\alpha t_{12})\varphi(t_{12},t_{23}).
\end{align*}
Thus $\alpha=0$, so we have $\varphi(f_{0},0)=1$.

Next by applying the map
\[
\widehat{U\mathfrak{t}_{4}}\to\widehat{U\mathfrak{t}_{3}}\ ;\ (t_{01},t_{02},t_{03},t_{12},t_{13},t_{23})\mapsto(0,0,0,t_{12},t_{13},t_{23}),
\]
to the pentagon equation \eqref{eq: the pentagon equation}, 
we obtain
\[
\varphi(t_{12},t_{23})\varphi(0,t_{12}+t_{13})
=\varphi(0,t_{12})\varphi(0,t_{13}+t_{23})\varphi(t_{12},t_{23}).
\]
Since $\varphi$ is group-like then $\varphi(0,f_{1})=\exp(\beta f_{1})$
for some $\beta$. Then we have
\begin{align*}
\varphi(t_{12},t_{23}) & =\exp(\beta(t_{12}+t_{13}+t_{23}))
\varphi(t_{12},t_{23})\exp(-\beta(t_{12}+t_{13}))\\
 & =\varphi(t_{12},t_{23})\exp(\beta t_{23}).
\end{align*}
Thus $\beta=0$, so we have $\varphi(0,f_{1})=1$.
 
Whence we see that there are no linear terms on $\varphi$,
which means that it is  commutator group-like.

(2). 
The second assertion is proved in the same way.
\end{proof}

We prepare the following lemmas.

 \begin{lem}
 Let $u_{1},u_{2},v_{1},v_{2}$ be elements of $\widehat{U{\mathfrak{t}}_{4}}$
 without constant terms satisfying $[u_{1},v_{1}]=[u_{1},v_{2}]=[u_{2},v_{1}]=[u_{2},v_{2}]=0$.
 Then
 $$
 m\circ(s_{u_{1},u_{2}}\otimes s_{v_{1},v_{2}})\circ\Delta=s_{u_{1}+v_{1},u_{2}+v_{2}}
 $$
 as a map from $\widehat{U{\mathfrak{t}}_{4}}$ to $\widehat{U{\mathfrak{t}}_{4}}$
 where $s_{u,v}:\widehat{U{\mathfrak{f}}_{2}}\to\widehat{U{\mathfrak{t}}_{4}}$
 is the specialization map defined by $s_{u,v}(\varphi(f_{0},f_{1})):=\varphi(u,v)$
 and $m:\widehat{U{\mathfrak{t}}_{4}}^{\dag}\otimes\widehat{U{\mathfrak{t}}_{4}}^{\dag}\to\widehat{U{\mathfrak{t}}_{4}}^{\dag}\otimes\widehat{U{\mathfrak{t}}_{4}}^{\dag}$
 is the map defined by the multiplication.
 \end{lem}

 \begin{proof}
 It follows from the following calculation:
 \begin{align*}
  & m\circ(s_{u_{1},u_{2}}\otimes s_{v_{1},v_{2}})\circ\Delta(f_{i_{1}}\cdots f_{i_{k}})\\
  & =m\circ(s_{u_{1},u_{2}}\otimes s_{v_{1},v_{2}})(\sum_{n_{1},\dots,n_{k}\in\{0,1\}}f_{i_{1}}^{n_{1}}\cdots f_{i_{k}}^{n_{k}}\otimes f_{i_{1}}^{1-n_{1}}\cdots f_{i_{k}}^{1-n_{k}})\\
  & =\sum_{n_{1},\dots,n_{k}\in\{0,1\}}u_{i_{1}}^{n_{1}}\cdots u_{i_{k}}^{n_{k}}v_{i_{1}}^{1-n_{1}}\cdots v_{i_{k}}^{1-n_{k}}\\
  & =(u_{i_{1}}+v_{i_{1}})\cdots(u_{i_{k}}+v_{i_{k}})\\
  & =s_{u_{1}+v_{1},u_{2}+v_{2}}(f_{i_{1}}\cdots f_{i_{k}}).\qedhere
 \end{align*}
 \end{proof}

 The lemma above implies the following lemma.
 
 \begin{lem} \label{lem:phi_change}
 Let $u$, $v$ and $w$ be elements of $\widehat{U{\mathfrak{t}}_{4}}$
without constant terms satisfying $[u,v]=0$ and $[u-v,w]=0$. Then,
for $\varphi(f_{0},f_{1})\in\widehat{U{\mathfrak{f}}_{2}}^{\dag}$,
we have
\[
\varphi(u,w)=\varphi(v,w).
\]
\end{lem}

The following gives a \lq quasi-involutive' characterization
$\ASTR$ and  $\grt_1$  in terms of the operator  $\flip$ in \eqref{eq: flip},
which is required to our mould-theoretic formulation in \S \ref{subsec:def and prop of GARI as pent}.

\begin{thm}\label{thm: flip characterization}
(1).
The set $\ASTR\cdot \exp\Q f_1$ is equal to the set
of 
series $g\in \exp \hat{\mathfrak f}_2^\dag$
for which  there exists a commutator group-like series
$h\in\widehat{U\mathfrak{f}_2}$ such that
 $$
 \flip(g^{0,1,2}g^{0,12,3})=g^{0,1,2}g^{0,12,3}\cdot h^{1,2,3},
 $$
 (for $\flip$ see \eqref{eq: flip})
or equivalently 
$$g^{01,2,3}g^{0,1,23} = g^{0,1,2}g^{0,12,3}\cdot h^{1,2,3}.$$

(2).
The $\Q$-linear space $\grt_1+\Q f_1$ is equal to the set
of series $\alpha\in  \hat{\mathfrak f}_2^\dag$
for which there exists a commutator Lie-like
series $\beta\in\widehat{\mathfrak{f}_2}$ such that
 $$
 \flip(\alpha^{0,1,2}+\alpha^{0,12,3})=\alpha^{0,1,2}+\alpha^{0,12,3}+ \beta^{1,2,3}
 $$
 or equivalently 
$$\alpha^{01,2,3}+\alpha^{0,1,23} = \alpha^{0,1,2}+\alpha^{0,12,3}+\beta^{1,2,3}.$$

\end{thm}

\begin{proof}
(1).  Let $g$ and $h$ be as in the statement.
Put $t_{4i}=t_{i4}:
=-\sum_{j=0}^3t_{ji}$.
Since  $g\in \exp \hat{\mathfrak f}_2^\dag$,  Lemma \ref{lem:phi_change} yields
\begin{align*}
\flip(g^{0,1,2}g^{0,12,3})&=g^{4,3,2}g^{4,32,1}=g^{01,2,3}g^{0,1,23}.
\end{align*}
Therefore  we have
\begin{equation}\label{eq: mixed pentagon g and h}
g^{01,2,3}g^{0,1,23}=g^{0,1,2}g^{0,12,3}\cdot h^{1,2,3}.
\end{equation}
By applying the algebra morphism
\[
\widehat{U\mathfrak{t}_{4}}\to\widehat{U\mathfrak{t}_{3}}\ ;\ (t_{01},t_{02},t_{03},t_{12},t_{13},t_{23})\mapsto(0,0,0,t_{12},t_{13},t_{23}),
\]
to \eqref{eq: mixed pentagon g and h}, we have
\[
g(t_{12},t_{23})g(0,t_{12}+t_{13})
=g(0,t_{12})g(0,t_{13}+t_{23})h(t_{12},t_{23}).
\]
Since $g$ is group-like, we have $g(0,f_{2})=\exp(\alpha f_{2})$
for some $\alpha$. Then
\begin{align*}
g(t_{12},t_{23}) & =\exp(\alpha(t_{12}+t_{13}+t_{23}))
h(t_{12},t_{23})\exp(-\alpha(t_{12}+t_{13}))\\
 & =h(t_{12},t_{23})\exp(\alpha t_{23}),
\end{align*}
and thus
\[
g=h\cdot \exp(\alpha f_{1}).
\]
By putting $g=h\cdot\exp(\alpha f_{1})$ in \eqref{eq: the pentagon equation},
we obtain
\[
h^{01,2,3}h^{0,1,23}=h^{0,1,2}
h^{0,12,3}h^{1,2,3},
\]
and thus $h$ belongs to $\ASTR$. 
So we obtain $g=h\cdot\exp(\alpha f_{1})\in\ASTR\cdot\exp\Q f_1$.

Conversely  assume $g=\varphi\cdot \exp(\alpha f_1)$ with $\varphi\in\ASTR$ and $\alpha\in\Q$.
Then 
\begin{align*}
\flip(g^{0,1,2}g^{0,12,3})&=\flip(\varphi^{0,1,2}\exp(\alpha t_{12})\varphi^{0,12,3}\exp(\alpha t_{12,3})) \\
&=\flip(\varphi^{0,1,2}\varphi^{0,12,3}\exp(\alpha (t_{12}+t_{13}+t_{23}))) \\
&=\varphi^{4,3,2}\varphi^{4,32,1}\exp(\alpha (t_{32}+t_{31}+t_{21})), \\
\intertext{by Lemma \ref{lem: commutator-like} and Lemma \ref{lem:phi_change}}
&=\varphi^{01,2,3}\varphi^{0,1,23}\exp(\alpha (t_{12}+t_{13}+t_{23})), \\
\intertext{by \eqref{eq: the pentagon equation}}
&=\varphi^{0,1,2}\varphi^{0,12,3}\varphi^{1,2,3}\exp(\alpha (t_{12}+t_{13}+t_{23})) \\
&=g^{0,1,2}\exp(-\alpha t_{12})
g^{0,12,3}\exp(-\alpha t_{12,3})\varphi^{1,2,3}\exp(\alpha (t_{12}+t_{13}+t_{23}))\\
&=g^{0,1,2}g^{0,12,3}\cdot \varphi^{1,2,3}.
\end{align*}
By Lemma \ref{lem: commutator-like},  $\varphi$ is commutator group-like.
So  $g$ belongs to the set described in (1).

(2). The proof is analogous to that of (1), with the group-like setting
replaced by the Lie-like one.
\end{proof}

\subsection{The balance map $\bal$}
\label{subsec: the map dec}

We introduce the balance map
$
\bal:\mathcal{P}_{4}(\mathcal{F})\to
\mathcal{P}_{4}(\mathcal{F})
$
in Theorem-Definition \ref{thm-defn: dec}
and the moulds $\Zag^w$ and $\Zag^{xy,w}$ ($w=x$ or $y$)
in Definition \ref{def:Zag and Zag}
which will be required  in the presentation of
the mould-theoretic variant $\GARI(\mathcal{F})_{\as+\abal}$ of ${\ASTR}$
in the next subsection.
In Proposition \ref{prop: bal= flip},
we  show how  $\bal$
is related with the involution $\flip$ introduced in \eqref{eq: flip}
via another involution $\rev$.

We consider the series given by the following iterated integrals
\begin{equation}\label{eq:power series I}
I\varia{u_1,\ \dots,\ u_d}{\epsilon_1,\ \dots,\ \epsilon_d}
:=
\int_{0<s_1<\cdots< s_d<1}
\frac{s_1^{-u_1}ds_1}{\epsilon_1-s_1}\wedge\cdots\wedge
\frac{s_d^{-u_d}ds_d}{\epsilon_d-s_d}
\in \C[[u_1,\dots,u_d]]
\end{equation}
with suitable complex numbers $\epsilon_1,\ \dots,\ \epsilon_d$.
Put $I\varia{\emptyset}{\emptyset}=1$.
We note that it is expanded as
\begin{align}\label{eq: I= sum Li}
I&\varia{u_1,\ \dots,\ u_d}{\epsilon_1,\ \dots,\ \epsilon_d} \\ \notag
&=\sum_{n_1,\dots, n_d>0}
\Li_{n_1,\dots,n_d}(\frac{\epsilon_2}{\epsilon_1},\dots,\frac{\epsilon_{d}}{\epsilon_{d-1}},\frac{1}{\epsilon_d})
u_1^{n_1-1}(u_1+u_2)^{n_2-1}\cdots
(u_1+\cdots+u_d)^{n_d-1}
\\ \notag
&=(-1)^{d}
\sum_{n_1,\dots, n_d>0}
I(0;\epsilon_1,{\bf 0}^{n_1-1},\dots,\epsilon_d,{\bf 0}^{n_d-1} ;1) \\ \notag
&\qquad\qquad\qquad\qquad \qquad\qquad
\cdot u_1^{n_1-1}(u_1+u_2)^{n_2-1}\cdots
(u_1+\cdots+u_d)^{n_d-1} \nonumber	 
\end{align}
where
$\Li_{n_1,\dots,n_d}$ is
in \eqref{eq:Li} and
$I(c_0;c_1,\dots, c_N;c_{N+1})$ with $c_0,\dots, c_{N+1}\in\C$
stands for the iterated integral
\begin{equation}\label{eq: symbol iterated integral}
I(c_0;c_1,\dots, c_N;c_{N+1})=
\int_{0<t_1<\cdots< t_N<1}
\frac{d\gamma(t_1)}{\gamma(t_1)-c_1}\wedge\cdots\wedge\frac{d\gamma(t_N)}{\gamma(t_N)-c_N}
\end{equation}
associated with each topological path $\gamma(t)$ ($t\in [0,1]$) in $\C$ with
$\gamma(0)=c_0$ and $\gamma(1)=c_{N+1}$.


\begin{lem}
As a multi-variable complex function with respect to the variables $\epsilon_1,\dots,\epsilon_d$,
it satisfies the following recursive differential equation
\begin{align}\label{eq:recdiff}
dI\varia{u_1,\dots, u_d}{\epsilon_1, \dots,\epsilon_d}
=&-\sum_{i=1}^d d\log(\frac{\epsilon_i-\epsilon_{i+1}}{\epsilon_i})
I\varia{u_1,\dots, u_{i-1},u_{i}+u_{i+1},\dots u_d}{\epsilon_1, \dots,\epsilon_{i-1},\quad \epsilon_{i+1},\dots,\epsilon_d}\\ \notag
&+\sum_{i=2}^d d\log(\frac{\epsilon_i-\epsilon_{i-1}}{\epsilon_i})
I\varia{u_1,\dots, u_{i-1}+u_{i},u_{i+1},\dots u_d}{\epsilon_1, \dots,\ \epsilon_{i-1},\quad \epsilon_{i+1},\dots,\epsilon_d}\\  \notag
&-\sum_{i=1}^d u_id\log \epsilon_iI\varia{u_1,\dots, u_d}{\epsilon_1, \dots,\epsilon_d}.
\end{align}
Here we understand the $i=d$ part in the summation of the first term of the right hand side as
\[
d\log(\frac{\epsilon_d-1}{\epsilon_d})
I\varia{u_1,\dots, u_{d-1}}{\epsilon_1, \dots,\epsilon_{d-1}}.
\]
\end{lem}

\begin{proof}
Let $\epsilon_{d+1}:=1$.
By \cite[Theorem 2.1]{G-MPL} and \cite[Lemma 3.3.30]{P},
we have
\begin{align*}
d&I\varia{u_1,\dots, u_d}{\epsilon_1, \dots,\epsilon_d} \\
&=
(-1)^d\sum_{n_1,\dots,n_d>0}
dI(0;\epsilon_1,{\bf 0}^{n_1-1},\dots,\epsilon_d,{\bf 0}^{n_d-1} ;1)
u_1^{n_1-1}(u_1+u_2)^{n_2-1}\cdots
(u_1+\cdots+u_d)^{n_d-1} \\
&
=(-1)^d\sum_{k=1}^d(-d\log \epsilon_k)
\sum_{\substack{n_1,\dots,n_d>0 \\ n_k>1}}
I(0;\epsilon_1,{\bf 0}^{n_1-1},\dots,
{\bf 0}^{n_{k-1}-1},
\epsilon_k,{\bf 0}^{n_k-2}, \dots,
\epsilon_d,{\bf 0}^{n_d-1} ;1) \\
&\qquad\qquad\cdot u_1^{n_1-1}(u_1+u_2)^{n_2-1}\cdots
(u_1+\cdots+u_d)^{n_d-1} \\
&
+(-1)^d\sum_{k=2}^d(d\log \epsilon_k)
\sum_{\substack{n_1,\dots,n_d>0 \\ n_{k-1}>1}}
I(0;\epsilon_1,{\bf 0}^{n_1-1},\dots,
{\bf 0}^{n_{k-1}-2},
\epsilon_k,{\bf 0}^{n_k-1}, \dots,
\epsilon_d,{\bf 0}^{n_d-1} ;1) \\
&\qquad\qquad\cdot u_1^{n_1-1}(u_1+u_2)^{n_2-1}\cdots
(u_1+\cdots+u_d)^{n_d-1} \\
&
+(-1)^d\sum_{k=1}^d(-d\log \epsilon_k)
\sum_{\substack{n_1,\dots,n_d>0 \\ n_k=1}}
I(0;\epsilon_1,{\bf 0}^{n_1-1},\dots,
\epsilon_{k-1},{\bf 0}^{n_{k-1}-1},
\epsilon_{k+1},
, \dots,
\epsilon_d,{\bf 0}^{n_d-1} ;1) \\
&\qquad\qquad\cdot u_1^{n_1-1}(u_1+u_2)^{n_2-1}\cdots
(u_1+\cdots+u_d)^{n_d-1} \\
&
+(-1)^d\sum_{k=2}^d(d\log \epsilon_k)
\sum_{\substack{n_1,\dots,n_d>0 \\ n_{k-1}=1}}
I(0;\epsilon_1,{\bf 0}^{n_1-1},\dots,
\epsilon_{k-1},{\bf 0}^{n_{k}-1},
\epsilon_{k+1},
\dots,
\epsilon_d,{\bf 0}^{n_d-1} ;1) \\
&\qquad\qquad\cdot u_1^{n_1-1}(u_1+u_2)^{n_2-1}\cdots
(u_1+\cdots+u_d)^{n_d-1}
\end{align*}
\begin{align*}
&
+(-1)^d\sum_{k=1}^dd\log (\epsilon_{k+1}-\epsilon_k)
\sum_{\substack{n_1,\dots,n_d>0 \\ n_k=1}}
I(0;\epsilon_1,{\bf 0}^{n_1-1},\dots,
\epsilon_{k-1},{\bf 0}^{n_{k-1}-1},
\epsilon_{k+1},
\dots,
\epsilon_d,{\bf 0}^{n_d-1} ;1) \\
&\qquad\qquad\cdot u_1^{n_1-1}(u_1+u_2)^{n_2-1}\cdots
(u_1+\cdots+u_d)^{n_d-1} \\
&
+(-1)^d\sum_{k=2}^d(-d\log (\epsilon_{k}-\epsilon_{k-1}))
\sum_{\substack{n_1,\dots,n_d>0 \\ n_{k-1}=1}}
I(0;\epsilon_1,{\bf 0}^{n_1-1},\dots,
\epsilon_{k-1},{\bf 0}^{n_{k}-1},
\epsilon_{k+1},
\dots,
\epsilon_d,{\bf 0}^{n_d-1} ;1) \\
&\qquad\qquad\cdot u_1^{n_1-1}(u_1+u_2)^{n_2-1}\cdots
(u_1+\cdots+u_d)^{n_d-1}.
\end{align*}
By regrouping the first and second terms,
the third and fifth ones, and the fourth and the last ones,
we obtain the required differential equation.
\end{proof}

By the above differential equation,
the analytic continuation of
$I\varia{u_1,\dots, u_d}{\epsilon_1, \dots,\epsilon_d}$
can be done with respect to the parameters $\epsilon_1,\dots,\epsilon_d$.

\begin{thm}\label{thm: rec diff eq for I pre}
Let $d\in\N$ and $x$ and $y$  be complex parameters. 
Let  ${\bf\epsilon}=(\epsilon_1,\dots,\epsilon_d)\in\{y,xy\}^d$.
Then the above
$I\varia{u_1,\dots, u_d}{\epsilon_1, \dots,\epsilon_d}$
can be expressed 
as finite sum
\begin{equation}\label{eq:I dec pre}
I\varia{u_1,\dots, u_d}{\epsilon_1, \dots,\epsilon_d}
=\sum_{{\bf j}={\bf j}({\bf{\epsilon}})} c_{\bf j}
\cdot
I\varia{v_{{\bf j},1}, \dots, v_{{\bf j},r_{\bf j}}}{\ y,\dots\dots,y}
\cdot
I\varia{w_{{\bf j},1}, \dots, w_{{\bf j},s_{\bf j}}}{\ \tau_{{\bf j},1}, \dots, \tau_{{\bf j},s_{\bf j}} }
\end{equation}
with $c_{\bf j}\in\Z$,
\begin{equation}\label{eq:epsilon set}
\tau_{{\bf j},1}, \dots, \tau_{{\bf j},s_{\bf j}}
\in\{0,x,xy\},
\end{equation}
and linearly independent  tuples
$(v_{{\bf j},k})_{1\leq k\leq r_{\bf j}}$,
$(w_{{\bf j},l})_{1\leq k\leq s_{\bf j}}$
with $r_{\bf j}+s_{\bf j} = d$,  $r_{\bf j}, s_{\bf j} \geq 0$
which are $\Z$-linear combinations of $u_1,\dots, u_d$ for each $\bf j$.
\end{thm}



\begin{proof}
We prove the more stronger statement with additional condition
\[
w_{{\bf j},1} + \cdots + w_{{\bf j},s_{\bf j}} = \sum_{\substack{\epsilon_i=xy}} u_i,
\]
by induction on $d$. The case $d=1$ follows from
$I\varia{u}{xy}=I\varia{\emptyset}{\emptyset}\cdot I\varia{u}{xy}$
and
$I\varia{u}{y}=I\varia{u}{y}\cdot I\varia{\emptyset}{\emptyset}$.
Assume that $d\geq 2$.
By \eqref{eq:recdiff}, we have
\begin{align}\label{eq:recdiff1}
\frac{d}{dx}I\varia{u_1,\dots, u_d}{\epsilon_1, \dots,\epsilon_d}=
\sum_{\bf p}h_{\bf p}(\frac{1}{x-a_{\bf p}}-\frac{1}{x})I\varia{u_{{\bf p},1},\dots, u_{{\bf p},d-1}}{\epsilon_{{\bf p},1}, \dots,\epsilon_{{\bf p},d-1}}
-(\sum_{\substack{\epsilon_i=xy}} u_i) \frac{1}{x}
I\varia{u_1,\dots, u_d}{\epsilon_1, \dots,\epsilon_d}
\end{align}
with $h_{\bf p}\in\Z$, $\epsilon_{{\bf p},j}\in\{y,xy\}$, $a_{\bf p}\in\{1,y^{-1},\infty\}$
and $u_{{\bf p},1},\dots,u_{{\bf p},d-1}\in \bigoplus_i \Z u_i $ of the form
\[
(u_{{\bf p},1},\dots,u_{{\bf p},d-1}) = (u_1,\dots,u_{j({\bf p})-2},u_{j({\bf p})-1}+u_{j({\bf p})},u_{j({\bf p})+1},\dots,u_d)
\]
where $j({\bf p})$ is an integer satisfying $2\leq j({\bf p})\leq d$ and $\epsilon_{j({\bf p})-1}\neq\epsilon_{j({\bf p})}$
or \[
(u_{{\bf p},1},\dots,u_{{\bf p},d-1})=(u_{1},\dots,u_{d-1})
\]
where $\epsilon_{d}=xy$.
By our induction hypothesis, for each  index $\bf p$,
it is written as
$$
I\varia{u_{{\bf p},1},\dots, u_{{\bf p},d-1}}{\epsilon_{{\bf p},1}, \dots,\epsilon_{{\bf p},d-1}}=
\sum_{{\bf q}}
c_{\bf q}
\cdot
I\varia{v_{{\bf q},1},\dots, v_{{\bf q},r_{{\bf q}}}}{\quad y,\dots\dots,y}
\cdot
I\varia{w_{{\bf q},1},\dots, w_{{\bf q},s_{{\bf q}}}}{\ \tau_{{\bf q},1}, \dots,\tau_{{\bf q},s_{{\bf q}}}}
$$
with $r_{{\bf q}}+s_{{\bf q}}\leq d-1$,
$\tau_{{\bf p},j}\in\{0,x,xy\}$.
Here $v_{{\bf q},j}$ for $j=1,\dots,r_{\bf q}$ and
$w_{{\bf q},j}$ for $j=1,\dots,s_{\bf q}$ are linearly independent $\Z$-linear sum of
$u_{{\bf p},1},\dots,u_{{\bf p},d-1}$ satisfying
\[
w_{{\bf q},1}+\cdots+w_{{\bf q},s_{\bf q}} = \sum_{\substack{\epsilon_{{\bf p},i}=xy}} u_{{\bf p},i}.
\]
We put
$$
J\varia{u_1,\dots, u_d}{\epsilon_1, \dots,\epsilon_d}
:=-\sum_{{\bf p}}h_{\bf p}\sum_{{\bf q} =\bf q(\bf p)}
c_{\bf q}
\cdot
I\varia{v_{{\bf q},1},\dots, v_{{\bf q},r_{{\bf q}}}}{\quad y,\dots\dots,y}
\cdot
I\varia{w_{{\bf q},1},\dots, w_{{\bf q},s_{{\bf q}}},W_{\bf q}}{\ \tau_{{\bf q},1}, \dots,\tau_{{\bf q},s_{{\bf q}}},T_{\bf q}}
$$
with
$$
w_{{\bf q},1}+\cdots+w_{{\bf q},s_{\bf q}} +W_{\bf q}=\sum_{\substack{\epsilon_i=xy}} u_i
\quad\text{and}\quad
T_{\bf q} = xa_{\bf p}^{-1} =
\begin{cases}
0   &( a_{\bf p}=\infty), \\
x   & (a_{\bf p}=1), \\
xy  & (a_{\bf p}=y^{-1}).
\end{cases}
$$
Then by \eqref{eq:recdiff}
one can calculate
\begin{align}\notag
\frac{d}{dx}J\varia{u_1,\dots, u_d}{\epsilon_1, \dots,\epsilon_d}
&=\sum_{{\bf p}}h_{\bf p}\sum_{{\bf q} =\bf q(\bf p)}
(\frac{1}{x-a_{\bf p}}-\frac{1}{x})\cdot c_{\bf q}
\cdot
I\varia{v_{{\bf q},1},\dots, v_{{\bf q},r_{{\bf q}}}}
{\quad y,\dots\dots,y}
\cdot
I\varia{w_{{\bf q},1},\dots, w_{{\bf q},s_{{\bf q}}}}{\tau_{{\bf q},1}, \dots,\tau_{{\bf q},s_{{\bf q}}}}
\\ \notag
+\sum_{{\bf p}}h_{\bf p}\sum_{{\bf q} =\bf q(\bf p)}
&(w_{{\bf q},1}+\cdots+w_{{\bf q},s_{\bf q}} +W_{\bf q})\cdot \frac{1}{x}
c_{\bf q}\cdot
I\varia{v_{{\bf q},1},\dots, v_{{\bf q},r_{{\bf q}}}}
{\quad y,\dots\dots,y}
\cdot
I\varia{w_{{\bf q},1},\dots, w_{{\bf q},s_{{\bf q}}},W_{\bf q}}{\tau_{{\bf q},1}, \dots,\tau_{{\bf q},s_{{\bf q}}},T_{\bf q}}
 \\ \label{eq:recdiff2}
&=\sum_{\bf p}h_{\bf p}(\frac{1}{x-a_{\bf p}}-\frac{1}{x})I\varia{u_{{\bf p},1},\dots, u_{{\bf p},d-1}}{\epsilon_{{\bf p},1}, \dots,\epsilon_{{\bf p},d-1}}
-(\sum_{\substack{\epsilon_i=xy}} u_i)\frac{1}{x}
J\varia{u_1,\dots, u_d}{\epsilon_1, \dots,\epsilon_d}.
\end{align}
We consider
$$
K\varia{u_1,\dots, u_d}{\epsilon_1, \dots,\epsilon_d}:=
I\varia{u_1,\dots, u_d}{\epsilon_1, \dots,\epsilon_d}
-J\varia{u_1,\dots, u_d}{\epsilon_1, \dots,\epsilon_d}.
$$
Then by \eqref{eq:recdiff1} and \eqref{eq:recdiff2}, we have
$$
\frac{d}{dx} K\varia{u_1,\dots, u_d}{\epsilon_1, \dots,\epsilon_d}
=-(\sum_{\substack{\epsilon_i=xy}} u_i)\frac{1}{x}
K\varia{u_1,\dots, u_d}{\epsilon_1, \dots,\epsilon_d}.
$$
So $K$ is expressed as $f(y)e^{-(\sum_{\substack{\epsilon_i=xy}} u_i)/x}$ with some function $f(y)$. \\
Since $\lim_{x\to \infty}K\varia{u_1,\dots, u_d}{\epsilon_1, \dots,\epsilon_d}=0$,
we see  $K=0$. Therefore
$I\varia{u_1,\dots, u_d}{\epsilon_1, \dots,\epsilon_d}=
J\varia{u_1,\dots, u_d}{\epsilon_1, \dots,\epsilon_d}$.
The remaining of the proof is a check of linear independence of
$v_{{\bf q},1},\dots,v_{{\bf q},r_{\bf q}}$ and
$w_{{\bf q},1},\dots,w_{{\bf q},s_{\bf q}},W_{\bf q}$.
We have
\[
W_{\bf q} = \sum_{\substack{\epsilon_i=xy}} u_i - \sum_{\substack{\epsilon_{{\bf p},i}=xy}} u_{{\bf p},i}
\]
Thus $W_{\bf q}$ is $\pm u_{j({\bf p})-1}$ or $\pm u_{j({\bf p})}$. Thus the theorem is proved.
\end{proof}

Since we have
\begin{align*}
I& \varia{u_1,\dots, u_{i-1},u_i,u_{i+1},\dots,u_d}{\epsilon_1, \dots,\epsilon_{i-1},0,\epsilon_{i+1},\dots,\epsilon_d} \\
&=
\frac{1}{u_i}\left\{ I\varia{u_1,\dots, u_{i-1},u_i+u_{i+1},u_{i+2},\dots,u_d}{\epsilon_1, \dots,\epsilon_{i-1},\ \ \ \ \ \ \epsilon_{i+1},\epsilon_{i+2},\dots,\epsilon_d}
- I\varia{u_1,\dots, u_{i-2},u_{i-1}+u_{i},u_{i+1},\dots,u_d}{\epsilon_1, \dots,\epsilon_{i-2},\epsilon_{i-1}\ \ \ \ \ \ ,\epsilon_{i+1},\dots,\epsilon_d}
\right\},
\end{align*}
or equivalently
\begin{align}
\label{eq: erase0}
N&\varia{u_{1},\dots,u_{d}}{\epsilon_{1},\dots,\epsilon_{i-1},0,\epsilon_{i+1},\dots,\epsilon_{d}} \\ \notag
& \
=\frac{1}{u_{i}-u_{i-1}}\left(N\varia{u_{1},\dots,u_{i-2},u_{i},u_{i+1},\dots,u_{d}}{\quad \epsilon_{1},\dots,\epsilon_{i-1},\epsilon_{i+1},\dots,\epsilon_{d}}-N\varia{u_{1},\dots,u_{i-2},u_{i-1},u_{i+1},\dots,u_{d}}{\quad \epsilon_{1},\dots,\epsilon_{i-1},\epsilon_{i+1},\dots,\epsilon_{d}}\right)
\end{align}
with
\[
N\varia{u_{1},\dots,u_{d}}{\epsilon_{1},\dots,\epsilon_{d}}:=I\varia{u_{1},u_{2}-u_{1},\dots,u_{d}-u_{d-1}}{\epsilon_{1},\epsilon_{2},\dots,\epsilon_{d}},
\]
we could  restrict the set in \eqref{eq:epsilon set} to $\{xy, x\}$
as follows:

\begin{thm}\label{thm: rec diff eq for I}
Let $s\in\N$ and $x$ and $y$  be complex parameters. 
Let  ${\bf\epsilon}=(\epsilon_1,\dots,\epsilon_s)\in\{y,xy\}^s$
and $r\geq 0$.
Then
$
I\varia{u_1,\dots, u_r}{\ x,\dots,x}\cdot
I\varia{u_{r+1},\dots,u_{r+s}}{\quad \epsilon_1, \dots,\epsilon_s}
$
can be expressed 
as finite sum
\begin{equation}\label{eq:I dec}
I\varia{u_1,\dots, u_r}{\ x,\dots,x}\cdot
I\varia{u_{r+1},\dots,u_{r+s}}{\quad \epsilon_1, \dots,\epsilon_s}
=\sum_{{\bf j}
}
c_{\bf j}\cdot
I\varia{v_{{\bf j},1}, \dots, v_{{\bf j},r_{\bf j}}}
{\ y,\dots\dots,y}
\cdot
I\varia{w_{{\bf j},1}, \dots, w_{{\bf j},s_{\bf j}}}
{\ \tau_{{\bf j},1}, \dots, \tau_{{\bf j},s_{\bf j}}}
\end{equation}
with $c_{\bf j}\in\Q(u_1,\dots,u_d)$,
\begin{equation}\label{eq:epsilon set 2}
\tau_{{\bf j},1}, \dots, \tau_{{\bf j},s_{\bf j}}
\in\{x,xy\},
\end{equation}
and linearly independent  tuples
$(v_{{\bf j},k})_{1\leq k\leq r_{\bf j}}$,
$(w_{{\bf j},l})_{1\leq k\leq s_{\bf j}}$
with $r_{\bf j}+s_{\bf j} \leq d$
which are $\Z$-linear combinations of $u_1,\dots, u_d$ for each $\bf j$.
Furthermore, the map
$$M\in\mathcal{M}_2(\mathcal{F}_{\ser};\{x\},\{xy,y\})
\mapsto N\in\mathcal{M}_2(\mathcal{F}_{\ser};\{y\},\{xy,x\}),
$$
defined by
\[
N
\varia{u_{1},\dots,u_{r};\ u_{r+1},\dots,u_{r+s}}{ x,\dots\dots,x; \
\epsilon_{1},\dots,\epsilon_{s}}
:=\sum_{\bf j} c_{\bf j}
 M\varia{v_{{\bf j},1}, \dots, v_{{\bf j},{r_{\bf j}}} ; \ w_{{\bf j},1}, \dots,  v_{{\bf j},s_{\bf j}}}
 {\quad y,\dots\dots,y  ; \quad \tau_{{\bf j},1},\dots,\tau_{{\bf j},s_{\bf j}}},
\]
is mould proper.
\end{thm}

\begin{proof}
It follows from Theorem \ref{thm: rec diff eq for I pre}, 
the shuffle product of iterated integrals,
and the repeated applications of \eqref{eq: erase0}.
\end{proof}

There are slight differences between the above theorem and Theorem \ref{thm: rec diff eq for I pre}.
The condition \eqref{eq:epsilon set 2} is stricter than \eqref{eq:epsilon set}
and $c_j$ is in  $\Q(u_1,\dots,u_d)$ in the theorem while $c_j\in\Z$ in
Theorem \ref{thm: rec diff eq for I pre}.

\begin{eg}
The following functional relations are obtained from \eqref{eq:I dec}
by \eqref{eq: I= sum Li}
$$
\Li_{n_1,n_2}({x}^{-1},{y}^{-1})=
\Li_{n_1}({x}^{-1})\Li_{n_2}({y}^{-1})
-\Li_{n_1+n_2}({x}^{-1}{y}^{-1})
-\Li_{n_2,n_1}({y}^{-1},{x}^{-1}).
$$
\end{eg}

\begin{eg}\label{eg: cj_const}

The depth $\leq1$ case of Theorem \ref{thm: rec diff eq for I pre} is as follows.
\[
I\varia{\emptyset}{\emptyset}I\varia{\emptyset}{\emptyset}=I\varia{\emptyset}{\emptyset}I\varia{\emptyset}{\emptyset},
\]
\[
I\varia u{xy}I\varia{\emptyset}{\emptyset}=I\varia u{xy}I\varia{\emptyset}{\emptyset},
\]
\[
I\varia uyI\varia{\emptyset}{\emptyset}=I\varia{\emptyset}{\emptyset}I\varia uy,
\]
\[
I\varia{\emptyset}{\emptyset}I\varia ux=I\varia uxI\varia{\emptyset}{\emptyset}.
\]
\end{eg}

Let ${\mathcal O}={\mathcal O}({\mathbb D}^2)$ be  the $\C$-algebra of
complex analytic functions on the 2-dimensional disk ${\mathbb D}^2\subset \C^2$
with respect to the variables $(x,y)\in\mathbb D^2$.

\begin{defn}\label{def:Zag and Zag}
We denote by ${}_\mathcal{O}\mathcal{F}_{\ser}$ the family of functions
defined by
$${}_\mathcal{O}F_{\mathcal{F}_{\ser}}(V)=\mathcal{O}\otimes F_{\mathcal{F}_{\ser}}(V).$$
Let $w=x$ or $y$. Put
\begin{align*}
&
\Zag^{w}\varia{u_1,\dots, u_d}{\ 1, \dots,1}
:=I\varia{u_1,\dots,u_d}{{w^{-1},\dots, w^{-1}}},\\
&\Zag^{xy,w}\varia{u_1,\dots, u_d}{\sigma_1, \dots,\sigma_d}
:=I\varia{\qquad u_1,\dots, u_d}{{
\epsilon(\sigma_1)^{-1}, \dots,\epsilon(\sigma_d)^{-1}}}
\qquad (\sigma_1,\dots,\sigma_d\in [2]:=\{1,2\}) ,
\end{align*}
where 
we put
$$\epsilon(\sigma)=
\begin{cases}w  &(\sigma=1), \\
xy & (\sigma=2).
\end{cases}$$
We regard
$$
\Zag^{w}\in\mathcal{M}({}_\mathcal{O}{\mathcal F}_{\ser};[1])
\quad \text{  and  }\quad
\Zag^{xy,w}\in\mathcal{M}({}_\mathcal{O}{\mathcal F}_{\ser};[2]), 
$$
by putting their constant parts to be $1$.
By taking their tensor product, we also have
 $$
 \Zag^{x}\otimes\Zag^{xy,y}
 \quad\text{ and }\quad
  \Zag^{y}\otimes\Zag^{xy,x}
 \in\mathcal P_4({}_\mathcal{O}\mathcal{F}_{\ser})
 $$
with $\mathcal P_4({}_\mathcal{O}\mathcal{F}_{\ser})=
  \mathcal{M}_2({}_\mathcal{O}\mathcal{F}_{\ser};[1],[2])$
  given in \eqref{eq: P n+2 F}.
\end{defn}

Similarly to Example \ref{eg: Zig and Zag}, by \eqref{eq: I= sum Li},
we have expansions
\begin{align*}
\Zag^{w}\varia{u_1,\dots, u_d}{\ 1, \dots,1}
=\sum_{n_1,\dots, n_d>0}&
\Li_{n_1,\dots,n_d}(1,\dots,1,w)
 \\
&
\cdot u_1^{n_1-1}(u_1+u_2)^{n_2-1}\cdots
(u_1+\cdots+u_d)^{n_d-1},\\
\Zag^{xy,w}\varia{u_1,\dots, u_d}{\sigma_1, \dots,\sigma_d}
=\sum_{n_1,\dots, n_d>0}&
\Li_{n_1,\dots,n_d}\left(\frac{\epsilon(\sigma_1)}{\epsilon(\sigma_2)},\dots,\frac{\epsilon(\sigma_{d-1})}{\epsilon(\sigma_{d})},{\epsilon(\sigma_d)}\right) \\
&\cdot u_1^{n_1-1}(u_1+u_2)^{n_2-1}\cdots
(u_1+\cdots+u_d)^{n_d-1}
\end{align*}
for
$\sigma_1,\dots,\sigma_d\in [2]:=\{1,2\}$
by \eqref{eq: I= sum Li}.


\begin{prop}\label{prop: group-likeness for Zag}
  We have
  \begin{align*}
  \shmap(\Zag^{xy,x})&=\Zag^{xy,x}\otimes\Zag^{xy,x}, \quad
  \shmap(\Zag^{x})=\Zag^{x}\otimes\Zag^{x}, \\
  \shmap(\Zag^{y})&=\Zag^{y}\otimes\Zag^{y}, \qquad\quad
  \shmap(\Zag^{xy,y})=\Zag^{xy,y}\otimes\Zag^{xy,y}.
  \end{align*}
\end{prop}

\begin{proof}
  They are consequence of  (\ref{eq:power series I}).
\end{proof}

\begin{lem}\label{lem: f(Zag)=0 then f=0}
Let $f:\mathcal{P}_{4}(\mathcal{F}_{\ser})\to\mathcal{F}_{\ser,m}$
be a mould-proper map such that
\[
f_{\mathcal{O}}(\Zag^{y}\otimes\Zag^{xy,x})=0
\]
where $f_{\mathcal{O}}:\mathcal{P}_{4}({}_{\mathcal{O}}\mathcal{F}_{\ser})\to _{\mathcal{O}}\mathcal{F}_{\ser,m}$ is the $\mathcal{O}$-linear extension of
$f$. Then $f=0$.
\end{lem}

\begin{proof}
Since the functions
\[
I(0;y^{-1},{\bf 0}^{k_{1}},\dots,y^{-1},{\bf 0}^{k_{r}};1)\cdot I(0;\epsilon_{1}^{-1},{\bf 0}^{l_{1}}\dots,\epsilon_{r}^{-1},{\bf 0}^{l_{s}};1)
\]
with $r,s,k_{i},l_{i}\in\Z_{\geq0}$ and $\epsilon_{i}\in\{x,xy\}$,
are of the form
\[
\Li_{k_{1}+1,\dots,k_{r}+1}(1,\dots,1,y)\cdot\Li_{l_{1}+1,\dots,l_{s}+1}(\frac{\epsilon_{1}}{\epsilon_{2}},\dots,\frac{\epsilon_{s-1}}{\epsilon_{s}},\epsilon_{s}),
\]
they are linearly independent over $\Q$. Thus for each $M\in\mathcal{P}_{4}({}_{\mathcal{O}}\mathcal{F}_{\ser})$,
there exists a $\Q$-linear map (not uniquely determined) $\theta_{M}:\mathcal{O}\to\Q$
which induces
\[
\Theta_{M}:\mathcal{P}_{4}({}_{\mathcal{O}}\mathcal{F}_{\ser})\to\mathcal{P}_{4}(\mathcal{F}_{\ser})
\]
such that
\[
\Theta_{M}(\Zag^{y}{\otimes}\Zag^{xy,x})=M.
\]
Then the following diagram is commutative
\[
\begin{CD}\mathcal{P}_{4}({}_{\mathcal{O}}\mathcal{F}_{\ser})@>\Theta_{M}>>\mathcal{P}_{4}(\mathcal{F}_{\ser})\\
@V{f_{\mathcal{O}}}VV@VV{f}V\\
_{\mathcal{O}}\mathcal{F}_{\ser,m}@>\Theta_{M}>>\mathcal{F}_{\ser,m}
\end{CD}
\]
because $\Theta_{M}$ is induced from $\theta_{M}$. Thus for all
$M\in\mathcal{P}_{4}(\mathcal{F}_{\ser})$, we have
\[
f(M)=f(\Theta_{M}(\Zag^{y}{\otimes}\Zag^{xy,x}))=\Theta_{M}(f_{\mathcal{O}}(\Zag^{y}{\otimes}\Zag^{xy,x}))=0.
\]
\end{proof}

\begin{thm-defn}\label{thm-defn: dec}
There exists a unique $\Q$-linear mould-proper automorphism
\begin{equation*}
\bal:\mathcal P_4(\mathcal{F}_{\ser})
\to
\mathcal P_4(\mathcal{F}_{\ser})
\end{equation*}
such that
\begin{equation*}
 {
 {\bal_{\mathcal O}(\minus(\Zag^{y})\otimes \minus(\Zag^{xy,x}))
 =\minus(\Zag^x)\otimes \minus(\Zag^{xy,y})
 }
 }
\end{equation*}
 Here $\minus$ is given by  \eqref{eq:neg} and
$\bal_{\mathcal O}$ is the $\mathcal O$-linear extension of $\bal$
to $\mathcal P_4({}_\mathcal{O}\mathcal{F}_{\ser})$.
\end{thm-defn}

By
the unique prolongation theorem (Theorem \ref{thm:UniqueProlongation})
and Theorem-Definition \ref{thm-defn: dec}, for a divisible family $\mathcal{F}$  of functions, we have
\begin{equation}\label{eq:bal P4F}
\bal:\mathcal P_4(\mathcal{F})
\to
\mathcal P_4(\mathcal{F}).
\end{equation}

\begin{lem}
The automorphism $\bal$ is an involution.
\end{lem}

\begin{proof}
It is immediate by the symmetry of $x$ and $y$.
\end{proof}

\begin{rem}\label{rem: existence of bal}
We note that  the existence of $\bal$ is a consequence of Theorem \ref{thm: rec diff eq for I}.
More precisely $\bal$ is described by
$$
\bal(M)
\varia{{-u_{1},\dots,-u_{r};-\ u_{r+1},\dots,-u_{r+s}}}{ 1,\dots\dots,1 \ ; \
\sigma(\epsilon_{1}),\dots,\sigma(\epsilon_{s})}
=\sum_{\bf j} c_{\bf j}
 M\varia{{-v_{{\bf j},1}, \dots, -v_{{\bf j},{r_{\bf j}}} ; \ -w_{{\bf j},1}, \dots,  -w_{{\bf j},s_{\bf j}}}}
 {\quad 1,\dots\dots,1  ; \ \sigma(\tau_{{\bf j},1}),\dots,\sigma(\tau_{{\bf j},s_{\bf j}})},
$$
that is,
\begin{equation}\label{eq:decM}
\minus(\bal(M))
\varia{{u_{1},\dots,u_{r};\ u_{r+1},\dots,u_{r+s}}}{ 1,\dots\dots,1 \ ; \
\sigma(\epsilon_{1}),\dots,\sigma(\epsilon_{s})}
=\sum_{\bf j} c_{\bf j} \
\minus( M)\varia{{v_{{\bf j},1}, \dots, v_{{\bf j},{r_{\bf j}}} ; \ w_{{\bf j},1}, \dots,  w_{{\bf j},s_{\bf j}}}}
 {\quad 1,\dots\dots,1  ; \ \sigma(\tau_{{\bf j},1}),\dots,\sigma(\tau_{{\bf j},s_{\bf j}})},
\end{equation}
where 
we employ the symbols in Theorem \ref{thm: rec diff eq for I}
and put
\begin{equation}\label{eq: 1-2 rule}
\sigma(\epsilon):=
\begin{cases}
1 & (\epsilon=y), \\
2 & (\epsilon=xy),
\end{cases}
\qquad
\sigma(\tau):=
\begin{cases}
1 & (\tau=x), \\
2 & (\tau=xy).
\end{cases}
\end{equation}
\end{rem}

{\it Proof of Theorem-Definition \ref{thm-defn: dec}. }
The existence of such mould proper map is assured by Remark \ref{rem: existence of bal}.
The uniqueness follows from Lemma \ref{lem: f(Zag)=0 then f=0}.
\qed

			\medskip

\begin{eg}\label{eg: val_dep1}
  By Example \ref{eg: cj_const}, we have
\begin{align*}
\bal(M)\varia{\emptyset;\emptyset}{\emptyset;\emptyset} & =M\varia{\emptyset;\emptyset}{\emptyset;\emptyset},\\
\bal(M)\varia{u;\emptyset}{1;\emptyset} & =M\varia{u;\emptyset}{1;\emptyset},\\
\bal(M)\varia{u;\emptyset}{2;\emptyset} & =M\varia{\emptyset;u}{\emptyset;1},\\
\bal(M)\varia{\emptyset;u}{\emptyset;1} & =M\varia{u;\emptyset}{2;\emptyset}.
\end{align*}
\end{eg}

\begin{defn}\label{defn: Xi and rev}
Let $\rev$ be the involution of $\widehat{U{\mathfrak{t}}_{4}}^{\dag}$ induced by
the involution of
$\widehat{U{\mathfrak{f}}_{2}}^{\dag}\widehat{\otimes}\widehat{U{\mathfrak{f}}_{3}}^{\dag}$
 given by
 \[
w_{2}(f_{0},f_{1})\otimes w_{3}(f_{0},f_{1},f_{2})\mapsto
w_{2}(-f_{0}-f_{1},f_{1})\otimes  w_{3}(-f_{0}-f_{1}-f_{2},f_{2},f_{1}),
\]
under the map $\dec^{-1}$ given in  \eqref{eq:dec isom for dag}.
\end{defn}

By definition,
we have
\begin{equation}\label{eq:w3w2=Xiw2w3}
\rev(w_{2}(t_{02},t_{12})w_{3}(t_{03},t_{13},t_{23}))=
w_{2}(-t_{02}-t_{12},t_{12})w_{3}(-t_{03}-t_{13}-t_{23},t_{23},t_{13}).
\end{equation}

Under the sequence of isomorphisms
\begin{equation}\label{eq:seq of isom: UT4 and M2F}
\madec_4:\widehat{U{\mathfrak t}_4}^\dag\quad
\overset{\dec}{\simeq}\quad
\widehat{U{\mathfrak f}_2}^\dag
\widehat\otimes
\widehat{U{\mathfrak f}_3}^\dag\quad
\overset{\ma_{[1]}\otimes \ma_{[2]}}{\simeq}
\quad
\mathcal{M}_2(\mathcal{F}_\ser;[1],[2])=
\mathcal P_4(\mathcal F_\ser),
\end{equation}
we transmit the involutions $\flip$ on
$\widehat{U{\mathfrak t}_4}^\dag$ in \eqref{eq: flip}
and $\rev$  on $\widehat{U{\mathfrak f}_2}^\dag
\widehat\otimes
\widehat{U{\mathfrak f}_3}^\dag$
to the ones on $\mathcal P_4(\mathcal F_\ser)$
and denote their $\mathcal O$-linear extensions
by the same symbols $\flip$ and $\rev$.
We note that $\flip$ and $\rev$ are mould proper maps
and they are not commutative.
We add that the involutions
$\minus$ defined in \eqref{eq:neg} on
$\mathcal{M}_2(\mathcal{F}_\ser;[1])$ and
$\mathcal{M}_2(\mathcal{F}_\ser;[2])$
induce an involution
on $\mathcal P_4(\mathcal F_\ser)$, hence
on $\widehat{U{\mathfrak t}_4}^\dag$,
which again we denote the same symbol $\minus$.

\begin{prop}\label{prop: bal= flip}
Under the identification of isomorphisms \eqref{eq:seq of isom: UT4 and M2F},
we have the following commutative diagram
$$
\xymatrix{
\mathcal P_4(\mathcal F_\ser)
\ar@{->}[d]_{\rev}  \ar@{->}[rr]^\bal
&& \mathcal P_4(\mathcal F_\ser)\ar@{->}[d]^{\rev}
\\
\mathcal P_4(\mathcal F_\ser)
\ar@{->}[rr]_{\flip}&&
\mathcal P_4(\mathcal F_\ser).
}
$$
\end{prop}

\begin{proof}
By Theorem-Definition \ref{thm-defn: dec},
it is enough to prove the following equality
holds in 
$\mathcal P_4({}_{\mathcal O}\mathcal F_\ser)$
\begin{equation}\label{eq: flip Zag}
\rev\circ\flip\circ\rev \ \left(\minus(\Zag^{y}){\otimes} \minus(\Zag^{xy,x})\right)=
\minus(\Zag^{x})\otimes\minus(\Zag^{xy,y}).
\end{equation}
We consider the fundamental solutions
 $G_0(f_0,f_1)(y)$ and $G_y(f_0,f_1,f_2)(x)$ of
the fundamental solutions of the following KZ-type differential equations
$$
\frac{dG_0}{dy}=\left\{\frac{f_0}{y}+\frac{f_1}{y-1}\right\}G_0(y)
$$
and
$$
\frac{dG_y}{dx}=\left\{\frac{f_0}{x}+\frac{f_1}{x-1}+\frac{f_2}{x-y^{-1}}\right\}G_y(x)
$$
(with $y$ fixed)
with the prescribed approximation condition
$G_0(f_0,f_1)(y)\approx y^{f_0}$ ($y\to 0$ in $\R_+$)
and
$G_y(f_0,f_1,f_2)(x)\approx x^{f_0}$  ($x\to 0$ in $\R_+$)  respectively.
In the same way, we define
$G_0(f_0,f_1)(x)$ and $G_x(f_0,f_1,f_2)(y)$.

Fix small $x$ and $y$ in $\C$.
We have
\begin{align*}
G_0(f_0,f_1)(y)=1+\sum_{d=1}^\infty\sum_{(n_1,\dots,n_d)\in\N^d}& (-1)^d\Li_{n_1,\dots,n_d}(1,\dots,1,y) \\
\cdot f_0^{n_d-1}f_1\cdots f_0^{n_1-1}f_1
&+\text{(the terms ending on $f_0$),} \\
G_y(f_0,f_1,f_2)(x)=1+
\sum_{d=1}^\infty\sum_{\substack{(n_1,\dots,n_d)\in\N^d \\ (\sigma_1,\dots,\sigma_d)\in [2 ]^d}}&
(-1)^d\Li_{n_1,\dots,n_d}\left(\frac{\epsilon(\sigma_1)}{\epsilon(\sigma_2)},\dots,\frac{\epsilon(\sigma_{d-1})}{\epsilon(\sigma_{d})},{\epsilon(\sigma_d)}\right) \\
\cdot f_0^{n_d-1}f_{\sigma_d}\cdots f_0^{n_1-1}f_{\sigma_1} &
+\text{(the terms ending on $f_0$)},
\end{align*}
where the coefficients of 'the terms ending on $f_0$' can be calculated by
regularization method (cf. \cite{IKZ, G-MPL, R}).
Since $G_0(f_0,f_1)(y)$ and $G_y(f_0,f_1,f_2)(x)$ are group-like,
we have
\begin{align*}
&G_0(f_0,f_1)(y)^{-1}=1+\sum_{d=1}^\infty\sum_{(n_1,\dots,n_d)\in\N^d} (-1)^{n_1+\cdots+n_d+d}\Li_{n_1,\dots,n_d}(1,\dots,1,y) \\
&\qquad\qquad \cdot f_1f_0^{n_1-1}f_1\cdots f_1f_0^{n_d-1}
+\text{(the terms starting from $f_0$),} \\
&G_y(f_0,f_1,f_2)(x)^{-1} \\
&\qquad =1+
\sum_{d=1}^\infty \sum_{\substack{(n_1,\dots,n_d)\in\N^d \\
(\sigma_1,\dots,\sigma_d)\in [2 ]^d}}
(-1)^{n_1+\cdots+n_d+d}
\Li_{n_1,\dots,n_d}\left(\frac{\epsilon(\sigma_1)}{\epsilon(\sigma_2)},\dots,\frac{\epsilon(\sigma_{d-1})}{\epsilon(\sigma_{d})},{\epsilon(\sigma_d)}\right) \\
&\qquad\qquad \cdot f_{\sigma_1}f_0^{n_1-1}\cdots f_{\sigma_d}f_0^{n_d-1}
+\text{(the terms starting from $f_0$)},
\end{align*}

We have
$$
y^{f_0}G_0(f_0,f_1)(y)^{-1}\in \widehat{U{\mathfrak f}_2}^\dag, \qquad 
x^{f_0}G_y(f_0,f_1,f_2)(x)^{-1}\in\widehat{U{\mathfrak f}_3}^\dag.
$$
By Definition \ref{def:ma}, Definition \ref{def:Zag and Zag} and \cite{FK} Lemma 2.5,
we have
\begin{align*}
&\ma_{[1]}\left(y^{f_0}G_0(f_0,f_1)(y)^{-1}\right)
=\minus(\Zag^{y}), \\
&\ma_{[2]}\left(x^{f_0}G_y(f_0,f_1,f_2)(x)^{-1}\right)
=\minus(\Zag^{xy,x}).
\end{align*}

Whence we have
\begin{align*}
(\ma_{[1]}\otimes \ma_{[2]})^{-1}&
\left(\minus(\Zag^{y}){\otimes}\minus(\Zag^{xy,x})\right) \\
&=
y^{f_0} G_0(f_0,f_1)(y)^{-1}\otimes x^{f_0} G_y(f_0,f_1,f_2)(x)^{-1}.
\end{align*}
Similarly we have
\begin{align*}
(\ma_{[1]}\otimes \ma_{[2]})^{-1}&
\left(\minus(\Zag^{x}){\otimes}\minus( \Zag^{xy,y})\right) \\
&=
x^{f_0} G_0(f_0,f_1)(x)^{-1}
\otimes y^{f_0} G_x(f_0,f_1,f_2)(y)^{-1}.
\end{align*}
Whence \eqref{eq: flip Zag} is reduced to
\begin{align}\label{eq:revfliprev}
\rev \circ \flip \ \circ & \rev\circ\dec^{-1} \left( y^{f_{0}}G_0(f_0,f_1)(y)^{-1}\otimes x^{f_{0}}G_y(f_0,f_1,f_2)(x)^{-1}\right)  \\
&= \dec^{-1} \left(x^{f_{0}}G_0(f_0,f_1)(x)^{-1}\otimes y^{f_{0}} G_x(f_0,f_1,f_2)(y)^{-1}\right).\notag
\end{align}
By \eqref{eq:w3w2=Xiw2w3}, we have
\begin{align} \label{eq:flip rev dec-1=xyG0-1Gy-1}
&\flip  \circ\rev\circ (\dec)^{-1}
 \left( y^{f_{0}}G_0(f_0,f_1)(y)^{-1}\otimes x^{f_{0}}G_y(f_0,f_1,f_2)(x)^{-1}
\right) \\ \notag
&
=y^{(t_{02}+t_{12})}G_0(t_{02}+t_{12},t_{32})(y)^{-1}\cdot x^{t_{01}}G_y(t_{01},t_{21},t_{31})(x)^{-1} \\  \notag
&
=y^{(t_{01}+t_{02}+t_{12})}G_0(t_{01}+t_{02}+t_{12},t_{32})(y)^{-1}\cdot x^{t_{01}}G_y(t_{01},t_{21},t_{31})(x)^{-1} \\ \notag
&
={x^{t_{01}}}\cdot y^{(t_{01}+t_{02}+t_{12})}\cdot G_0(t_{01}+t_{02}+t_{12},t_{32})(y)^{-1}\cdot G_y(t_{01},t_{21},t_{31})(x)^{-1}. 
\end{align}

We consider the fundamental solution  $W(x,y)=W(x,y)(\{t_{ij}\})$
of the following differential equation
\begin{align*}
dW(x,y)=\bigl\{d\log(xy)t_{01}&+d\log(y)t_{02}+d\log(xy-y)t_{12} \\
&+d\log(1-xy)t_{13}+d\log(y-1)t_{23}
\bigr\}\cdot W(x,y)
\end{align*}
appearing in \cite[p.347]{F11}
with the prescribed approximation condition
$W(x,y)\approx x^{t_{01}}y^{t_{01}+t_{02}+t_{12}}$ ($x,y\to 0$ in $\R_+$)
over
the space
\begin{equation*}
\mathcal M_{0,5}:=
\{(x,y)\in\mathbb G_m^2 \bigm| x,y\neq 0,1,xy\neq 1\}.
\end{equation*}
By \cite[Proposition 12.2]{OU},
we have
\begin{equation}\label{eq:W=Gy G0}
W(x,y)=
G_y(t_{01},t_{21},t_{31})(x)\cdot G_0(t_{01}+t_{02}+t_{12},t_{32})(y).
\end{equation}

By \eqref{eq:flip rev dec-1=xyG0-1Gy-1} and \eqref{eq:W=Gy G0},
we obtain
\begin{align}\label{eq:flip rev dec-1=W-1}
\flip\circ\rev\circ (\dec)^{-1} &
 \left( y^{f_{0}}G_0(f_0,f_1)(y)^{-1}\otimes x^{f_{0}}G_y(f_0,f_1,f_2)(x)^{-1}\right) \\ \notag
&\qquad\qquad
=\left( W(x,y)(\{t_{ij}\})\cdot x^{-t_{01}}y^{-(t_{01}+t_{02}+t_{12})}\right)^{-1} .
\end{align}

%
By definition, we have
$$
\flip\left(W(x,y)(\{t_{ij}\})\right)=
W(y,x)(\{t_{ij}\})\cdot(xy)^{-c_4}
$$
where $c_4$ is the center $\sum_{i<j} t_{ij}$
given in \eqref{eq: center of braid Lie algebra}.
So we have
\begin{equation}\label{eq: flip W(x,y)=W(y,x)}
\flip\left(
W(x,y)(\{t_{ij}\})\cdot x^{-t_{01}}y^{-(t_{01}+t_{02}+t_{12})}\right)
=W(y,x)(\{t_{ij}\})\cdot y^{-t_{01}}x^{-(t_{01}+t_{02}+t_{12})}.
\end{equation}
Whence by \eqref{eq:flip rev dec-1=W-1} and \eqref{eq: flip W(x,y)=W(y,x)},
we have
\begin{align*}
\rev & \circ (\dec)^{-1}
\left( y^{f_{0}}G_0(f_0,f_1)(y)^{-1}\otimes x^{f_{0}}G_y(f_0,f_1,f_2)(x)^{-1}
\right) \\
&=\flip\circ\flip\circ\rev\circ (\dec)^{-1}
\left( y^{f_{0}}G_0(f_0,f_1)(y)^{-1}\otimes x^{f_{0}}G_y(f_0,f_1,f_2)(x)^{-1}
\right) \\
&=\flip\left(
W(x,y)(\{t_{ij}\})\cdot x^{-t_{01}}y^{-(t_{01}+t_{02}+t_{12})} \right)^{-1} \\
&=\left(
W(y,x)(\{t_{ij}\})\cdot y^{-t_{01}}x^{-(t_{01}+t_{02}+t_{12})}
\right)^{-1}.
\end{align*}
While similarly to \eqref{eq:flip rev dec-1=W-1}, we have
\begin{align*}
\flip &\circ\rev  \circ (\dec)^{-1}
\left( x^{f_{0}}G_0(f_0,f_1)(x)^{-1} \otimes y^{f_{0}}G_x(f_0,f_1,f_2)(y)^{-1}
\right) \\
&=\left(W(y,x)(\{t_{ij}\})\cdot y^{-t_{01}}x^{-(t_{01}+t_{02}+t_{12})}
\right)^{-1}.
\end{align*}
Thus \eqref{eq:revfliprev}, and hence \eqref{eq: flip Zag}, is obtained.
\end{proof}

\subsection{
$\GARI(\mathcal{F})_{\as+\abal}$ and $\ASTR$}
\label{subsec:def and prop of GARI as pent}
We introduce the sets $\GARI(\mathcal{F})_{\as+\abal}$,
$\GARI(\mathcal{F})_{\underline{\as+\abal}}$
and $\ARI(\mathcal{F})_{\underline{\al+\abal}}$, in Definition \ref{defn: GARI as+bal}
by using the map $\bal$ introduced  in \eqref{eq:bal P4F}.
We also introduce 
the sets
$\GARI(\mathcal{F})_{\as+\pent}$,
$\GARI(\mathcal{F})_{\underline{\as+\pent}}$ 
and $\ARI(\mathcal{F})_{\underline{\al+\pent}}$ 
using the pentagon equation in Definition  \ref{defn: GARI as+pent}.
By proving in Theorem \ref{pentagonal characterization theorem} that
$\GARI(\mathcal{F})_{\as+\abal} = \GARI(\mathcal{F})_{\as+\pent}$, 
$\GARI(\mathcal{F})_{\underline{\as+\abal}} = \GARI(\mathcal{F})_{\underline{\as+\pent}}$, 
$\ARI(\mathcal{F})_{\underline{\al+\abal}} = \ARI(\mathcal{F})_{\underline{\al+\pent}}$,
we show that, 
when ${\mathcal F}={\mathcal F}_\ser$,
they recover $\ASTR$, $\GRT_1$, and $\grt_1$, respectively, as stated in Theorem \ref{thm: recovery theorem of ASTR and GRT}.

\begin{defn}\label{def: well-balanced}
(1). Let $\mathcal{F}$ be a divisible family of functions.
A mould $M\in \mathcal M(\mathcal F)$ is called {\it well-balanced}
when
$$
\bal (M_{[1]}\otimes M_{[2]})=M_{[1]}\otimes M_{[2]}
$$
holds in $\mathcal P_4(\mathcal F)$.
It is called {\it balanced} when
\begin{equation}\label{eq:almost-balanced condition}
\bal (M_{[1]}\otimes M_{[2]})=M_{[1]}\otimes (M_{[2]}\times C)
\end{equation}
holds  with a constant mould  $C\in\mathcal{M}({{\mathcal F}};[2])$.

(2).
A mould $M\in \mathcal M(\mathcal F)$ is called {\it linearly balanced}
when
\begin{equation}\label{eq: linearly almost-balanced condition}
\bal (M_{[1]}\otimes 1 +1 \otimes M_{[2]})=M_{[1]}\otimes 1 +1\otimes (M_{[2]}+ C)
\end{equation}
holds  with a constant mould  $C\in\mathcal{M}({{\mathcal F}};[2])$.
\end{defn}

\begin{lem}\label{lem: maC is commutator group-like}
(1).
Assume that $M\in\GARI(\mathcal{F})_{\as}$ is balanced as \eqref{eq:almost-balanced condition},
Then, the constant mould $C$ is symmetral and $C\varia u1=C\varia u2=0$, i.e., $\ma_{[2]}^{-1}(C)\in\mathbb{Q}\langle\langle f_{1},f_{2}\rangle\rangle\subset\widehat{U\mathfrak{f}_{3}}$
is commutator group-like.

(2). Assume that $M\in\ARI(\mathcal{F})_{\al}$ is linearly balanced as \eqref{eq: linearly almost-balanced condition}.
Then, the constant mould $C$ is alternal and $C\varia u1=C\varia u2=0$, i.e., $\ma_{[2]}^{-1}(C)\in\mathbb{Q}\langle\langle f_{1},f_{2}\rangle\rangle\subset\widehat{U\mathfrak{f}_{3}}$
is commutator Lie-like.
\end{lem}

\begin{proof}
(1).
The symmetrality follows from the symmetrality of $\bal(M_{[2]}\otimes M_{[1]})$
and $M_{[2]}\otimes M_{[1]}$.
The condition $C\varia u1=C\varia u2=0$
follows from the direct calculation using Example \ref{eg: val_dep1}.

(2). The proof proceeds in the same way as the proof of (1).
\end{proof}

\begin{defn}\label{defn: GARI as+bal}
Let $\mathcal{F}$ be a divisible family of functions.
We consider the following two subsets
$\GARI(\mathcal{F})_{\as+\abal}$ ($\supset$)
$\GARI(\mathcal{F})_{\underline{\as+\abal}}$ of $\GARI(\mathcal{F})_\as$ (cf. Definition \ref{def:al-il-as-is}):
\begin{align*}
\GARI(\mathcal{F})_{\as+\abal}
&:=\left\{M\in\GARI(\mathcal{F})_{\as}\ \middle|\
M \text{ is balanced}
\right\}, \\
\GARI(\mathcal{F})_{\underline{\as+\abal}}
&:=\{M\in\GARI(\mathcal{F})_{\as+\abal}\bigm| M(x_1) \text{ is even }
\},
\end{align*}
and the following subset
$\ARI(\mathcal{F})_{\underline{\al+\abal}}$ of $\ARI(\mathcal{F})_{\al}$ 
(cf. Definition \ref{def:al-il-as-is}):
$$
\ARI(\mathcal{F})_{\underline{\al+\abal}}
:=\left\{M\in\ARI(\mathcal{F})_{\al}\ \middle|\
M \text{ is  linearly balanced and } M(x_1) \text{ is even}
\right\}.
$$
\end{defn}

\begin{rem}
We note that
another (more pentagon-like) characterization of $\GARI(\mathcal{F})_{\as+\abal}$ will be given in
Theorem \ref{pentagonal characterization theorem}.
\end{rem}

For $\psi\in\widehat{U\mathfrak{f}_{2}}$
and disjoint subsets $S_{1},S_{2},S_{3}$ of $\{1,\dots,n-1\}$,
we write
$$\psi_{\mathcal P}^{S_{1},S_{2},S_{3}}\in {\mathcal P}_{n}(\mathcal{F})$$
for $\madec_{n}(\psi^{S_{1},S_{2},S_{3}}) $
(cf. Definition \ref{defn:polymoulds of pure braid type}).

\begin{lem}\label{lem: calc of as+bal}
For $M\in\mathcal{M}(\mathcal{F})$ and a constant mould
$$
C=\ma_{[2]}(\psi(-f_{1}-f_{2}, f_1))
$$
with a commutator group-like series $\psi\in\widehat{U\mathfrak{f}_{2}}$, we have
\begin{align*}
\bal(M_{[1]}\otimes M_{[2]}) & =
\rev  \ (M^{01,2,3}\mulp M^{0,1,23}),\\
M_{[1]}\otimes (M_{[2]}\times C)
& =
\rev  \ (M^{0,1,2}\mulp M^{0,12,3}\mulp \psi_{\mathcal P}^{1,2,3} )
\end{align*}
\end{lem}

\begin{proof}
It is enough to show the equality for the case $M=\ma(\varphi)$ for
$\varphi\in\widehat{U{\mathfrak{f}}_{2}}^{\dag}$. Then
\begin{align*}
M_{[1]}\otimes & (M_{[2]}\times C)
=(\ma_{[1]}\otimes\ma_{[2]})
(\varphi(f_{0},f_{1})\otimes \varphi(f_{0},f_{1}+f_{2})\psi(-f_{1}-f_2,f_1)
).\\
\intertext{Under the identification \eqref{eq:seq of isom: UT4 and M2F}, we have}
& =\varphi(t_{02},t_{12})\varphi(t_{03},t_{13}+t_{23})\psi(-t_{13}-t_{23},t_{13}) \\
 & = \rev (\varphi(-t_{02}-t_{12},t_{12})\varphi(t_{43},t_{13}+t_{23})\psi(-t_{13}-t_{23},t_{23}) )\\
 &= \rev (\varphi(t_{01},t_{12})\varphi(t_{01}+t_{02},t_{13}+t_{23})\psi(t_{12},t_{23})) \\
 & = \rev  \ (\varphi^{0,1,2}\varphi^{0,12,3}\psi^{1,2,3} )\\
 & = \rev  \ (M^{0,1,2}\mulp M^{0,12,3}\mulp \psi_{\mathcal P}^{1,2,3} ),
\end{align*}
with $t_{4i}=t_{i4}:=-\sum_{0\leq k<4, k\neq i}t_{4k}$.
And we have
\begin{align*}
\bal
(M_{[1]}& \otimes M_{[2]})
=
\bal
\circ(\ma_{[1]}\otimes\ma_{[2]})
(\varphi(f_{0},f_{1})\otimes \varphi(f_{0},f_{1}+f_{2})). \\
\intertext{By the above computations,
we have}
& = {\bal\circ \rev (\varphi(t_{01},t_{12})\varphi(t_{43},t_{13}+t_{23}))}\\
& = \rev\circ\flip (\varphi(t_{01},t_{12})\varphi(t_{43},t_{13}+t_{23})) \\
& = \rev  \ (\varphi(t_{43},t_{32})\varphi(t_{01},t_{31}+t_{21}))\\
& = \rev  \ (\varphi(t_{02}+t_{12},t_{32})\varphi(t_{01},t_{31}+t_{21}))\\
& = \rev  \ (\varphi^{01,2,3}\varphi^{0,1,23})\\
& = \rev  \ (M^{01,2,3}\mulp M^{0,1,23}).
\end{align*}
\end{proof}

\begin{lem}\label{lem: calc of al+bal}
For $M\in\ARI(\mathcal{F})$ and a constant mould
$$
C=\ma_{[2]}(\psi(-f_{1}-f_{2}, f_1))
$$
with a commutator Lie-like series $\psi\in\widehat{U\mathfrak{f}_{2}}$, we have
\begin{align*}
\bal(M_{[1]}\otimes1 +1\otimes M_{[2]}) & =
\rev  \ (M^{01,2,3}+ M^{0,1,23}),\\
M_{[1]}\otimes 1 +1\otimes  (M_{[2]}+ C)
& =
\rev  \ (M^{0,1,2}+ M^{0,12,3}+ \psi_{\mathcal P}^{1,2,3} ).
\end{align*}
\end{lem}
\begin{proof}
It can be proved in the same way as that of Lemma \ref{lem: calc of as+bal}.
\end{proof}

We consider the following pentagonal variants of
$\GARI(\mathcal{F})_{\as+\bal}$,
$\GARI(\mathcal{F})_{\underline{\as+\bal}}$
and
$\ARI(\mathcal{F})_{\underline{\al+\bal}}$:

\begin{defn}\label{defn: GARI as+pent}
For a divisible family $\mathcal{F}$ of functions, we define
the following  two subsets of $\GARI(\mathcal{F})_{\as}$ (cf. Definition \ref{def:al-il-as-is}):
\begin{align*}
\GARI(\mathcal{F})_{\as+\pent}
&:=\left\{M\in\GARI(\mathcal{F})_{\as}\ \middle|\
\begin{array}{c}
\text{there is a  commutator group-like series } \\ \psi\in\widehat{U\mathfrak{f}_{2}} \
\text{satisfying} \ \eqref{eq: MM=MMpsi} \ \text{below  in} \ \mathcal P_4(\mathcal F)
\end{array}
\right\}, \\
\GARI(\mathcal{F})_{\underline{\as+\pent}}
&:=\{M\in\GARI(\mathcal{F})_{\as+\pent}\bigm| M(x_1) \text{ is even }
\},
\end{align*}
\begin{equation}\label{eq: MM=MMpsi}
M^{01,2,3}\mulp M^{0,1,23}=M^{0,1,2}\mulp M^{0,12,3}\mulp\psi_{\mathcal{P}}^{1,2,3},
\end{equation}
(see \eqref{eq: mulp:Pn} for $\mulp$)
and the following subset
of $\ARI(\mathcal{F})_{\al}$ (cf. Definition \ref{def:al-il-as-is}):
$$
\ARI(\mathcal{F})_{\underline{\al+\pent}}
:=\left\{M\in\ARI(\mathcal{F})_{\al}\ \middle|\
\begin{array}{c}
\text{there is a  commutator group-like series } \\ \psi\in\widehat{U\mathfrak{f}_{2}} \
\text{satisfying} \ \eqref{eq: linMM=MMpsi} \ \text{below  in} \ \mathcal P_4(\mathcal F) \\
\text{ and }M(x_1)\text{ is even}
\end{array}
\right\},
$$
\begin{equation}\label{eq: linMM=MMpsi}
M^{01,2,3}+M^{0,1,23}=M^{0,1,2}+ M^{0,12,3}+\psi_{\mathcal{P}}^{1,2,3}.
\end{equation}
\end{defn}

The following ensures that they are identical to our previous objects:

\begin{thm}\label{pentagonal characterization theorem}
We have
\begin{align*}
\GARI(\mathcal{F})_{\as+\bal} = \GARI(\mathcal{F})_{\as+\pent},\\
\GARI(\mathcal{F})_{\underline{\as+\bal}} = \GARI(\mathcal{F})_{\underline{\as+\pent}},\\
\ARI(\mathcal{F})_{\underline{\al+\bal}} = \ARI(\mathcal{F})_{\underline{\al+\pent}}.
\end{align*}
More precisely,  under the identification,
the constant mould $C$ appearing in the definition of
$\GARI(\mathcal{F})_{\as+\bal}$ 
(resp. $\ARI(\mathcal{F})_{\underline{\al+\bal}}$)
is given by
$$C=\ma_{[2]}(\psi(-f_{1}-f_{2},f_1))$$ with
the series $\psi$ appearing in $\GARI(\mathcal{F})_{\as+\pent}$
(resp. $\ARI(\mathcal{F})_{\underline{\al+\pent}}$)
\end{thm}

\begin{proof}
It follows from
 Lemmas \ref{lem: maC is commutator group-like} and \ref{lem: calc of as+bal}.
\end{proof}

\begin{thm}\label{thm: recovery theorem of ASTR and GRT}
We have
\[
(\ma)^{-1}(\GARI(\mathcal{F}_{\ser})_{\as+\bal})={\ASTR}\cdot \exp{\Q f_{1}},
\]
\[
(\ma)^{-1}(\GARI(\mathcal{F}_{\ser})_{\underline{\as+\bal}})=\GRT_{1}\cdot\exp{\Q f_{1}},
\]
\[
(\ma)^{-1}(\ARI(\mathcal{F}_{\ser})_{\underline{\al+\bal}})=\grt_{1}\oplus{\Q f_{1}}.
\]
\end{thm}

\begin{proof}
By Theorem \ref{pentagonal characterization theorem},
it is reduced to show the following equalities:
\begin{align*}
(\ma)^{-1}(\GARI(\mathcal{F}_{\ser})_{\as+\pent})={\ASTR}\cdot \exp{\Q f_{1}}, \\
(\ma)^{-1}(\GARI(\mathcal{F}_{\ser})_{\underline{\as+\pent}})=\GRT_{1}\cdot\exp{\Q f_{1}},\\
(\ma)^{-1}(\ARI(\mathcal{F}_{\ser})_{\underline{\al+\pent}})=\grt_{1}\oplus{\Q f_{1}}.
\end{align*}
By definition, the left-hand side of the first equality
$(\ma)^{-1}(\GARI(\mathcal{F}_{\ser})_{\as+\pent})$
is the set of group-like elements $\varphi$ of $\widehat{U\mathfrak{f}_{2}}^{\dagger}$
such that there exists a commutator group-like series $\psi$ satisfying
\begin{equation*}\label{eq:pent_phi_psi}
\varphi^{01,2,3}\varphi^{0,1,23}=
\varphi^{0,1,2}\varphi^{0,12,3}\psi^{1,2,3}.
\end{equation*}
By Theorem \ref{thm: flip characterization},
this set coincides with  ${\ASTR}\cdot \exp{\Q f_{1}}$
and hence the claim follows.
The second equality follows from the first one.
The proof of third one proceed similarly.
%
\end{proof}


\begin{rem}
Drinfeld \cite{Dr} showed that  $\GRT_1$ forms a group, $\ASTR$ forms a $\GRT_1$-set
under the operation $\circledast$
and $\grt_1$ forms a Lie algebra under  the bracket $\{,\}$
(cf. Definition \ref{defn:ASTR and GRT1}).
In our setting, however, it is not known whether the set
the set $\GARI(\mathcal{F})_{\underline{\as+\abal}}$ forms a group,
whether $\GARI(\mathcal{F})_{\as+\abal}$  ((or an appropriate subset thereof) 
forms a $\GARI(\mathcal{F})_{\underline{\as+\abal}}$-set
(or a torsor)
under the product $\gari$ (Definition \ref{def:gari})
and  $\ARI(\mathcal{F})_{\underline{\al+\abal}}$  form a Lie algebra
under the $\ari$-bracket (Definition \ref{defn:ari-bracket}).
\end{rem}

\subsection{Associators accompanied with moulds in $\GARI(\mathcal{F})_{\as+\pent}$ }
\label{subsec: pentagon equation for the associated constant moulds}
We prepare the fact that the series $\psi$ appearing  in the definition of
$\GARI(\mathcal{F})_{\as+\pent}$ (Definition \ref{defn: GARI as+pent})
is actually an associator in Proposition \ref{prop:psi satisfies the pentagon equation}.
This fact will be required in the next subsection.

For distinct subsets $S_{1},S_{2},S_{3},S_{4}$ of $\{1,2,3,4,5\}$,
define the map 
\[
\ev^{S_{1},S_{2},S_{3},S_{4}}:\widehat{U\mathfrak{t}_{4}}\to\widehat{U\mathfrak{t}_{5}}
\]
by $t_{ij}\mapsto\sum_{k\in S_{i},l\in S_{j}}t_{kl}$. Put
\[
\ev_{1}=\ev^{12,3,4,5},\quad \ev_{2}=\ev^{1,23,4,5},\quad \ev_{3}=\ev^{1,2,34,5},
\quad \ev_{4}=\ev^{1,2,3,45}.
\]
Then $\ev_{i}(\widehat{U\mathfrak{t}_{4}}^\dag)\subset\widehat{U\mathfrak{t}_{5}}^\dag$ for $1\leq i\leq4$.
Since $\widehat{U\mathfrak{t}_{n}}^\dag$ is canonically
isomorphic to $\mathcal{M}(\mathcal{F}_{{\ser}};\{1,\dots,n-1\})$,
$\ev_{i}$ ($1\leq i\leq4$) induce the map
\[
\Ev_{i}:\mathcal{M}(\mathcal{F}_{{\ser}};\{1,2,3\})\to\mathcal{M}(\mathcal{F}_{{\ser}};\{1,2,3,4\}).
\]
By Lemma \ref{lem:mould-proper-B}, the map $\Ev_{i}$ is mould-proper. Thus by Theorem \ref{thm:UniqueProlongation}, $\Ev_{i}$ uniquely
extends to a mould-map from $\mathcal{M}(\mathcal{F};\{1,2,3\})$
to $\mathcal{M}(\mathcal{F};\{1,2,3,4\})$ for any divisible family of functions of $\mathcal{F}$.

\begin{lem}\label{lem:g1234}
Let $\mathcal{F}$ be a divisible family of functions.
Let $M\in\mathcal{M}(\mathcal{F})$ and $\psi\in\widehat{U\mathfrak{f}_{2}}$. For $i=0,1,2,3$, define $f_{i}:\{0,1,2,3,4\}\to\{0,1,2,3\}$ by $f(p)=p$
for $p\leq i$ and $f(p)=p-1$ for $p\geq i+1$.
Then we have
\begin{align*}f_{0}^{*}(M^{01,2,3}\mulp M^{0,1,23}&-M^{0,1,2}\mulp M^{0,12,3}\mulp\psi_{\mathcal P}^{1,2,3} ) \\
& = M^{012,3,4}\mulp M^{01,2,34}-M^{01,2,3}\mulp M^{01,23,4}\mulp \psi_{\mathcal P}^{2,3,4},\\
f_{1}^{*}(M^{01,2,3}\mulp M^{0,1,23}&-M^{0,1,2}\mulp M^{0,12,3}\mulp\psi_{\mathcal P}^{1,2,3} ) \\
& = M^{012,3,4}\mulp M^{0,12,34}-M^{0,12,3} \mulp M^{0,123,4}\mulp \psi_{\mathcal P}^{12,3,4},\\
f_{2}^{*}(M^{01,2,3}\mulp M^{0,1,23}&-M^{0,1,2}\mulp M^{0,12,3}\mulp\psi_{\mathcal P}^{1,2,3} ) \\
& =M^{01,23,4} \mulp M^{0,1,234}-M^{0,1,23}\mulp M^{0,123,4}\mulp \psi_{\mathcal P}^{1,23,4},\\
f_{3}^{*}(M^{01,2,3}\mulp M^{0,1,23}&-M^{0,1,2}\mulp M^{0,12,3}\mulp\psi_{\mathcal P}^{1,2,3} ) \\
& =M^{01,2,34} \mulp M^{0,1,234}-M^{0,1,2} \mulp M^{0,12,34}\mulp \psi_{\mathcal P}^{1,2,34}
\end{align*}
in ${\mathcal P}_5({\mathcal F})$.
For the definition of $f_i^{*}$, see Definition \ref{def:the induced f*}.
\end{lem}

\begin{proof}
The case $\mathcal{F}=\mathcal{F}_{{\ser}}$ easily follows
from the corresponding identity in $\widehat{U\mathfrak{t}_{5}}$.
Since both sides in each equation can be viewed as mould-maps of $M$,
the lemma follows from Theorem \ref{thm:UniqueProlongation}.
\end{proof}

\begin{prop}\label{prop:psi satisfies the pentagon equation}
Let $\mathcal{F}$ be a divisible family of functions.
Let $(M,\psi)\in \GARI(\mathcal{F})_{\as+\pent}$,
that is, the pair $(M,\psi)$ satisfies \eqref{eq: MM=MMpsi}.
Then we have $\psi\in\ASTR$.
\end{prop}

\begin{proof}
Since $\psi$ is a group-like, it is enough to show the pentagon equation for $\psi$.
By Lemma \ref{lem:g1234}, we have
\begin{align*}
  M^{012,3,4}\mulp M^{01,2,34}=M^{01,2,3}\mulp M^{01,23,4}\mulp \psi_{\mathcal P}^{2,3,4},\\
M^{012,3,4}\mulp M^{0,12,34}=M^{0,12,3} \mulp M^{0,123,4}\mulp \psi_{\mathcal P}^{12,3,4},\\
M^{01,23,4} \mulp M^{0,1,234}=M^{0,1,23}\mulp M^{0,123,4}\mulp \psi_{\mathcal P}^{1,23,4},\\
M^{01,2,34} \mulp M^{0,1,234}=M^{0,1,2} \mulp M^{0,12,34}\mulp \psi_{\mathcal P}^{1,2,34}
\end{align*}
in ${\mathcal P}_5({\mathcal F})$.
Furthermore, by Theorem \ref{thm:UniqueProlongation}, we have the following commutativities
\begin{align*}
(M^{0,1,2})^{\mulp -1}\mulp M^{012,3,4} & =M^{012,3,4}\mulp (M^{0,1,2})^{\mulp -1},\\
M^{0,1,234}\mulp \psi_{\mathcal P}^{2,3,4} & =\psi_{\mathcal P}^{2,3,4}\mulp M^{0,1,234},\\
M^{0,123,4}\mulp \psi_{\mathcal P}^{1,2,3} & =\psi_{\mathcal P}^{1,2,3} \mulp M^{0,123,4}
\end{align*}
in ${\mathcal P}_5({\mathcal F})$,
since we can check these equalities for general $M\in {\rm GARI}(\mathcal{F}_{{\ser}})$,
whose validities are deduced from the ones in  $\widehat{U\mathfrak{t}_{5}}^\dag$ via $\madec_5$.
Thus
\begin{align*}\psi_{\mathcal{P}}^{12,3,4}&\mulp\psi_{\mathcal{P}}^{1,2,34}  =(M^{0,12,3}\mulp M^{0,123,4})^{\mulp-1}\mulp M^{012,3,4}\mulp M^{0,12,34} \\
&\qquad\qquad\qquad\qquad
\mulp(M^{0,1,2}\mulp M^{0,12,34})^{\mulp{-1}}\mulp M^{01,2,34}\mulp M^{0,1,234}\\
 & =(M^{0,1,2}\mulp M^{0,12,3}\mulp M^{0,123,4})^{\mulp-1}\mulp M^{012,3,4}\mulp M^{01,2,34}\mulp M^{0,1,234}\\
 & =(M^{0,1,2}\mulp M^{0,12,3}\mulp M^{0,123,4})^{\mulp-1}\mulp M^{01,2,3}\mulp M^{01,23,4}\mulp\psi_{\mathcal{P}}^{2,3,4}\mulp M^{0,1,234}\\
 & =(M^{0,1,2}\mulp M^{0,12,3}\mulp M^{0,123,4})^{\mulp-1}\mulp M^{01,2,3}\mulp M^{01,23,4}\mulp M^{0,1,234}\mulp\psi_{\mathcal{P}}^{2,3,4}\\
 & =(M^{0,1,2}\mulp M^{0,12,3}\mulp M^{0,123,4})^{\mulp-1}\mulp M^{01,2,3}\mulp M^{0,1,23}\mulp M^{0,123,4}\\
 &\qquad\qquad\qquad\qquad\qquad
 \mulp\psi_{\mathcal{P}}^{1,23,4}\mulp\psi_{\mathcal{P}}^{2,3,4}\\
 & =(M^{0,123,4})^{\mulp-1}\mulp\psi_{\mathcal{P}}^{1,2,3}\mulp M^{0,123,4}
 \mulp\psi_{\mathcal{P}}^{1,23,4}\mulp\psi_{\mathcal{P}}^{2,3,4}\\
 & =\psi_{\mathcal{P}}^{1,2,3}\mulp\psi_{\mathcal{P}}^{1,23,4}\mulp\psi_{\mathcal{P}}^{2,3,4}
\end{align*}
in ${\mathcal P}_5({\mathcal F})$.
Hence $\psi$ satisfies the pentagon equation.
\end{proof}

We call such  $\varphi\in\mathsf{ASTR}$ as in Proposition
\ref{prop:psi satisfies the pentagon equation}
as the {\it associator accompanied with} 
$M\in\GARI(\mathcal{F})_{\as + \mathsf{pent} }$,


\subsection{Inclusions to \'{E}calle's sets $\GARI(\mathcal{F})_{\as\ast\is}$}
\label{sec: inclusion from GARI as pent to GARI as is}
We introduce a series $\Zig$ in Definition \ref{def: Zig} and
construct inclusions
$\GARI(\mathcal{F})_{\as+\pent}\hookrightarrow \GARI(\mathcal{F})_{\as\ast\is}$
(cf. Definition \ref{def:GARIas*is})
and
 $\GARI(\mathcal{F})_{\protect\underline{\as+\pent}}\hookrightarrow \GARI(\mathcal{F})_{\protect\underline{\as\ast\is}}$ in Theorem \ref{thm: inclusion GARIas+pent to GARIas*is}.


\begin{defn}\label{defn:Mini}
Following \cite[(1.29) and (1.30)]{E-flex},
for an associator $\varphi\in{\rm ASTR}$, we consider the constant mould
${\Mini}_{\varphi}\in\overline{\mathcal{M}}(\mathcal{F})$ defined by

\[
{\Mini}_{\varphi}(x_{1},\dots,x_{m})={\Mono}_{\varphi,m}
\]
with
\[
\sum_{r=0}^{\infty}{\Mono}_{\varphi,r}t^{r}=\exp(\sum_{k=2}^{\infty}(-1)^{k-1}\frac{\zeta_{\varphi}(k)}{k}t^{k}),
\]
and 
{$\zeta_{\varphi}(k):={-\langle\varphi| f_1f_{0}^{k-1}\rangle}$}.
\end{defn}

\begin{lem}\label{lem:I rel imply decM rel}
Suppose that the following identity holds
as complex functions with respect to  complex parameters $x$ and $y$
\begin{equation}\label{eq:sumcI=sumcII}
\sum_{{{\bf i}\in S_1}}c_{\bf i}\cdot I\varia{u_{{\bf i},1},\dots,u_{{\bf i},d_{{\bf i}}}}{\epsilon_{{\bf i},1},\dots,\epsilon_{{\bf i},d_{{\bf i}}}}
=\sum_{{{\bf j}\in S_2}}c_{{\bf j}}\cdot
I\varia{v_{{\bf j},1},\dots,v_{{\bf j},r_{{\bf j}}}}{\ y,\dots\dots,y}
\cdot
I\varia{w_{{\bf j},1},\dots,w_{{\bf j},s_{{\bf j}}}}{\ \tau_{{\bf j},1},\dots,\tau_{{\bf j},s_{{\bf j}}}}.
\end{equation}
with finite sets $S_1$, $S_2,$
$c_{\bf i},c_{\bf j}\in{\mathcal K}
=\Q(x_i\ |\ i\in\N)
$, $\epsilon_{{\bf i},p}\in\{xy,y\}$, $\tau_{{\bf j},q}\in\{xy,x\}$,
and $I$ defined by \eqref{eq:power series I}.
Then the following identity
\[
{
\sum_{{\bf i}\in S_1}c_{\bf i}\cdot \minus(\bal(M))\varia{{\emptyset;}\quad u_{{\bf i},1},\dots,u_{{\bf i},d_{{\bf i}}}}{{\emptyset;}\ \sigma(\epsilon_{{\bf i},1}),\dots,\sigma(\epsilon_{{\bf i},d_{{\bf i}}})}
=\sum_{{\bf j}\in S_2}c_{{\bf j}}\cdot \minus(M)\varia{v_{{\bf j},1},\dots,v_{{\bf j},{r_{{\bf j}}}};\quad w_{{\bf j},1},\dots,w_{{\bf j},s_{{\bf j}}}}{\quad 1,\dots\dots,1\ ;\ \sigma(\tau_{{\bf j},1}),\dots,\sigma(\tau_{{\bf j},s_{{\bf j}}})}
}
\]
holds in $\mathcal P_4(\mathcal F)$
for any mould $M\in\mathcal{M}(\mathcal{F}_{{\rm Lau}})$,
where we follow the rule in \eqref{eq: 1-2 rule}.
\end{lem}

\begin{proof}
By Theorem \ref{thm: rec diff eq for I}, for each ${\bf i}$, $I\varia{u_{{\bf i},1},\dots,u_{{\bf i},d_{{\bf i}}}}{\epsilon_{{\bf i},1},\dots,\epsilon_{{\bf i},d_{{\bf i}}}}$
can be expressed as
\[
I\varia{u_{{\bf i},1},\dots,u_{{\bf i},d_{{\bf i}}}}{\epsilon_{{\bf i},1},\dots,\epsilon_{{\bf i},d_{{\bf i}}}}
=\sum_{{\bf j}({\bf i})}c_{{\bf j}({\bf i})}\cdot
I\varia{v_{{\bf j}({\bf i}),1},\dots,v_{{\bf j}({\bf i}),r_{{\bf j}({\bf i})}}}{\qquad y,\dots\dots,y}
\cdot
I\varia{w_{{\bf j}({\bf i}),1},\dots,w_{{\bf j}({\bf i}),s_{{\bf j}({\bf i})}}}
{\tau_{{\bf j}({\bf i}),1},\dots,\tau_{{\bf j}({\bf i}),s_{{\bf j}({\bf i})}}}
\]
with $\tau_{{\bf j}({\bf i}),q}\in\{xy,x\}$, and by (\ref{eq:decM}),
\begin{align*}
\minus(\bal(M)) & \varia{\emptyset; \quad u_{{\bf i},1},\dots,u_{{\bf i},d_{{\bf i}}}}{\emptyset;\ \sigma(\epsilon_{{\bf i},1}),\dots,\sigma(\epsilon_{{\bf i},d_{{\bf i}}})} \\
&=\sum_{{\bf j}({\bf i})}c_{{\bf j}({\bf i})}\cdot
\minus(M)\varia{v_{{\bf j}({\bf i}),1},\dots,v_{{\bf j}({\bf i}),{r_{{\bf j}({\bf i})}}}\ ; \ w_{{\bf j}({\bf i}),1},\dots,w_{{\bf j}({\bf i}),s_{{\bf j}({\bf i})}}}
{\qquad 1,\dots\dots,1 \qquad ;\ \sigma(\tau_{{\bf j}({\bf i}),1}),\dots,\sigma(\tau_{{\bf j}({\bf i}),s_{{\bf j}({\bf i})}})}.
\end{align*}
Since we have
\begin{align*}
 \sum_{{\bf i}\in S_1}c_{\bf i}I\varia{u_{{\bf i},1},\dots,u_{{\bf i},d_{{\bf i}}}}{\epsilon_{{\bf i},1},\dots,\epsilon_{{\bf i},d_{{\bf i}}}}
  =\sum_{{\bf i}\in S_1}c_{\bf i}\sum_{{\bf j}({\bf i})}c_{{\bf j}({\bf i})}\cdot I\varia{v_{{\bf j}({\bf i}),1},\dots,v_{{\bf j}({\bf i}),r_{{\bf j}({\bf i})}}}{\quad y,\dots\dots,y}\cdot
  I\varia{w_{{\bf j}({\bf i}),1},\dots,w_{{\bf j}({\bf i}),s_{{\bf j}({\bf i})}}}
  {\tau_{{\bf j}({\bf i}),1},\dots,\tau_{{\bf j}({\bf i}),s_{{\bf j}({\bf i})}}},
\end{align*}
we obtain
\begin{align*}
\sum_{{\bf j}\in S_2}c_{{\bf j}}\cdot &
I\varia{v_{{\bf j},1},\dots,v_{{\bf j},r_{{\bf j}}}}{\ y,\dots\dots,y}
\cdot
I\varia{w_{{\bf j},1},\dots,w_{{\bf j},s_{{\bf j}}}}{\ \tau_{{\bf j},1},\dots,\tau_{{\bf j},s_{{\bf j}}}} \\
&=\sum_{{\bf i}\in S_1}c_{\bf i}\sum_{{\bf j}({\bf i})}c_{{\bf j}({\bf i})}\cdot I\varia{v_{{\bf j}({\bf i}),1},\dots,v_{{\bf j}({\bf i}),r_{{\bf j}({\bf i})}}}{\quad y,\dots\dots,y}\cdot
  I\varia{w_{{\bf j}({\bf i}),1},\dots,w_{{\bf j}({\bf i}),s_{{\bf j}({\bf i})}}}
  {\tau_{{\bf j}({\bf i}),1},\dots,\tau_{{\bf j}({\bf i}),s_{{\bf j}({\bf i})}}}
\end{align*}
by \eqref{eq:sumcI=sumcII}.
By change of variables $x$ and $y$ with $x^{-1}$ and $y^{-1}$ respectively, we obtain
\begin{align*}
\sum_{{\bf j}\in S_2}c_{{\bf j}}\cdot & (\Zag^{{y}}\otimes\Zag^{{xy},{x}})\varia{v_{{\bf j},1},\dots,v_{{\bf j},r_{{\bf j}}}\ ;\quad w_{{\bf j},1},\dots,w_{{\bf j},s_{{\bf j}}}}{\quad1,\dots\dots,1\ ;\ \sigma(\tau_{{\bf j},1}),\dots,\sigma(\tau_{{\bf j},s_{{\bf j}}})}\\
 & =\sum_{{\bf i}\in S_1}c_{{\bf i}}\sum_{{\bf j}({\bf i})}c_{{\bf j}({\bf i})}\cdot(\Zag^{{y}}\otimes\Zag^{{xy},{x}})\varia{v_{{\bf j}({\bf i}),1},\dots,v_{{\bf j}({\bf i}),r_{{\bf j}({\bf i})}}\ ;\quad w_{{\bf j}({\bf i}),1},\dots,w_{{\bf j}({\bf i}),s_{{\bf j}({\bf i})}}}{\qquad1,\dots\dots,1\qquad;\ \sigma(\tau_{{\bf j}({\bf i}),1}),\dots,\sigma(\tau_{{\bf j}({\bf i}),s_{{\bf j}({\bf i})}})}.
\end{align*}
Thus by Lemma \ref{lem: f(Zag)=0 then f=0}, we have 
\begin{align*}
\sum_{{\bf j}\in S_2}c_{{\bf j}}\cdot
&M\varia{v_{{\bf j},1},\dots,v_{{\bf j},r_{{\bf j}}}\ ;\quad w_{{\bf j},1},\dots,w_{{\bf j},s_{{\bf j}}}}
{\quad 1,\dots\dots,1\ ; \ \sigma(\tau_{{\bf j},1}),\dots,\sigma(\tau_{{\bf j},s_{{\bf j}}})} \\
&=\sum_{{\bf i}\in S_1}c_{\bf i}\sum_{{\bf j}({\bf i})}c_{{\bf j}({\bf i})}\cdot M\varia{v_{{\bf j}({\bf i}),1},\dots,v_{{\bf j}({\bf i}),r_{{\bf j}({\bf i})}}\ ; \quad w_{{\bf j}({\bf i}),1},\dots,w_{{\bf j}({\bf i}),s_{{\bf j}({\bf i})}}}{\qquad 1,\dots\dots,1\qquad ; \ \sigma(\tau_{{\bf j}({\bf i}),1}),\dots,\sigma(\tau_{{\bf j}({\bf i}),s_{{\bf j}({\bf i})}})}
\end{align*}
for any mould $M$,
as well as $\minus(M)$,
which implies the theorem.
\end{proof}

\begin{defn}\label{def: Zig}
Similarly to Definition \ref{def:Zag and Zag}, we consider the following series:
\begin{align}\label{eq: Zig}
\Zig\binom{\epsilon_{1},\dots,\epsilon_{d}}{u_{1},\dots,u_{d}}
& :=\sum_{n_{1},\dots,n_{d}\geq0}\Li_{n_{d},\dots,n_{1}}(\epsilon_{d},\dots,\epsilon_{1})u_{1}^{n_{1}-1}\cdots u_{d}^{n_{d}-1} \\ \notag
& =\sum_{0<m_{d}<\cdots<m_{1}}\frac{\epsilon_{d}^{m_{d}}\cdots\epsilon_{1}^{m_{1}}}{(m_{d}-u_{d})\cdots(m_{1}-u_{1})}
\in\C[[u_1,\dots,u_d]],
\end{align}
where $\epsilon_1,\dots,\epsilon_d$ are chosen to be suitable complex numbers.
\end{defn}

By \eqref{eq: I= sum Li},  we have
\begin{equation}\label{eq: Zig=I}
\Zig\binom{\epsilon_{1},\dots,\epsilon_{d}}{u_{1},\dots,u_{d}}
  =I\binom{\qquad u_{d},\qquad u_{d-1}-u_{d},\ \dots\ ,u_{1}-u_{2}}{\epsilon_{1}^{-1}\cdots\epsilon_{d}^{-1},\epsilon_{1}^{-1}\cdots\epsilon_{d-1}^{-1},\ \dots \ ,\epsilon_{1}^{-1}}.
\end{equation}

\begin{thm}\label{thm: inclusion GARIas+pent to GARIas*is}
The operation $\minus$ in \eqref{eq:neg} induces the inclusion
$$\GARI(\mathcal{F})_{\as+\pent}\subset\GARI(\mathcal{F})_{\as*\is}.$$
Particularly we have
$$
\minus:\GARI(\mathcal{F})_{\underline{\as+\pent}}\subset\GARI(\mathcal{F})_{\underline{\as*\is}}.$$
\end{thm}


\begin{proof}
\underline{\sf{Step 1:}}
Let $x$ and $y$ be complex numbers whose absolute value is smaller
than $1$. Then we have

\begin{align*}
\Zig\binom{x,1,\dots,1}{x_{1},\dots,x_{p}}\cdot \Zig\binom{y,1,\dots,1}{x_{p+1},\dots,x_{p+q}} & =\sum_{\alpha\in Y_{\mathbb{Z}}^{\bullet}}{\rm Sh}_{*}\binom{\left(_{x_{1},\dots,x_{p}}^{x,1,\dots,1}\right);\left(_{x_{p+1},\dots,x_{p+q}}^{\ \ y,1,\dots,1}\right)}{\alpha}\cdot \Zig\left(\alpha\right).
\end{align*}
Hereafter, for $\sum_{\alpha}c_{\alpha}\alpha\in\mathcal{K}\langle Y_{\Z}\rangle$, we write $\Zig(\sum_{\alpha}c_{\alpha}\alpha)$ for $\sum_{\alpha}c_{\alpha}\cdot\Zig(\alpha)$.
We have
\begin{align*}
&\sum_{\alpha\in Y_{\mathbb{Z}}^{\bullet}}
{\rm Sh}_{*}\binom{\left(_{x_{1},\dots,x_{p}}^{x,1,\dots,1}\right);\left(_{x_{1},\dots,x_{p}}^{y,1,\dots,1}\right)}{\alpha}\cdot \Zig\left(\alpha\right) \\
&\ = \Zig\left( \varia{x,1,\dots,1}{x_{1},\dots,x_{p}} \shuffle_* \varia{y,1,\dots,1}{x_{p+1},\dots,x_{p+q}}  \right) \\
&\ = \frac{1}{x_1-x_{p+1}}\Zig\left( \left(\varia{xy}{x_1}-\varia{xy}{x_{p+1}}\right) \left( \varia{1,\dots,1}{x_{2},\dots,x_{p}} \shuffle_* \varia{1,\dots,1}{x_{p+2},\dots,x_{p+q}} \right)  \right) \\
&\quad + \Zig\left( \varia{x}{x_1} \left( \varia{1,\dots,1}{x_{2},\dots,x_{p}} \shuffle_* \varia{y,1,\dots,1}{x_{p+1},\dots,x_{p+q}} \right)\right)
 + \Zig\left( \varia{y}{x_{p+1}} \left( \varia{x,1,\dots,1}{x_{1},\dots,x_{p}} \shuffle_* \varia{1,\dots,1}{x_{p+2},\dots,x_{p+q}} \right)\right).
\end{align*}
We rewrite the above three terms simply by
\begin{align*}
\sum_{j}c_{j} \Zig &\binom{xy,1,\dots,1}{u_{j,1},\dots,u_{j,d_{j}}}+\sum_{j}c_{j}'\Zig\binom{x,\{1\}^{r_{j}-1},y,\{1\}^{s_{j}-1}}{v_{j,1},\dots,v_{j,r_{j}+s_{j}}}\\
&+\sum_{j}c_{j}''\Zig\binom{y,\{1\}^{a_{j}-1},x,\{1\}^{b_{j}-1}}{w_{j,1},\dots,w_{j,a_{j}+b_{j}}}
\end{align*}
with {$c_{j}, c_{j}', c_{j}''\in\mathcal K$.}
Thus by changing  variables $x$ and $y$ with $x^{-1}$ and $y^{-1}$,
we deduce from  \eqref{eq: Zig=I},
\begin{align*}
 & \sum_{j}c_{j}I\binom{u_{j,d_{j}},u_{j,d_{j-1}}-u_{j,d_{j}},\dots,u_{j,1}-u_{j,2}}{xy,\dots,\dots, xy} \\
  & \qquad\qquad +\sum_{j}c_{j}'I\binom{v_{j,r_{j}+s_{j}},v_{j,r_{j}+s_{j}-1}-v_{j,r_{j}+s_{j}},\dots,v_{j,1}-v_{j,2}}{\{xy\}^{s_{j}},\{x\}^{r_{j}}} \\ \notag
 &\qquad\qquad + \sum_{j}c_{j}''I\binom{w_{j,a_{j}+b_{j}},w_{j,a_{j}+b_{j}-1}-w_{j,a_{j}+b_{j}},\dots,w_{j,1}-w_{j,2}}{\{xy\}^{b_{j}},\{y\}^{a_{j}}}\\ \notag
 &\quad =I\binom{x_{p},x_{p-1}-x_{p},\dots,x_{1}-x_{2}}{x,\dots,x}\cdot I\binom{x_{p+q},x_{p+q-1}-x_{p+q},\dots,x_{p+1}-x_{p+2}}{y,\dots,y}.
\end{align*}
holds as functions with respect to complex variables.
By considering the change of variables $x$ and $y$ with $x^{-1}$ and $y^{-1}$,
we obtain, by Lemma \ref{lem:I rel imply decM rel},
\begin{align} \label{eq: func eq for N}
 & \sum_{j}c_{j}\minus(N)\binom{\emptyset;u_{j,d_{j}},u_{j,d_{j-1}-u_{j,d_{j}}},\dots,u_{j,1}-u_{j,2}}{\emptyset; \qquad\qquad\{2\}^{d_j}\qquad\qquad \qquad\qquad } \\ \notag
  & \qquad\qquad +\sum_{j}c_{j}'\minus(N)\binom{\emptyset; v_{j,r_{j}+s_{j}},v_{j,r_{j}+s_{j}-1}-v_{j,r_{j}+s_{j}},\dots,v_{j,1}-v_{j,2}}{\emptyset; \qquad \qquad \qquad \{2\}^{s_{j}},\{1\}^{r_{j}}\qquad\qquad \qquad\qquad  }\\ \notag
 &\qquad\qquad + \sum_{j}c_{j}''\minus(\bal(N))\binom{\emptyset;\quad w_{j,a_{j}+b_{j}},w_{j,a_{j}+b_{j}-1}-w_{j,a_{j}+b_{j}},\dots,w_{j,1}-w_{j,2}}{\emptyset;\qquad\qquad\qquad\qquad \{2\}^{b_{j}},\{1\}^{a_{j}}\qquad\qquad\qquad\qquad }\\ \notag
 &\quad =\minus(N)\binom{x_{p+q},x_{p+q-1}-x_{p+q},\dots,x_{p+1}-x_{p+2}\ ; \ x_{p},x_{p-1}-x_{p},\dots,x_{1}-x_{2}}
 {\quad\quad\{1\}^q\quad\quad\quad\quad;\quad\{1\}^p} 
\end{align}
for any $N \in \mathcal P_4(\mathcal F_\Lau)$.

\underline{\sf{Step 2:}}
Let $M\in\GARI(\mathcal{F})_{\as+\pent}$.
Then by Theorem \ref{pentagonal characterization theorem},
we have $M\in \GARI(\mathcal{F})_{\as+\abal}$.
Put
$N=M_{[1]}\otimes M_{[2]}\in \mathcal P_4(\mathcal F)$.
Since  we have
$$\bal (M_{[1]}\otimes M_{[2]})=M_{[1]}\otimes (M_{[2]}\times C)$$
by definition of $\GARI(\mathcal{F})_{\as+\abal}$,
we deduce from \eqref{eq: func eq for N} the following
\begin{align}
 & \sum_{j}c_{j}\minus(M)(u_{j,d_{j}},u_{j,d_{j-1}}-u_{j,d_{j}},\dots,u_{j,1}-u_{j,2}) \\ \notag
  & \qquad +\sum_{j}c_{j}'\minus(M)(v_{j,r_{j}+s_{j}},v_{j,r_{j}+s_{j}-1}-v_{j,r_{j}+s_{j}},\dots,v_{j,1}-v_{j,2})\\ \notag
 &\qquad +\sum_{j}c_{j}''\left(\minus(M)_{[2]}\times C\right)\binom{w_{j,a_{j}+b_{j}},w_{j,a_{j}+b_{j}-1}-w_{j,a_{j}+b_{j}},\dots,w_{j,1}-w_{j,2}}{\{2\}^{b_{j}},\{1\}^{a_{j}}}\\ \notag
 &\quad =\minus(M)(x_{p},x_{p-1}-x_{p},\dots,x_{1}-x_{2})\cdot \minus(M)(x_{p+q},x_{p+q-1}-x_{p+q},\dots,x_{p+1}-x_{p+2}).
\end{align}
Put
\begin{equation}\label{eq:beta}
\beta_{C}(k,l)=C\binom{x_{1},\dots,x_{k+l}}{{\{2\}^{k},\{1\}^{l}}}\in\mathbb{Q}.
\end{equation}
Then since $\beta_{C}(i,0)=\beta_{C}(0,i)=\delta_{i,0}$ by Lemma \ref{lem: maC is commutator group-like}, we have
\begin{align*}
 & \sum_{j} c_{j}''\left(\minus(M)_{[2]}\times C\right)\binom{w_{j,a_{j}+b_{j}},w_{j,a_{j}+b_{j}-1}-w_{j,a_{j}+b_{j}},\dots,w_{j,1}-w_{j,2}}{{\{2\}^{b_{j}},\{1\}^{a_{j}}}}\\
 & =\sum_{j}c_{j}''\minus(M)(w_{j,a_{j}+b_{j}},w_{j,a_{j}+b_{j}-1}-w_{j,a_{j}+b_{j}},\dots,w_{j,1}-w_{j,2})\\
 & \quad +\sum_{j}c_{j}''\sum_{k=1}^{b_{j}}\minus(M)(w_{j,a_{j}+b_{j}},w_{j,a_{j}+b_{j}-1}-w_{j,a_{j}+b_{j}},\dots,w_{j,a_{j}+k+1}-w_{j,a_{j}+k+2}) \\
&\qquad\qquad\qquad\qquad\qquad \cdot\beta_{C}(k,a_{j}).
\end{align*}
Then
\begin{align*}
  \sum_{j}&c_{j}\swap(\minus(M))(u_{j,1},\dots,u_{j,d_{j}})+\sum_{j}c_{j}'\swap(\minus(M))(v_{j,1},\dots,v_{j,r_{j}+s_{j}})\\ \notag
 &+\sum_{j}c_{j}''\swap(\minus(M))(w_{j,1},\dots,w_{j,a_{j}+b_{j}}) \\ \notag
 & +\sum_{j}c_{j}''\sum_{k=1}^{b_{j}}\swap(\minus(M))(w_{j,a_{j}+k+1},\dots,w_{j,a_{j}+b_{j}})\cdot\beta_{C}(k,a_{j})\\ \notag
 &\qquad =\swap(\minus(M))(x_{1},\dots,x_{p})\cdot\swap(\minus(M))(x_{p+1},\dots,x_{p+q}).
\end{align*}
So we have
\begin{align*}
\shmap_{*}&(\swap(\minus(M)))(x_{1},\dots,x_{p};x_{p+1},\dots,x_{p+q})
\\ \notag
&+\sum_{j}c_{j}''\sum_{k=1}^{b_{j}}\swap(\minus(M))(w_{j,a_{j}+k+1},\dots,w_{j,a_{j}+b_{j}})\cdot\beta_{C}(k,a_{j})\\ \notag
&=\swap(\minus(M))(x_{1},\dots,x_{p})\cdot\swap(\minus(M))(x_{p+1},\dots,x_{p+q}).
\end{align*}

\underline{\sf{Step 3:}}
By definition, we have
\begin{align*}
 & \sum_{j}c_{j}''\varia{y,\{1\}^{a_{j}-1},x,\{1\}^{b_{j}-1}}{w_{j,1},\dots,w_{j,a_{j}+b_{j}}}
  =\varia y{x_{p+1}}\left(\varia{x,1,\dots,1}{x_{1},\dots,x_{p}}\shuffle_{*}\varia{1,\dots,1}{x_{p+2},\dots,x_{p+q}}\right). \\
\intertext{A repeated application of \eqref{eqn:shufflestar product} yields}
 & =\sum_{a=1}^{q}\varia{y,\{1\}^{a-1},x}{x_{p+1},\dots,x_{p+a},x_{1}}\left(\varia{1,\dots,1}{x_{2},\dots,x_{p}}\shuffle_{*}\varia{1,\dots,1}{x_{p+a+1},\dots,x_{p+q}}\right)\\
 & \quad+\sum_{a=1}^{q-1}\frac{1}{x_{1}-x_{p+a+1}}\varia{y,\{1\}^{a-1}}{x_{p+1},\dots,x_{p+a}}\left(\varia x{x_{1}}-\varia x{x_{p+a+1}}\right)\left(\varia{1,\dots,1}{x_{2},\dots,x_{p}}\shuffle_{*}\varia{1,\dots,1}{x_{p+a+2},\dots,x_{p+q}}\right).
\end{align*}
Thus
\[
\sum_{\substack{j\\a_{j}=a}}c_{j}''\varia{\{1\}^{b_{j}-1}}{w_{j,a_{j}+2},\dots,w_{j,a_{j}+b_{j}}} =\left(\varia{1,\dots,1}{x_{2},\dots,x_{p}}\shuffle_{*}\varia{1,\dots,1}{x_{p+a+1},\dots,x_{p+q}}\right).
\]
Define $\mathcal{K}$-linear map
$
D:\mathcal{K}\langle Y_{\Z}\rangle\to\mathcal{K}\langle Y_{\Z}\rangle
$
by
$
D\left(\varia{\epsilon_{1},\dots,\epsilon_{n}}{u_{1},\dots,u_{n}}\right)=\begin{cases}
\varia{\epsilon_{2},\dots,\epsilon_{n}}{u_{2},\dots,u_{n}} & n>0\\
0 & n=0.
\end{cases}
$
Then by definition,
\[
D\left(\varia{\epsilon_{1},\dots,\epsilon_{n}}{u_{1},\dots,u_{n}}\shuffle_{*}\varia{\epsilon_{1}',\dots,\epsilon_{n}'}{u_{1}',\dots,u_{n}'}\right)=\varia{\epsilon_{2},\dots,\epsilon_{n}}{u_{2},\dots,u_{n}}\shuffle_{*}\varia{\epsilon_{1}',\dots,\epsilon_{n}'}{u_{1}',\dots,u_{n}'}+\varia{\epsilon_{1},\dots,\epsilon_{n}}{u_{1},\dots,u_{n}}\shuffle_{*}\varia{\epsilon_{2}',\dots,\epsilon_{n}'}{u_{2}',\dots,u_{n}'}.
\]
Thus we have
\begin{align*}
\sum_{\substack{j\\a_{j}=a,b_j\geq k}}& c_{j}''\varia{\{1\}^{b_{j}-k}}{w_{j,a_{j}+k+1},\dots,w_{j,a_{j}+b_{j}}}  =D^{k-1}\left(\sum_{\substack{j\\a_{j}=a}}c_{j}''\varia{\{1\}^{b_{j}-1}}{w_{j,a_{j}+2},\dots,w_{j,a_{j}+b_{j}}}\right)\\
 & =D^{k-1}\left(\varia{1,\dots,1}{x_{2},\dots,x_{p}}\shuffle_{*}\varia{1,\dots,1}{x_{p+a+1},\dots,x_{p+q}}\right)\\
 & =\sum_{\substack{l_{1}+l_{2}=k-1\\
0\leq l_{1}\leq p-1,0\leq l_{2}\leq q-a
}
}\binom{l_{1}+l_{2}}{l_{1}}\cdot\varia{1,\dots,1}{x_{l_{1}+2},\dots,x_{p}}\shuffle_{*}\varia{1,\dots,1}{x_{p+a+l_{2}+1},\dots,x_{p+q}}.
\end{align*}
Thus
\begin{align*}
 & \sum_{j}c_{j}''\sum_{k=1}^{b_{j}}\swap(\minus(M))(w_{j,a_{j}+k+1},\dots,w_{j,a_{j}+b_{j}})\cdot\beta_{C}(k,a_{j})\\
 & = \sum_{k=1}^{\infty}\sum_{a=1}^{q}\beta_C(k,a) \sum_{\substack{j\\a_j=a,b_j\geq k} }c_{j}''\cdot\swap(\minus(M))(w_{j,a_{j}+k+1},\dots,w_{j,a_{j}+b_{j}})\\
  & = \sum_{k=1}^{\infty}\sum_{a=1}^{q}\beta_C(k,a)
  \sum_{\substack{l_{1}+l_{2}=k-1\\0\leq l_{1}\leq p-1,0\leq l_{2}\leq q-a}} \binom{l_{1}+l_{2}}{l_{1}} \\
  &\qquad\qquad\qquad \cdot{\rm swap}(\minus(M))\left( \varia{1,\dots,1}{x_{l_{1}+2},\dots,x_{p}}\shuffle_{*}\varia{1,\dots,1}{x_{p+a+l_{2}+1},\dots,x_{p+q}} \right)\\
 & =\sum_{a=1}^{q}\sum_{l_{1}=0}^{p-1}\sum_{l_{2}=0}^{q-a}{l_{1}+l_{2} \choose l_{2}}\beta_{C}(l_{1}+l_{2}+1,a) \\
  &\qquad\qquad\qquad \cdot{\rm swap}(\minus(M))\left( \varia{1,\dots,1}{x_{l_{1}+2},\dots,x_{p}}\shuffle_{*}\varia{1,\dots,1}{x_{p+a+l_{2}+1},\dots,x_{p+q}} \right).\\
 \intertext{Since $\beta_{C}(i,0)=\beta_{C}(0,i)=\delta_{i,0}$ by Lemma \ref{lem: maC is commutator group-like}, we have}
 & =\sum_{r_{1}=1}^{p}\sum_{r_{2}=0}^{q}\left(\sum_{l=0}^{r_{2}}{r_{1}-1+l \choose l}\beta_{C}(r_{1}+l,r_{2}-l)\right) \\
   &\qquad\qquad\qquad
 \cdot{\rm swap}(\minus(M))\left( \left(_{x_{1+r_{1}},\dots,x_{p}}^{1,\dots,1}\right) \shuffle_{*} \left(_{x_{p+r_{2}+1},\dots,x_{p+q}}^{1,\dots,1}\right) \right)\\
 & =(A_{\varphi}\times\shmap_{*}(\swap(\minus(M))))(x_{1},\dots,x_{p};x_{p+1},\dots,x_{q}).
\end{align*}
Here $A_{\varphi}\in\overline{\mathcal{M}_{2}}(\mathcal{F})$ is
a constant mould defined by
\[
A_{\varphi}(x_{1},\dots,x_{r_{1}};x_{r_{1}+1},\dots x_{r_{1}+r_{2}})=\begin{cases}
0 & r_{1}=0\\
\sum_{l=0}^{r_{2}}{r_{1}-1+l \choose l}\beta_{C}(r_{1}+l,r_{2}-l) & r_{1}>0,
\end{cases}
\]
with $\beta_C$ given by \eqref{eq:beta}
and $\shmap_{*}:\overline{\mathcal M}(\mathcal{F}_{{\rm Lau}})\to \overline{\mathcal M}_{2}(\mathcal{F}_{{\rm Lau}})$
is the map in Definition \ref{def:shmap}.
Thus we have
\begin{align*}
\swap(\minus(M))&(x_{1},\dots,x_{p})\cdot\swap(\minus(M))(x_{p+1},\dots,x_{p+q})\\
&=\left((I+A_{\varphi})\times\shmap_{*}\left(\swap(\minus(M))\right)\right)(x_{1},\dots,x_{p};x_{p+1},\dots,x_{q}),
\end{align*}
or equivalently
\begin{equation}\label{eq: swap swap = I+A Delta swap}
\swap(\minus(M))\otimes\swap(\minus(M))=(I+A_{\varphi})\times\shmap_{*}(\swap(\minus(M))).
\end{equation}

\underline{\sf{Step 4:}}
For $\varphi=\varphi(f_0,f_1)\in \widehat{U{\mathfrak f}_2}$,
we put
$$\iota_0\varphi:=\varphi(-f_0,f_1).$$
By \eqref{eq: swap swap = I+A Delta swap}, we have
\begin{align}\label{eq: mini swap minus M times}
&({\Mini}_{\iota_0\varphi}\times\swap(\minus(M)))\otimes({\Mini}_{\iota_0\varphi}\times\swap(\minus(M))) \\ \notag
&=({\Mini}_{\iota_0\varphi}\otimes{\Mini}_{\iota_0\varphi})\times(I+A_{\varphi})\times\shmap{*}({\Mini}_{\iota_0\varphi})^{-1}\times\shmap_{*}({\Mini}_{\iota_0\varphi}\times\swap(\minus(M))),
\end{align}
where $\Mini_{\iota_0\varphi}$ is the constant mould defined in Definition \ref{defn:Mini}.
In particular, by considering the case where $M=\ma(\varphi)$,
and using $\ma(\iota_0\varphi)=\minus\circ\ma(\varphi)$,
we have
\begin{align}\label{eq:mini swap minus ma}
&({\Mini}_{\iota_0\varphi}\times\swap\circ\ma(\iota_0\varphi))\otimes
({\Mini}_{\iota_0\varphi}\times\swap\circ\ma(\iota_0\varphi)) \\ \notag
=&({\Mini}_{\iota_0\varphi}\otimes{\Mini}_{\iota_0\varphi})\times(I+A_{\varphi})\times\shmap_{*}({\Mini}_{\iota_0\varphi})^{-1}\times\shmap_{*}({\Mini}_{\iota_0\varphi}\times\swap\circ\ma(\iota_0\varphi)).
\end{align}
On the other hand,
since the accompanied associator $\varphi$ is an associator
by Proposition \ref{prop:psi satisfies the pentagon equation},
we have
$\Delta_\ast((\iota_1\varphi)_\ast)=(\iota_1\varphi)_\ast\otimes (\iota_1\varphi)_\ast$
with $\iota_1\varphi:=\varphi(f_0,-f_1)$
by \cite[Theorem 0.2]{F11},
(for $\varphi_\ast$ and $\Delta_\ast$,
see \eqref{eq: varphi ast} and \eqref{eq: harmonic coproduct}
in Appendix \ref{sec: appendix GARIasis} respectively).
Whence we have $\Delta_\ast((\iota_0\varphi)_\ast)=(\iota_0\varphi)_\ast\otimes (\iota_0\varphi)_\ast$.
Then by Proposition \ref{prop: equivalent formulation of double shuffle relations} in our appendix, we obtain
\begin{equation}\label{eq: Mini swap ma times}
({\Mini}_{\iota_0\varphi}\times\swap(\ma(\iota_0\varphi)))\otimes({\Mini}_{\iota_0\varphi}\times\swap(\ma(\iota_0\varphi)))=\shmap_{*}({\Mini}_{\iota_0\varphi}\times\swap(\ma(\iota_0\varphi))).
\end{equation}
By \eqref{eq:mini swap minus ma} and \eqref{eq: Mini swap ma times}, we obtain
we have
\begin{equation}\label{eq: Mini Mini I+A Delta Mini inv =1}
({\Mini}_{\iota_0\varphi}\otimes{\Mini}_{\iota_0\varphi})\times (I+A_{\varphi})\times \shmap_{*}({\Mini}_{\iota_0\varphi})^{-1}=I.
\end{equation}
By \eqref{eq: mini swap minus M times}  and \eqref{eq: Mini Mini I+A Delta Mini inv =1}, we have
\[
({\Mini}_{\iota_0\varphi}\times\swap(\minus(M)))\otimes({\Mini}_{\iota_0\varphi}\times\swap(\minus(M)))=\shmap_{*}({\Mini}_{\iota_0\varphi}\times\swap(\minus(M))),
\]
i.e. ${\Mini}_{\iota_0\varphi}\times\swap(\minus(M))$ is symmetril in general.
\end{proof}

\begin{rem}
The main arguments in Step 1-3 of the above proof are based on  the same idea  of
\cite{F11} \S 3 where we find the arguments  deducing  the
harmonic product formula 
from that of two variable multiple polylogarithms.
Also, we expect  that 
an embedding  from $\ARI(\mathcal{F})_{\underline{\al+\pent}}$
to $\ARI(\mathcal{F})_{\underline{\al\ast\il}}$ is
obtained by the similar arguments.
\end{rem}

As a corollary of  Theorem \ref{thm: inclusion GARIas+pent to GARIas*is}, we obtain the following

\begin{cor}
When $\mathcal F=\mathcal F_\ser$, we have the following commutative diagram:
\begin{equation*}
\xymatrix{
 \ASTR \cdot \exp{\Q f_{1}}\ar@{->}[rr]^{\simeq}_{\ma_{[1]}}\ar@{^{(}->}^{\iota_0}[d]
 && \GARI(\mathcal{F})_{{\as+\abal}}\ar@{^{(}->}^{\minus}[d] \\
\DMR\cdot \exp{\Q f_{1}} \ar@{->}[rr]^{\simeq}_{\ma_{[1]}}&&\GARI(\mathcal{F})_{\as\ast\is}
}
\end{equation*}
where $\iota_0$ is the map sending $\varphi\mapsto\varphi(-f_0,f_1)$.
\end{cor}

\begin{proof}
It follows from
Theorems \ref{pentagonal characterization theorem},
\ref{thm: recovery theorem of ASTR and GRT},
\ref{thm: inclusion GARIas+pent to GARIas*is}
and
Theorem \ref{theorem: GARI and DMR} in Appendix \ref{sec: appendix GARIasis}.
\end{proof}

\subsection{$\paj$  and $\GARI(\mathcal{F}_\Lau)_{\as+\abal}$}
\label{sec: specific element paj}
In this subsection, we recall \'{E}calle's definition of $\paj\in\GARI$
(\cite[(4.72)]{E-flex})
and we prove that $\paj$ gives an element of $\GARI(\mathcal{F}_\Lau)_{\as+\abal}$
in Theorem \ref{thm: paj in GARI as+pent}.

We consider the $\mathcal K$-linear subspace $\mathcal K\langle X^{\Gamma}_{\Z} \rangle^{\rm lindep}$ of $\mathcal K\langle X^{\Gamma}_{\Z} \rangle$ generated by
\[
\left\{ \varia{u_1,\dots,u_d}{\varepsilon_1,\dots, \varepsilon_d}\in \left( X^{\Gamma}_{\Z} \right)^\bullet \ \middle|\ \mbox{$u_1,\dots,u_d$: linearly independent} \right\}.
\]

We define the $\Q$-linear map
$$
\psi:\Q\langle X^{\{xy,y\}}_{\Z} \rangle^{\rm lindep} \rightarrow \Q\langle X^{\{y\}}_{\Z} \rangle^{\rm lindep} \otimes\Q\langle X^{\{xy,x,0\}}_{\Z} \rangle^{\rm lindep}
$$
inductively by 
\begin{equation*}\label{eqn:special definition of psi}
\psi\varia{u_1,\dots,u_d}{y,\dots,y}:= \varia{u_1,\dots,u_d}{y,\dots,y}\otimes\emptyset
\end{equation*}
for $d\geq 0$, and
\begin{align}\label{eqn:general definition of psi}
&\psi(\omega_1,\dots,\omega_d)
=\psi\varia{u_1,\dots,u_d}{\varepsilon_1,\dots, \varepsilon_d} \\  \nonumber
&\qquad:=\sum_{i=1}^d\psi(\omega_1,\cdots,\omega_{i-1},\urflex{\omega_i}{\omega_{i+1}},\cdots, \omega_d) \\  \nonumber
&\qquad\qquad\cdot \Bigl( \emptyset\otimes\bigl( \delta_{\varepsilon_i,xy}(\delta_{\varepsilon_{i+1},y}+\delta_{\varepsilon_{i+1},1})\llflex{\omega_i}{\omega_{i+1}} + \delta_{\varepsilon_i,y}\delta_{\varepsilon_{i+1},xy} \left( \varia{-u_i}{x}- \varia{-u_i}{0}\right) \bigr)  \Bigr) \nonumber \\  \nonumber
&\qquad\quad -\sum_{i=2}^d\psi(\omega_1,\cdots,\ulflex{\omega_{i-1}}{\omega_i},\omega_{i+1},\cdots, \omega_d)\\
&\qquad\qquad\quad\cdot
\Bigl( \emptyset \otimes\bigl( \delta_{\varepsilon_i,xy}\delta_{\varepsilon_{i-1},y}\lrflex{\omega_{i-1}}{\omega_i} + \delta_{\varepsilon_i,y}\delta_{\varepsilon_{i-1},xy} \left( \varia{-u_i}{x}-\varia{-u_i}{0} \right) \bigr) \Bigr), \nonumber
\end{align}
for 
$(\varepsilon_1,\dots, \varepsilon_d)\neq(y,\dots,y)$.
Here we use $\omega_{d+1}=\emptyset$ and $\varepsilon_{d+1}=1$.
We note that
$$
\psi\varia{u_1,\dots,u_d}{xy,\dots,xy}:=\emptyset\otimes \varia{u_1,\dots,u_d}{xy,\dots,xy}.
$$

\begin{rem}
By rephrasing the arguments of the proof of Theorem \ref{thm: rec diff eq for I pre},
we could reformulate
\eqref{eq:I dec pre} as
$$
I(\omega)=\sum_{(\omega)}I(\omega')I(\omega''),
$$
for any word $\omega\in \Q\langle X^{\{xy,y\}}_{\Z} \rangle^{\rm lindep}$,
where we employ the Sweedler's notation of  the above $\psi$:
\begin{equation}\label{eqn:Sweedler notation of psi}
\psi(\omega)=\sum_{(\omega)}\omega'\otimes \omega''.
\end{equation}
\end{rem}


We define the $\mathcal K$-linear map
$$
\zerocut:\mathcal K\langle X^{\{xy,x,0\}}_{\Z} \rangle^{\rm lindep}\rightarrow\mathcal K\langle X^{\{xy,x\}}_{\Z} \rangle^{\rm lindep}
$$
by $\zerocut(\emptyset):=\emptyset$ and if
$(\omega_1,\cdots, \omega_d)=\varia{u_1,\dots,u_d}{\varepsilon_1,\dots, \varepsilon_d} $
is with
$\varepsilon_1,\dots, \varepsilon_d\neq0$,
\begin{align*}
\zerocut\left(\omega_1,\dots,\omega_d\right)
:=(\omega_1,\dots,\omega_d),
\end{align*}
and if there is $1\leq i\leq d$ such that $\varepsilon_i=0$ and $\varepsilon_k\neq0$ for $i<k$,
\begin{align*}
\zerocut\left(\omega_1,\dots,\omega_d\right)
:=\frac{1}{u_i}\left\{ \zerocut(\omega_1,\dots,\ulflex{\omega_i}{\omega_{i+1}},\omega_{i+2},\dots,\omega_d) - \zerocut(\omega_1,\dots,\omega_{i-2},\urflex{\omega_{i-1}}{\omega_{i}},\dots,\omega_d) \right\}.
\end{align*}
We define the $\Q$-linear map
$$
\widetilde{\psi}:\Q\langle X^{\{xy,y\}}_{\Z} \rangle^{\rm lindep} \rightarrow \Q\langle X^{\{y\}}_{\Z} \rangle^{\rm lindep} \otimes\Q\langle X^{\{xy,x\}}_{\Z} \rangle^{\rm lindep}
$$
by
$$
\widetilde{\psi}:=({\rm Id} \otimes {\zerocut})\circ\psi.
$$

By definition, we know  $\widetilde{\psi}(\emptyset)=\emptyset\otimes \emptyset$ and
\begin{equation*}
\widetilde{\psi}\varia{u_1}{xy}:=\emptyset\otimes \varia{u_1}{xy},
\hspace{10ex} \widetilde{\psi}\varia{u_1,\dots,u_d}{y,\dots,y}:= \varia{u_1,\dots,u_d}{y,\dots,y}\otimes\emptyset
\end{equation*}
for $d\geq1$.

We define the $\Q$-linear map
$$
\overline{\psi}:\Q\langle X^{\{x\}}_{\Z} \rangle^{\rm lindep} \otimes \Q\langle X^{\{xy,y\}}_{\Z} \rangle^{\rm lindep} \rightarrow \mathcal K\langle X^{\{y\}}_{\Z} \rangle^{\rm lindep}\otimes \mathcal K\langle X^{\{xy,x\}}_{\Z} \rangle^{\rm lindep}
$$
by
\begin{equation}\label{eq:overline psi eta otimes omega}
\overline{\psi}(\eta\otimes \omega):=\widetilde{\psi}(\omega)\shuffle (\emptyset\otimes \eta)
\end{equation}
for $\eta\in\left( X^{\{x\}}_{\Z} \right)^\bullet$ and $\omega\in \left( X^{\{xy,y\}}_{\Z} \right)^\bullet$.

\begin{rem}\label{rem: reform minus bal}
By rephrasing the arguments of the proof of Theorem \ref{thm: rec diff eq for I},
we could reformulate
\eqref{eq:I dec}  as 
$$
I(\eta)I(\omega)=\sum_{\bf j} c_{\bf j} \cdot I(a_{\bf j})I(b_{\bf j})
$$
for $\eta\in\left( X^{\{x\}}_{\Z} \right)^\bullet$ and $\omega\in \left( X^{\{xy,y\}}_{\Z} \right)^\bullet$
when
$$\overline{\psi}(\eta\otimes \omega)=\sum_{{\bf j}}c_{\bf j} (a_{\bf j}\otimes b_{\bf j}) \qquad (c_{\bf j}\in\mathcal K).
$$

By Remark \ref{rem: existence of bal},
the map $\bal$ in \eqref{eq:bal P4F} is described as,
for $M\in\mathcal P_4(\mathcal{F})$,
\begin{equation}\label{eq:minus bal = sum minus}
\minus(\bal(M))(\tilde\sigma(\eta);\tilde\sigma(\omega))=\sum_{{\bf j}}c_{{\bf j}}\cdot \minus(M)({\tilde\sigma}(a_{\bf j});{\tilde\sigma}(b_{\bf j}))
\end{equation}
where $\tilde\sigma$ is the identification of
$X^{\{\epsilon\}}_{\Z}$ ($\epsilon=x,y$)
and $X^{\{\epsilon_1,\epsilon_2\}}$
 ($(\epsilon_1,\epsilon_2)=(xy,x), (xy,y)$)
with  $X^{\{1\}}_{\Z}$ and $X^{\{1,2\}}_{\Z}$ respectively
under the notation \eqref{eq: 1-2 rule}.
\end{rem}

\begin{defn}
The mould $\paj \in \GARI(\mathcal{F}_\Lau)$ is defined by $\paj(\emptyset):=1$ and
$$
\paj(x_1,\dots,x_m)
:=\frac{1}{x_1(x_1+x_2)\cdots(x_1+\cdots+x_m)}
$$
for $m\geq1$.
\end{defn}

\begin{rem}
We put $\pic:=\swap(\paj)\in \overline{\GARI}(\mathcal{F}_\Lau)$.
Then $\pic$ is explicitly given by
$$
\pic({x_1,\dots,x_m})
:=\frac{1}{x_1x_2\cdots x_m}
$$
for $m\geq1$.
\end{rem}

In the beginning of \cite[\S 4.3]{E-flex},
it is stated that the mould $\paj$ is symmetral and $\pic=\swap(\paj)$ is symmetril
though its proof does not look presented there.
We give a complete proof below:

\begin{prop}\label{prop:paj in GRAI asis}
We have $\paj\in\GARI(\mathcal{F}_\Lau)_{\as*\is}$.
\end{prop}


\begin{proof}
We first prove that $\paj$ is symmetral.
Because we have $\paj(\emptyset)=1$, it suffices to prove
\begin{equation*}
\sum_{\alpha\in X_\Z^\bullet}\Sh{\omega}{\eta}{\alpha}\paj(\alpha)=\paj(\omega)\paj(\eta),
\end{equation*}
for $\omega,\eta\in X_\Z^\bullet$ with $l(\omega),l(\eta)\geq1$.
We prove this by induction on $r=l(\omega)+l(\eta)$ ($\geq2$) below.
The case of $l(\omega)=l(\eta)=1$ is obvious.
Assume that it holds for the case of $r\ (\geq2)$.
For $\omega=(u,x)$ and $\eta=(v,y)$ with $u,v\in X_\Z^\bullet$ and $x,y\in X_\Z$, we have
\begin{align*}
&\sum_{\alpha\in X_\Z^\bullet}\Sh{(u,x)}{(v,y)}{\alpha}\paj(\alpha) \\
&=\sum_{\alpha\in X_\Z^\bullet}
\left\{\Sh{u}{(v,y)}{\alpha}\paj(\alpha,x)+\Sh{(u,x)}{v}{\alpha}\paj(\alpha,y)\right\}.
\intertext{For $\omega=(x_1,\dots,x_r)$ ($x_i\in X$), put $|\omega|:=x_1+\cdots+x_r$.
Then we have $\paj(\omega)=\frac{1}{|\omega|}\paj(x_1,\dots,x_{r-1})$.
So we get }
&=\sum_{\alpha\in X_\Z^\bullet}
\left\{\Sh{u}{(v,y)}{\alpha}\frac{1}{|\alpha|+x}\paj(\alpha)+\Sh{(u,x)}{v}{\alpha}\frac{1}{|\alpha|+y}\paj(\alpha)\right\}.
\intertext{When $\Sh{\omega}{\eta}{\alpha}\neq0$ holds, we have $|\alpha|=|\omega|+|\eta|$. Therefore, we calculate}
&=\frac{1}{|u|+|v|+x+y}\sum_{\alpha\in X_\Z^\bullet}
\left\{\Sh{u}{(v,y)}{\alpha}\paj(\alpha)+\Sh{(u,x)}{v}{\alpha}\paj(\alpha)\right\}.
\intertext{By induction hypothesis, we get}
&=\frac{1}{|u|+|v|+x+y}\left\{\paj(u)\paj(v,y)+\paj(u,x)\paj(v)\right\} \\
&=\frac{1}{|u|+|v|+x+y}\left\{\frac{1}{|v|+y}\paj(u)\paj(v)+\frac{1}{|u|+x}\paj(u)\paj(v)\right\} \\
&=\frac{1}{|u|+x}\frac{1}{|v|+y}\paj(u)\paj(v) \\
&=\paj(u,x)\paj(v,y).
\end{align*}
Hence, $\paj$ is symmetral.

We next prove that $\pic=\swap(\paj)$ is symmetril by induction on $r=l(\omega)+l(\eta)\geq2$.
When $l(\omega)=l(\eta)=1$, i.e., $\omega=x_1$ and $\eta=x_2$, we have
\begin{align*}
\pic(x_1,x_2)+\pic(x_2,x_1)+\frac{1}{x_1-x_2}\left\{\pic(x_1)-\pic(x_2)\right\}
=&\frac{2}{x_1x_2}+\frac{1}{x_1-x_2}\left(\frac{1}{x_1}-\frac{1}{x_2}\right) \\
=&\frac{1}{x_1x_2} \\
=&\pic(x_1)\pic(x_2).
\end{align*}
Assume that symmetrility of $\pic$ holds when the total length is less than $r(\geq2)$.
For $r=l(\omega)+l(\eta)$ with $l(\omega),l(\eta)\geq1$, put $\omega=xu$ and $\eta=yv$ ($x,y\in  Y_\Z$ and $u,v\in  Y_\Z^\bullet$).
Then we have
\begin{align*}
&\sum_{\alpha\in  Y_\Z^\bullet}\Shstar{xu}{yv}{\alpha}\pic(\alpha) \\
&=\sum_{\alpha\in Y_\Z^\bullet}
\left[\Shstar{u}{yv}{\alpha}\pic(x,\alpha)
+\Shstar{xu}{v}{\alpha}\pic(y,\alpha)\right. \\
&\hspace{5cm}\left.+\frac{1}{x-y}\Shstar{u}{v}{\alpha}\left\{\pic(x,\alpha)-\pic(y,\alpha)\right\}\right].
\intertext{By definition, we have $\pic(x_1,\dots,x_r)=\frac{1}{x_1}\pic(x_2,\dots,x_r)$ ($r\geq1$). So we get}
&=\frac{1}{x}\sum_{\alpha\in  Y_\Z^\bullet}\Shstar{u}{yv}{\alpha}\pic(\alpha)
+\frac{1}{y}\sum_{\alpha\in  Y_\Z^\bullet}\Shstar{xu}{v}{\alpha}\pic(\alpha) \\
&\hspace{5cm}+\frac{1}{x-y}\left(\frac{1}{x}-\frac{1}{y}\right)\sum_{\alpha\in Y_\Z^\bullet}\Shstar{u}{v}{\alpha}\pic(\alpha).
\intertext{By induction hypothesis, we have}
&=\frac{1}{x}\pic(u)\pic(y,v)
+\frac{1}{y}\pic(x,u)\pic(v)
-\frac{1}{xy}\pic(u)\pic(v) \\
&=\pic(x,u)\pic(y,v).
\end{align*}
Therefore, $\pic$ is symmetril.
Hence, we obtain $\paj\in\GARI(\mathcal{F}_\Lau)_{\as*\is}$.
\end{proof}

We define $\rho=\rho_{\Gamma}:\mathcal{K}\langle X_{\Z}^{\Gamma}\rangle\to\mathcal{K}\langle X_{\Z}^{\{1\}}\rangle$
by
\[
\rho\binom{u_{1},\dots,u_{d}}{\epsilon_{1},\dots,\epsilon_{d}}=\binom{u_{1},\dots,u_{d}}{1,\dots,1}.
\]

\begin{lem}\label{lem:paj_reduce}
For $\omega\in\mathcal{K}\langle X_{\Z}^{\{xy,x,0\}}\rangle^{{\rm lindep}}$,
we have
\[
\paj_{\{xy,x\}}(\zerocut(\omega))=\paj(\rho(w)).
\]
\end{lem}

\begin{proof}
Put $\omega=\binom{u_{1},\dots,u_{d}}{\epsilon_{1},\dots,\epsilon_{d}}$.
We prove the claim by induction on $\#\{i\mid\epsilon_{i}=0\}$. The
case when $\{i\mid\epsilon_{i}=0\}$ is empty follows from definition.
Let $i$ be a minimum such that $\epsilon_{i}=0$. Then by definition,
\begin{align*}
\paj{}_{\{xy,x\}}(\zerocut(\omega)) & =\frac{1}{u_{i}}\paj{}_{\{xy,x\}}\left(\zerocut\binom{u_{1},\dots u_{i-1},u_{i}+u_{i+1},u_{i+2},\dots,u_{d}}{\epsilon_{1},\dots,\epsilon_{i-1,}\ \ \ \ \epsilon_{i+1},\epsilon_{i+2},\dots,\epsilon_{d}}\right)\\
 & \quad-\frac{1}{u_{i}}\paj{}_{\{xy,x\}}\left(\zerocut\binom{u_{1},\dots u_{i-2},u_{i-1}+u_{i},u_{i+1},\dots,u_{d}}{\epsilon_{1},\dots,\epsilon_{i-2,}\ \ \ \ \epsilon_{i-1},\epsilon_{i+1},\dots,\epsilon_{d}}\right),
\end{align*}
and by induction hypothesis,
\begin{align*}
 & =\frac{1}{u_{i}}\paj\left(\rho\binom{u_{1},\dots u_{i-1},u_{i}+u_{i+1},u_{i+2},\dots,u_{d}}{\epsilon_{1},\dots,\epsilon_{i-1,}\ \ \ \ \epsilon_{i+1},\epsilon_{i+2},\dots,\epsilon_{d}}\right)\\
 & \quad-\frac{1}{u_{i}}\paj\left(\rho\binom{u_{1},\dots u_{i-2},u_{i-1}+u_{i},u_{i+1},\dots,u_{d}}{\epsilon_{1},\dots,\epsilon_{i-2,}\ \ \ \ \epsilon_{i-1},\epsilon_{i+1},\dots,\epsilon_{d}}\right)\\
 & =\frac{1}{u_{i}}\left((u_{1}+\cdots+u_{i})-(u_{1}+\cdots+u_{i-1})\right)\paj(\rho(\omega))\\
 & =\paj(\rho(\omega))
\end{align*}
which completes the proof.
\end{proof}

\begin{thm}\label{thm: paj in GARI as+pent}
$\minus(\paj)\in\GARI(\mathcal{F}_{\Lau})_{\as+\abal}$.
\end{thm}

\begin{proof}
Since we know that $\minus(\paj)\in\GARI(\mathcal{F}_{\Lau})_{\as}$
from Proposition \ref{prop:paj in GRAI asis},
our claim is reduced to prove that $\minus(\paj)$ is balanced.
Actually  we prove  below that it is well-balanced, that is,
$$
\bal (\minus(\paj)_{[1]}\otimes \minus(\paj)_{[2]})=\minus(\paj)_{[1]}\otimes \minus(\paj)_{[2]}.
$$
By applying $\minus$ to both sides, we see that it is equivalent to
\begin{equation*}
\minus \circ\bal (\minus(\paj)_{[1]}\otimes \minus(\paj)_{[2]})=\minus\circ(\minus(\paj)_{[1]}\otimes \minus(\paj)_{[2]}).
\end{equation*}
Thus it is enough to show
\begin{align*}
\minus  \circ\bal & (\minus(\paj)_{[1]}\otimes \minus(\paj)_{[2]})({\tilde\sigma}(\eta);{\tilde\sigma}(\omega)) \\
&=
\minus\circ(\minus(\paj)_{[1]}\otimes \minus(\paj)_{[2]})({\tilde\sigma}(\eta);{\tilde\sigma}(\omega)).
\end{align*}
By Remark \ref{rem: reform minus bal}, the left hand side
$$
\minus\circ\bal (\minus(\paj)_{[1]}\otimes \minus(\paj)_{[2]})({\tilde\sigma}(\eta);{\tilde\sigma}(\omega) ))
$$
is equal to
$$
(\paj_{[1]}\otimes \paj_{[2]})(\bar{\psi}({\tilde\sigma}(\eta)\otimes {\tilde\sigma}(\omega) )) = (\paj_{\{y\}}\otimes \paj_{\{xy,x\}})(\bar{\psi}(\eta\otimes\omega)).
$$
While the right hand side is equal to
$$
\paj_{\{x\}}(\eta)\paj_{\{xy,y\}}(w).
$$

Therefore our claim is reduced to show
\begin{equation*}\label{eq: paj paj=paj paj}
(\paj_{\{y\}}\otimes\paj_{\{xy,x\}})(\bar{\psi}(\eta\otimes w))
=\paj_{\{x\}}(\eta)\paj_{\{xy,y\}}(w).
\end{equation*}
By \eqref{eq:overline psi eta otimes omega},
the left hand side is calculated to be
\begin{align*}
  (\paj_{\{y\}}&\otimes\paj_{\{xy,x\}})(\bar{\psi}(\eta\otimes w))
  =(\paj_{\{y\}}\otimes\paj_{\{xy,x\}})(\tilde{\psi}(w)\shuffle(\emptyset\otimes\eta)).
\end{align*}
By $\paj\in\GARI(\mathcal{F}_{\Lau})_{\as}$,
\begin{align*}
 & =\paj_{\{xy,x\}}(\eta)(\paj_{\{y\}}\otimes\paj_{\{xy,x\}})(\tilde{\psi}(w))\\
  & =\paj_{\{xy,x\}}(\eta)(\paj_{\{y\}}\otimes\paj_{\{xy,x\}})((\mathrm{Id}\otimes\zerocut)\psi(w)),\\
\intertext{and by Lemma \ref{lem:paj_reduce},}
&
=\paj_{\{xy,x\}}(\eta)(\paj_{\{y\}}\otimes\paj)\left((\mathrm{Id}\otimes\rho)\psi(w)\right).
\end{align*}

We have  $\paj_{\{x\}}(\eta)=\paj_{\{xy,x\}}(\eta)$ by definition.
Thus our claim 
is deduced to show
\[
(\paj_{\{y\}}\otimes\paj)(\psi_{\mathrm{red}}(w))
=\paj_{\{xy,y\}}(w)
\]
where we put $\psi_{\mathrm{red}}=(\mathrm{Id}\otimes\rho)\circ\psi$.
Then, by \eqref{eqn:general definition of psi}, we have
\begin{align*}
 & \psi_{\mathrm{red}}(\omega_{1},\dots,\omega_{d})=\psi_{\mathrm{red}}\varia{u_{1},\dots,u_{d}}{\varepsilon_{1},\dots,\varepsilon_{d}}\\
 & \quad =\sum_{\substack{1\leq i\leq d\\
\epsilon_{i}=xy,\epsilon_{i+1}\in\{y,1\}
}
}\psi_{\mathrm{red}}(\omega_{1},\cdots,\omega_{i-1},\urflex{\omega_{i}}{\omega_{i+1}},\cdots,\omega_{d})\cdot\Bigl(\emptyset\otimes\binom{u_{i}}{1}\Bigr)\\
 & \qquad\qquad -\sum_{\substack{2\leq i\leq d\\
\epsilon_{i}=xy,\epsilon_{i-1}=y
}
}\psi_{\mathrm{red}}(\omega_{1},\cdots,\ulflex{\omega_{i-1}}{\omega_{i}},\omega_{i+1},\cdots,\omega_{d})\cdot\Bigl(\emptyset\otimes\binom{u_{i}}{1}\Bigr)
\end{align*}
when
$(\omega_1,\cdots, \omega_d)=\varia{u_1,\dots,u_d}{\varepsilon_1,\dots, \varepsilon_d} $
is with
$(\epsilon_{1},\dots,\epsilon_{d})\neq(y,\dots,y)$.
Then we can easily check that
\[
\psi_{\mathrm{red}}\varia{u_{1},\dots,u_{d}}{\varepsilon_{1},\dots,\varepsilon_{d}}\in\mathbb{Q}\langle X_{\Z}^{\{y\}}\rangle\otimes\mathbb{Q}\langle X_{\Z}^{\{1\},a}\rangle
\]
where
$$a=\sum_{\substack{1\leq i\leq d,\omega_{i}=xy}}u_{i}$$
and
$\Q\langle X_{\Z}^{\{1\},a}\rangle$ is the $\Q$-linear subspace spanned by
$\varia{v_{1},\dots,v_{d}}{\varepsilon_{1},\dots,\varepsilon_{d}}$
with $v_{1}+\cdots+v_{d}=a$.

Let us show the claim
\[
(\paj_{\{y\}}\otimes\paj)\left(\psi_{\mathrm{red}}\varia{u_{1},\dots,u_{d}}{\varepsilon_{1},\dots,\varepsilon_{d}}\right)
=\paj_{\{xy,y\}}\varia{u_{1},\dots,u_{d}}{\varepsilon_{1},\dots,\varepsilon_{d}}
\]
by induction on $d$. The case $(\epsilon_{1},\dots,\epsilon_{d})=(y,\dots,y)$
follows from definition. Assume that $(\epsilon_{1},\dots,\epsilon_{d})\neq(y,\dots,y)$.
Then we have
\begin{align*}
&  (\paj_{\{y\}}\otimes\paj)\left(\psi_{\mathrm{red}}\varia{u_{1},\dots,u_{d}}{\varepsilon_{1},\dots,\varepsilon_{d}}\right)\\
 & =\sum_{\substack{1\leq i\leq d\\
\epsilon_{i}=xy,\epsilon_{i+1}\in\{y,1\}
}
}(\paj_{\{y\}}\otimes\paj)\left(\psi_{\mathrm{red}}(\omega_{1},\cdots,\omega_{i-1},\urflex{\omega_{i}}{\omega_{i+1}},\cdots,\omega_{d})\cdot\Bigl(\emptyset\otimes\binom{u_{i}}{1}\Bigr)\right)\\
 & \qquad\quad  -\sum_{\substack{2\leq i\leq d\\
\epsilon_{i}=xy,\epsilon_{i-1}=y
}
}(\paj_{\{y\}}\otimes\paj)\left(\psi_{\mathrm{red}}(\omega_{1},\cdots,\ulflex{\omega_{i-1}}{\omega_{i}},\omega_{i+1},\cdots,\omega_{d})\cdot\Bigl(\emptyset\otimes\binom{u_{i}}{1}\Bigr)\right),
\end{align*}
and since $\psi_{\mathrm{red}}(\omega_{1},\cdots,\omega_{i-1},\urflex{\omega_{i}}{\omega_{i+1}},\cdots,\omega_{d})$
and $\psi_{\mathrm{red}}(\omega_{1},\cdots,\ulflex{\omega_{i-1}}{\omega_{i}},\omega_{i+1},\cdots,\omega_{d})$
are in $\mathbb{Q}\langle X_{\Z}^{\{y\}}\rangle\otimes\mathbb{Q}\langle X_{\Z}^{\{1\},a-u_{i}}\rangle$
with $a=\sum_{\epsilon_{j}=xy}u_{j}$,
\begin{align*}
 & =\sum_{\substack{1\leq i\leq d\\
\epsilon_{i}=xy,\epsilon_{i+1}\in\{y,1\}
}
}(\paj_{\{y\}}\otimes\paj)\left(\psi_{\mathrm{red}}(\omega_{1},\cdots,\omega_{i-1},\urflex{\omega_{i}}{\omega_{i+1}},\cdots,\omega_{d})\right)\frac{1}{\sum_{\epsilon_{j}=xy}u_{j}}\\
 & \qquad -\sum_{\substack{2\leq i\leq d\\
\epsilon_{i}=xy,\epsilon_{i-1}=y
}
}(\paj_{\{y\}}\otimes\paj)\left(\psi_{\mathrm{red}}(\omega_{1},\cdots,\ulflex{\omega_{i-1}}{\omega_{i}},\omega_{i+1},\cdots,\omega_{d})\right)\frac{1}{\sum_{\epsilon_{j}=xy}u_{j}},
\end{align*}
and by induction hypothesis,
\begin{align*}
 & =\sum_{\substack{1\leq i\leq d\\
\epsilon_{i}=xy,\epsilon_{i+1}\in\{y,1\}
}
}\paj_{\{xy,y\}}\left(\omega_{1},\cdots,\omega_{i-1},\urflex{\omega_{i}}{\omega_{i+1}},\cdots,\omega_{d}\right)\frac{1}{\sum_{\epsilon_{j}=xy}u_{j}}\\
 & \qquad -\sum_{\substack{2\leq i\leq d\\
\epsilon_{i}=xy,\epsilon_{i-1}=y
}
}\paj_{\{xy,y\}}\left(\omega_{1},\cdots,\ulflex{\omega_{i-1}}{\omega_{i}},\omega_{i+1},\cdots,\omega_{d}\right)\frac{1}{\sum_{\epsilon_{j}=xy}u_{j}}\\
 & =\sum_{\substack{1\leq i\leq d\\
\epsilon_{i}=xy,\epsilon_{i+1}\in\{y,1\}
}
}\frac{u_{1}+\cdots+u_{i}}{u_{1}(u_{1}+u_{2})\cdots(u_{1}+\cdots+u_{d})}\frac{1}{\sum_{\epsilon_{j}=xy}u_{j}}\\
 & \qquad -\sum_{\substack{2\leq i\leq d\\
\epsilon_{i}=xy,\epsilon_{i-1}=y
}
}\frac{u_{1}+\cdots+u_{i-1}}{u_{1}(u_{1}+u_{2})\cdots(u_{1}+\cdots+u_{d})}\frac{1}{\sum_{\epsilon_{j}=xy}u_{j}}\\
 & =\sum_{\substack{1\leq i\leq d\\
\epsilon_{i}=xy
}
}\frac{(u_{1}+\cdots+u_{i})-(u_{1}+\cdots+u_{i-1})}{u_{1}(u_{1}+u_{2})\cdots(u_{1}+\cdots+u_{d})}\frac{1}{\sum_{\epsilon_{j}=xy}u_{j}}\\
 &=\paj_{\{xy,y\}}\varia{u_{1},\dots,u_{d}}{\varepsilon_{1},\dots,\varepsilon_{d}},
\end{align*}
which completes the proof.
\end{proof}

We remark that actually
the statement of Proposition \ref{prop:paj in GRAI asis}
can also be obtained by
Theorems \ref{thm: inclusion GARIas+pent to GARIas*is} and
\ref{thm: paj in GARI as+pent}.

\appendix
\section{Correspondence between $\exp^\circledast$ and $\expari$ under $\ma_\Gamma$}
\label{app:correspondence between exp and expari}


In this appendix, we give complete proofs of the fact
that the Lie bracket $\{, \}$ corresponds to the $\ari$-bracket in Theorem \ref{thm:ma(<,>)=ari(ma,ma)}
and the fact
that Racinet's exponential map $\exp^\circledast$ corresponds to \'{E}calle's $\expari$ under the map
$\ma_\Gamma$ in Theorem \ref{thm:ma(exp)=expari(ma)},
both of which are sometimes implicitly used in the literature
without proof.

In this appendix, we assume $\Gamma$ is an abelian group.
For $\psi\in\widehat{U{\mathfrak f}_\Gamma}$, we define the derivation $D_\psi$ on $\widehat{U{\mathfrak f}_\Gamma}$ by
$$
D_\psi(f_0):=0, \qquad\qquad
{D_\psi(f_\sigma):=[t_\sigma(\psi), f_\sigma] \quad(\sigma\in\Gamma),}
$$
where $t_\sigma$ is the group action of $\Gamma$ on $\widehat{U{\mathfrak f}_\Gamma}$ defined by
$$
t_\sigma(f_0):=f_0, \qquad\qquad
t_\sigma(f_\tau):=f_{\sigma\tau} \quad(\tau\in\Gamma).
$$
We put {$s_\varphi(\psi):=\psi\varphi + D_\varphi(\psi)$} and we define the bracket
\begin{equation*}\label{eq:Ihara bracket on UfGamma}
\{\psi,  \varphi \}
:= s_\psi(\varphi) - s_\varphi(\psi)
=[\varphi,\psi]  +  D_\psi(\varphi) - D_\varphi(\psi) 
\end{equation*}
for $\varphi,\psi\in\widehat{U{\mathfrak f}_\Gamma}$.
We note that it agrees with \eqref{eq: Ihara bracket} when $\Gamma=\{e\}$.
The $\Q$-linear space $(\widehat{U\frak f_{\Gamma}}^\dag)_0$ (resp. $(\widehat{U\frak f_{\Gamma}}^\dag)_1$)
 is defined to be the subset of elements in $\widehat{U\frak f_{\Gamma}}^\dag$ whose constant term is equal to 0 (resp. 1).
In \cite{R} \S 3.1,
the exponential map $\exp^\circledast:(\widehat{U\frak f_{\Gamma}}^\dag)_0 \rightarrow (\widehat{U\frak f_{\Gamma}}^\dag)_1$
is  defined by
$$
\exp^\circledast(\varphi)
:=\sum_{k\geq0}\frac{1}{k!}s_\varphi^k(1)
=1 + \varphi + \frac{1}{2!}s_\varphi(\varphi) + \cdots.
$$

\begin{prop}\label{prop:ma_D=arit'ma)(ma)}
For $\mathcal F={\mathcal F}_\ser$ and for $\psi\in\widehat{U{\mathfrak f}_\Gamma}$,
we have the following commutative diagram:
\begin{equation}\label{eq:CD ma_D=arit'ma)(ma)}
\xymatrix{
\widehat{U\frak f_{\Gamma}}^\dag \ar@{->}[rr]^{\simeq}_{\ma_{\Gamma}}\ar@{->}[d]_{D_\psi} && \mathcal{M}(\mathcal F;\Gamma) \ar@{->}[d]^{\arit(\ma_\Gamma(\psi))} \\
\widehat{U\frak f_{\Gamma}}^\dag  \ar@{->}[rr]^{\simeq}_{\ma_{\Gamma}}&&\mathcal{M}(\mathcal F;\Gamma).
}
\end{equation}
That is, we have $\ma_{\Gamma} \circ D_\psi = \arit(\ma_\Gamma(\psi)) \circ \ma_{\Gamma}$.
\end{prop}

\begin{proof}
We consider elements $\varphi,\psi\in \widehat{U{\mathfrak f}_\Gamma}^\dag$
given by
\begin{align*}
\varphi
&:= \sum_{s\geq0}\sum_{\sigma_i\in\Gamma}\sum_{k_i\in\N_0}
\coeff{\varphi}{k_0,\dots,k_s}{\sigma_1,\dots,\sigma_s}f_0^{k_0}f_{\sigma_1}\cdots f_{\sigma_s}f_0^{k_s}, \\
\psi
&:= \sum_{r\geq0}\sum_{\tau_i\in\Gamma}\sum_{l_i\in\N_0}
\coeff{\psi}{l_0,\dots,l_r}{\tau_1,\dots,\tau_r}f_0^{l_0}f_{\tau_1}\cdots f_{\tau_r}f_0^{l_r}.
\end{align*}
We have
\footnote{Because we have $D_\psi(f_0^k)=0$ for $k\geq0$ by the definition of $D_\psi$, the case of $s=0$ does not appear in the summation of $D_\psi(\varphi)$.}
\begin{align*}
D_\psi(\varphi)
&= \sum_{s\geq1}\sum_{\sigma_i\in\Gamma}\sum_{k_i\in\N_0}
\coeff{\varphi}{k_0,\dots,k_s}{\sigma_1,\dots,\sigma_s}
\sum_{p=1}^s f_0^{k_0}f_{\sigma_1}\cdots f_0^{k_{p-1}}[t_{\sigma_p}(\psi),f_{\sigma_p}]f_0^{k_{p}}f_{\sigma_{p+1}} \cdots f_{\sigma_s}f_0^{k_s} \\
&= \sum_{s\geq p\geq1}\sum_{\sigma_i\in\Gamma}\sum_{k_i\in\N_0}
\coeff{\varphi}{k_0,\dots,k_s}{\sigma_1,\dots,\sigma_s}
f_0^{k_0}f_{\sigma_1}\cdots f_0^{k_{p-1}}[t_{\sigma_p}(\psi),f_{\sigma_p}]f_0^{k_{p}}f_{\sigma_{p+1}} \cdots f_{\sigma_s}f_0^{k_s} \\
&=L_\psi(\varphi) - R_\psi(\varphi),
\end{align*}
where we put
\begin{align*}
R_\psi(\varphi)
&:= \sum_{s\geq p\geq1}\sum_{\sigma_i\in\Gamma}\sum_{k_i\in\N_0}
\coeff{\varphi}{k_0,\dots,k_s}{\sigma_1,\dots,\sigma_s}
f_0^{k_0}f_{\sigma_1}\cdots f_0^{k_{p-1}}f_{\sigma_p}t_{\sigma_p}(\psi)f_0^{k_{p}}f_{\sigma_{p+1}} \cdots f_{\sigma_s}f_0^{k_s}, \\
L_\psi(\varphi)
&:= \sum_{s\geq p\geq1}\sum_{\sigma_i\in\Gamma}\sum_{k_i\in\N_0}
\coeff{\varphi}{k_0,\dots,k_s}{\sigma_1,\dots,\sigma_s}
f_0^{k_0}f_{\sigma_1}\cdots f_0^{k_{p-1}}t_{\sigma_p}(\psi)f_{\sigma_p}f_0^{k_{p}}f_{\sigma_{p+1}} \cdots f_{\sigma_s}f_0^{k_s}.
\end{align*}
We calculate $R_\psi(\varphi)$ in the following:
\begin{align*}
&R_\psi(\varphi) \\
&= \sum_{\substack{s\geq p\geq1 \\ r\geq0}}\sum_{\sigma_i,\tau_i\in\Gamma}\sum_{k_i,l_i\in\N_0}
\coeff{\varphi}{k_0,\dots,k_s}{\sigma_1,\dots,\sigma_s}\coeff{\psi}{l_0,\dots,l_r}{\tau_1,\dots,\tau_r} \\
&\hspace{3cm}\cdot  f_0^{k_0}f_{\sigma_1}\cdots f_0^{k_{p-1}}f_{\sigma_p}
\left( f_0^{l_0}f_{\sigma_p\tau_1}\cdots f_0^{l_{r-1}}f_{\sigma_p\tau_r}f_0^{l_r} \right) f_0^{k_{p}}f_{\sigma_{p+1}} \cdots f_{\sigma_s}f_0^{k_s}.
\intertext{By replacing $s=p+q$ ($q\geq0$), we have}
&= \sum_{\substack{q,r\geq0 \\ p\geq1}}\sum_{\sigma_i,\tau_i\in\Gamma}\sum_{k_i,l_i\in\N_0}
\coeff{\varphi}{k_0,\dots,k_{p+q}}{\sigma_1,\dots,\sigma_{p+q}}\coeff{\psi}{l_0,\dots,l_r}{\tau_1,\dots,\tau_r} \\
&\hspace{3cm}\cdot  f_0^{k_0}f_{\sigma_1}\cdots f_0^{k_{p-1}}f_{\sigma_p}
f_0^{l_0}f_{\sigma_p\tau_1}\cdots f_0^{l_{r-1}}f_{\sigma_p\tau_r}f_0^{l_r+k_{p}}f_{\sigma_{p+1}} \cdots f_{\sigma_{p+q}}f_0^{k_{p+q}} \\
&= \sum_{m\geq1}\sum_{\substack{p+q+r=m \\ q,r\geq0,\ p\geq1}}\sum_{\sigma_i,\tau_i\in\Gamma}\sum_{k_i,l_i\in\N_0}
\coeff{\varphi}{k_0,\dots,k_{p+q}}{\sigma_1,\dots,\sigma_{p+q}}\coeff{\psi}{l_0,\dots,l_r}{\sigma_p^{-1}\tau_1,\dots,\sigma_p^{-1}\tau_r} \\
&\hspace{3cm}\cdot  f_0^{k_0}f_{\sigma_1}\cdots f_0^{k_{p-1}}f_{\sigma_p}
f_0^{l_0}f_{\tau_1}\cdots f_0^{l_{r-1}}f_{\tau_r}f_0^{l_r+k_{p}}f_{\sigma_{p+1}} \cdots f_{\sigma_{p+q}}f_0^{k_{p+q}}.
\intertext{By replacing $\tau_i\mapsto\sigma_{p+i}$ ($1\leq i\leq r$) and $\sigma_{p+j}\mapsto\sigma_{p+j+r}$ ($1\leq j\leq q$), we get}
&= \sum_{m\geq1}\sum_{\substack{p+q+r=m \\ q,r\geq0,\ p\geq1}}\sum_{\sigma_i\in\Gamma}\sum_{k_i,l_i\in\N_0}
\coeff{\varphi}{k_0,\dots,k_{p+q}}{\sigma_1,\dots,\sigma_{p},\sigma_{p+r+1},\dots,\sigma_{m}}\coeff{\psi}{l_0,\dots,l_r}{\sigma_p^{-1}\sigma_{p+1},\dots,\sigma_p^{-1}\sigma_{p+r}}
 \\
&\hspace{3cm}\cdot  f_0^{k_0}f_{\sigma_1}\cdots f_0^{k_{p-1}}f_{\sigma_p}
f_0^{l_0}f_{\sigma_{p+1}}\cdots f_0^{l_{r-1}}f_{\sigma_{p+r}}f_0^{l_r+k_{p}}f_{\sigma_{p+r+1}} \cdots f_{\sigma_{m}}f_0^{k_{p+q}}.
\end{align*}
So, for $m\geq1$ and for $\sigma_1,\dots,\sigma_m\in\Gamma$, we get
\begin{align*}
&\vimo^m_{\Gamma, R_\psi(\varphi)}\varia{z_0,\dots,z_m}{\sigma_1,\dots,\sigma_m} \\
&= \sum_{\substack{p+q+r=m \\ q,r\geq0,\ p\geq1}}\sum_{k_i,l_i\in\N_0}
\coeff{\varphi}{k_0,\dots,k_{p+q}}{\sigma_1,\dots,\sigma_{p},\sigma_{p+r+1},\dots,\sigma_{m}}\coeff{\psi}{l_0,\dots,l_r}{\sigma_p^{-1}\sigma_{p+1},\dots,\sigma_p^{-1}\sigma_{p+r}} \\
&\hspace{4cm}\cdot  z_0^{k_0}\cdots z_{p-1}^{k_{p-1}}
z_p^{l_0}\cdots z_{p+r-1}^{l_{r-1}}z_{p+r}^{l_r+k_{p}}
z_{p+r+1}^{k_{p+1}} \cdots z_m^{k_{p+q}} \\
&= \sum_{\substack{p+q+r=m \\ q,r\geq0,\ p\geq1}}
\left\{\sum_{k_i\in\N_0}
\coeff{\varphi}{k_0,\dots,k_{p+q}}{\sigma_1,\dots,\sigma_{p},\sigma_{p+r+1},\dots,\sigma_{m}}
z_0^{k_0}\cdots z_{p-1}^{k_{p-1}}
z_{p+r}^{k_{p}}
z_{p+r+1}^{k_{p+1}} \cdots z_m^{k_{p+q}} \right\} \\
&\hspace{4cm}\cdot
\left\{\sum_{l_i\in\N_0}\coeff{\psi}{l_0,\dots,l_r}{\sigma_p^{-1}\sigma_{p+1},\dots,\sigma_p^{-1}\sigma_{p+r}}
z_p^{l_0}\cdots z_{p+r-1}^{l_{r-1}}z_{p+r}^{l_r}
\right\} \\
&= \sum_{\substack{p+q+r=m \\ q,r\geq0,\ p\geq1}}
\vimo^{p+q}_{\Gamma, \varphi}\varia{z_0,\cdots, z_{p-1},z_{p+r},z_{p+r+1}, \cdots, z_m}{\sigma_1,\dots,\sigma_{p},\sigma_{p+r+1},\dots,\sigma_{m}}
\vimo^r_{\Gamma, \psi}\varia{z_p, z_{p+1}, \cdots, z_{p+r}}{\sigma_p^{-1}\sigma_{p+1},\dots,\sigma_p^{-1}\sigma_{p+r}}.
\intertext{By using Lemma \ref{lem:UfGamma} for $\alpha=-z_p$, we obtain}
&= \sum_{\substack{p+q+r=m \\ q,r\geq0,\ p\geq1}}
\vimo^{p+q}_{\Gamma, \varphi}\varia{z_0,\cdots, z_{p-1},z_{p+r},z_{p+r+1}, \cdots, z_m}{\sigma_1,\dots,\sigma_{p},\sigma_{p+r+1},\dots,\sigma_{m}}
\vimo^r_{\Gamma, \psi}\varia{0, z_{p+1}-z_p, \cdots, z_{p+r}-z_p}{\sigma_p^{-1}\sigma_{p+1},\dots,\sigma_p^{-1}\sigma_{p+r}}.
\end{align*}
Therefore, by substituting $z_0=0$ and $z_i=x_1+\cdots+x_i$ ($1\leq i\leq m$), we get
\begin{align*}
&\ma^m_{\Gamma, R_\psi(\varphi)}(\vecx_m)
=\ma^m_{\Gamma, R_\psi(\varphi)}\varia{x_1,\dots,x_m}{\sigma_1,\dots,\sigma_m} \\
&= \sum_{\substack{p+q+r=m \\ q,r\geq0,\ p\geq1}}
\ma^{p+q}_{\Gamma, \varphi}\varia{x_1,\cdots, x_{p-1},x_p+\cdots +x_{p+r},x_{p+r+1}, \cdots, x_m}{\sigma_1,\dots,\sigma_{p-1},\sigma_{p+r},\sigma_{p+r+1},\dots,\sigma_{m}}
\ma^r_{\Gamma, \psi}\varia{x_{p+1}, \cdots, x_{p+r}}{\sigma_p^{-1}\sigma_{p+1},\dots,\sigma_p^{-1}\sigma_{p+r}}.
\intertext{By putting $\omega_i=\varia{x_i}{\sigma_i}$, we calculate}
&= \sum_{\substack{p+q+r=m \\ q,r\geq0,\ p\geq1}}
\ma^{p+q}_{\Gamma, \varphi}(\omega_1,\cdots, \omega_{p-1},\ulflex{\omega_p}{(\omega_{p+1},\dots ,\omega_{p+r})},\omega_{p+r+1}, \cdots, \omega_m) \\
&\qquad\qquad\qquad\qquad\qquad\qquad\qquad\qquad\qquad\qquad\qquad
\cdot\ma^r_{\Gamma, \psi}(\lrflex{\omega_p}{(\omega_{p+1},\dots ,\omega_{p+r})}).
\end{align*}By putting $\alpha=(\omega_1,\cdots,\omega_p)$, $\beta=(\omega_{p+1},\cdots,\omega_{p+r})$ and $\gamma=(\omega_{p+r+1},\cdots,\omega_m)$, we obtain
\begin{align}\label{eq:ma(Rpsi(phi))}
\ma^m_{\Gamma, R_\psi(\varphi)}(\vecx_m)
= \sum_{\substack{\vecx_m=\alpha\beta\gamma \\ \alpha\neq\emptyset}}
\ma^{l(\alpha)+l(\gamma)}_{\Gamma, \varphi}(\ulflex{\alpha}{\beta}\gamma)
\ma^{l(\beta)}_{\Gamma, \psi}(\lrflex{\alpha}{\beta}).
\end{align}
By similar computation, we get
\begin{align}\label{eq:ma(Lpsi(phi))}
\ma^m_{\Gamma, L_\psi(\varphi)}(\vecx_m)
= \sum_{\substack{\vecx_m=\alpha\beta\gamma \\ \gamma\neq\emptyset}}
\ma^{l(\alpha)+l(\gamma)}_{\Gamma, \varphi}(\alpha\urflex{\beta}{\gamma})
\ma^{l(\beta)}_{\Gamma, \psi}(\llflex{\beta}{\gamma}).
\end{align}
Hence, for $m\geq1$, we obtain
\begin{align*}
&\ma^m_{\Gamma, D_\psi(\varphi)}(\vecx_m) \\
&=\ma^m_{\Gamma, L_\psi(\varphi)}(\vecx_m) - \ma^m_{\Gamma, R_\psi(\varphi)}(\vecx_m) \\
&=\sum_{\substack{\vecx_m=\alpha\beta\gamma \\ \gamma\neq\emptyset}}
\ma^{l(\alpha)+l(\gamma)}_{\Gamma, \varphi}(\alpha\urflex{\beta}{\gamma})
\ma^{l(\beta)}_{\Gamma, \psi}(\llflex{\beta}{\gamma})
- \sum_{\substack{\vecx_m=\alpha\beta\gamma \\ \alpha\neq\emptyset}}
\ma^{l(\alpha)+l(\gamma)}_{\Gamma, \varphi}(\ulflex{\alpha}{\beta}\gamma)
\ma^{l(\beta)}_{\Gamma, \psi}(\lrflex{\alpha}{\beta}) \\
&=\sum_{\substack{\vecx_m=\alpha\beta\gamma \\ \beta,\gamma\neq\emptyset}}
\ma^{l(\alpha)+l(\gamma)}_{\Gamma, \varphi}(\alpha\urflex{\beta}{\gamma})
\ma^{l(\beta)}_{\Gamma, \psi}(\llflex{\beta}{\gamma})
- \sum_{\substack{\vecx_m=\alpha\beta\gamma \\ \alpha,\beta\neq\emptyset}}
\ma^{l(\alpha)+l(\gamma)}_{\Gamma, \varphi}(\ulflex{\alpha}{\beta}\gamma)
\ma^{l(\beta)}_{\Gamma, \psi}(\lrflex{\alpha}{\beta}) \\
&=(\arit(\ma_{\Gamma, \psi})(\ma_{\Gamma, \varphi}))(\vecx_m).
\end{align*}
The case of $m=0$ is obvious, so we get
$$
\ma_{\Gamma} \circ D_\psi(\varphi)
= \ma_{\Gamma, D_\psi(\varphi)}
= (\arit(\ma_{\Gamma, \psi})(\ma_{\Gamma, \varphi}))
= \arit(\ma_\Gamma(\psi)) \circ \ma_{\Gamma}(\varphi).
$$
Thus we complete  the proof.
\end{proof}

We recall the pre-Lie bracket $\preari$ of $\mathcal M(\mathcal F;\Gamma)$ introduced in \cite[(2.46)]{E-flex} by
$$
\preari(A,B):=\arit(B)(A) + A\times B,
$$
for $A,B\in\mathcal M(\mathcal F;\Gamma)$.
In \cite[Proposition A.6]{FK}, it is shown that it satisfies the axiom of pre-Lie bracket, that is,
\begin{align*}
	&\preari(A,\preari(B,C))-\preari(\preari(A,B),C) \\
	&=\preari(A,\preari(C,B))-\preari(\preari(A,C),B)
\end{align*}
for $A,B,C\in\ARI(\Gamma)$.
Actually, we have the following lemma.
\begin{lem}\label{lem:ma(s)=preari(ma,ma)}
For $\varphi,\psi\in\widehat{U\frak f_{\Gamma}}^\dag$, we have
$$
\ma_\Gamma(s_\varphi(\psi))=\preari(\ma_\Gamma(\psi),\ma_\Gamma(\varphi))
$$
\end{lem}

\begin{proof}
By Proposition \ref{prop: isom ma} and Proposition \ref{prop:ma_D=arit'ma)(ma)}, we get the above equation.
\end{proof}

By Proposition \ref{prop: isom ma}, we know the map $\ma_\Gamma$ induces the Lie algebra isomorphisms $(\widehat{U\frak f_{\Gamma}}^\dag)_0 \rightarrow \ARI(\mathcal F_\ser;\Gamma)$ and the group isomorphism $(\widehat{U\frak f_{\Gamma}}^\dag)_1 \rightarrow \GARI(\mathcal F_\ser;\Gamma)$.


\begin{thm}\label{thm:ma(<,>)=ari(ma,ma)}
We have the following commutative diagram:
\begin{equation}\label{eq:CD ma(<,>)=ari(ma,ma)}
\xymatrix{
(\widehat{U\frak f_{\Gamma}}^\dag)_0 \times (\widehat{U\frak f_{\Gamma}}^\dag)_0 \ar@{->}[rr]^{\simeq}_{\ma_{\Gamma}\times\ma_{\Gamma}\quad }\ar@{->}[d]_{-\{\ ,\  \}\ } && \ARI(\mathcal F_\ser;\Gamma) \times \ARI(\mathcal F_\ser;\Gamma) \ar@{->}[d]^{\ari} \\
(\widehat{U\frak f_{\Gamma}}^\dag)_0  \ar@{->}[rr]^{\simeq}_{\ma_{\Gamma}}&&\ARI(\mathcal F_\ser;\Gamma).
}
\end{equation}
That is, we have $\ma_{\Gamma} ({\{ \psi, \varphi \}}) = \ari(\ma_\Gamma(\varphi), \ma_\Gamma(\psi))$ for $\varphi,\psi\in (\widehat{U\frak f_{\Gamma}}^\dag)_0$.
\end{thm}

\begin{proof}
Note that we have
$$
\ari(A,B)=\preari(A,B) - \preari(B,A)
$$
for $A,B\in\mathcal{M}(\mathcal F;\Gamma)$.
By Lemma \ref{lem:ma(s)=preari(ma,ma)}, we get
\begin{align*}
\ma_{\Gamma} ({\{ \psi, \varphi \}})
&=\ma_\Gamma(s_\psi(\varphi) - s_\varphi(\psi)) \\
&=\preari(\ma_\Gamma(\varphi),\ma_\Gamma(\psi)) - \preari(\ma_\Gamma(\psi),\ma_\Gamma(\varphi)) \\
&=\ari(\ma_\Gamma(\varphi), \ma_\Gamma(\psi)).
\end{align*}
Hence, we obtain the above commutative diagram.
\end{proof}

For $k\in\N_0$, we define the map $\preari_k: \mathcal{M}(\mathcal F;\Gamma) \rightarrow \mathcal{M}(\mathcal F;\Gamma)$ inductively by
\begin{equation*}
\preari_k(A):=
\left\{\begin{array}{ll}
	\unitmould & (k=0), \\
	\preari(\preari_{k-1}(A),A) & (k\geq1),
\end{array}\right.
\end{equation*}
for $A\in\mathcal{M}(\mathcal F;\Gamma)$.
The exponential map $\expari:\ARI(\Gamma)\rightarrow\GARI(\Gamma)$ is defined in \cite{E-flex} by
\begin{align*}
\expari(A)
&:=\sum_{k\geq0}\frac{1}{k!}\preari_k(A),
\end{align*}
for $A\in\ARI(\Gamma)$.

\begin{thm}\label{thm:ma(exp)=expari(ma)}
We have the following commutative diagram:
\begin{equation}\label{eq:CD ma(exp)=expari(ma)}
\xymatrix{
(\widehat{U\frak f_{\Gamma}}^\dag)_0 \ar@{->}[rr]^{\simeq}_{\ma_{\Gamma}}\ar@{->}[d]_{\exp^\circledast} && \ARI(\mathcal F_\ser;\Gamma) \ar@{->}[d]^{\expari} \\
(\widehat{U\frak f_{\Gamma}}^\dag)_1  \ar@{->}[rr]^{\simeq}_{\ma_{\Gamma}}&&\GARI(\mathcal F_\ser;\Gamma).
}
\end{equation}
That is, we have $\ma_{\Gamma} \circ \exp^\circledast = \expari \circ \ma_{\Gamma}$.
\end{thm}

\begin{proof}
By definition, we have
$$
\ma_\Gamma(s_\varphi^0(1))=\ma_\Gamma(1)=\unitmould=\preari_0(\ma_\Gamma(\varphi)).
$$
Let $k\in\N$.
By Lemma \ref{lem:ma(s)=preari(ma,ma)} for $\psi=s_\varphi^{k-1}(1)$, we have
$$
\ma_\Gamma(s_\varphi^{k}(1))
=\ma_\Gamma(s_\varphi(s_\varphi^{k-1}(1)))
=\preari(\ma_\Gamma(s_\varphi^{k-1}(1)),\ma_\Gamma(\varphi)).
$$
So by induction on $k$, we get
$$
\ma_\Gamma(s_\varphi^{k}(1))=\preari_k(\ma_\Gamma(\varphi)).
$$
Hence, by using this equation, we obtain
\begin{align*}
\ma_\Gamma \left( \exp^\circledast(\varphi) \right)
=\sum_{k\geq0}\frac{1}{k!} \ma_\Gamma \left(s_\varphi^k(1) \right)
=\sum_{k\geq0}\frac{1}{k!} \preari_k(\ma_\Gamma(\varphi))
=\expari(\ma_\Gamma(\varphi)).
\end{align*}
Thus we get the above commutative diagram.
\end{proof}

\section{Correspondence between $\circledast$ and $\gari$ under $\ma_\Gamma$}
\label{app:correspondence between circledast and gari}
In this appendix, we give an explicit proof that Racinet's product $\circledast$ corresponds to \'{E}calle's $\gari$-product under the map $\ma_\Gamma$
in Theorem \ref{thm:ma(circledast)=gari(ma,ma)}, which looks implicitly employed in the literature
without proofs.

In this appendix, we again assume $\Gamma$ is an abelian group.
For $\varphi\in (\widehat{U\frak f_{\Gamma}}^\dag)_1$
(defined in Appendix \ref{app:correspondence between exp and expari}),
we define the map $\kappa_\psi:\widehat{U\mathfrak{f}_{\Gamma}}^\dag \rightarrow \widehat{U\mathfrak{f}_{\Gamma}}^\dag$ by
$$
\kappa_\psi(\varphi)\left(f_0,f_\sigma  \middle|\, \sigma\in\Gamma \right)
:=\varphi \left(f_0,t_\sigma(\psi)f_\sigma t_\sigma(\psi^{-1}) \middle|\, \sigma\in\Gamma \right)
$$
for $\varphi\in\widehat{U\mathfrak{f}_{\Gamma}}^\dag$.
We define the product $\circledast$  in \eqref{eq:circled-ast-product}
given by
\begin{equation*}
\psi \circledast \varphi:=\kappa_\psi(\varphi) \psi
\end{equation*}
for $\varphi, \psi \in (\widehat{U\frak f_{\Gamma}}^\dag)_1$.
We note that it agrees with \eqref{eq:circled-ast-product}  when $\Gamma=\{e\}$.
We prepare two notations to prove Proposition \ref{prop:ma_kappa=garit(ma)(ma)}.

\begin{nota}\label{lem:variable transformation of sigma}
Let $p\in\N$ and $r_1,r_2,\dots,r_{2p-1},r_{2p}\in\N_0$.
Put $R_0:=0$ and
\begin{equation*}
\left\{\begin{array}{l}
R_{2i-1}:=r_1+\cdots+r_{2i-1} +i-1, \\
\hphantom{{}_{-1}}R_{2i}:=r_1+\cdots+r_{2i} \hphantom{{}_{-1}} +i,
\end{array}\right.
\end{equation*}
for $1\leq i\leq p$.
Define two sets $X,Y$ of indeterminates by
\begin{align*}
X&:=\bigcup_{i=1}^p \left\{ \sigma_{i},\sigma_{2i-1,j_{2i-1}},\sigma_{2i,j_{2i}} \ |\ 1\leq j_k\leq r_k,\ k=2i-1,2i \right\}, \\
Y&:=\{ \sigma_i \ |\ 1\leq i\leq R_{2p} \}.
\end{align*}
We have
$
\#X = p+\sum_{k=1}^{2p}r_k = \#Y,
$
and we consider the following
bijection $g:X\rightarrow Y$:
\begin{center}
\begin{tabular}{|c||c|c|c|c|c|c|c|}
\hline
X & $\sigma_{2i-1,1}$ & $\cdots$ & $\sigma_{2i-1,r_{2i-1}}$ & $\sigma_i$ & $\sigma_{2i,1}$ & $\cdots$ & $\sigma_{2i,r_{2i}}$ \\ \hline
Y & $\sigma_{R_{2i-2}+1}$ & $\cdots$ & $\sigma_{R_{2i-2}+r_{2i-1}}$ & $\sigma_{R_{2i-1}+1}$ & $\sigma_{R_{2i-1}+2}$ & $\cdots$ & $\sigma_{R_{2i-1}+r_{2i}+1}$ \\
 &  &  & $=\sigma_{R_{2i-1}}$ &  &  &  & $=\sigma_{R_{2i}}$ \\ \hline
\end{tabular}
\end{center}
Here  $i$ is in $\{1,\dots,p\}$.
\end{nota}

\begin{nota}\label{lem:D(x_m)=D'(x_m)}
We consider the tensor algebra $T\left(\Q\langle X_\Z^\bullet\rangle\right)$ of $\Q\langle X_\Z^\bullet\rangle$.
For $m\geq1$, we define the two subsets ${\mathcal S}(\vecx_m)$,
${\mathcal S}'(\vecx_m)$ of $T\left(\Q\langle X_\Z^\bullet\rangle\right)$
with $\vecx_m=\varia{x_1,\ \dots,\ x_m}{\sigma_1,\ \dots,\ \sigma_m}$
by
\begin{align*}
{\mathcal S}(\vecx_m)
&:= \left\{
 \alpha_1\otimes\beta_1\otimes\gamma_1\otimes \cdots \otimes \alpha_s\otimes\beta_s\otimes\gamma_s 
\ \middle|\
\begin{array}{l}
1 \leq s,\\
\alpha_i,\beta_i,\gamma_i\in X_\Z^\bullet\ (1\leq i\leq s), \\
\beta_i\neq\emptyset \ (1\leq i\leq s), \\
\gamma_j\alpha_{j+1}\neq\emptyset\ (1\leq j\leq s-1), \\
\vecx_m= \alpha_1\beta_1\gamma_1 \cdots  \alpha_s\beta_s\gamma_s
\end{array}
\right\}, \\
{\mathcal S}' (\vecx_m)
&:=\left\{
 \alpha_1\otimes\omega_{k_1}\otimes\gamma_1\otimes \cdots \otimes \alpha_p\otimes\omega_{k_p}\otimes\gamma_p 
\ \middle|\
\begin{array}{l}
1 \leq p,\\
\alpha_i,\gamma_i\in X_\Z^\bullet\ (1\leq i\leq p), \\
\{k_1,\dots,k_p\}\subset[m], \\
\vecx_m= \alpha_1\omega_{k_1}\gamma_1 \cdots  \alpha_p\omega_{k_p}\gamma_p
\end{array}
\right\},
\end{align*}
where we put $\omega_k:=\varia{x_k}{\sigma_k}$ for $k\in\N$ and
denote $\{i_1,\dots,i_p\}$ to be a subset of $[m]:=\{1,\dots,m\} (\subset\N)$ with $1\leq i_1<\cdots<i_p\leq m$.

For $m\geq1$, we consider  the  following map $g_m:{\mathcal S}(\vecx_m)\rightarrow {\mathcal S}'(\vecx_m)$ defined by
\begin{align*}
g_m(\omega)
&:=\alpha_1\otimes(\omega_{k_1}\otimes\emptyset\otimes\emptyset\otimes\omega_{k_1+1}\otimes\emptyset\otimes\cdots\otimes\emptyset\otimes\omega_{k_1+l_1})\otimes\gamma_1\otimes \\
&\qquad \cdots \otimes \alpha_s\otimes(\omega_{k_s}\otimes\emptyset\otimes\emptyset\otimes\omega_{k_s+1}\otimes\emptyset\otimes\cdots\otimes\emptyset\otimes\omega_{k_s+l_s})\otimes\gamma_s,
\end{align*}
for $\omega= \alpha_1\otimes\beta_1\otimes\gamma_1\otimes \cdots \otimes \alpha_s\otimes\beta_s\otimes\gamma_s\in {\mathcal S}(\vecx_m)$ with $\beta_j=(\omega_{k_j},\omega_{k_j+1},\dots,\omega_{k_j+l_j})$ ($l_j\geq0$).
It is easy to show that $g_m$ forms a bijection.
\end{nota}

\begin{eg}
For $m=1$, we have
$$
{\mathcal S}(\vecx_1)=\{ \emptyset\otimes\omega_1\otimes\emptyset \}
={\mathcal S}'(\vecx_1).
$$
Hence, $g_1$ is the identity map.
For $m=2$, we have
\begin{align*}
{\mathcal S}(\vecx_2)
&=\{ \emptyset\otimes(\omega_1,\omega_2)\otimes \emptyset,\ \omega_1\otimes\omega_2\otimes \emptyset,\ \emptyset\otimes\omega_1\otimes\omega_2  \}, \\
{\mathcal S}'(\vecx_2)
&=\{ \emptyset\otimes\omega_1\otimes\emptyset\otimes\emptyset\otimes\omega_2\otimes \emptyset,\ \omega_1\otimes\omega_2\otimes \emptyset,\ \emptyset\otimes\omega_1\otimes\omega_2  \}.
\end{align*}
Hence, we can take $g_2$ as
\begin{align*}
\left\{\begin{array}{rl}
g_2(\emptyset\otimes(\omega_1,\omega_2)\otimes \emptyset)
&=\emptyset\otimes\omega_1\otimes\emptyset\otimes\emptyset\otimes\omega_2\otimes \emptyset , \\
g_2(\omega_1\otimes\omega_2\otimes \emptyset)&=\omega_1\otimes\omega_2\otimes \emptyset , \\
g_2(\emptyset\otimes\omega_1\otimes\omega_2)&=\emptyset\otimes\omega_1\otimes\omega_2 .
\end{array}\right.
\end{align*}
\end{eg}

By using the above notations, we prove the following proposition.
\begin{prop}\label{prop:ma_kappa=garit(ma)(ma)}
For $\mathcal F={\mathcal F}_\ser$ and for $\psi\in (\widehat{U\frak f_{\Gamma}}^\dag)_1$,
we have the following commutative diagram:
\begin{equation}\label{eq:CD ma_kappa=garit(ma)(ma)}
\xymatrix{
\widehat{U\frak f_{\Gamma}}^\dag \ar@{->}[rr]^{\simeq}_{\ma_{\Gamma}}\ar@{->}[d]_{\kappa_\psi} && \mathcal{M}(\mathcal F;\Gamma) \ar@{->}[d]^{\garit(\ma_\Gamma(\psi))} \\
\widehat{U\frak f_{\Gamma}}^\dag  \ar@{->}[rr]^{\simeq}_{\ma_{\Gamma}}&&\mathcal{M}(\mathcal F;\Gamma).
}
\end{equation}
That is, we have
\[
\ma_{\Gamma} \circ \kappa_\psi = \garit(\ma_\Gamma(\psi)) \circ \ma_{\Gamma}.
\]
\end{prop}

\begin{proof}
For $\psi\in (\widehat{U\frak f_{\Gamma}}^\dag)_1$ and $\varphi\in\widehat{U\mathfrak{f}_{\Gamma}}^\dag$, we prove
\begin{equation}\label{eqn:ma_kappa=garit(ma)(ma)}
\ma_\Gamma (\kappa_\psi(\varphi))(\vecx_m)
=\garit(\ma_\Gamma(\psi)) (\ma_{\Gamma}(\varphi))(\vecx_m).
\end{equation}
for all $m\geq 0$.
We have
\begin{align}\label{eqn:kappa_psi(phi)}
&\kappa_\psi(\varphi) \\
&=\sum_{p\geq0}\sum_{\sigma_{i}\in\Gamma}\sum_{k_{i}\in\N_0}
\coeff{\varphi}{k_0,\dots,k_p}{\sigma_1,\dots,\sigma_p}
f_0^{k_0}\left[ t_{\sigma_1}(\psi)f_{\sigma_1} t_{\sigma_1}(\psi^{-1}) \right]f_0^{k_1} \cdots \left[ t_{\sigma_p}(\psi)f_{\sigma_p} t_{\sigma_p}(\psi^{-1}) \right]f_0^{k_p} \nonumber\\
&=\sum_{k_{0}\in\N_0}
\coeff{\varphi}{k_0}{}
f_0^{k_0} \nonumber\\
&\quad +\sum_{p\geq1}\sum_{\sigma_{i}\in\Gamma}\sum_{k_{i}\in\N_0}
\coeff{\varphi}{k_0,\dots,k_p}{\sigma_1,\dots,\sigma_p}
f_0^{k_0}\left[ t_{\sigma_1}(\psi)f_{\sigma_1} t_{\sigma_1}(\psi^{-1}) \right]f_0^{k_1} \cdots \left[ t_{\sigma_p}(\psi)f_{\sigma_p} t_{\sigma_p}(\psi^{-1}) \right]f_0^{k_p}. \nonumber
\end{align}
By Definition \ref{def:garit} and Definition \ref{def:ma}, we know
$$
\ma_\Gamma (\kappa_\psi(\varphi))(\emptyset)
=\ma_{\Gamma,\kappa_\psi(\varphi)}^0(\emptyset)
=\coeff{\varphi}{0}{}
=\bigl(\garit(\ma_\Gamma(\psi)) (\ma_{\Gamma}(\varphi))\bigr)(\emptyset),
$$
so the case for $m=0$ of \eqref{eqn:ma_kappa=garit(ma)(ma)} holds.

Next  we  prove the case for $m\geq1$.
Note that we have
\begin{align*}
&t_{\sigma_i}(\psi)f_{\sigma_i} t_{\sigma_i}(\psi^{-1}) \\
&=\sum_{r_{2i-1},r_{2i}\geq0}\sum_{\sigma_{2i-1,j},\sigma_{2i,j}\in\Gamma}\sum_{l_{2i-1,j},l_{2i,j}\in\N_0} \\
&\quad \cdot \coeff{\psi}{l_{2i-1,0},\ \dots,\ l_{2i-1,r_{2i-1}}}{\sigma_i^{-1}\sigma_{2i-1,1},\ \dots,\ \sigma_i^{-1}\sigma_{2i-1,r_{2i-1}}}
\coeff{\psi^{-1}}{l_{2i,0},\ \dots,\ l_{2i,r_{2i}}}{\sigma_i^{-1}\sigma_{2i,1},\ \dots,\ \sigma_i^{-1}\sigma_{2i,r_{2i}}} \\
&\hspace{3cm} \cdot
f_0^{l_{2i-1,0}} \left( \prod_{j=1}^{r_{2i-1}} f_{\sigma_{2i-1,j}}f_0^{l_{2i-1,j}} \right)f_{\sigma_i}
f_0^{l_{2i,0}} \left( \prod_{j=1}^{r_{2i}} f_{\sigma_{2i,j}}f_0^{l_{2i,j}} \right),
\end{align*}
for $i\in\N$.
So by  \eqref{eqn:kappa_psi(phi)}, we have
\begin{align*}
&\kappa_\psi(\varphi)
-\sum_{k_{0}\in\N_0}
\coeff{\varphi}{k_0}{}
f_0^{k_0} \\
&=\sum_{\substack{r_{2i-1},r_{2i}\geq0 \\ p\geq1}}\sum_{\substack{\sigma_{2i-1,j},\sigma_{2i,j}\in\Gamma \\ \sigma_i\in\Gamma}}\sum_{\substack{l_{2i-1,j},l_{2i,j}\in\N_0 \\ k_i\in\N_0}}
\coeff{\varphi}{k_0,\dots,k_p}{\sigma_1,\dots,\sigma_p} \\
&\quad \cdot \prod_{i=1}^p\left\{
\coeff{\psi}{l_{2i-1,0},\ \dots,\ l_{2i-1,r_{2i-1}}}{\sigma_i^{-1}\sigma_{2i-1,1},\ \dots,\ \sigma_i^{-1}\sigma_{2i-1,r_{2i-1}}}
\coeff{\psi^{-1}}{l_{2i,0},\ \dots,\ l_{2i,r_{2i}}}{\sigma_i^{-1}\sigma_{2i,1},\ \dots,\ \sigma_i^{-1}\sigma_{2i,r_{2i}}}
\right\} \\
&\qquad \cdot f_0^{k_0} \prod_{i=1}^p\left\{
f_0^{l_{2i-1,0}} \left( \prod_{j=1}^{r_{2i-1}} f_{\sigma_{2i-1,j}}f_0^{l_{2i-1,j}} \right)f_{\sigma_i}
f_0^{l_{2i,0}} \left( \prod_{j=1}^{r_{2i}} f_{\sigma_{2i,j}}f_0^{l_{2i,j}} \right) f_0^{k_i}
\right\} \\
\end{align*}
\begin{align*}
&=\sum_{m\geq1}\sum_{\substack{p+ \sum_i(r_{2i-1}+r_{2i})=m \\ r_{2i-1},r_{2i}\geq0 \\ p\geq1}}\sum_{\substack{\sigma_{2i-1,j},\sigma_{2i,j}\in\Gamma \\ \sigma_i\in\Gamma}}\sum_{\substack{l_{2i-1,j},l_{2i,j}\in\N_0 \\ k_i\in\N_0}}
\coeff{\varphi}{k_0,\dots,k_s}{\sigma_1,\dots,\sigma_s} \\
&\quad \cdot \prod_{i=1}^p\left\{
\coeff{\psi}{l_{2i-1,0},\ \dots,\ l_{2i-1,r_{2i-1}}}{\sigma_i^{-1}\sigma_{2i-1,1},\ \dots,\ \sigma_i^{-1}\sigma_{2i-1,r_{2i-1}}}
\coeff{\psi^{-1}}{l_{2i,0},\ \dots,\ l_{2i,r_{2i}}}{\sigma_i^{-1}\sigma_{2i,1},\ \dots,\ \sigma_i^{-1}\sigma_{2i,r_{2i}}}
\right\} \\
&\qquad \cdot f_0^{k_0} \prod_{i=1}^p\left\{
f_0^{l_{2i-1,0}} \left( \prod_{j=1}^{r_{2i-1}} f_{\sigma_{2i-1,j}}f_0^{l_{2i-1,j}} \right)f_{\sigma_i}
f_0^{l_{2i,0}} \left( \prod_{j=1}^{r_{2i}} f_{\sigma_{2i,j}}f_0^{l_{2i,j}} \right) f_0^{k_i}
\right\}.
\end{align*}
Therefore, for $m\geq1$ and for $\sigma_{i},\sigma_{2i-1,j_{2i-1}},\sigma_{2i,j_{2i}}\in\Gamma$ ($1\leq i\leq p$, $1\leq j_k\leq r_k$, $k=2i-1,2i$), we have
\begin{align*}
&\vimo^m_{\Gamma, \kappa_\psi(\varphi)}\varia{z_0,\dots,z_m}{(\sigma_{2i-1,1},\cdots,\sigma_{2i-1,r_{2i-1}},\sigma_i,\sigma_{2i,1},\cdots,\sigma_{2i,r_{2i}})_{1\leq i\leq p}} \\
&=\sum_{\substack{p+ \sum_i(r_{2i-1}+r_{2i})=m \\ r_{2i-1},r_{2i}\geq0 \\ p\geq1}}\sum_{\substack{l_{2i-1,j},l_{2i,j}\in\N_0 \\ k_i\in\N_0}}
\coeff{\varphi}{k_0,\dots,k_p}{\sigma_1,\dots,\sigma_p} \\
&\quad \cdot \prod_{i=1}^p\left\{
\coeff{\psi}{l_{2i-1,0},\ \dots,\ l_{2i-1,r_{2i-1}}}{\sigma_i^{-1}\sigma_{2i-1,1},\ \dots,\ \sigma_i^{-1}\sigma_{2i-1,r_{2i-1}}}
\coeff{\psi^{-1}}{l_{2i,0},\ \dots,\ l_{2i,r_{2i}}}{\sigma_i^{-1}\sigma_{2i,1},\ \dots,\ \sigma_i^{-1}\sigma_{2i,r_{2i}}}
\right\} \\
&\qquad \cdot z_0^{k_0} \prod_{i=1}^p\left\{
z_{R_{2i-2}}^{l_{2i-1,0}} \left( \prod_{j=1}^{r_{2i-1}} z_{R_{2i-2}+j}^{l_{2i-1,j}} \right)z_{R_{2i-1}+1}^{l_{2i,0}} \left( \prod_{j=1}^{r_{2i}} z_{R_{2i-1}+1+j}^{l_{2i,j}} \right) z_{R_{2i}}^{k_i}
\right\}
\intertext{where see \eqref{eq:vimo} for $\vimo$. By rearranging each term, we calculate}
&=\sum_{\substack{p+ \sum_i(r_{2i-1}+r_{2i})=m \\ r_{2i-1},r_{2i}\geq0 \\ p\geq1}}
\left\{
\sum_{k_i\in\N_0}
\coeff{\varphi}{k_0,\dots,k_p}{\sigma_1,\dots,\sigma_p} z_0^{k_0} z_{R_{2}}^{k_1}\cdots z_{R_{2p}}^{k_p}
\right\} \\
&\quad \cdot \prod_{i=1}^p
\left\{
\sum_{l_{2i-1,j}\in\N_0}
\coeff{\psi}{l_{2i-1,0},\ \dots,\ l_{2i-1,r_{2i-1}}}{\sigma_i^{-1}\sigma_{2i-1,1},\ \dots,\ \sigma_i^{-1}\sigma_{2i-1,r_{2i-1}}}
 \prod_{j=0}^{r_{2i-1}} z_{R_{2i-2}+j}^{l_{2i-1,j}}
\right\} \\
&\qquad \cdot \prod_{i=1}^p
\left\{
\sum_{l_{2i,j}\in\N_0}
\coeff{\psi^{-1}}{l_{2i,0},\ \dots,\ l_{2i,r_{2i}}}{\sigma_i^{-1}\sigma_{2i,1},\ \dots,\ \sigma_i^{-1}\sigma_{2i,r_{2i}}}
 \prod_{j=0}^{r_{2i}} z_{R_{2i-1}+1+j}^{l_{2i,j}}
\right\}.
\intertext{By definition of $\vimo_\Gamma$, we get}
&=\sum_{\substack{p+ \sum_i(r_{2i-1}+r_{2i})=m \\ r_{2i-1},r_{2i}\geq0 \\ p\geq1}}
\vimo^p_{\Gamma, \varphi}\varia{z_0, z_{R_{2}},\dots, z_{R_{2p}}}{\sigma_1,\dots,\sigma_p} \\
&\quad \cdot \prod_{i=1}^p
\left\{
 \vimo^{r_{2i-1}}_{\Gamma, \psi}\varia{z_{R_{2i-2}}, z_{R_{2i-2}+1},\dots, z_{R_{2i-2}+r_{2i-1}}}{\sigma_i^{-1}\sigma_{2i-1,1},\ \dots,\ \sigma_i^{-1}\sigma_{2i-1,r_{2i-1}}}
\right\} \\
&\qquad \cdot \prod_{i=1}^p
\left\{
 \vimo^{r_{2i}}_{\Gamma, \psi^{-1}}\varia{z_{R_{2i-1}+1}, z_{R_{2i-1}+2},\dots, z_{R_{2i-1}+1+r_{2i}}}{\sigma_i^{-1}\sigma_{2i,1},\ \dots,\ \sigma_i^{-1}\sigma_{2i,r_{2i}}}
\right\}.
\intertext{By using Lemma \ref{lem:UfGamma}.(4), we obtain}
&=\sum_{\substack{p+ \sum_i(r_{2i-1}+r_{2i})=m \\ r_{2i-1},r_{2i}\geq0 \\ p\geq1}}
\vimo^p_{\Gamma, \varphi}\varia{z_0, z_{R_{2}},\dots, z_{R_{2p}}}{\sigma_1,\dots,\sigma_p} \\
&\quad \cdot \prod_{i=1}^p
\left\{
 \vimo^{r_{2i-1}}_{\Gamma, \psi}\varia{0, z_{R_{2i-2}+1}-z_{R_{2i-2}},\dots, z_{R_{2i-2}+r_{2i-1}}-z_{R_{2i-2}}}{\sigma_i^{-1}\sigma_{2i-1,1},\ \dots,\ \sigma_i^{-1}\sigma_{2i-1,r_{2i-1}}}
\right\} \\
&\qquad \cdot \prod_{i=1}^p
\left\{
 \vimo^{r_{2i}}_{\Gamma, \psi^{-1}}\varia{0, z_{R_{2i-1}+2}-z_{R_{2i-1}+1},\dots, z_{R_{2i-1}+1+r_{2i}}-z_{R_{2i-1}+1}}{\sigma_i^{-1}\sigma_{2i,1},\ \dots,\ \sigma_i^{-1}\sigma_{2i,r_{2i}}}
\right\}.
\end{align*}

Change of variables under the bijection $g$ in Notation \ref{lem:variable transformation of sigma} and  substitution of $z_0=0$ and $z_i=x_1+\cdots+x_i$ ($1\leq i\leq m$) yield
\footnote{We note that $m=p+ \sum_{i=1}^p(r_{2i-1}+r_{2i})=R_{2p}$.}
\begin{align*}
&\ma^m_{\Gamma}(\kappa_\psi(\varphi))(\vecx_m)
=\ma^m_{\Gamma, \kappa_\psi(\varphi)}\varia{x_1,\dots,x_m}{\sigma_1,\dots,\sigma_m} \\
&=\sum_{\substack{p+ \sum_i(r_{2i-1}+r_{2i})=m \\ r_{2i-1},r_{2i}\geq0 \\ p\geq1}}
\ma^p_{\Gamma, \varphi}\varia{x_{R_0+1}+\cdots+x_{R_2},\ \dots,\ x_{R_{2p-2}+1}+\cdots+x_{R_{2p}}}{\sigma_{R_1+1},\ \dots,\ \sigma_{R_{2p-1}+1}} \\
&\quad \cdot \prod_{i=1}^p
\left\{
 \ma^{r_{2i-1}}_{\Gamma, \psi}\varia{x_{R_{2i-2}+1},\ \dots,\ x_{R_{2i-1}}}{\sigma_{R_{2i-1}+1}^{-1}\sigma_{R_{2i-2}+1},\ \dots,\ \sigma_{R_{2i-1}+1}^{-1}\sigma_{R_{2i-1}}}
\right\} \\
&\qquad \cdot \prod_{i=1}^p
\left\{
 \ma^{r_{2i}}_{\Gamma, \psi^{-1}}\varia{x_{R_{2i-1}+2},\ \dots,\ x_{R_{2i}}}{\sigma_{R_{2i-1}+1}^{-1}\sigma_{R_{2i-1}+2},\ \dots,\ \sigma_{R_{2i-1}+1}^{-1}\sigma_{R_{2i}}}
\right\}.
\intertext{By putting $\omega_i=\binom{x_i}{\sigma_i}$ and by using flexions in Definition \ref{def:flexion}, we get}
&=\sum_{\substack{p+ \sum_i(r_{2i-1}+r_{2i})=m \\ r_{2i-1},r_{2i}\geq0 \\ p\geq1}} \\
&\ma^p_{\Gamma, \varphi}\Bigl( \urflex{(\omega_{R_0+1},\dots, \omega_{R_1})}{\ulflex{\omega_{R_1+1}}{(\omega_{R_1+2},\dots, \omega_{R_2})}},
\dots, \urflex{(\omega_{R_{2p-2}+1},\dots, \omega_{R_{2p-1}})}{\ulflex{\omega_{R_{2p-1}+1}}{(\omega_{R_{2p-1}+2},\dots, \omega_{R_{2p}})}} \Bigr) \\
&\quad \cdot \prod_{i=1}^p
\left\{
 \ma^{r_{2i-1}}_{\Gamma, \psi} \left( \llflex{(\omega_{R_{2i-2}+1},\dots, \omega_{R_{2i-1}})}{\omega_{R_{2i-1}+1}} \right)
 \ma^{r_{2i}}_{\Gamma, \psi^{-1}} \left( \lrflex{\omega_{R_{2i-1}+1}}{(\omega_{R_{2i-1}+2},\dots, \omega_{R_{2i}})} \right)
\right\}.
\intertext{By putting $\alpha_i=(\omega_{R_{2i-2}+1},\dots, \omega_{R_{2i-1}})$, $\gamma_i=(\omega_{R_{2i-1}+2},\dots, \omega_{R_{2i}})$ and $k_i=R_{2i-1}+1$ ($1\leq i\leq p$), we have}
&=\sum_{\alpha_1\otimes\omega_{k_1}\otimes\gamma_1\otimes \cdots \otimes \alpha_p\otimes\omega_{k_p}\otimes\gamma_p \in {\mathcal S}'(\vecx_m)}
\ma^p_{\Gamma, \varphi}\left( \urflex{\alpha_1}{\ulflex{\omega_{k_1}}{\gamma_1}}, \dots, \urflex{\alpha_p}{\ulflex{\omega_{k_p}}{\gamma_p}} \right) \\
&\hspace{4cm} \cdot \prod_{i=1}^p
\left\{
 \ma^{l(\alpha_i)}_{\Gamma, \psi} \left( \llflex{\alpha_i}{\omega_{k_i}} \right)
 \ma^{l(\gamma_i)}_{\Gamma, \psi^{-1}} \left( \lrflex{\omega_{k_i}}{\gamma_i} \right)
\right\}.
\intertext{By changing  variables under the bijection $g_m$ in  Notation \ref{lem:D(x_m)=D'(x_m)}, we calculate}
&=\sum_{\alpha_1\otimes\beta_{1}\otimes\gamma_1\otimes \cdots \otimes \alpha_s\otimes\beta_{s}\otimes\gamma_s \in {\mathcal S}(\vecx_m)}
\ma^{l(\beta_1,\dots,\beta_s)}_{\Gamma, \varphi}\left( \urflex{\alpha_1}{\ulflex{\beta_{1}}{\gamma_1}} \cdots \urflex{\alpha_s}{\ulflex{\beta_{s}}{\gamma_s}} \right) \\
&\hspace{4cm} \cdot \prod_{i=1}^s
\left\{
 \ma^{l(\alpha_i)}_{\Gamma, \psi} \left( \llflex{\alpha_i}{\beta_{i}} \right)
 \ma^{l(\gamma_i)}_{\Gamma, \psi^{-1}} \left( \lrflex{\beta_{i}}{\gamma_i} \right)
\right\}.
\intertext{Note that $\ma_{\Gamma}(\psi^{-1})=\ma_{\Gamma}(\psi)^{\times-1}$ by
\eqref{eq: ma and inverse}. So, we obtain}
&=\sum_{s\geq1}\sum_{\substack{\vecx_m=\alpha_1\beta_{1}\gamma_1 \cdots\alpha_s\beta_{s}\gamma_s \\ \beta_i,\gamma_j\alpha_{j+1}\neq\emptyset \\ 1\leq i\leq s, 1\leq j\leq s-1}}
\ma^{l(\beta_1,\dots,\beta_s)}_{\Gamma}(\varphi)\left( \urflex{\alpha_1}{\ulflex{\beta_{1}}{\gamma_1}} \cdots \urflex{\alpha_s}{\ulflex{\beta_{s}}{\gamma_s}} \right) \\
&\hspace{4cm} \cdot \prod_{i=1}^s
\left\{
 \ma^{l(\alpha_i)}_{\Gamma}(\psi) \left( \llflex{\alpha_i}{\beta_{i}} \right)
 \ma^{l(\gamma_i)}_{\Gamma}(\psi)^{\times-1} \left( \lrflex{\beta_{i}}{\gamma_i} \right)
\right\}. 
\intertext{By Definition \ref{def:garit},}
&=\bigl(\garit(\ma_\Gamma(\psi)) (\ma_{\Gamma}(\varphi))\bigr)(\vecx_m).
\end{align*}
Hence, we get \eqref{eqn:ma_kappa=garit(ma)(ma)} for $m\geq0$, and we finish the proof.
\end{proof}

\begin{thm}\label{thm:ma(circledast)=gari(ma,ma)}
We have the following commutative diagram:
\begin{equation}\label{eq:CD ma(circledast)=gari(ma,ma)}
\xymatrix{
(\widehat{U\frak f_{\Gamma}}^\dag)_1 \times (\widehat{U\frak f_{\Gamma}}^\dag)_1 \ar@{->}[rr]^{\simeq\qquad}_{\ma_{\Gamma} \times \ma_{\Gamma}\qquad}\ar@{->}[d]_{\circledast\circ\mathrm{switch}} && \GARI(\mathcal F_\ser;\Gamma) \times \GARI(\mathcal F_\ser;\Gamma) \ar@{->}[d]^{\gari} \\
(\widehat{U\frak f_{\Gamma}}^\dag)_1  \ar@{->}[rr]^{\simeq}_{\ma_{\Gamma}}&&\GARI(\mathcal F_\ser;\Gamma).
}
\end{equation}
That is, we have $\ma_{\Gamma} ({\psi \circledast \varphi}) = \gari( \ma_\Gamma(\varphi), \ma_\Gamma(\psi))$ for $\varphi,\psi\in(\widehat{U\frak f_{\Gamma}}^\dag)_1$.
\end{thm}
\begin{proof}
By Proposition \ref{prop: isom ma} and Proposition \ref{prop:ma_kappa=garit(ma)(ma)}, we have
\begin{align*}
\ma_{\Gamma} ({\psi \circledast \varphi})
&=\ma_\Gamma (\kappa_\psi(\varphi) \psi) \\
&=\ma_\Gamma (\kappa_\psi(\varphi)) \times \ma_\Gamma (\psi) \\
&=\garit(\ma_\Gamma(\psi)) (\ma_{\Gamma}(\varphi)) \times \ma_\Gamma (\psi) \\
&=\gari( \ma_\Gamma(\varphi), \ma_\Gamma(\psi)).
\end{align*}
Hence we obtain the above commutative diagram.
\end{proof}


\section{On $\GARI(\mathcal{F})_{\protect\underline{\as\ast\is}}$}\label{sec: appendix GARIasis}
We will give a complete proof that \'{E}calle's set $\GARI(\mathcal{F})_{{\as\ast\is}}$ (resp. $\GARI(\mathcal{F})_{\underline{\as\ast\is}}$) (cf. Definition \ref{def:GARIas*is})
recovers
Racinet's double shuffle set $\DMR$ (resp. $\DMR_0$)
(cf. Definition \ref{def:DMR})
under the map $\ma$
in Theorem \ref{theorem: GARI and DMR},
which extends the results \cite{S-ARIGARI} in  the Lie algebra  setting.

Consider the set $Y:=\{Y_k\}_{k\in\N}$.
Let $\Q\langle\langle Y \rangle\rangle$ to be the non-commutative formal power series ring over $\Q$ generated by $Y$,
which is equipped with a structure of Hopf algebra with the  harmonic (stuffle) coproduct
given by
\begin{equation}\label{eq: harmonic coproduct}
\Delta_\ast (Y_n)=Y_n\otimes1 +1\otimes Y_n +\sum_{i=1}^{n-1} Y_i\otimes Y_{n-i}
\end{equation}
for $n\geq 1$.

\begin{defn}
Following \cite{S-ARIGARI},
we consider two $\Q$-linear maps
$$\mi:\widehat{U\frak f_2}^\dagger\rightarrow \overline{\mathcal M}(\mathcal F_\ser)
\
\footnote{
It looks that there is an error on the definition of $\mi$ in \S 3.2 of \cite{S-ARIGARI}. The map  $\iota_Y$ in (3.2.3) should be corrected by
$\iota_Y:y_{a_1}\cdots y_{a_r}\mapsto v_1^{a_r-1}\cdots v_r^{a_1-1}$.
}
$$
and
$$\overline{\mi}:\Q\langle\langle Y \rangle\rangle \rightarrow \overline{\mathcal M}(\mathcal F_\ser)$$
by
\begin{align*}
&\mi(\varphi)
:=(\mi_\varphi^r(x_1,\dots,x_r))_{r\geq0}, \\
&\mi_\varphi^r(x_1,\dots,x_r)
:=\vimo_\varphi^r\varia{0,x_r,\dots,x_1}{1,\dots,1}, \\
&\overline{\mi}(\Phi)
:=(\overline{\mi}_\Phi^r(x_1,\dots,x_r))_{r\geq0}, \\
&\overline{\mi}_\Phi^r(x_1,\dots,x_r)
:=\sum_{k_1,\dots,k_r\geq1}\langle \Phi | k_1,\dots,k_r \rangle x_1^{k_r-1}\cdots x_r^{k_1-1},
\end{align*}
for $\varphi=\sum_{r\geq0}\sum_{k_1,\dots,k_r\geq0}\langle \varphi | k_1,\dots,k_r \rangle f_0^{k_0}f_1f_0^{k_1}\cdots f_1f_0^{k_r}\in \widehat{U\frak f_2}^\dagger$ and\\
$\Phi=\sum_{r\geq0}\sum_{k_1,\dots,k_r\geq1}\langle \Phi | k_1,\dots,k_r \rangle Y_{k_1}\cdots Y_{k_r}\in \Q\langle\langle Y \rangle\rangle$.
\end{defn}

\begin{lem}\label{lem:algebra hom of mi}
The maps $\mi$ and $\overline{\mi}$ are {anti-}isomorphisms of algebras.
Particularly, for $\varphi,\psi\in\widehat{U\frak f_2}^\dagger$ and $\Phi,\Psi\in\Q\langle\langle Y \rangle\rangle$, we have
$$
\mi(\varphi\psi)=\mi(\psi)\times \mi(\varphi), \qquad
\overline{\mi}(\Phi\Psi)=\overline{\mi}(\Psi)\times \overline{\mi}(\Phi)
$$
\end{lem}

\begin{proof}
It is clear that the maps $\mi$ and $\overline{\mi}$ are bijections.
By direct calculation, we know that these $\mi$ are {anti-}algebra homomorphisms.
\end{proof}

We define the $\Q$-linear map
$\pi_Y:\widehat{U\frak f_2}^\dagger \rightarrow \Q\langle\langle Y \rangle\rangle$
by
\begin{align*}
\pi_Y(w)&:=\left\{\begin{array}{ll}
1 & (w=1), \\
{Y_{k_1}\cdots Y_{k_r}}&(w=f_1f_0^{k_1-1}\cdots f_1f_0^{k_r-1}), \\
0&(\mbox{otherwise}),
\end{array}\right.
\end{align*}
for $r,k_1,\dots,k_r\geq1$.

\begin{lem}\label{lem:swap ma=mi pi ret}
The following diagram commutes:
\begin{equation}\label{eq:CD swap ma=mi pi ret}
\xymatrix{
 \widehat{U\frak f_2}^\dag \ar@{->}[rr]^{\ma}\ar@{->}[d]_{\pi_Y}\ar[drr]^{\mi} && \mathcal M(\mathcal F_\ser)\ar@{->}[d]^{\swap} \\
\Q\langle\langle Y \rangle\rangle \ar@{->}[rr]_{\overline{\mi}}&&\overline{\mathcal M}(\mathcal F_\ser).
}
\end{equation}
\end{lem}
\begin{proof}
Consider $\varphi\in\widehat{U\frak f_2}^\dag$ with
$$
\varphi
=\sum_{r\geq0}\sum_{k_0,\dots,k_r\geq0}\langle \varphi | k_0,k_1,\dots,k_r \rangle f_0^{k_0}f_1f_0^{k_1}\cdots f_1f_0^{k_r}.
$$
By definition of $\vimo$, we have
$$
\vimo_\varphi^r\varia{z_0,z_1,\dots,z_r}{1,\dots,1}
=\sum_{k_0,\dots,k_r\geq0}\langle \varphi | k_0,k_1,\dots,k_r \rangle z_0^{k_0}z_1^{k_1}\cdots z_r^{k_r},
$$
for $r\geq0$.
So by definition of the maps $\ma$ and $\mi$, we have
\begin{align}
\ma_\varphi^r(x_1,\dots,x_r)
&=\sum_{k_1,\dots,k_r\geq0}\langle \varphi | 0,k_1,\dots,k_r \rangle x_1^{k_1}(x_1+x_2)^{k_2}\cdots (x_1+\cdots+x_r)^{k_r}, \\
\label{eqn:mi expression for varphi}
\mi_\varphi^r(x_1,\dots,x_r)
&=\sum_{k_1,\dots,k_r\geq0}\langle \varphi | 0,k_1,\dots,k_r \rangle x_r^{k_1}x_{r-1}^{k_2}\cdots x_1^{k_r}.
\end{align}
Therefore, we get
\begin{align}\label{eqn:ma expression for varphi}
\swap (\ma(\varphi))(x_1,\dots,x_r)
&=\sum_{k_1,\dots,k_r\geq0}\langle \varphi | 0,k_1,\dots,k_r \rangle x_r^{k_1}x_{r-1}^{k_2}\cdots x_1^{k_r},
\end{align}
for $r\geq0$.
On the other hand, we calculate
\begin{align*}
\pi_Y(\varphi)
&=\pi_Y\left(
\sum_{r\geq0}\sum_{k_0,\dots,k_r\geq0}\langle \varphi | k_0,k_1,\dots,k_r \rangle f_0^{k_0}f_1f_0^{k_1}\cdots f_1f_0^{k_r}
\right) \\
&=\sum_{r\geq0}\sum_{k_1,\dots,k_r\geq1}\langle \varphi | 0,k_1-1,\dots,k_r-1 \rangle Y_{k_1}\cdots Y_{k_r}.
\end{align*}
So we get
\begin{align}\label{eqn:barmi expression for varphi}
\overline{\mi}(\pi_Y(\varphi))(x_1,\dots,x_r)
=\sum_{k_1,\dots,k_r\geq1}\langle \varphi | 0,k_1-1,\dots,k_r-1 \rangle x_1^{k_r-1}\cdots x_{r-1}^{k_2-1}x_r^{k_1-1}.
\end{align}
Hence, by \eqref{eqn:mi expression for varphi}, \eqref{eqn:ma expression for varphi} and \eqref{eqn:barmi expression for varphi}, we obtain the claim.
\end{proof}

\begin{defn}
We define the $\Q$-linear map $\overline{\dimi}:\Q\langle\langle Y \rangle\rangle \widehat{\otimes} \Q\langle\langle Y \rangle\rangle \rightarrow \overline{\mathcal M}_2(\mathcal F_\ser)$ by
\begin{align*}
&\overline{\dimi}(\Phi):=(\dimi_\Phi^{p,q}(x_1,\dots,x_p;x_{p+1},\dots,x_{p+q}))_{p,q\geq0}, \\
&\overline{\dimi}_\Phi^{p,q}(x_1,\dots,x_p;x_{p+1},\dots,x_{p+q}) \\
& :=
\sum_{d'_i,d''_j\in\N}\left\langle \Phi \middle| (d'_1,\dots,d'_p);(d''_1,\dots,d''_q) \right\rangle
x_1^{d'_p-1}\cdots x_p^{d'_1-1}x_{p+1}^{d''_q-1}\cdots x_{p+q}^{d''_1-1}
\end{align*}
for $\Phi\in\Q\langle\langle Y \rangle\rangle \widehat{\otimes} \Q\langle\langle Y \rangle\rangle$ which is described as
$$
\Phi=\sum_{p,q\geq0}\sum_{d'_i,d''_j\in\N}\left\langle \Phi \middle| (d'_1,\dots,d'_p);(d''_1,\dots,d''_q) \right\rangle
Y_{d'_1}\cdots Y_{d'_p}\otimes Y_{d''_1}\cdots Y_{d''_q}.
$$
\end{defn}
Similarly to \eqref{eq: dima=ma otimes ma}, for $\Phi=h_1\otimes h_2$ with $h_1,h_2\in \Q\langle\langle Y \rangle\rangle$, we have
\begin{align}\label{eqn:dimibar=mibar otimes mibar}
\overline{\dimi}(\Phi)=\overline{\mi}(h_1)\otimes \overline{\mi}(h_2).
\end{align}

\begin{lem}\label{lem:dimi delta*=sh* mi}
We have the following commutative diagram:
\begin{equation}\label{eq:CD dimi delta*=sh* mi}
\xymatrix{
\Q\langle\langle Y \rangle\rangle \ar@{->}[rr]^{\overline{\mi}}\ar@{->}[d]_{\Delta_*} && \overline{\mathcal M}(\mathcal F_\ser)\ar@{->}[d]^{\shmap_*} \\
\Q\langle\langle Y \rangle\rangle^{\widehat{\otimes}2} \ar@{->}[rr]_{\overline{\dimi}}&&\overline{\mathcal M}_2(\mathcal F_\ser).
}
\end{equation}
Here, the coproduct $\Delta_*$ is given by
$$
\Delta_*(Y_n):=Y_n\otimes 1 + 1\otimes Y_n + \sum_{\substack{i+j=n \\ i,j\geq1}}Y_i\otimes Y_j \qquad (n\geq1).
$$
\end{lem}
\begin{proof}
We get this claim in the same way as Lemma \ref{lem: dima delta=sh ma}.
\end{proof}

\begin{defn}
For $\varphi \in \widehat{U\frak f_2}^\dag$, we define $\varphi_{\rm corr} \in \Q\langle\langle Y \rangle\rangle$ by
$$
\varphi_{\rm corr}
:=\exp\left(
\sum_{k=2}^\infty (-1)^{k}\frac{\langle \varphi | f_1f_0^{k-1} \rangle}{k} Y_1^k
\right).
$$
\end{defn}

\begin{lem}\label{lem:mi phicorr=Miniphi}
For $\varphi \in \widehat{U\frak f_2}^\dag$, we have
$$
\overline{\mi}(\varphi_{\rm corr})
=\Mini_\varphi,
$$
where $\Mini_\varphi$ is the mould in Definition \ref{defn:Mini}.
\end{lem}

\begin{proof}
Put $\varphi_{\rm corr}=\sum_{r\geq0}c_r Y_1^r$ ($c_r\in \Q$).
Then by definition of map $\overline{\mi}$, we have
$$
\overline{\mi}^r_{\varphi_{\rm corr}}(x_1,\dots,x_r)
=c_r x_1^0\cdots x_r^0
=c_r.
$$
So we get
\begin{align*}
\sum_{r\geq0}\overline{\mi}^r_{\varphi_{\rm corr}}(x_1,\dots,x_r)t^r
&=\exp\left(
\sum_{k=2}^\infty (-1)^k\frac{\langle \varphi | f_1f_0^{k-1} \rangle}{k} t^k
\right)
=\sum_{r\geq0}\Mono_{\varphi,r}t^r.
\end{align*}
Hence, we obtain $\overline{\mi}_{\varphi_{\rm corr}}^r(x_1,\dots,x_r)=\Mini_\varphi(x_1,\dots,x_r)$,
which proves the claim.
\end{proof}

Following Racinet \cite{R}, we consider  the element
\footnote{
For our technical reason, we reverse the order of the product in \cite{R}.
}
\begin{equation}\label{eq: varphi ast}
\varphi_*:=\pi_Y(\varphi)\cdot\varphi_{\rm corr}
\end{equation}
in $\Q\langle\langle Y \rangle\rangle$ for $\varphi \in \widehat{U\frak f_2}^\dag$.

\begin{prop}\label{prop: equivalent formulation of double shuffle relations}
Let $\varphi\in\widehat{U\frak f_2}^\dag$.
The following two are equivalent:
\begin{enumerate}
\item $\Delta_*(\varphi_*)=\varphi_*\otimes \varphi_*$,
\item $\shmap_*\Bigl(\Mini_\varphi\times \swap(\ma(\varphi))\Bigr)
=\Bigl(\Mini_\varphi\times \swap(\ma(\varphi))\Bigr)\otimes \Bigl(\Mini_\varphi\times \swap(\ma(\varphi))\Bigr)$.
\end{enumerate}
\end{prop}

\begin{proof}
By Lemmas \ref{lem:algebra hom of mi}, \ref{lem:swap ma=mi pi ret} and \ref{lem:mi phicorr=Miniphi}, we have
\begin{equation}\label{eq:mi=mini x swap ma}
\overline{\mi}(\varphi_*)
=\overline{\mi}(\varphi_{\rm corr})\times \overline{\mi}(\pi_Y(\varphi))
=\Mini_\varphi\times \swap(\ma(\varphi)).
\end{equation}
Assume (1).
Then by the above equation, we calculate
\begin{align*}
\shmap_*\Bigl(\Mini_\varphi\times \swap(\ma(\varphi))\Bigr)
&=\shmap_*(\overline{\mi}(\varphi_*)).
\intertext{By Lemma \ref{lem:dimi delta*=sh* mi}, we have}
&=\overline{\dimi} \circ \Delta_*(\varphi_*).
\intertext{By the assumption (1) and the equation \eqref{eqn:dimibar=mibar otimes mibar}, we get}
&=(\overline{\mi} \otimes \overline{\mi})(\varphi_*\otimes \varphi_*) \\
&=\overline{\mi}(\varphi_*)\otimes \overline{\mi}(\varphi_*).
\intertext{By \eqref{eq:mi=mini x swap ma}, we obtain}
&=\Bigl(\Mini_\varphi\times \swap(\ma(\varphi))\Bigr)\otimes \Bigl(\Mini_\varphi\times \swap(\ma(\varphi))\Bigr).
\end{align*}
Hence, we obtain the claim (2) from (1).

Assume (2).
As by \eqref{eq:mi=mini x swap ma}, we have
\begin{align*}
&\overline{\dimi} \circ \Delta_*(\varphi_*)
=\shmap_*\Bigl(\Mini_\varphi\times \swap(\ma(\varphi))\Bigr), \\
&(\overline{\mi} \otimes \overline{\mi})(\varphi_*\otimes \varphi_*)
=\Bigl(\Mini_\varphi\times \swap(\ma(\varphi))\Bigr)\otimes \Bigl(\Mini_\varphi\times \swap(\ma(\varphi))\Bigr).
\end{align*}
By using these equations, the assumption (2) and \eqref{eqn:dimibar=mibar otimes mibar}, we get
\begin{align*}
\overline{\dimi} \circ \Delta_*(\varphi_*)
&=\shmap_*\Bigl(\Mini_\varphi\times \swap(\ma(\varphi))\Bigr) \\
&=\Bigl(\Mini_\varphi\times \swap(\ma(\varphi))\Bigr)\otimes \Bigl(\Mini_\varphi\times \swap(\ma(\varphi))\Bigr) \\
&=(\overline{\mi} \otimes \overline{\mi})(\varphi_*\otimes \varphi_*) \\
&=\overline{\dimi}(\varphi_*\otimes \varphi_*).
\end{align*}
since $\overline{\dimi}$ is bijective, we obtain the claim (1).
\end{proof}

\begin{defn}[\cite{R}]\label{def:DMR}
The double shuffle set $\DMR$ is defined to be the set of series
$\varphi=\varphi(f_0,f_1)$ in
$\widehat{U\mathfrak f_2}$
which satisfies
$\varphi(0,0)=1$,
$\Delta(\varphi)=\varphi\otimes\varphi$,
$\Delta_\ast(\varphi_\ast)=\varphi_\ast\otimes \varphi_\ast$
and
$\langle\varphi|f_0\rangle=\langle\varphi|f_1\rangle=0$.
We define
$\DMR_0$ to be its subset defined by
$\langle\varphi|f_1f_0\rangle=0$.
\end{defn}

In \cite{R} it is shown that $\DMR_0$ forms a group under $\circledast$.

\begin{thm}[cf. {\cite[Theorem 3.4.4]{S-ARIGARI}}]
\label{theorem: GARI and DMR}
We have
\begin{align*}
&\ma^{-1}(\GARI(\mathcal{F}_\ser)_{{\as\ast\is}})=\DMR \cdot \exp{\Q f_{1}}, \\
&\ma^{-1}(\GARI(\mathcal{F}_\ser)_{\underline{\as\ast\is}})=\DMR_0 \cdot \exp{\Q f_{1}}.
\end{align*}
\end{thm}

For $\GARI(\mathcal{F})_{{\as\ast\is}}$ and $\GARI(\mathcal{F})_{\underline{\as\ast\is}}$,
see  Definition \ref{def:GARIas*is}.

\begin{proof}
It is enough to show the first equality.
Let $M\in \GARI(\mathcal{F}_\ser)_{{\as\ast\is}}$ such that
$C\times \swap(M)$ is symmetril with a constant mould $C$.
Put $\varphi'=\ma^{-1}(M)$ and $\alpha'(f_1):=\mi^{-1}(C)\in\Q\langle\langle f_1\rangle\rangle$.
Put $\kappa=\langle \varphi'\bigm| f_1 \rangle$,
 $\varphi=\varphi'\cdot \exp\{-\kappa f_1\}$
 and
 $\alpha(f_1)=\exp(\kappa f_1)\alpha'(f_1)$.
Then $\pi_Y(\varphi)\cdot\alpha(Y_1)
=\pi_Y(\varphi')\cdot \alpha'(Y_1)
\in\Q\langle\langle Y\rangle\rangle$ is group-like with respect to $\Delta_\ast$
by Lemmas \ref{lem:swap ma=mi pi ret} and \ref{lem:dimi delta*=sh* mi}.
Since we have $\langle\varphi|f_1\rangle=0$,
$\alpha(Y_1)$ is  given by $\varphi_{\rm corr}(Y_1)$
by the arguments in \cite{IKZ}.
So we have $\varphi\in\DMR$, whence $\varphi'\in \DMR \cdot \exp{\Q f_{1}}$.

Let $\varphi'=\varphi\cdot\exp(\kappa f_1)$ with $\varphi\in\DMR$ and $\kappa\in\Q$.
Put $C=\mi(\exp(-\kappa f_1))$.
Note that $\mi(\varphi)=C\times \mi(\varphi')$.
Then, by Lemmas \ref{lem:algebra hom of mi}, \ref{lem:swap ma=mi pi ret} and \ref{lem:dimi delta*=sh* mi}, we have
\begin{align*}
\shmap_*(C\times \mi(\varphi'))
&=\shmap_*(\mi(\varphi'\cdot\exp(-\kappa f_1)))
=\shmap_*\circ\mi(\varphi)
=\overline{\dimi}\circ\Delta_*(\varphi).
\intertext{By \eqref{eqn:dimibar=mibar otimes mibar}, we get}
&=(\overline{\mi}\otimes \overline{\mi})(\varphi \otimes \varphi)
=(C\times \mi(\varphi')) \otimes (C\times \mi(\varphi')).
\end{align*}
So we obtain $C\times \mi(\varphi')=C\times \swap(\ma(\varphi'))$ is symmetril, whence $\ma(\varphi')\in\GARI(\mathcal{F}_\ser)_{{\as\ast\is}}$.
\end{proof}

\end{document}